\documentclass[11pt, reqno]{amsart}

\usepackage{textcomp} 
\usepackage{amsmath, amsfonts,amsthm,amssymb,amsxtra}
\usepackage{amscd}
\usepackage{enumitem}

\usepackage{xcolor}
\usepackage{graphicx}
\usepackage{verbatim}

\makeatletter
\def\@secnumfont{\bfseries}
\def\section{\@startsection{section}{1}%
  \z@{.7\linespacing\@plus\linespacing}{.5\linespacing}%
  {\normalfont\Large\bfseries}}
\def\subsection{\@startsection{subsection}{2}%
  \z@{.5\linespacing\@plus.7\linespacing}{-.5em}%
  {\normalfont\large\bfseries}}
  \def\subsubsection{\@startsection{subsubsection}{3}%
  \z@{.5\linespacing\@plus.7\linespacing}{-.5em}%
  {\normalfont\bfseries}}
\makeatother

\newtheorem{theorem}{Theorem}[section]
\newtheorem{proposition}[theorem]{Proposition}
\newtheorem{lemma}[theorem]{Lemma}
\newtheorem{corollary}[theorem]{Corollary}

\theoremstyle{definition}

\newtheorem{definition}[theorem]{Definition}
\newtheorem{examples}[theorem]{Example}

\theoremstyle{remark}

\newtheorem{remark}[theorem]{Remark}

\numberwithin{equation}{section}

\renewcommand{\epsilon}{\varepsilon}

\renewcommand{\phi}{\varphi}

\begin{document}

{
\title{Contact Manifolds with Flexible Fillings}
\author{Oleg Lazarev}
} 

\maketitle

\begin{abstract}
We prove that all flexible Weinstein fillings of a given contact manifold with vanishing first Chern class have isomorphic integral cohomology; in certain cases, we prove that all flexible fillings are symplectomorphic. As an application, we show that in dimension at least 5 any almost contact class that has an almost Weinstein filling
has infinitely many different contact structures.
Similar methods are used to construct the first known infinite family of almost symplectomorphic Weinstein domains whose contact boundaries are not contactomorphic. 
We also prove relative analogs of our results,
which we apply to Lagrangians in cotangent bundles. 
\end{abstract}
\tableofcontents

\section{Introduction}

This paper is mainly concerned with the two related problems of distinguishing contact structures and classifying symplectic fillings of a given contact structure. We focus on distinguishing contact structures that have the same  bundle-theoretic data, or \textit{almost contact class}, and hence cannot be distinguished via algebraic topology. Non-contactomorphic contact structures in the same almost contact class are called \textit{exotic}. 
This problem has a long history. 
Bennequin \cite{Benn} constructed the first example of an exotic contact structure in the standard almost contact class on $\mathbb{R}^3$. In higher dimensions, Eliashberg \cite{Eexotic} constructed an exotic contact structure in the standard almost contact class on $S^{4k+1}$ for $k \ge 1$. 
This was generalized by Ustilovsky \cite{U} who proved that \textit{every} almost contact class on $S^{4k+1}$ has infinitely many different contact structures; also see 
\cite{Fauck, Gutt}. van Koert \cite{vK}  showed that many simply-connected 5-manifolds have infinitely many contact structures in the same almost contact class; see  \cite{DG, MMgrowth, Ue} for  more examples.

One way to study contact manifolds is through their symplectic fillings, i.e symplectic manifolds whose contact boundary is contactomorphic to the given contact manifold. 
In this paper, we consider only fillings that are 
\textit{Liouville domains}, which are certain exact symplectic manifolds, or \textit{Weinstein domains}, which are Liouville domains that admit a compatible Morse function; see Section \ref{ssec: domains}. 
One phenomenon is that if a contact manifold has a symplectic filling that satisfies an h-principle, i.e. is governed by algebraic topology, then the contact manifold itself is very rigid and remembers the topology of its fillings; see the discussion at the end of Section \ref{sec: mainresults}. The first result of this type is the Eliashberg-Floer-McDuff theorem \cite{McD}: all Liouville fillings of the standard contact structure $(S^{2n-1}, \xi_{std}) = \partial (B^{2n}, \lambda_{std})$ are diffeomorphic to $B^{2n}$. Using this result,  Eliashberg proved that the contact structures on $S^{4k+1}$ in  \cite{Eexotic} are exotic. 
Similarly, quite a lot is known about contact manifolds with \textit{subcritical} Weinstein fillings, which satisfy an h-principle; see Section \ref{ssec: domains}. 
For example, M.-L.Yau \cite{MLYau} showed that linearized contact homology, which a priori depends on the filling,  is a contact invariant for such contact manifolds. This can be used to prove that all fillings with vanishing symplectic homology of such contact manifolds  have the same rational homology. Later Barth, Geiges, and Zehmisch \cite{Geiges_subcritical} showed  that in fact all Liouville fillings of simply-connected subcritically-filled contact manifolds are diffeomorphic; since $(S^{2n-1}, \xi_{std})$ has a subcritical filling $(B^{2n}, \lambda_{std})$, this result generalizes the Eliashberg-McDuff-Floer theorem.

However not much was known beyond the subcritical case. There are contact manifolds with many different fillings 
\cite{Oba_infinite_fillings, OS} so the results in the subcritical setting do not hold in general. The main purpose of this paper is to extend those results to \textit{flexible} fillings, which  generalize subcritical fillings and also satisfy an h-principle; see Section \ref{subsection: loose}. Flexible Weinstein domains are only defined for $n \ge 3$ and so many of our results below require $n \ge 3$. 

\subsection{Contact manifolds with flexible fillings}\label{sec: mainresults}

In this paper, we will denote almost contact structures by $(Y, J)$ 
and almost symplectic structures by $(W,J)$; see Section \ref{ssec: domains}. We also assume the first Chern classes $c_1(Y, J) \in H^2(Y; \mathbb{Z}), c_1(W, J) \in H^2(W; \mathbb{Z})$ always vanish, even when this is not stated explicitly. Many of our results concern only Weinstein fillings and so it will often suffice to assume $c_1(Y, J) = 0$; indeed by Proposition \ref{prop: c1equivalence}, if $W^{2n}$ is a Weinstein filling of $(Y^{2n-1}, \xi)$ with $n \ge 3$, then $c_1(Y, J)$ vanishes if and only if $c_1(W, J)$ does. 

We begin by stating our main geometric result and some of its applications; its proof will be briefly discussed at the end of this section.
\begin{theorem}\label{thm: main}
All flexible Weinstein fillings of a contact manifold $(Y, \xi)$ with $c_1(Y, \xi) = 0$ have isomorphic integral cohomology; that is, if $W_1, W_2$ are flexible fillings of $(Y, \xi)$, then $H^*(W_1; \mathbb{Z}) \cong H^*(W_2;\mathbb{Z})$ as abelian groups.
\end{theorem}
\begin{remark}\label{rem: fieldcoefficients,homology}
Theorem \ref{thm: main} also holds with any field coefficients and with homology instead of cohomology.
\end{remark}
The cohomology long exact sequence of the pair $(W, Y)$ and the fact that $H^k(W, Y) \cong H_{2n-k}(W; \mathbb{Z}) = 0$ for $k \le  n-1$ for Weinstein domains show that 
$H^k(W; \mathbb{Z}) \cong H^k(Y; \mathbb{Z})$ for $k \le n-2$.
In particular, $H^k(W_1; \mathbb{Z}) \cong H^k(W_2; \mathbb{Z})$ for $k \le n-2$ because of purely topological reasons.  Therefore, Theorem \ref{thm: main} is interesting only for $k = n-1$ and $k=n$. On the other hand, its stronger variant Corollary \ref{thm: mainstronger}  applies to a larger class of Liouville domains and so is interesting for all $k$.
The long exact sequence also shows that for Weinstein $W$,
the rank of the intersection form 
$H_n(W;\mathbb{Z})\otimes H_n(W; \mathbb{Z})\rightarrow \mathbb{Z}$ equals
$$
\dim H^n(W; \mathbb{Q}) + \dim H^{n-1}(W; \mathbb{Q}) - \dim H^{n}(Y; \mathbb{Q}).
$$
So Theorem \ref{thm: main} shows that the contact boundary also remembers the intersection form rank of its flexible fillings. 

There are analogs of Theorem \ref{thm: main} in smooth topology, where fillings play a similar role. For example, Kervaire and Milnor \cite{KMhomotopy} showed that the diffeomorphism type of an exotic $(4k-1)$-dimensional sphere bounding a parallelizable manifold is determined by the signature of this manifold (modulo some other integer). In particular, the smooth structure remembers the signature  of its  parallelizable fillings (again, modulo some fixed integer). Similarly, Theorem \ref{thm: main} shows that the contact structure remembers the cohomology of its flexible fillings.  Work in progress with Y. Eliashberg and S. Ganatra \cite{EGL2} shows that the contact structure also remembers the signature of its fillings as an integer (not just as a residue modulo some other integer). 

We also note that there are contact manifolds that do not have any flexible fillings, e.g. overtwisted manifolds. In this case, Theorem \ref{thm: main} is vacuous. In fact, there are contact structures that have Weinstein fillings but no flexible Weinstein fillings, e.g.  $ST^*M^n = \partial T^*M^n$ if $M$ is simply-connected and spin or the Ustilovsky contact structures on $S^{4k+1}$; see Remark \ref{rem: ustilovskynotflex} and Remark \ref{rem: boundedinfinite}. 

As seen in the Eliashberg-Floer-McDuff theorem \cite{McD}, certain contact manifolds remember the \textit{diffeomorphism type} of their fillings. This is because in some special cases cohomology is enough to determine diffeomorphism type. Sometimes cohomology can even determine almost symplectomorphism type. 
 Since almost symplectomorphic flexible Weinstein domains are genuinely symplectomorphic (see Section \ref{subsection: loose}), we can also use Theorem \ref{thm: main} to prove results about symplectomorphism type.

One application is to simply-connected 5-manifolds. 
In \cite{S5}, Smale showed that any simply-connected 5-manifold $Y^5$  with $w_2(Y) = 0$ admits a smooth 2-connected filling $W^6$ that has a handle decomposition with only 0 and 3-handles. Smale also showed that such fillings of $Y^5$ are unique up to boundary connected sum with $S^3\times S^3 \backslash D^6$; we will suppose that every filling of $Y^5$ is of the form 
 $W_n := W \natural (S^3 \times S^3 \backslash D^6) \natural \cdots \natural (S^3 \times S^3 \backslash D^6)$, the boundary connected sum of $W$ with n copies of $S^3 \times S^3 \backslash D^6$, for some $n \in \mathbb{Z}_{\ge 0}$. Note that $W_n$ is determined by its cohomology. We also note that $W_n$ admits a unique almost complex structure by obstruction theory since $\pi_3(O(6)/U(3)) \cong \pi_3(O/U) \cong \pi_4(O)$ vanishes. Hence  by the uniqueness h-principle in Section \ref{subsection: loose}, $W_n$ admits a unique flexible Weinstein structure. 
More precisely, all flexible Weinstein structures on $W_n$  have symplectomorphic completions; the completion of $W$ is the open symplectic manifold $\widehat{W} := W \cup Y \times [1, \infty)$ with a conical symplectic form on the cylindrical end $Y \times [1, \infty)$, see Section \ref{ssec: domains} for details. 
So in this case, the cohomology of the flexible Weinstein domain determines the symplectomorphism type of its completion.  

The boundary of $W_n$ has a contact structure $(Y, \xi_n)$, which by definition has $W_n$ as a filling; this was first proven by Geiges in \cite{G5}. Since $c_1(W_n) \in H^{2}(W_n) = 0$, these structures have $c_1(Y, \xi_n) = i^*c_1(W_n) = 0$. Geiges also  showed that there is a unique \textit{almost} contact structure $(Y, J)$ with $c_1(Y, J) = 0$ so these are all in the same almost contact class. Using Theorem \ref{thm: main}, we can prove that these contact structures have unique flexible fillings. 
\begin{corollary}\label{cor: example}
Any contact structure on $Y^5$ with a  2-connected  flexible filling is of the form $(Y, \xi_n)$ for some $n\in \mathbb{Z}_{\ge 0}$ and all 2-connected flexible  fillings of $(Y^5, \xi_n)$ have completions that are symplectomorphic to $\widehat{W}_n$. 
In particular, $(Y^5, \xi_n)$ are in the same almost contact class but are pairwise non-contactomorphic. 
\end{corollary}
We will provide more examples of exotic contact structures in Section \ref{sub: exotic}.

Theorem \ref{thm: main} can also be used to show that the contact boundary remembers the symplectomorphism type of its fillings when the original flexible filling is smoothly displaceable in its completion, i.e. there is an embedding 
$\phi: W \hookrightarrow \widehat{W}$ smoothly isotopic to the standard inclusion $i: W\hookrightarrow \widehat{W}$ such that  
$\phi(W) \cap i(W) = \emptyset$.
To prove this result, we follow the approach of Barth, Geiges, and Zehmisch \cite{Geiges_subcritical} of finding an h-cobordism between the standard filling and the new filling.  

\begin{corollary}\label{cor: displaceable_flexible_diffeomorphism}
Suppose that $(Y, \xi)$ has a flexible filling $W_{flex}$ such that $W_{flex}$ is smoothly displaceable in its completion and $H^{n-1}(W_{flex}; \mathbb{Z}), H_{n-1}(W_{flex}; \mathbb{Z}), \pi_1(W_{flex}), c_1(W_{flex})$ all vanish. Then all flexible Weinstein fillings of $(Y, \xi)$ have symplectomorphic completions. 
\end{corollary}

The condition that $W_{flex}$ is smoothly displaceable in its completion restricts the topology of $W_{flex}$; for example, the intersection form of $W_{flex}$ must vanish. 

\begin{examples}
If $M^n, n \ge 3,$ is a closed manifold with $\chi(M) = 0$ and $\pi_1(M) = 0$, then all flexible Weinstein fillings of $\partial T^*M_{flex}$ have symplectomorphic completions. In particular, for n odd, all flexible fillings of $\partial T^*S^n_{flex}$ have symplectomorphic completions. Similarly, all flexible fillings of 
$\partial (T^*S^n_{flex} \natural \cdots \natural T^*S^n_{flex})$ have symplectomorphic completions.  
\end{examples}

We will now sketch the proof of Theorem \ref{thm: main}. The main technical tool in this paper is a certain Floer-theoretic invariant of Liouville domains called positive symplectic homology $SH^+$; see Section \ref{ssec: postive_sym_hom}. If we fix a contact manifold $(Y, \xi)$, then a priori positive symplectic homology of a Liouville filling $W$ of $(Y, \xi)$ depends on $W$. In Section \ref{ssec: independencelin}, we define a certain collection of \textit{asymptotically dynamically convex} contact structures, which generalize the dynamically convex contact structures from \cite{Abreu_Mac, CieliebakOancea, HWZ_convex}, and show that $SH^+$ is independent of the filling for these structures and is therefore a contact invariant; see Definition \ref{def: semigood} and Proposition \ref{prop: nice_sh_independent}. Our main result is that asymptotically dynamically convex contact structures are preserved under flexible contact surgery, i.e. surgery along a \textit{loose} Legendrian \cite{Murphy11}. This extends a similar result of M.-L.Yau \cite{MLYau} for subcritical surgery; see Theorems  \ref{thm: MLYau}, \ref{thm: MLYauindex2}, and \ref{thm: semi-surgery}. In particular, contact manifolds with flexible Weinstein fillings are asymptotically dynamically convex and so have $SH^+$ independent of the filling; see Corollary \ref{cor: flexsemigood}. It is a standard fact that $SH^+$ of a flexible domain always equals the singular cohomology of the domain (see Proposition \ref{prop: shcomputation}) and so all flexible fillings have the same cohomology. 
We note that positive symplectic homology is the (non-equivariant) Floer-theoretic version of linearized contact homology and so this generalizes the result of M.-L.Yau mentioned earlier.

As we explain in Definition \ref{def: semigood}, asymptotically dynamically convex contact structures are essentially characterized by the fact that their Reeb orbits have positive degree. Hence to prove our main result that these structures are preserved under flexible surgery, we use 
Proposition \ref{prop: BEE} (see \cite{BEE12}) which gives a correspondence between the new Reeb orbits after flexible surgery and words of Reeb chords of the loose Legendrian attaching sphere. So it is enough to prove that Reeb chords of loose Legendrians have positive degree (possibly after Legendrian isotopy), which we do in our main geometric result Lemma \ref{lem: maingeo}. This lemma depends crucially on Murphy's h-principle for loose Legendrians \cite{Murphy11}.
\\

In certain special cases, there is an alternative approach to Theorem \ref{thm: main} which may provide some more geometric intuition for why it is true. In these cases, Theorem \ref{thm: main} for flexible domains can be reduced to the subcritical case. The idea is that flexible domains obey an embedding h-principle \cite{EM} so any flexible domain satisfying the appropriate topological conditions can embedded into a subcritical domain. Hence if its contact boundary had many different fillings, then  the contact boundary of the subcritical domain would also have many different fillings as well by a cut-and-paste argument, violating known results in the subcritical case  \cite{Geiges_subcritical, MLYau, OV}. So if there is any contact rigidity (in the sense that some contact manifold remembers the topology of its filling), then the contact boundary of a flexible domain should also inherit this rigidity. 

We will give a proof of this result here since it is independent of all other results in this paper. In the following, an almost symplectic embedding $\phi: (X, \omega_X) \rightarrow (Y, \omega_Y)$  is an embedding such that $\phi^* \omega_Y$ can be deformed through non-degenerate two-forms to $\omega_X$; this is a purely topological notion.

\begin{proposition}\label{thm: reduce_subcritical}
Suppose $(Y^{2n-1}, \xi)$ has a flexible Weinstein filling $W_{flex}$  that has an almost symplectic embedding into a subcritical Weinstein domain. Then all Liouville fillings of $(Y, \xi)$  have the same integral homology as $W_{flex}$.  
\end{proposition}
\begin{remark}
The condition that $W_{flex}$ has a smooth embedding into a subcritical domain implies that its intersection form vanishes. 
\end{remark}
\begin{proof}
By the embedding h-principle \cite{EM}, the topological condition that $W_{flex}$   has an almost symplectic embedding into a subcritical Weinstein domain $W_{sub}$ implies the geometric condition that $W_{flex}$  admits a symplectic (in fact Liouville) embedding $\phi: W_{flex} \hookrightarrow  W_{sub}$. Suppose $X$ is another filling of $(Y, \xi)$ and consider the cut-and-pasted domain 
$X' :=  X \cup (W_{sub} \backslash \phi(W_{flex}))$, which is another filling of $(Y_{sub}, \xi_{sub}) := \partial W_{sub}$. By Theorem 1.2 of \cite{Geiges_subcritical}, $H_k(X') \cong H_k(W_{sub})$ for all $k$. Using the Mayer-Vietoris long exact sequence of the pair 
$X' = X \cup (W_{sub}\backslash \phi(W_{flex}))$ and the fact that $H_k(X') \cong H_k(W_{sub}) = 0$ for $k \ge n$, we get $H_k(Y) \cong H_k(X) \oplus H_k(W_{sub} \backslash \phi(W_{flex}))$ for $k \ge n$.
Similarly, the Mayer-Vietoris sequence of the pair $W_{sub} = W_{flex}\cup (W_{sub}\backslash \phi(W_{flex}))$ shows that $H_k(Y) \cong H_k(W_{flex}) \oplus H_k(W_{sub} \backslash \phi(W_{flex}))$.
So 
$H_k(X) \oplus H_k(W_{sub} \backslash \phi(W_{flex})) \cong H_k(W_{flex}) \oplus H_k(W_{sub} \backslash \phi(W_{flex}))$ for $k \ge n$. 
By the classification of finitely-generated abelian groups, we can cancel 
$ H_k(W_{sub} \backslash \phi(W_{flex}))$ from both sides and get 
$ H_k(X) \cong H_k(W_{flex}) = 0$ for $k \ge n$. 
For the case $k \le n-1$, we follow the same approach but use the \textit{relative} Mayer-Vietoris  sequences of the pairs 
$(W_{flex}, \emptyset)  \cup 
(W_{sub}\backslash \phi(W_{flex}), \partial W_{sub}) = (W_{sub}, \partial W_{sub})$  and 
$(X, \emptyset)  \cup 
(W_{sub}\backslash \phi(W_{flex}), \partial W_{sub}) = (X', \partial W_{sub})$ and the facts 
that $H_k(W_{sub}, \partial W_{sub})$ and 
$H_k(X', \partial W_{sub})$ both vanish for 
$k \le n$; the latter fact comes from Theorem 1.2 of \cite{Geiges_subcritical} 
which shows that $H_{k}(\partial W_{sub}) \rightarrow H_{k}(X')$ is an isomorphism for $k \le n-1$ 
and $H_n(X') = 0$.
\end{proof}
\begin{examples}
Any closed manifold $M^n$ embeds into $\mathbb{C}^n$ and let $N(M^n) \subset \mathbb{C}^n$ be its tubular neighborhood.
For example, $N(S^n)$ is diffeomorphic to  $S^n \times D^n$ for any embedding $S^n \subset \mathbb{C}^n$. 
Then $N(M^n)$  admits a Morse function with critical points of index at most $n$ and an almost complex structure obtained by restricting the almost complex structure on $\mathbb{C}^n$. By the existence h-principle, $N(M^n)$ has a flexible Weinstein structure; see Section \ref{subsection: loose}. This flexible structure automatically has an almost symplectic embedding into the subcritical domain $\mathbb{C}^n$ and so Theorem \ref{thm: reduce_subcritical} shows that all Liouville fillings of $\partial N(M^n)$ (with the induced contact structure) have the same homology. 
Also, the flexible domain $W_{flex} \times T^*S^1$ has an almost symplectic embedding into the subcritical domain $W_{flex} \times \mathbb{C}$ and so all Liouville fillings of  
$\partial(W_{flex} \times T^*S^1)$ have the same homology. 
\end{examples}

The proof of Proposition \ref{thm: reduce_subcritical} works more generally than just for flexible Weinstein domains; any contact manifold with a Liouville filling that symplectically embeds into a subcritical domain remembers the homology of its fillings. 
But not any Weinstein domain that has an almost symplectic embedding into a subcritical domain has a genuine symplectic embedding there (and not every contact manifold has a unique filling \cite{Oba_infinite_fillings}). So the flexible condition cannot be removed. 
We also note that Proposition \ref{thm: reduce_subcritical}, when it applies, is stronger than Theorem \ref{thm: main}. For example,  it restricts all Liouville fillings of $(Y, \xi)$, not just flexible Weinstein fillings. Furthermore, the condition that $c_1(Y, \xi)$ vanishes is not necessary in Proposition \ref{thm: reduce_subcritical}. 
Also, it is possible to combine 
Proposition \ref{thm: reduce_subcritical} and
Corollary \ref{cor: displaceable_flexible_diffeomorphism}  to get new examples of contact manifolds where all Liouville fillings are diffeomorphic. 

Despite the strength of Proposition \ref{thm: reduce_subcritical}, the $SH^+$ approach used to prove Theorem \ref{thm: main} has several important advantages, particularly for applications. First of all, not all flexible Weinstein domains admit smooth embeddings into subcritical domains; the non-vanishing of the intersection form obstructs this. Hence the condition in Proposition \ref{thm: reduce_subcritical} is quite special. The constructions of exotic contact structures in the next section (see Theorem \ref{thm: inf_contact_flex}) all require using flexible domains with non-degenerate intersection forms and so these results cannot be proven using Proposition \ref{thm: reduce_subcritical}. Perhaps more importantly, the $SH^+$ approach 
even works for certain non-flexible domains.
For example, in Theorem \ref{thm: boundedinfinite} we will study non-flexible domains that are constructed by attaching flexible handles to asymptotically dynamically convex contact structures. These domains do not symplectically embed into subcritical domains even if they admit almost symplectic embeddings. 

\subsection{Exotic contact structures}\label{sub: exotic}

Theorem \ref{thm: main} and its variants can also be used to construct  exotic contact structures.
As we noted in the Introduction, there are infinitely many contact structures in the standard almost contact class $(S^{4k+1}, J_{std}), k \ge 1,$ and on certain $5$-manifolds \cite{Eexotic, U, vK}. Also, McLean \cite{MMgrowth} showed that if $(Y^{2n-1}, \xi)$, $n \ge 8$, has a certain special Weinstein filling (an \textit{algebraic} Stein filling with subcritical handles attached), then $Y$ has at least two Weinstein-fillable contact structures in the same almost contact class. Here we extend these results by constructing infinitely many contact structures in a general setting. 
 
In the following theorem, an \textit{almost Weinstein manifold} is an almost symplectic manifold with a compatible Morse function; see Section \ref{ssec: domains} for details. Note that this is a purely algebraic topological notion. Also, let $X\natural Y$ denote the boundary connected sum of $X$ and $Y$. 
\begin{theorem}\label{thm: inf_contact_flex}
Suppose $(Y^{2n-1}, J), n \ge 3$ with $c_1(Y, J) = 0$, has an almost Weinstein filling $W$. Then for any almost Weinstein filling $M^{2n}$ of $(S^{2n-1}, J_{std})$, there is a contact structure $(Y,\xi_M)$ such that the following hold 
\begin{itemize}
\item $(Y, \xi_M)$ is almost contactomorphic to $(Y, J)$
\item if $H^*(M; \mathbb{Z}) \not\cong H^*(N; \mathbb{Z})$, then $\xi_M, \xi_N$ are non-contactomorphic
\item $(Y, \xi_M)$ has a flexible Weinstein filling $W_M$ almost symplectomorphic to $W\natural M$.
\end{itemize}
There are infinitely many such $M^{2n}$ with different integral cohomology and hence there are infinitely many contact structures in 
$(Y, J)$ with flexible fillings. 
\end{theorem}
\begin{remark}\label{rem: Cieliebak_Eliash_McLean}
 Using work of McLean \cite{MM}, Cieliebak and Eliashberg \cite{CE12} proved  that any almost Weinstein domain $(W^{2n},J), n \ge 3,$ admits infinitely many non-symplectomorphic Weinstein structures $W_k^{2n}$. The contact boundaries $(Y, \xi_k) = \partial W_k^{2n}$ are in the same almost contact class but it is unknown whether they are contactomorphic. 
We show that at most $\dim H^1(Y; \mathbb{Z}/2) + 1$ of the Cieliebak-Eliashberg-McLean contact structures can have flexible fillings and so at most finitely many of them 
can coincide with the structures $(Y, \xi_M)$ from Theorem \ref{thm: inf_contact_flex}. 
\end{remark}

The second part of Theorem \ref{thm: inf_contact_flex} fails for $n = 2$ (the first part does not make sense for $n = 2$ since flexible Weinstein domains are defined only for $n \ge 3$). For example, $S^3$, the 3-torus $T^3,$ and the lens space $L(p,1)$ have finitely many Weinstein-fillable contact structures; see Chapter 16 of \cite{CE12}. Also note that the condition $H^*(M; \mathbb{Z}) \not \cong H^*(N; \mathbb{Z})$ is interesting only in degree $n$ since $M, N$ are Weinstein fillings of $S^{2n-1}$ and hence have zero cohomology except in degree $0$ and $n$; see Remark \ref{rem: ustilovksy} for a generalization where all degrees matter. 

We also note that the question of whether an almost contact manifold admits an almost Weinstein filling, and hence a flexible one, is purely topological. This was first explored by Geiges  \cite{Geiges} and then by Bowden, Crowley, and Stipsicz \cite{BCS, BCSII}, who gave a bordism-theoretic characterization of such almost contact manifolds. Combining this with Theorem \ref{thm: inf_contact_flex}, we get a topological criterion for almost contact manifolds to admit infinitely many different contact structures in the same formal class. The following application follows immediately from Theorem \ref{thm: inf_contact_flex} and \cite{BCS}, \cite{Geiges}.
\begin{corollary}\label{cor: 57example}
If $(Y,J)$ with $c_1(Y, J) = 0$  is a simply-connected 5-manifold or a
simply-connected 7-manifold with torsion-free $\pi_2(Y)$, then $(Y, J)$ has infinitely many different contact structures with 
flexible Weinstein fillings. 
\end{corollary}
Also see Corollary \ref{cor: example}. The 5-dimensional case of Corollary \ref{cor: 57example} (without the statement about flexible fillings) was proven by van Koert in  Corollary 11.14 of \cite{vK}. As we noted before, there is at most one almost contact structure $(Y^5, J)$ with $c_1(Y^5, J) = 0$ on a simply-connected 5-manifold $Y^5$  (see Lemma 7 of \cite{G5}) and hence there is at most one structure to which Corollary \ref{cor: 57example} applies. 
On the other hand, Bowden, Crowley, and Stipsicz \cite{BCSII} showed that many manifolds admit almost contact structures that have no almost Weinstein fillings, in which case Theorem \ref{thm: inf_contact_flex} does not apply; for example, 
$S^{8k-1}, k \ge 2, $  admits an almost contact structure which has no almost Weinstein filling. 

Another application of Theorem \ref{thm: inf_contact_flex} is to Question 6.12 of \cite{KvK}, where Kwon and van Koert asked whether there are infinitely many different contact structures in $(S^{2n-1}, J_{std})$ for $n \ge 3$. Ustilovsky  \cite{U} proved this for $S^{4k+1}, k \ge 1,$ and Uebele \cite{Ue} proved this for $S^7, S^{11}, S^{15}$. Since $(B^{4k+4}, \lambda_{std})$ is a Weinstein filling of $(S^{4k+3}, \xi_{std})$, Theorem \ref{thm: inf_contact_flex} provides an affirmative answer to this question, even within the possibly smaller class of Weinstein-fillable contact structures. 
\begin{corollary}\label{cor: s4k+3}
For $n\ge 3$, $(S^{2n-1}, J_{std})$ has infinitely many contact structures
with flexible Weinstein fillings. For $n$ odd, this is true for \textit{all} almost contact classes; furthermore, these contact structures are not contactomorphic to the Ustilovsky structures.
\end{corollary}
\begin{remark}\label{rem: ustilovskynotflex}
The proof of this corollary also shows that although the Ustilovsky contact structures have Weinstein fillings, they do not have any flexible Weinstein fillings. 
\end{remark}

For an almost Weinstein filling $M$ of $(S^{2n-1}, J_{std})$, the contact structure $(Y, \xi_M)$  in Theorem \ref{thm: inf_contact_flex} has a   flexible Weinstein filling $W_M$ almost symplectomorphic to $W\natural M$.  Furthermore, $(Y, \xi_M)$ with different $H^*(M; \mathbb{Z})$ are non-contactomorphic. Therefore different $(Y, \xi_M)$ have fillings $W_M$ that are not even homotopy equivalent. Indeed, if the $W_M$ were almost symplectomorphic for different $M$, then they would be Weinstein homotopy equivalent by the h-principle for flexible Weinstein domains and so their contact boundaries would be contactomorphic; see Section \ref{ssec: domains}.

In light of this, one can ask whether there are exotic contact structures that bound almost symplectomorphic Weinstein domains, i.e. exotic Weinstein domains whose contact boundaries are also exotic. 
One such example is provided by $\partial \mathbb{C}_{k}^n$, where $\mathbb{C}_k^n$ are the exotic Weinstein structures on
$\mathbb{C}^n$ constructed by McLean \cite{MM} that are non-symplectomorphic for different $k \in \mathbb{N}$;  see Remark \ref{rem: Cieliebak_Eliash_McLean}. 
Although $\partial \mathbb{C}_k^n$ and $(S^{2n-1}, \xi_{std})$ are almost contactomorphic and admit almost symplectomorphic Weinstein fillings, Oancea and Viterbo \cite{OV} show that they are not contactomorphic. 
However, as noted in Remark \ref{rem: Cieliebak_Eliash_McLean}, it is unknown whether $\partial \mathbb{C}_k^n$ are non-contactomorphic for different $k$. More generally, it was unknown whether there exist infinitely many different contact structures bounding almost symplectomorphic Weinstein domains. Here we show that such examples do exist.

For the following theorem, let $\Omega^n$ denote the set of smooth manifolds $M^n$ such that 
\begin{itemize}
\item  $M^n$ is closed, simply-connected, and stably parallelizable
\item if $n$ is even, $\chi(M) = 2$ and if $n$ is odd, $\chi_{1/2}(M) \equiv 1 \mod 2$
\end{itemize}
Here $\chi_{1/2}(M)$ denotes the semi-characteristic $\sum_{i = 0}^{(n-1)/2}\dim H_i(M) \mod 2$.
Note these conditions are purely algebraic topological. We will also use $\Lambda M$ to denote the free loop space of $M$. 
\begin{theorem}\label{thm: boundedinfinite}
Suppose $(Y^{2n-1}, J), n \ge 4$ with $c_1(Y,J) = 0$, has an almost Weinstein filling $W$. Then for any $M^n \in \Omega^n$, there is a contact structure $(Y, \xi_M)$  such that 
\begin{itemize}
\item $\xi_M$ is in $(Y,J)$ for $n$ odd and in some fixed  $(Y, J')$ for $n$ even (depending only on $(Y, J)$ and not $M$)
\item if 
$
\dim H_k(\Lambda M; \mathbb{Q}) - \dim H_k(\Lambda N; \mathbb{Q}) > 2 \dim H^{n-k}(Y; \mathbb{Q}) + 
2\dim H^{n-k+1}(Y; \mathbb{Q})
$
for some $k$, then $\xi_M, \xi_N$ are non-contactomorphic 
\item $(Y, \xi_M)$ has a Weinstein filling $W_M$ almost symplectomorphic to $W\natural P$, where $P$ is a certain plumbing of $T^*S^n$ depending only the dimension $n$ 
(and not $M$).
\end{itemize}
In particular, for $n \ge 4$ there are infinitely many different contact structures in $(Y, J)$ or $(Y, J')$ that admit almost symplectomorphic Weinstein fillings. 
\end{theorem}
\begin{remark}\label{rem: boundedinfinite}\
\begin{enumerate}[leftmargin=*]
\item The fillings $W_M$ are \textit{not} flexible since otherwise they would be Weinstein homotopic by the h-principle for flexible domains and so $(Y, \xi_M)$ would be contactomorphic.  
In fact, $(Y, \xi_M)$ does not have \textit{any} flexible fillings. As a result, the contact structures in Theorem \ref{thm: boundedinfinite} are different from the structures in Theorem \ref{thm: inf_contact_flex}.
\item In the last part of Theorem \ref{thm: boundedinfinite} involving the infinite collection of contact structures, we do not claim that \textit{all}  fillings of a given contact manifold in this collection are almost symplectomorphic but rather than every contact manifold in this collection admits some Weinstein domain filling and these particular Weinstein domains are all almost symplectomorphic.
\end{enumerate}
\end{remark}

\subsection{Legendrians with flexible Lagrangian fillings}
\label{ssec: intro_legendrians}
 We also prove relative analogs of our results for Legendrians. 
In particular, we define the class of \textit{asymptotically dynamically convex Legendrians}, show that positive wrapped Floer homology $WH^+$ (the relative analog of $SH^+$) is an invariant for these Legendrians, and prove that Legendrians with \textit{flexible Lagrangian} fillings are asymptotically dynamically convex; see Definition \ref{def: leggood}, Proposition \ref{prop: Legindependence}, and Corollary \ref{cor: semiflexlaggood}.

We now give some geometric applications of these results. In this paper, we will assume all Legendrians $\Lambda$ and Lagrangians $L$ are connected, oriented, spin, and $c_1(Y, \Lambda) \in H^2(Y, \Lambda; \mathbb{Z})$ and $c_1(W, L) \in H^2(W, L; \mathbb{Z})$ vanish. 

As for contact manifolds, it is known that certain Legendrians remember the topology of their exact Lagrangian fillings. For example, a relative analog of the Eliashberg-McDuff-Floer theorem states that any exact Lagrangian filling $L^n \subset B^{2n}$ of the standard Legendrian unknot in $(S^{2n-1}, \xi_{std})$ is diffeomorphic to $B^n$; see 
\cite{CRGGConcordance}, \cite{Abouzaid}. Recently, Eliashberg, Ganatra, and the author introduced the class of \textit{flexible Lagrangians} \cite{EGL}.  
These Lagrangians are the relative analog of flexible Weinstein domains. 
In fact, they are always contained in flexible Weinstein domains and so are defined only for $n \ge 3$. 
Problem 4.13 of \cite{EGL} asked whether the Legendrian boundary of a flexible Lagrangian remembers the topology of the filling. The following relative analog of Theorem \ref{thm: main} gives an affirmative answer to this question assuming some minor topological conditions. 
\begin{theorem}\label{thm: flexlag}
If $\Lambda^{n-1} \subset (Y^{2n-1}, \xi), n \ge 5$, has a  flexible Lagrangian filling $L^n \subset W^{2n}$ and $\pi_1(L, \Lambda) = \pi_1(Y, \Lambda) = c_1(W, L) = c_1(Y, \Lambda) = 0$, then all exact Lagrangian fillings of $\Lambda$ in all flexible Weinstein fillings of $(Y,\xi)$ have the isomorphic cohomology, i.e. if $W'$ is a flexible filling of $(Y, \xi)$ and $L' \subset W'$ is an exact Lagrangian filling of $\Lambda$, then $H^*(L'; \mathbb{Z}) \cong H^*(L; \mathbb{Z})$.   
\end{theorem}
In fact, Theorem \ref{thm: flexlag} holds for $n = 3, 4$ 
if $L^n$ admits a proper Morse function whose critical points have index less than $n-1$; for $n \ge 5$, this is equivalent to 
$\pi_1(L, \Lambda) = 0$ by Smale's handle-trading trick. The result also holds with any field coefficients. Note that here $L'$ does not have be flexible. So this result is slightly stronger than its contact analog Theorem \ref{thm: main} which is concerned only with flexible Weinstein fillings. We also note that work of Ekholm and Lekili \cite{Ekholm_Lekili} implies that certain Legendrians (with special Reeb chord conditions that should be satisfied for Legendrians with flexible fillings) also remember the homology of the based loop space of their Lagrangian fillings. 

As in the contact case, we can use Theorem \ref{thm: flexlag} to construct infinitely many non-isotopic Legendrians in the same formal Legendrian class, i.e. with the same algebraic topological data; see Section \ref{subsection: loose}. The first exotic examples were 1-dimensional Legendrians in $(\mathbb{R}^3, \xi_{std})$
constructed by Chekanov \cite{Che}. Higher dimensional examples were found in \cite{EES} by Ekholm, Etnyre, and Sullivan who produced infinitely many non-isotopic Legendrians spheres and tori in $P^{2n} \times \mathbb{R}$ in the same formal class. However these Legendrians are all nullhomologous in $P^{2n} \times \mathbb{R}$. Bj\"{o}rklund \cite{Bj} showed that if $\Sigma$ is a closed surface, there are arbitrarily many 1-dimensional  Legendrians 
in $\Sigma \times \mathbb{R}$ that are formally isotopic  representing \textit{any} class in $H_1(\Sigma \times \mathbb{R})$. The following analog of Theorem \ref{thm: inf_contact_flex} generalizes these results. 
Here $\Omega^n$ is as in Theorem \ref{thm: boundedinfinite}. 

\begin{theorem}\label{thm: leginfinite}
Suppose $\Lambda^{n-1} \subset (Y^{2n-1},\xi), n \ge 5,$ is a formal Legendrian that has a formal Lagrangian filling $L^n$ in a flexible filling $W$ of $(Y, \xi)$
and $\pi_1(Y, \Lambda) = \pi_1(L, \Lambda) = c_1(W, L) = c_1(Y, L) 
= 0$. 
Then for any $M^n\in \Omega^n$, there is a Legendrian $\Lambda_M \subset (Y, \xi)$ such that 
\begin{itemize}
\item $\Lambda_M$ is formally isotopic to $\Lambda$
\item if $H^*(M;\mathbb{Z}) \not\cong H^*(N;\mathbb{Z})$, then $\Lambda_M, \Lambda_N$ are not Legendrian isotopic
\item $\Lambda_M$ has a flexible Lagrangian filling $L_M\subset W$ diffeomorphic to $L \natural (M \backslash D^n)$.
\end{itemize}
 In particular, there are infinitely many non-isotopic Legendrians in $(Y, \xi)$ that are formally isotopic to $\Lambda$ and have flexible Lagrangian fillings in $W$. 
\end{theorem}
\begin{remark}
A similar result was proven by Eliashberg, Ganatra, and the author in \cite{EGL}. In their situation, $Y^{2n-1} = (S^{2n-1}, \xi_{std})$ and if $H(\Lambda M; \mathbb{Z}) \not \cong  H(\Lambda N; \mathbb{Z})$, then  $\Lambda_M, \Lambda_N$ are not Legendrian isotopic.
\end{remark}

As for contact manifolds, this theorem fails for $n = 2$.
 Because of the Bennequin inequality, there are formal Legendrians with no genuine Legendrian representatives
and hence the first part of Theorem \ref{thm: leginfinite} fails; see \cite{CE12} for example. Furthermore, there are formal Legendrians that have a unique Legendrian representation up to isotopy, in which case the second part fails; for example, all Legendrians in $(\mathbb{R}^3, \xi_{std})$ that are topologically unknotted and formally isotopic are Legendrian isotopic \cite{EFraser}. However, like Theorem \ref{thm: flexlag}, the first part of Theorem \ref{thm: leginfinite} does hold for $n = 3, 4$ if $L \backslash D^n$ admits a Morse function whose critical points have index less than $n - 1$; the second part about infinitely many Legendrians also holds for $n = 4$.

We can also use Theorem \ref{thm: flexlag} to 
give a new proof of the homotopy equivalence version of the nearby Lagrangian conjecture for simply-connected Lagrangians intersecting a cotangent fiber once. 
In the following, let $\pi: T^*M \rightarrow M$ be the projection to the zero-section. 
 
\begin{corollary}\label{cor: nearby}
Suppose $M^n, n \ge 2,$ is simply-connected and $L \subset T^*M$ is a  closed exact Lagrangian with zero Maslov class intersecting $T_q^*M$  transversely in a single point for some $q\in M$. Then 
$\pi_*: H_*(L;\mathbb{Z})\rightarrow H_*(M; \mathbb{Z})$ is an isomorphism; hence if $L$ is simply-connected, $\pi$ is a homotopy equivalence. 
\end{corollary}
The fact that $\pi$ is homotopy equivalence  was first proven by Fukaya, Seidel, and Smith \cite{FukSS}, whose result did not require $L$ to intersect $T_q^*M$ in a single point. 
Although our intersection condition is quite restrictive and immediately implies that $\pi_*$ is surjective,  our approach seems to be fairly elementary and does not use Fukaya categories or spectral sequences..
In addition, our approach seems to be adaptable and can be generalized to certain other Weinstein domains like plumbings; see Section \ref{sec: ADS_Weinstein}.
\\

This paper is organized as follows. In Section \ref{sec: background}, we present some background material. In Section \ref{ssec: independencelin}, we introduce asymptotically dynamically convex contact structures and prove their main properties. In Section \ref{sec: applications}, we prove the results stated in Section \ref{sec: mainresults} and \ref{sub: exotic}, assuming Theorem \ref{thm: semi-surgery}  that flexible surgery preserves  asymptotically dynamically convex contact structures. In Section \ref{sec: reeb_chord_loose}, we prove our main geometric result  about Reeb chords of loose Legendrians. In Section
\ref{sec: surgerysemigood}, we prove Theorem \ref{thm: semi-surgery}
modulo a technical lemma that we prove in the Appendix. In Section 
\ref{sec: legendrian_results}, we prove the relative versions of our results for Legendrians stated in Section \ref{ssec: intro_legendrians}. In Section \ref{sec: open_problems}, we present some open problems.

\section*{Acknowledgements}
I would like to thank my Ph.D. advisor Yasha Eliashberg for suggesting this problem and for many inspiring discussions. I am also grateful to Kyler Siegel and Laura Starkston for providing many helpful comments on earlier drafts of this paper. I also thank Matthew Strom Borman, Tobias Ekholm, Sheel Ganatra, Jean Gutt, Alexandru Oancea, Joshua Sabloff, and  Otto van Koert for valuable discussions.  This work was partially supported by a National Science Foundation Graduate Research Fellowship under grant number DGE-114747.

\section{Background}\label{sec: background}

\subsection{Liouville and Weinstein domains}\label{ssec: domains}

We first review the relevant symplectic manifolds and their relationship to contact manifolds. All our symplectic manifolds will be exact and either compact with boundary or open.  A \textit{Liouville domain} is a pair $(W^{2n}, \lambda)$ such that \begin{itemize}
\item $W^{2n}$ is a compact manifold with boundary
\item $d\lambda$ is a symplectic form on $W$ 
\item the Liouville field $X_\lambda$, defined by $i_{X}d\lambda = \lambda$, is outward transverse along $\partial W$
\end{itemize}
A \textit{Weinstein domain} is a triple $(W^{2n}, \lambda, \phi)$ such that
\begin{itemize}
\item $(W, \lambda)$ is a Liouville domain
\item $\phi: W \rightarrow \mathbb{R}$ is a Morse function with maximal level set $\partial W$ 
\item $X_\lambda$ is a gradient-like vector field for $\phi$.
\end{itemize}
Since $W$ is compact and $\phi$ is a Morse function with maximal level set $\partial W$, $\phi$ has finitely many critical points. Liouville and Weinstein \text{cobordisms} are defined similarly.  

The fact that $X_\lambda$ is outward transverse to $\partial W$ implies that $(Y, \alpha):= (\partial W, \lambda|_{\partial W})$ is a contact manifold. 
If a contact manifold $(Y, \xi)$ is contactomorphic to $\partial (W, \lambda)$, then  we say that $(W, \lambda)$ is a Liouville or Weinstein \textit{filling} of $(Y, \xi)$.
Because $X_\lambda$ is defined on all of $W$ and points outward along $\partial W$, the flow of $X_\lambda$ is defined for all negative time. In particular, the negative flow of $X_\lambda$ identifies a subset of $W$ with the \textit{negative} symplectization 
$(Y \times (0,1], d(t \alpha))$ of $(Y, \alpha)$; here $t$ is the second coordinate on $Y\times (0,1]$. We can also glue the positive symplectization $(Y \times \mathbb [1,\infty), d(t \alpha))$ of $(Y, \alpha)$ to $W$ along $\partial W$. 
The result is the \textit{completion} $(\widehat{W}, d\hat{\lambda})$ of $W$, an open exact symplectic manifold; here $\hat{\lambda} = \lambda$ in $W$ and $\hat{\lambda} = t\alpha$ in $Y\times [1, \infty)$. Note that $X_{\hat{\lambda}}$ is a complete vector field in $\widehat{W}$. 
In order to avoid trivial invariants like volume, one usually speaks of symplectomorphisms of completed  Weinstein domains rather than symplectomorphisms of domains themselves. 
In this paper, the negative symplectization of $Y$ in $W$ will play a more important role than the completion of $W$; see Remark \ref{rem: positivedeg}.

The natural notion of equivalence between Weinstein domains
$(W, \lambda_0, \phi_0),$ $(W, \lambda_1, \phi_1)$ is a 
\textit{Weinstein homotopy}, i.e. a 1-parameter family
of Weinstein structures $(W, \lambda_t, \phi_t)$, 
$t\in [0, 1],$ connecting them, where $\phi_t$ is allowed to have birth-death critical points. 
Weinstein domains that are Weinstein homotopic have exact  symplectomorphic completions and contactomorphic contact boundaries; see \cite{CE12}. We note that the notion of Weinstein homotopy between Weinstein \textit{manifolds} is more general and does not necessarily imply that the contact manifolds at infinity are contactomorphic (or even diffeomorphic); see \cite{Courte}.

\subsubsection{Weinstein handle attachment and contact surgery}

A Weinstein structure yields a special handle-body decomposition for $W$. First, recall that $\lambda$ vanishes on the $X_\lambda$-stable disc $D_p$ of a critical point $p$; see \cite{CE12}. In particular, $D_p$ is isotropic with respect to $d\lambda$ and so all critical points of $\phi$ have index less than or equal to $n$. If all critical points of $\phi$ have index \textit{strictly less} than $n$, then the Weinstein domain is \textit{subcritical}.  Also, $X_\lambda$ is transverse to any regular level $Y^c = \phi^{-1}(c)$ of $\phi$ and so $(Y^c, \lambda|_{Y^c})$ is a contact manifold; similarly $W^c = \{\phi \le c\}$ is a Weinstein subdomain of $W$. Since $\lambda$ vanishes on $D_p$, then  $\Lambda_p: = D_p \cap Y^c \subset (Y^c, \lambda|_{Y^c})$ is an isotropic sphere, where $c = \phi(p) - \epsilon$ for sufficiently small $\epsilon$. Furthermore, $\Lambda_p$ comes with a parametrization and framing, i.e. a trivialization of its normal bundle. Note that a framing of $\Lambda_p$ is equivalent to the framing of the conformal symplectic normal bundle of $\Lambda_p$; see \cite{Gbook}. Hence parametrized Legendrians come with a canonical framing.

Suppose that $c_1 < c_2$ are regular values of $\phi$ and 
$W^{c_2} \setminus W^{c_1}$ contains a unique critical point $p$ of $\phi$. Then $W^{c_2} \backslash W^{c_1}$ is an elementary Weinstein cobordism between $Y^{c_1}$ and $Y^{c_2}$ and the symplectomorphism type of $W^{c_2}$ is determined by the symplectomorphism type of $W^{c_1}$ along with the framed isotopy class of the isotropic sphere $\Lambda_p \subset Y^{c_1}$.
If the critical values of $\phi$ are distinct, then $W$ can be viewed as the concatenation of such elementary Weinstein cobordisms. 

On the other hand, one can explicitly construct such cobordisms and use them to modify Liouville domains or contact manifolds. Given a Liouville domain $X_-$ and a framed isotropic sphere $\Lambda$ in its contact boundary $Y_-$, we can attach 
an elementary Weinstein cobordism with critical point $p$ 
and $\Lambda_p = \Lambda$ to $X_-$ and obtain a new Liouville domain $X_+$. This operation is called \textit{Weinstein handle attachment} and $\Lambda$ is called the \textit{attaching sphere} of the Weinstein handle.
If $X_-$ was Weinstein, then so is $X_+$. 
The contact boundary $Y_+$ of $X_+$ is the result of \textit{contact surgery} along $\Lambda \subset Y_-$ and the Weinstein handle gives an elementary Weinstein cobordism between $Y_-$ and $Y_+$; see Proposition \ref{prop: techsurgery} for details.  
If the dimension of $\Lambda \subset Y_-^{2n-1}$ is less than $n-1$, the handle attachment, surgery, and $\Lambda$ itself are all called \textit{subcritical}.
Therefore, any  (subcritical) Weinstein domain can be obtained by  attaching (subcritical) Weinstein handles to the standard Weinstein structure on $B^{2n}$; similarly, the contact boundary of any (subcritical) Weinstein domain can be obtained by doing (subcritical) contact surgery to $(S^{2n-1}, \xi_{std})$.

\subsubsection{Formal structures}\label{sec: formal structures}
There are also formal versions of symplectic, Weinstein, and contact structures that depend on just the underlying algebraic topological data. 
 For example, an \textit{almost symplectic structure} $(W, J)$ on $W$ is an almost complex structure $J$ on $W$; this is equivalent to having a 
non-degenerate (but not necessarily closed) 2-form on $W$. 
An almost symplectomorphism between two almost symplectic manifolds $(W_1, J_1), (W_2, J_2)$ is a diffeomorphism $\phi: W_1 \rightarrow W_2$ such that $\phi^*J_2 $ can be deformed to $J_1$ through almost complex structures on $W_1$. 
An \textit{almost Weinstein domain} is a triple $(W, J, \phi)$, where $(W,J)$ is a compact almost symplectic manifold with boundary 
and $\phi$ is a Morse function on $W$ with no critical points of index greater than $n$ and maximal level set $\partial W$. An \textit{almost contact structure} $(Y, J)$ on $Y$ is an almost complex structure $J$ on the stabilized tangent bundle $TY \oplus \epsilon^1$ of $Y$. Therefore an almost symplectic domain $(W, J)$ has almost contact boundary $(\partial W, J|_{\partial W})$; it is an almost symplectic filling of this almost contact manifold. 
Note that any symplectic, Weinstein, or contact structure can also be viewed as an almost symplectic, Weinstein, or contact structure by considering just the underlying  algebraic topological data. 

Note that the first Chern class $c_1(J)$ is an invariant of almost symplectic, almost Weinstein, or almost contact structures.
 In this paper, we will often need to assume that $c_1(J)$ vanishes. The following proposition, which will be used several times in this paper, shows  that the vanishing of $c_1(Y, J)$ is often preserved under contact surgery and furthermore implies the vanishing of $c_1(W, J)$.  

\begin{proposition}\label{prop: c1equivalence}
Let $(W^{2n}, J), n \ge 3,$ be an almost Weinstein cobordism between $\partial_-W = (Y_-, J_-), \partial_+W = (Y_+, J_+)$. If $H^2(W,Y_-) = 0$, the vanishing of  $c_1(J_-)$ and $c_1(J_+)$ and $c_1(J)$ are equivalent. If  $\partial_-W =\varnothing$, the vanishing of 
$c_1(J_+)$ and $c_1(J)$ are equivalent. 
\end{proposition}
\begin{proof}
Let $i_\pm: Y_\pm \hookrightarrow W$ be inclusions. Then $i_\pm^*c_1(J) = c_1(J_\pm)$ so the vanishing of $c_1(J)$ implies the vanishing of $c_1(J_-)$ and $c_1(J_+)$.  
To prove the converse, consider the cohomology long exact sequences of the pairs $(W, Y_-)$ and $(W, Y_+)$: 
$$
H^2(W, Y_\pm; \mathbb{Z})
\rightarrow 
H^2(W; \mathbb{Z}) \xrightarrow{i_\pm^*} H^2(Y_\pm; \mathbb{Z}).
$$
By assumption, $H^2(W, Y_-; \mathbb{Z})$ vanishes and hence $i_-^*$ is injective. By Poincar\'e-Lefschetz duality, 
$H^2(W, Y_+; \mathbb{Z}) \cong H_{2n-2}(W, Y_-;\mathbb{Z})$. Since $2n-2 \ge n+1$ for $n \ge 3$ and $W$ is a Weinstein cobordism, 
$H_{2n-2}(W, Y_-;\mathbb{Z})$ vanishes and hence $i_+^*$ is also injective. Then if either $c_1(J_-) = i_-^*c_1(J)$ or $c_1(J_+) = i_+^*c_1(J)$ vanish, so does $c_1(J)$.

If $\partial_-W = \varnothing$, 
we just need the vanishing of $H^2(W, Y_+; \mathbb{Z})$, which holds for $n\ge 3$.
\end{proof}

\subsection{Loose Legendrians and flexible Weinstein domains}\label{subsection: loose}
Murphy  \cite{Murphy11} discovered that a certain class of \textit{loose} Legendrians satisfy a h-principle. That is, the symplectic topology of these Legendrians is governed by algebraic topology.
There are several equivalent criteria for a Legendrian to be loose, all of which depend the existence of a certain local model inside this Legendrian. We will use the following local model from Section 2.1 of \cite{CMP}. 
Let $B^3 \subset (\mathbb{R}^3, \xi_{std} = \ker \alpha_{std})$ be a unit ball and let $\Lambda_0$ be the 1-dimensional Legendrian whose front projection is shown in Figure \ref{fig: stabilizationgeo}. Let $Q^{n-2}, n \ge 3,$ be a closed manifold
and $U$ a neighborhood of the zero-section $Q \subset T^*Q$. Then 
$\Lambda_0 \times Q \subset (B^3 \times U, \mbox{ker}(\alpha_{std} + \lambda_{std}))$ is a Legendrian submanifold. 
This Legendrian is the \textit{stabilization} over $Q$ of 
the Legendrian $\{y = z = 0 \} \times Q \subset (B^3 \times U, \mbox{ker}(\alpha_{std} + \lambda_{std}))$.
\begin{definition}
A Legendrian $\Lambda^{n-1} \subset (Y^{2n-1}, \xi), n \ge 3,$ is \textit{loose} if there is a neighborhood $V\subset (Y, \xi)$ of $\Lambda$ such that $(V, V\cap \Lambda)$ is contactomorphic to $(B^3 \times U, \Lambda_0 \times Q)$.
\end{definition}
\begin{remark}\label{rem: looseleg}
If $f: (U^{2n-1}, \xi_1) \rightarrow (V^{2n-1}, \xi_2)$ is an 
equidimensional contact embedding and $\Lambda \subset (U, \xi_1)$ is loose, then $f(\Lambda) \subset (V, \xi_2)$ is also loose.
\end{remark}

\begin{figure}
              \centering
              \includegraphics[scale=0.25]{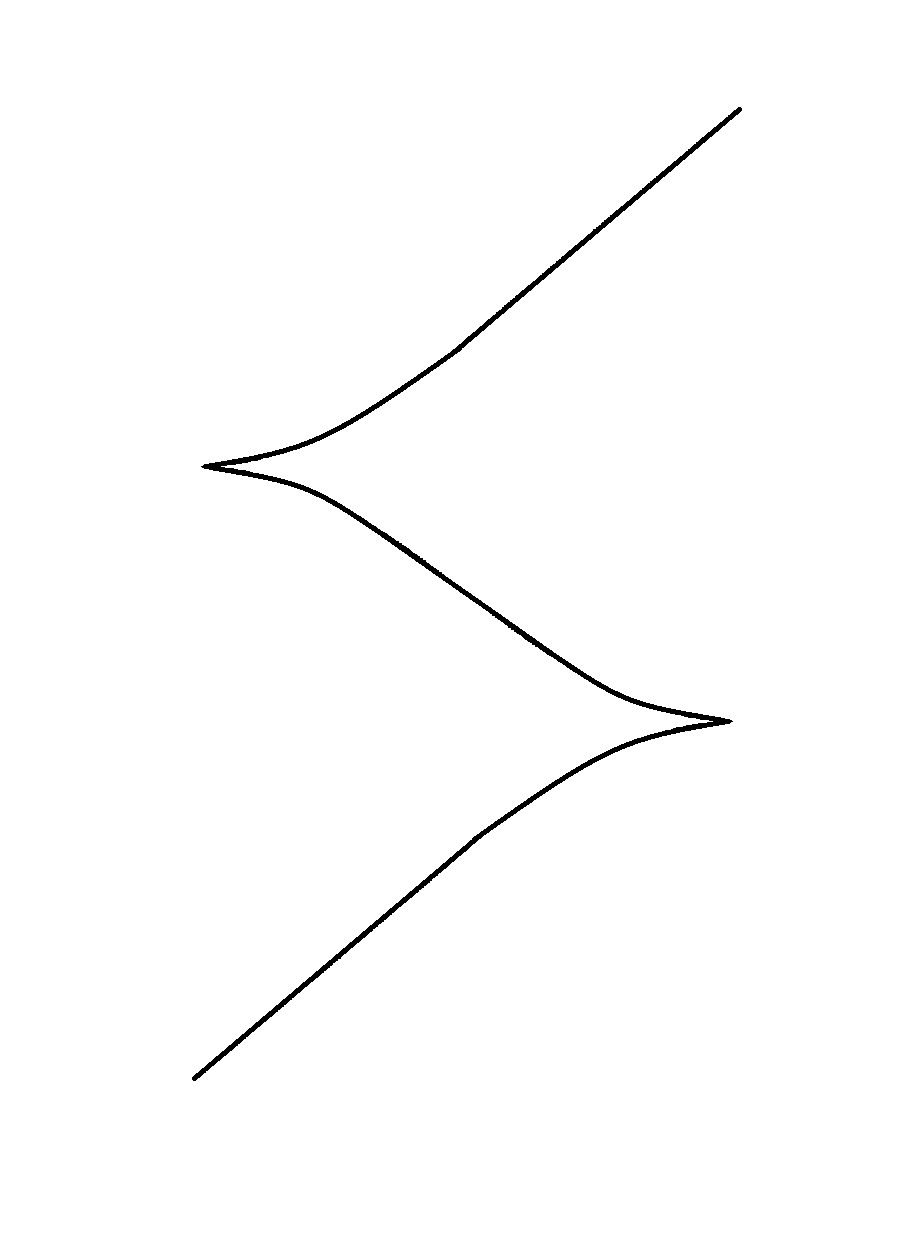}
              \caption{Front projection of $\Lambda_0$}
              \label{fig: stabilizationgeo}
 \end{figure}

A \textit{formal Legendrian embedding} is an embedding $f: \Lambda \rightarrow (Y, \xi)$ together with a homotopy of bundle monomorphisms $F_s: T\Lambda \rightarrow TY$ covering $f$ for all $s$ such that $F_0 = df$ and $F_1(T\Lambda)$ is a Lagrangian subspace of $\xi$.  A \textit{formal Legendrian isotopy} is an isotopy through formal Legendrian embeddings. Murphy's h-principle \cite{Murphy11} has an existence and uniqueness part:
\begin{itemize}
\item any formal Legendrian is formally Legendrian isotopic to a loose Legendrian
\item any two loose Legendrians that are formally Legendrian isotopic are genuinely Legendrian isotopic. 
\end{itemize}

We now define a class of Weinstein domains introduced in \cite{CE12} that are constructed by iteratively attaching Weinstein handles along loose Legendrians.

\begin{definition}\label{def: weinstein_flexible}
A Weinstein domain $(W^{2n}, \lambda, \phi)$ is  \textit{flexible} if there exist regular values $c_1, \cdots, c_{k}$ of $\phi$ such that $c_1 < \min \phi < c_2 < \cdots < c_{k-1} < \max \phi < c_{k}$ and for all $i = 1, \cdots, k-1$, $\{c_i \le \phi \le c_{i+1} \}$ is a Weinstein cobordism with a single critical point $p$ whose the attaching sphere $\Lambda_p$ is either subcritical or a loose Legendrian in $(Y^{c_i}, \lambda|_{Y^{c_i}})$.  
\end{definition} 

Flexible Weinstein \textit{cobordisms} are defined similarly. Also, Weinstein handle attachment or contact surgery is called flexible if the attaching Legendrian is loose. 
So any flexible Weinstein domain can be constructed by iteratively attaching subcritical or flexible handles to $(B^{2n}, \omega_{std})$.  A Weinstein domain that is Weinstein homotopic to a Weinstein domain  satisfying Definition \ref{def: weinstein_flexible} will also be called flexible.
Loose Legendrians have dimension at least $2$  
so if $(Y_+,\xi_+)$ is the result of flexible contact surgery on $(Y_-, \xi_-)$, then by Proposition \ref{prop: c1equivalence}
$c_1(Y_+)$ vanishes if and only if $c_1(Y_-)$ does.
Finally, we note that subcritical domains are automatically flexible.

Our definition of flexible Weinstein domains is a bit different from the original definition in \cite{CE12}, where several critical points are allowed in  $\{c_i \le \phi \le c_{i+1} \}$. 
There are no gradient trajectories between these critical points and their attaching spheres form a loose \textit{link} in $(Y^{c_i}, \lambda|_{Y^{c_i}})$, i.e each Legendrian is loose in the complement of the others. In this paper, we prefer to work with connected Legendrians, which is why we allow only one critical point in each cobordism  $\{c_i \le \phi \le c_{i+1} \}$; hence all critical points have distinct critical values. 
These two definitions are the same up to Weinstein homotopy.

Since they are built using loose Legendrians, which satisfy an h-principle, flexible Weinstein domains also satisfy an h-principle \cite{CE12}.
Again, the h-principle has an existence and uniqueness part:
\begin{itemize}
\item any almost Weinstein domain admits a flexible Weinstein structure in the same almost symplectic class
\item any two flexible Weinstein domains that are almost symplectomorphic are Weinstein homotopic (and hence have exact symplectomorphic completions and contactomorphic boundaries).
\end{itemize}

\subsection{Symplectic homology}\label{subsec: symhom}
In this section, we review the symplectic homology $SH(W)$ of a Liouville domain $(W, \lambda)$. We will follow the conventions for signs and grading in symplectic homology used in  \cite{CieliebakOancea}.
Let $(Y, \alpha):= (\partial W, \lambda|_{\partial W})$ be the contact boundary of $W$.  As before, let $(\widehat{W}, d\hat{\lambda})$ be the completion of $W$ obtained by attaching the positive symplectization $(Y \times [1, \infty), d(r \alpha))$ of $(Y, \alpha)$ to $\partial W$; here $r$ is the cylindrical coordinate on $Y \times [1, \infty)$. 
The contact manifold $(Y, \alpha)$ has a canonical \textit{Reeb} vector field $R_\alpha$ defined by $i_{R_\alpha}d\alpha = 0$ and $\alpha(R_\alpha) =1$; periodic orbits of $R_\alpha$ are called Reeb orbits. 
The \textit{action} of a Reeb orbit
 $\gamma$ is
\begin{equation}\label{eqn: contactactiondef}
A(\gamma):= \int_{S^1}\gamma^* \alpha.
\end{equation}
Note that $A(\gamma)$ is always positive and equals the period of $\gamma$. Let $Spec(Y, \alpha) \subset \mathbb{R}^+$ denote the set of actions of all Reeb orbits of $\alpha$. 

We say that a Reeb orbit $\gamma$ of $\alpha$ is \textit{non-degenerate} if the linearized Reeb flow from $\xi_p$ to $\xi_p$ for some  $p \in \gamma$ does not have $1$ as an eigenvalue; similarly, we say that the contact form $\alpha$ is \textit{non-degenerate} if all Reeb orbits of $\alpha$ are non-degenerate. A generic contact form is non-degenerate and we can assume that any contact form is non-degenerate after a $C^0$-small modification. If $\alpha$ is non-degenerate, then $Spec(Y, \alpha)$ is a discrete subspace of $\mathbb{R}^+$. 
All Reeb orbits that we discuss in this paper will be non-degenerate. However  we only work with orbits below a fixed action and so contact forms do not have to be non-degenerate, i.e there may be high-action orbits that are degenerate. 

\subsubsection{Admissible Hamiltonians and almost complex structures}
To define symplectic homology, we need to equip $W$ with a certain family of  functions and almost complex structures. Let  $\mathcal{H}_{std}(W)$ denote the class of \textit{admissible Hamiltonians}, which are functions on $\widehat{W}$ defined up to smooth approximation as follows:
\begin{itemize}
\item $H \equiv 0$ in $W$
\item $H$ is linear in $r$ with slope $s \not\in Spec(Y, \alpha)$ in $\widehat{W}\backslash W = Y\times [1, \infty)$.
\end{itemize}
More precisely, $H$ is a $C^2$-small Morse function in $W$ and $H=h(r)$ in $\widehat{W}\backslash W$ for some function $h$ that is increasing convex in a small collar $(Y \times [1, 1+\delta], r\alpha)$ of $Y$ and linear with slope $s$ outside this collar; for example, see \cite{Gutt} for details. Often, we will just say that $h$ is increasing convex near $Y$, by which we mean in such a collar. 

For $H \in \mathcal{H}_{std}(W)$,  the Hamiltonian vector field $X_H$ is defined by the condition $d\hat\lambda(\cdot, X_H ) = dH$. The time-1 orbits of $X_H$ are called the Hamiltonian orbits of $H$ and fall into two classes depending on their location in $\widehat{W}$: 
\begin{itemize}
\item In $W$, the only Hamiltonian orbits are constants corresponding to Morse critical points of $H|_W$
\item In $\widehat{W}\backslash W$, we have $X_H = h'(r)R_\alpha$, where $R_\alpha$ is the Reeb vector field of $(Y, \alpha)$. So all Hamiltonian orbits lie on level sets of $r$ and come in $S^1$-families corresponding to reparametrizations of some Reeb orbit of $\alpha$ with period $h'(r)$.
\end{itemize}
Since the slope $s$ of $H$ at infinity is not in $Spec(Y, \alpha)$, all Hamiltonian orbits lie in a small neighborhood of $Y$ in $\widehat{W}$.
After a $C^2$-small time-dependent perturbation of $H$, the orbits become \textit{non-degenerate}, i.e. 
the linearized Hamiltonian flow from $T_p W$ to $T_p W$, for some $p$ in the Hamiltonian orbit,  does not have $1$ as an eigenvalue.
These non-degenerate orbits also lie in a neighborhood of $W$ and so there are only finitely many of them. 
In fact, under this perturbation, each $S^1$-family of Hamiltonian orbits degenerates into two Hamiltonian orbits; see for example \cite{BO}.

We say that an almost complex structure $J$ is \textit{cylindrical} on the symplectization 
$(Y \times (0, \infty), r\alpha)$ if it preserves 
$\xi = \ker \alpha,$ $J|_\xi$ is independent of $r$ and compatible with $d(r\alpha)|_\xi$, and $J(r\partial_r) = R_\alpha$. Let $\mathcal{J}_{std}(W)$ denote the class of \textit{admissible} almost complex structures $J$ on $\widehat{W}$ which satisfy
\begin{itemize}
\item $J$ is compatible with $\omega$ on $\widehat{W}$
\item $J$ is cylindrical on  $\widehat{W}\backslash W = (Y \times [1, \infty), r\alpha)$. 
\end{itemize}

\subsubsection{Floer complex}
For $H \in \mathcal{H}_{std}(W), J\in \mathcal{J}_{std}(W)$, the Floer complex $SC(W,\lambda,  H, J)$  is generated as a free abelian group by Hamiltonian orbits of $H$ that are contractible in $W$. Note that we can work with integer coefficients rather than Novikov ring coefficients since all the symplectic manifolds in this paper are exact. 
We will often write this complex as $SC(H, J)$ when we do not need to specify $(W, \lambda)$. 

The differential is given by counts of Floer trajectories. In particular, for two Hamiltonian orbits $x_-, x_+$ of $H$, let $\widehat{\mathcal{M}}(x_-, x_+; H, J)$ be the moduli space of smooth maps 
$u: \mathbb{R}\times S^1  \rightarrow \widehat{W}$ such that $
\underset{s\rightarrow \pm \infty}{\lim}
 u(s, \cdot) = x_\pm$ and $u$ satisfies 
Floer's equation 
$$
\partial_s u + J(\partial_t u - X_H) =0.
$$
Here $s, t$ denote the $\mathbb{R}, S^1$ coordinates on $\mathbb{R}\times S^1$ respectively. 
Since the Floer equation is $\mathbb{R}$-invariant, there is a free $\mathbb{R}$-action on $\widehat{\mathcal{M}}(x_-, x_+; H, J)$ for $x_- \ne x_+$. 
Let $\mathcal{M}(x_-, x_+: H, J)$ be quotient by this $\mathbb{R}$-action, i.e. 
$\widehat{\mathcal{M}}(x_-, x_+; H, J)/ \mathbb{R}$.
After a small time-dependent perturbation of $(H,J)$, $\mathcal{M}(x_-, x_+, H, J)$ is a smooth finite-dimensional manifold. 

A maximal principle ensures that Floer trajectories do not escape to infinity in $\widehat{W}$. Here we use a quite general result sometimes called the `no escape' lemma, which we will also use frequently in Sections \ref{ssec: transfermaps} and \ref{ssec: independencelin}. This lemma holds for domains that are not necessarily cylindrical (for which it does not make sense to discuss maxima). It was  proven in Lemma 7.2 of \cite{Abouzaid_Seidel}; also, see Lemma 2.2 of \cite{CieliebakOancea}. 

For the following, let $V \subset (W, \lambda_W)$ be a \textit{Liouville subdomain}, i.e. $(V, \lambda_W|_V)$ is a Liouville domain. Then $(Z,\alpha_Z) = \partial (V, \lambda)$ is a contact manifold. 
Since $V$ is a Liouville subdomain, there is a collar of $Z$ in $W$ that is symplectomorphic to $(Z \times [1, 1+\delta], d(t\alpha_Z))$ for some small $\delta$. 
\begin{lemma}\cite{Abouzaid_Seidel}\label{lem: maximal_principle}
Consider $H: \widehat{W} \rightarrow \mathbb{R}$  such that  $H = h(t)$  is increasing near $Z$ and  $J \in \mathcal{J}_{std}(W)$ that is cylindrical near $Z$. If both asymptotic orbits of a $(H, J)$-Floer trajectory $u: \mathbb{R}\times S^1 \rightarrow \widehat{W}$ are contained in $V$, then $u$ is contained in $V$.  
\end{lemma}

Applying this result to $V = W$, we can proceed as if $W$ were closed and conclude by the Gromov-Floer compactness theorem that $\mathcal{M}(x_-, x_+; H, J)$ has a codimension one compactification. This implies that $\mathcal{M}_0(x_-, x_+; H, J)$, the zero-dimensional component of $\mathcal{M}(x_-, x_+; H, J)$, has finitely many elements and the map 
$
d: SC(H, J) \rightarrow SC(H, J)  
$
defined by
$$
d x_+ = \sum_{x_-} \sharp \mathcal{M}_0(x_-, x_+; H, J) x_-
$$
on generators and extended to $SC(H, J)$ by linearity is a differential. 
Here $\sharp \mathcal{M}_0(x_-, x_+; H, J)$ denotes the signed count of elements of $\mathcal{M}_0(x_-, x_+; H, J)$; signs are obtained via  the theory of coherent orientations \cite{FH93}. 
So $(SC(H,J), d)$ is a chain complex. Note that the underlying vector space $SC(H,J)$ depends only on $H$ while the differential $d$ depends on both $H$ and $J$, which are required to define Floer trajectories. 
The resulting homology $SH(H, J)$ is independent of $J$ and compactly supported deformations of $H$.

If $c_1(W, \omega) = 0$, as will always be the case in this paper, $SH(H, J)$ has a $\mathbb{Z}$-grading. 
More precisely, if $c_1(W, \omega) = 0$, the canonical line bundle of $(W, \omega)$ is trivial. 
After choosing a global trivialization of this bundle, 
we can assign an integer called the Conley-Zehnder index $\mu_{CZ}(x)$ to each Hamiltonian orbit $x$; since we have already taken a small time-dependent perturbation of $H$, all Hamiltonian orbits are non-degenerate and so this index is defined. Then the degree of $x$ in $SC(H, J)$ is 
$$
|x| := \mu_{CZ}(x).
$$
For a general orbit $x$, $\mu_{CZ}(x)$ depends on the choice of trivialization of the canonical bundle. We will only consider orbits $x$ that are contractible in $W$, for which $\mu_{CZ}(x)$ is independent of the trivialization.
For a Hamiltonian orbit 
corresponding to a critical point $p$ of the Morse function $H|_W$, the Conley-Zehnder index coincides with $n- Ind(p)$, where $Ind(p)$ is the Morse index of $H|_W$ at $p$.

Furthermore, $\dim \mathcal{M}(x,y; H, J) = |y|-|x|-1$ so the differential, which counts the zero-dimensional components of $\mathcal{M}(x,y; H, J)$, decreases the degree by one. Hence the induced grading on the homology $SH$ is well-defined.  
This grading on $SH$ coincides with the one in \cite{BEE12} and has the opposite sign of the one in \cite{S_06}. With this convention, we have $SH_*(T^*M) \cong H_*(\Lambda M; \mathbb{Z})$ for any closed spin manifold $M$ 
 \cite{K_07, V_96}.

\subsubsection{Continuation map}\label{sssec: continuation_map}
Although $SH(H, J)$ is independent of $J$ and compactly supported deformations of $H$, $SH(H, J)$ does depend on the slope of $H$ at infinity and so is not an invariant of $W$. In particular,  $SH(H, J)$ only sees Reeb orbits of period less than the slope of $H$ at infinity. To incorporate all Reeb orbits, we need to consider Hamiltonians with arbitrarily large slope. 
More formally, this can be done by considering continuation maps between $SC(H, J)$ for different $H$. Given $H_-, H_+ \in \mathcal{H}_{std}(W)$, let $H_s \in \mathcal{H}_{std}(W), s\in \mathbb{R},$ be a family of Hamiltonians such that $H_s = H_-$ for $s \ll 0$ and $H_s = H_+$ for $s\gg 0$. Similarly, let $J_s \in \mathcal{J}_{std}(W)$ interpolate between $J_-, J_+$. For  Hamiltonian orbits $x_-, x_+$ of $H_-, H_+$ respectively, let $\mathcal{M}(x_-, x_+; H_s, J_s)$ be the moduli space of parametrized Floer trajectories, i.e. maps 
$u:   \mathbb{R}\times S^1 \rightarrow \widehat{W}$ 
$$
\partial_s u + J_s(\partial_t u - X_{H_s}) = 0
$$

To ensure that parametrized Floer trajectories do not escape to infinity, we again need to use a maximal principle. For this principle to hold, it is crucial that the homotopy of Hamiltonian functions is decreasing, i.e. $\partial H_s/ \partial s \le 0$. 
If $J_s$ is $s$-independent, we use the following parametrized version of `no escape' Lemma \ref{lem: maximal_principle}, which is proven in Proposition 3.1.10 of \cite{Gutt}. 
If $J_s$ does depend on $s$ and $V = W$, then we use the maximal principle from \cite{S_06}. We state both in the following lemma. 
\begin{lemma}\label{lem: maximal_principle_param}
\cite{Gutt}, \cite{S_06}
Consider a decreasing homotopy $H_s: \widehat{W} \rightarrow \mathbb{R}$ such that $H_s = h_s(t)$ 
is increasing in $t$ near $Z = \partial V$ and $H_s|_Z$ is $s$-independent; let $J \in \mathcal{J}_{std}(W)$ be cylindrical near $Z$. If 
$u: \mathbb{R}\times S^1 \rightarrow \widehat{W}$ 
is a $(H_s, J)$-Floer trajectory with both asymptotes in $V$, then $u$  is contained in $V$. If $V = W$, the same claim also holds for a homotopy $J_s \in J_{std}(W)$  that is cylindrical near $Z$.
\end{lemma}
By applying the second part of Lemma \ref{lem: maximal_principle_param}, we can proceed as if $W$ were closed and conclude that $\mathcal{M}(x_-, x_+; H_s, J_s)$ has a codimension one compactification. Then the continuation map 
$
\phi_{H_s, J_s}: SC(H_+, J_+) \rightarrow SC(H_-, J_-)
$
defined by
$$
\phi_{H_s, J_s}(x_+) = \sum_{x_-} \sharp \mathcal{M}_0(x_-, x_+; H_s, J_s) x_-
$$
on generators and extended to $SC(H_+, J_+)$ by linearity is a chain map. Up to chain homotopy, this map is independent of $J_s$ and $H_s$. 
Note that there is no $\mathbb{R}$-action since the parametrized Floer equation is not $\mathbb{R}$-invariant. As a result, $\phi_{H_s, J_s}$ is degree-preserving.  
Finally, we define symplectic homology as the direct limit
$$
SH(W, \lambda):= \lim_{\rightarrow} SH(H, J). 
$$
The direct limit is taken over continuation maps 
$\phi_{H_s, J_s}: SH(H_+, J_+) \rightarrow SH(H_-, J_-)$ on homology. 
One key property is that $SH(W, \lambda)$ depends only on the symplectomorphism type of $(\widehat{W}, d\hat{\lambda})$
\cite{S_06}.

\begin{remark}
At several minor points in this paper, we will discuss a version of $SH(W)$ that is generated by \textit{all} Hamiltonian orbits, not just contractible ones; we will call this 
 \textit{full} $SH(W)$.   
\end{remark}

\subsection{Positive symplectic homology}\label{ssec: postive_sym_hom}

Positive symplectic homology $SH^+(W)$ is defined using the action functional. 
For a small time-dependent perturbation of 
$H\in\mathcal{H}_{std}(W)$, the action functional $A_H: C^\infty(S^1, \widehat{W}) \rightarrow \mathbb{R}$ is 
$$
A_H(x) :=  \int_{S^1} x^* \lambda - \int_{S^1} H(x(t)) dt.
$$
Under our conventions, the Floer equation is the \textit{positive} gradient flow of the action functional and so action increases along Floer trajectories, i.e.  if $u\in \mathcal{M}(x_-, x_+)$ is a non-constant Floer trajectory, then $A_H(x_+) > A_H(x_-)$. Let $SC^{<a}(H, J)$ be generated by orbits of action less than $a$. Since action increases along Floer trajectories, the differential decreases action and hence 
$SC^{<a}(H, J)$ is a subcomplex of $SC(H, J)$; we define $SC^{>a}(H, J)$ to be the quotient complex $SC(H, J)/ SC^{<a}(H, J)$. 

For $H\in \mathcal{H}_{std}(W)$, the constant orbits corresponding to Morse critical points $x \in W$ have action $-H(x)$.
The non-constant orbits that correspond to Reeb orbits 
have action close to the action of the corresponding Reeb orbit, which is positive.
In fact, for sufficiently small $\epsilon$, $SC^{< \epsilon}(H, J)$ corresponds to the Morse complex of $-H|_W$ with a grading shift. More precisely, 
$H_k(SC^{< \epsilon}(H, J)) \cong H^{n-k}(W; \mathbb{Z})$. 
Define $SC^+(H, J)$ to be the quotient complex 
$SC(H, J)/SC^{< \epsilon}(H, J)$ and let $SH^+(H, J)$ be the resulting homology. Using a direct limit construction, we can also define 
$SH^+(W)$. More precisely, suppose   $H_s$ is an decreasing homotopy such that $H_s = H_-$ for $s \ll 0$ and $H_s = H_+$ for $s\gg 0$. Then the continuation Floer trajectories are also action increasing and hence there is an induced chain map 
$\phi_{H_s, J_s}^+: SC^+(H_+, J_+) \rightarrow SC^+(H_-, J_-)$. 
As with $SH(W)$, we define $SH^+(W)$ by
$$
SH^+(W, \lambda):= \lim_{\rightarrow} SH^+(H, J).
$$
The direct limit is taken over the  continuation maps $\phi_{H_s, J_s}^+: SH^+(H_+, J_+) \rightarrow SH^+(H_-, J_-)$ on homology.

Like $SH(W)$, $SH^+(W)$ depends only on the symplectomorphism type of $(\widehat{W}, d\hat{\lambda})$. 
Note that as a vector space, $SC^+(H, J)$ essentially depends only on $(Y, \alpha)$ and not on the interior $(W, \lambda)$. This is because $SC^+(H,J)$ is generated by non-constant Hamiltonian orbits, which live in the cylindrical end of $W$ and correspond to Reeb orbits of $(Y, \alpha)$.
On the other hand, the differential for $SC^+(H, J)$ may depend on the filling $W$ of $(Y, \alpha)$ since Floer trajectories between non-constant orbits may go into the filling. So different Liouville fillings of $(Y, \xi)$ might have different $SH^+$. 
\begin{examples}\label{ex: surfaces}
The once-punctured genus g  surface $\Sigma_g$ is a Weinstein filling of $S^1$ for all g.
Since $\Sigma_0 = \mathbb{C}$ is subcritical,  $SH(\Sigma_0) = 0$ and so  $SH^+(\Sigma_0) \cong H(\Sigma_0; \mathbb{Z}) \cong \mathbb{Z}$ by the tautological long exact sequence, which we will discuss next. 
On the other hand, for $g\ge 1$, $SH(\Sigma_g)$ is infinite-dimensional by \cite{BO, S_06} and so $SH^+(\Sigma_g)$ is also infinite-dimensional by the tautological exact sequence. Therefore
 $SH^+(\Sigma_g)$ is different for $g = 0$ and $g \ge 1$.
\end{examples}

Of course the set of $SH^+$ of all Liouville  fillings of $(Y, \xi)$ is an invariant of $(Y, \xi)$; see Remark \ref{rem: semigood} (2) for details. But this requires knowing all fillings, which in general is completely beyond current technology.
However, in Subsection \ref{ssec: independencelin}, we will discuss certain contact structures $(Y, \xi)$ for which $SH^+(W)$ is independent of the filling and is an invariant of $(Y, \xi)$.

Now we compute $SH^+$ in a simple case.
The chain-level short exact sequence 
$$
0 \rightarrow SC^{< \epsilon}(H, J) \rightarrow SC(H, J) \rightarrow SC^+(H, J) \rightarrow 0
$$
induces the `tautological' long exact sequence in homology
$$
\cdots \rightarrow H^{n-k}(W; \mathbb{Z}) \rightarrow 
SH_k(W, \lambda) \rightarrow SH_k^+(W, \lambda) \rightarrow 
H^{n-k+1}(W; \mathbb{Z})\rightarrow \cdots .
$$
Using this long exact sequence, we can compute $SH^+(W)$ for Liouville domains with vanishing $SH$. 
\begin{proposition}\label{prop: shcomputation}
 If $W^{2n}$ is a Liouville domain with $SH(W) =0$, then 
\begin{equation}\label{eqn: linformula}
SH_k^+(W) 
\cong H^{n-k+1}(W; \mathbb{Z}).
\end{equation}
\end{proposition}
Since flexible $W$ have vanishing symplectic homology \cite{BEE12}, Proposition \ref{prop: shcomputation} shows that $SH^+(W)$ can be expressed in terms of the cohomology of $W$.  However, this is not enough to prove our main theorem that all flexible fillings of a given contact manifold have the same cohomology because in general $SH^+$ is not independent of the filling as we saw in Example \ref{ex: surfaces}.
The main purpose of this paper is show that in the setting of flexible fillings, this is indeed the case as we explain in Section \ref{ssec: independencelin}.
				
\begin{remark}\label{rem: flexibledegs}
If $W^{2n}$ is a Weinstein domain, then $W$ is homotopy equivalent to a $n$-dimensional CW complex and so $H^{n-k+1}(W; \mathbb{Z}) = 0$ for $k \le 0$. By Proposition \ref{prop: shcomputation}, this implies that $SH_k^+(W) = 0$ for $k \le 0$.
\end{remark}

\subsection{Transfer maps}\label{ssec: transfermaps}
 Viterbo \cite{viterbo1} showed that if $V \subset W$ is a Liouville subdomain, then there is a \textit{transfer map} $SH(W) \rightarrow SH(V)$. Gutt \cite{Gutt} defined a similar map for $SH^+$, which will be crucial to our later results. 
To define the transfer map, we introduce a new class of \textit{step} Hamiltonians and almost complex structures.
Since $V$ is a Liouville subdomain of $W$, there is a collar $U$ of $(Z, \alpha_Z) = (\partial V, \lambda_V|_{\partial V})$ in $W\backslash V$ such that $(U, \omega_W)$ is symplectomorphic to 
$(Z \times [1, 1+\epsilon_V], d(t \alpha_Z))$.
Let $\mathcal{H}_{step}(W, V)$ denote the class of smooth functions $H$ on
$\widehat{W}$ defined up to smooth approximation as follows: 
\begin{itemize}
\item $H \equiv 0$ in $V$
\item $H$ is linear in $t$ with slope $s_V$ in $U$
\item $H \equiv s_V \epsilon_V$ in $W  \backslash
( V \cup U)$
\item $H$ is linear in $r$ with slope $s_W$ in 
$\widehat{W} \backslash W = Y \times [1, \infty)$. 
\end{itemize}
See Figure \ref{fig: step_hamiltonian}. More precisely, $H$ is a $C^2$-small Morse function in $V$, $C^2$-close to $s_V\epsilon_V$ in $W\backslash (V\cup U)$, increasing convex in $t$ near $Z \times \{1\}$,  increasing concave in $t$ 
near $Z \times \{1 + \epsilon_V\}$, and 
increasing convex in $r$ near $Y \times \{1\}$; furthermore, $s_W, s_V \not \in Spec(Z, \alpha_Z) \cup Spec(Y, \alpha_Y)$.
 \begin{figure}
     \centering
     \includegraphics[scale=0.14]{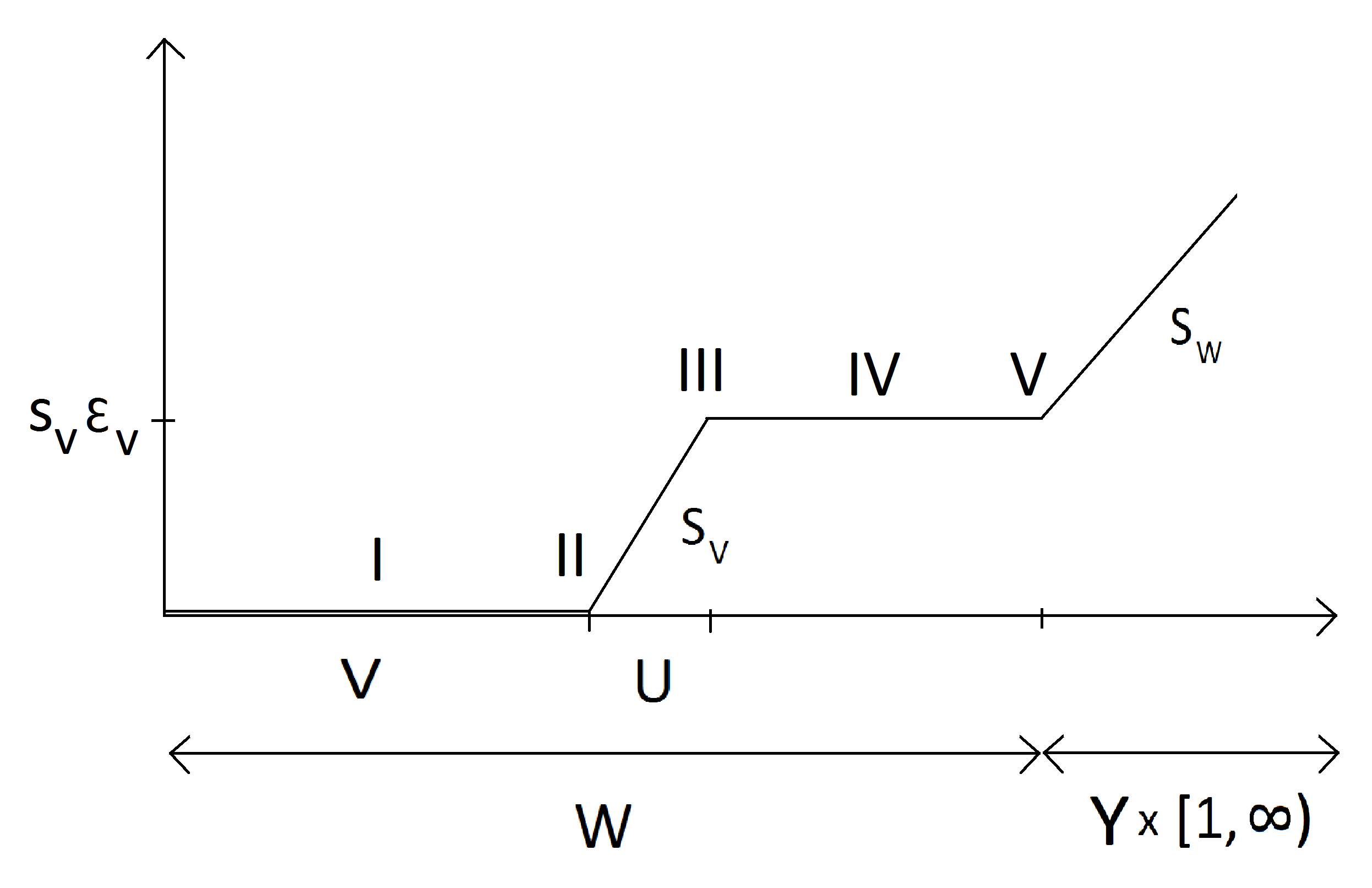}
     \caption{A step Hamiltonian in $\mathcal{H}_{step}(W,V)$ with slope $s_V,s_W$ near $\partial V, \partial W$ respectively and Hamiltonian orbits $I, II, III, IV, V$}
     \label{fig: step_hamiltonian} 
 \end{figure}
 As depicted in Figure \ref{fig: step_hamiltonian}, 
the Hamiltonian orbits of $X_{H}$ fall into five classes which we denote by I, II, III, IV, V:
\begin{itemize}
\item[(I)]  Morse critical points of $H|_V$
\item[(II)] Orbits near $\partial V  = Z \times \{1\}$ corresponding to parameterized Reeb orbits of $(Z, \alpha_Z)$
\item[(III)] Orbits near $\partial (V\cup U) = Z \times \{1+\epsilon_V\}$ corresponding to parameterized Reeb orbits of $(Z, \alpha_Z)$
\item[(IV)] Morse critical points of 
$H|_{W\backslash (V\cup U)}$
\item[(V)] Orbits near $\partial W= Y \times \{1\}$ corresponding to parameterized Reeb orbits of $(Y, \alpha_Y)$.
\end{itemize}

For $H_{W, V} \in \mathcal{H}_{step}(W, V)$, let $i(H_{W, V})\in \mathcal{H}_{std}(V)$ denote the Hamiltonian obtained by extending $H_{W,V}|_{V\cup U}$ to $\widehat{V}$. Similarly, let $\mathcal{J}_{step}(W,V)$ denote the class of almost complex structures $J \in \mathcal{J}_{std}(W)$ that are cylindrical in $U$, i.e. $J|_U$ preserves $\xi_Z = \ker \alpha_Z$, $J|_{\xi_Z}$ is $t$-independent, and $J(t\partial_t) = R_{\alpha_Z}$.
For $J_{W, V} \in \mathcal{J}_{step}(W, V)$, let $i(J_{W, V})\in \mathcal{J}_{std}(V)$ denote the almost complex structure obtained by extending $J_{W,V}|_{V\cup U}$ to $\widehat{V}$. 

Next we construct a map 
\begin{equation}\label{eqn: transfer_portion}
SH(H_{W,V}, J_{W, V})\rightarrow 
SH(i(H_{W , V}), i(J_{W , V})).
\end{equation}
This relies on the following proposition, some formulation of which is essential in all constructions of the transfer map. 
\begin{proposition}\cite{CieliebakOancea}
If $\epsilon_V\ge\frac{s_W}{s_V}$, then the subspace $SC^{III, IV, V}(H_{W, V}, J_{W,V})$ generated by III, IV, V orbits is a subcomplex of $SC(H_{W, V}, J_{W,V})$.
\end{proposition}
\begin{remark}
In the proof of this proposition, we  follow the approach taken in  Lemma 4.1 of \cite{CieliebakOancea},
which helps us achieve this particular bound on $\epsilon_V$; this bound will be important in Section \ref{ssec: independencelin}.
\end{remark}
\begin{proof}
We need to show that the differential of $SC(H_{W, V}, J_{W,V})$ does not map III, IV, V orbits to 
I, II orbits. We first show that this holds for IV, V orbits using action considerations. The I, II orbits have non-negative action as we noted in Section \ref{ssec: postive_sym_hom}. Since the differential decreases action, it is enough to show that the IV, V orbits have negative action. The IV orbits are constant and occur in $W\backslash (V \cup U)$, where  $H= s_V\epsilon_V$. So these orbits have action $-s_V\epsilon_V < 0$. 
The V orbits are located near $\partial W$, where  the action equals $rh'(r) - h(r)$. There we have  $r= 1, h = s_V \epsilon_V$, and $h'(r) < s_W$
so that the action is less than $s_W - s_V \epsilon_V$. Since $\epsilon_V \ge \frac{s_W}{s_V}$, the action of the V orbits is also negative. 

For III orbits, we use Lemma 2.3 of \cite{CieliebakOancea}, which summarizes an argument from  \cite{BO}. 
\begin{lemma}\label{lem: rise_above}\cite{BO}
Let $H = h(r)$ be a Hamiltonian on $(Y \times (0, \infty), r\alpha)$ and let $J$ be a cylindrical almost complex structure. Let $u: \mathbb{R} \times S^1\rightarrow Y \times (0, \infty)$ be a $(H,J)$-Floer trajectory with 
$\underset{s\rightarrow \pm \infty}{\lim} r\circ u(s, \cdot) = r_\pm$. If $\pm h''(r_{\pm})< 0$, then either there exists $(s_0, t_0) \in \mathbb{R} \times S^1$ such that $r\circ u(s_0, t_0) > r_{\pm}$ or $r\circ u$ is constant and equal to $r_\pm$.
\end{lemma}
Note that $H_{W, V}= h(t)$ is concave, i.e. $h''(t) < 0$, and $J_{W,V}$ is cylindrical near $Z \times \{1+ \epsilon_V\}$ where the III orbits occur. Hence Lemma \ref{lem: rise_above} shows that any $(H_{W,V}, J_{W,V})$-Floer trajectory which is asymptotic to a III orbit at its positive end must rise above this orbit. If this trajectory is asymptotic to a I or II orbit at its negative end, then the Floer trajectory must also travel below this III orbit since I, II orbits occur in $V \subset V\cup U$. This violates the `no escape' Lemma \ref{lem: maximal_principle} applied to $V\cup U$, which holds since $H_{W,V}$ is increasing and $J_{W,V}$ is cylindrical near $Z\times \{1+\epsilon_V\} = \partial (V\cup U)$. Therefore such trajectories cannot exist.
\end{proof}

Assuming that $\epsilon_V \ge \frac{s_W}{s_V}$, let $SC^{I,II}(H_{W, V}, J_{W,V})$ be the quotient complex\\
 $SC(H_{W, V}, J_{W,V})/ SC^{III, IV, V}(H_{W, V}, J_{W,V})$. This complex is  generated by the I, II orbits, which are precisely the generators of $SH(i(H_{W, V}), i(J_{W,V}))$. Note that $H_{W,V}= h(t)$ is increasing near $Z = \partial V$, where $I, II$ orbits occur.  Therefore, we can apply the `no escape'  Lemma \ref{lem: maximal_principle} to $V$, which shows that all $(H_{W,V}, J_{W,V})$-Floer trajectories in $\widehat{W}$ between $I,II$ orbits actually stay in $V$ and hence are $(i(H_{W, V}), i(J_{W,V}))$-trajectories. Therefore, we have an isomorphism
$SH^{I,II}(H_{W, V}, J_{W,V}) \cong SH(i(H_{W, V}), i(J_{W,V}))$. So the map in Equation \ref{eqn: transfer_portion} is obtained by projecting to the quotient $SH^{I,II}(H_{W, V}, J_{W,V}) $ and then using this isomorphism.

For any $H_W \in \mathcal{H}_{std}(W)$,  take $H_{W, V} \in \mathcal{H}_{step}(W, V)$ such that $\epsilon_V \ge \frac{s_W}{s_V}$ and $H_{W,V}, H_W$ differ by the constant $s_V \epsilon_W$ in $\widehat{W}\backslash W$. In particular, $H_{W, V} \ge H_V$ and so there exists a decreasing homotopy $H_s$ from $H_{W,V}$ to $H_W$; also take a homotopy $J_s \in \mathcal{J}_{std}(W)$ from $J_{W,V}$ to $J_W$. 
Then $(H_s, J_s)$-trajectories define the continuation map
$$
\phi_{H_s, J_s}: SH(H_W, J_W) \rightarrow SH(H_{W, V}, J_{W,V}).
$$
Since $H_s$ is decreasing, we can apply the second part of Lemma \ref{lem: maximal_principle_param} to $W$ to ensure that all trajectories stay in $W$ and hence that $\phi_{H_s}$ is well-defined; in Section \ref{ssec: independencelin}, we will also assume that $H_s = h_s(t)$ is increasing in $t$ and $J_s$ is $s$-independent near $Z$. By composing this continuation map with the projection to the quotient $SH^{I,II}(H_{W, V}, J_{W,V}) \cong SH(i(H_{W, V}), i(J_{W,V}))$, we get a map 
$$
SH(H_W, J_W) \rightarrow SH(i(H_{W, V}), i(J_{W,V})).
$$
This map commutes with continuation maps and hence induces a map on symplectic homology 
$$
\phi_{W, V}: SH(W) \rightarrow SH(V),
$$ 
which we call the transfer map. 

To construct a transfer map on $SH^+$, we first note that since the homotopy $H_s$ is decreasing, action increases along the parametrized Floer trajectories and hence we get a map 
$$
\phi_{H_s, J_s}^+: SH^{+}(H_W, J_W) \rightarrow SH^{+}(H_{W,V}, J_{W,V}).
$$
We can further quotient out III orbits from $SH^{+}(H_{W,V}, J_{W,V})$; this
is because the differential of $SC(H_{W,V}, J_{W,V})$ must map III orbits to III, IV, V orbits by Lemma \ref{lem: rise_above} and IV, V orbits have already been quotiented out in $SH^{>0}(H_{W, V}, J_{W,V})$ since they have negative action (for $\epsilon_V \ge \frac{s_W}{s_V}$). We use $SH^{II}(H_W, J_W)$ to denote the complex obtained by quotienting out III orbits from $SH^{+}(H_{W,V}, J_{W,V})$; $SH^{II}(H_W, J_W)$  is generated by II orbits since these have positive action. Also, note that II orbits are precisely the generators of $SH^+(i(H_{W, V}), i(J_{W,V}))$ and by the `no escape' Lemma \ref{lem: maximal_principle} applied to $V$, we have $SH^{II}(H_W, J_W) \cong SH^+(i(H_{W, V}), i(J_{W,V}))$. 
By composing $\phi_{H_s, J_s}^+$ with the 
projection to the quotient 
$SH^{II}(H_W, J_W) \cong SH^+(i(H_{W, V}), i(J_{W,V}))$, we get a map 
$$
SH^+(H_W, J_W) \rightarrow SH^+(i(H_{W,V}), i(J_{W,V})).
$$
Again this commutes with continuation maps in $W,V$ and hence we get the transfer map for $SH^+$
$$
\phi_{W,V}^+: SH^+(W) \rightarrow SH^+(V). 
$$

\subsubsection{Graphical subdomains}\label{sssec: graphical}
In the rest of this paper, we will be concerned with a special class of Liouville subdomains. Recall that the whole completion $(Y \times (0, \infty), d(r\alpha))$ of $(Y, \alpha) = (\partial W, \lambda|_{\partial W})$ embeds into $\widehat{W}$. Then any contact form $\beta$ for $(Y, \xi)$ gives rise to a \textit{graphical submanifold} $Y_\beta$ of $Y\times (0, \infty)$ defined by the condition $r \alpha|_{Y_{\beta}} = \beta$. Then the bounded component $W_\beta$ of $\widehat{W}\backslash Y_\beta$ is a \textit{graphical subdomain}. Note that the completion of any graphical subdomain is $\widehat{W}$. 

If $V \subset W$ is a graphical subdomain, then the transfer map 
$SH(W) \rightarrow SH(V)$ is an isomorphism. To see this, note that there is another graphical subdomain $W' \subset V \subset W$  defined by the condition $\lambda|_{\partial W'} = c\lambda|_{\partial W}$ for some constant $c<1$. Then $W\backslash W'$ is symplectomorphic to $(Y \times [c, 1], d(r\alpha))$, which is a part of the symplectization of $(Y, \alpha)$.
Hence the map $SH(W) \rightarrow SH(W')$ is an isomorphism. Since the transfer map is functorial with respect to embeddings of Liouville subdomains, the map $SH(W) \rightarrow SH(V)$ is injective; an analogous argument shows surjectivity. 
Similarly, $SH^+(W) \rightarrow SH^+(V)$ is an isomorphism for a graphical subdomain $V \subset W$.

\section{Asymptotically dynamically convex contact structures}\label{ssec: independencelin}

We now change our point of view and instead of focusing on Liouville domains themselves, we focus on contact manifolds and their Liouville fillings. As we saw in Example \ref{ex: surfaces}, in general $SH^+(W)$ depends on the filling $W$ and is not a contact invariant.  
Here we introduce a class of \textit{asymptotically dynamically convex} contact structures for which $SH^+$ is independent of the filling; see Definition \ref{def: semigood} and Proposition \ref{prop: nice_sh_independent} below. In the rest of the paper, we use ADC to abbreviate `asymptotically dynamically convex.'

\subsection{Properties of Reeb orbits}

We first review the degree of Reeb orbits, which is essential to the definition of ADC contact structures. For any contact manifold $(Y, \alpha)$ with $c_1(Y, \xi) = 0$, as will always be the case in this paper, the canonical line bundle of $\xi$ is trivial. After choosing a global trivialization of this bundle, we can assign an integer called the Conley-Zehnder index $\mu_{CZ}(\gamma)$ to each Reeb orbit $\gamma$ of $(Y, \alpha)$. 
Then we let the degree of $\gamma$ be the \textit{reduced} Conley-Zehnder index:
$$
|\gamma| := \mu_{CZ}(\gamma) + n-3.
$$
This grading convention coincides with the contact homology algebra grading \cite{EGH}. 
Note that to define $\mu_{CZ}(\gamma)$,
and hence $|\gamma|$, we need to assume that $\alpha$ is non-degenerate; see Section \ref{subsec: symhom}.
For a general Reeb orbit, $|\gamma|$ depends on the choice of trivialization of canonical bundle. In this paper we will only use orbits $\gamma$ that are contractible in $Y$, for which $|\gamma|$ is independent of the trivialization. Note that if $(Y, \alpha)$ is the contact boundary of some Liouville domain $(W, \lambda)$, the trivialization of the canonical bundle of $(W, \omega)$ induces a trivialization of the canonical bundle of $(Y, \xi)$.
In particular, $i^*c_1(W, \omega) = c_1(Y, \xi)$ where $i: Y\rightarrow W$ is the inclusion map and so $c_1(W, \omega) = 0$ implies $c_1(Y, \xi) = 0$. 

Let $\mathcal{P}^{< D}(Y, \alpha)$ be the set of Reeb orbits $\gamma$ of $(Y,\alpha)$ that are contractible in $Y$  and have  $A(\gamma)< D$; since $\alpha$ is non-degenerate, $\mathcal{P}^{< D}(Y, \alpha)$ is finite for any fixed $D$. 
Recall that $\phi: (Y, \alpha) \rightarrow (Y', \alpha')$ is a \textit{strict} contactomorphism if $\phi^*\alpha' = \alpha$. Then if $\phi: (Y, \alpha) \rightarrow (Y', \alpha')$ is a strict contactomorphism, then $\mathcal{P}^{< D}(Y, \alpha)$ and 
$\mathcal{P}^{< D}(Y', \alpha')$ are in grading-preserving bijection since $\phi$ preserves everything. We also have the following simple proposition, which will be used many times throughout. 

\begin{proposition}\label{prop: easy}
For any $D, s>0$,   
$\mathcal{P}^{< D}(Y, s\alpha)$ and $\mathcal{P}^{< D/s}(Y, \alpha)$ are in grading-preserving bijection. 
\end{proposition}
\begin{proof}
Note that $R_{s\alpha} = \frac{1}{s}R_\alpha$. So if $\gamma_\alpha:[0, T]\rightarrow Y$ is a Reeb trajectory of $\alpha$ with action $T$, then $\gamma_{s\alpha} = \gamma_{\alpha}\circ m_{\frac{1}{s}}: [0, s T] \rightarrow Y$ is a Reeb trajectory of $s\alpha$ with action $sT$; here  $m_{\frac{1}{s}}: [0, sT] \rightarrow [0, T]$ is multiplication by $\frac{1}{s}$. 
The map $\gamma_\alpha \rightarrow \gamma_{s\alpha}$ is a bijection between the set of Reeb orbits. If $T < D/s$, then $sT < D$ and so it is a bijection between $\mathcal{P}^{< D/s}(Y, \alpha)$ and $\mathcal{P}^{< D}(Y, s\alpha)$.  This bijection is grading-preserving since the Conley-Zehnder index of a Reeb orbit is determined by the linearized Reeb flow on the trivialized contact planes $\xi$ but  does not depend on the speed of the flow.  
\end{proof}

We will also need the following notation. If $\alpha_1, \alpha_2$ are contact forms for $\xi$, then there exists a unique $f: Y \rightarrow \mathbb{R}^+$ such that $\alpha_2 = f \alpha_1$. We write $\alpha_2 > \alpha_1, \alpha_2 \ge \alpha_1$ if $f > 1, f\ge 1$ respectively. Note that if $\alpha_2 > \alpha_1, \alpha_2 \ge \alpha_1$, then for any diffeomorphism $\phi: Y' \rightarrow Y$, we have
$\phi^*\alpha_2 > \phi^*\alpha_1, \phi^*\alpha_2 \ge \phi^*\alpha_1$ respectively.

\subsection{Background and definitions}

\subsubsection{Contact structures with $SH^+$ independent of the filling}\label{sssec: sh_independent}

Before defining ADC contact structures, we discuss other related contact structures that have $SH^+$ independent of the filling.
As we noted in Section \ref{ssec: postive_sym_hom}, the  underlying vector space 
$SC^+(H, J)$ depends only on $(Y, \xi)$ but the differential may depend on the filling $W$. 
This is because the differential counts Floer trajectories which can go into $W$. 
Using a stretch-the-neck argument, Bourgeois and Oancea \cite{BO}, \cite{BO2}  showed that for certain special $J$, these Floer trajectories are in bijection with punctured Floer trajectories in the symplectization $Y \times (0, \infty)$ capped off by rigid 
$J$-holomorphic disks in $\widehat{W}$; see Proposition \ref{prop: neckstretch} below for more details. The Floer trajectories in $Y\times (0,\infty)$ do not depend on $W$. Hence the differential depends on $W$ precisely because of these  J-holomorphic disks in $\widehat{W}$. 
The disks are asymptotic to Reeb orbits of $(Y, \alpha)$ and since the disks are  rigid, these orbits have degree zero, assuming the appropriate transversality results; note that these orbits have a well-defined degree since they bound disks in $W$ and hence are contractible in $W$.  
So if all orbits that are contractible in $W$ have positive degree, there are no such $J$-holomorphic discs and so all Floer trajectories stay in $\widehat{W}\backslash W = Y \times [1, \infty)$. Therefore in this case, the differential does not depend on the filling $W$ and $SH^+(W)$ is an invariant of $(Y, \alpha)$.

However, this criterion  depends on the filling $W$ and hence is not very helpful unless we already know the filling. In order to get a criterion depending only on $Y$, we restrict to \textit{$\pi_1$-injective} fillings $W$ of $Y$ for which $i_*: \pi_1(Y) \rightarrow \pi_1(W)$ is injective.
Then all Reeb orbits that are contractible in $W$ are also contractible \textit{in $Y$}; furthermore,  their degrees can be computed in $Y$. The same holds for Hamiltonian orbits that are contractible in $W$ and so the grading of $SH^+(W)$ can be determined in $Y$. 
In this case, the necessary conditions on $(Y, \xi)$ are summarized in the following definition: 
\begin{definition}\label{def: dyn_convex }
\cite{CieliebakOancea, HWZ_convex}, 
A contact manifold $(Y, \xi)$ is \textit{dynamically convex} if there is a contact form $\alpha$ for $\xi$ such that all of contractible Reeb orbits of $\alpha$ have positive degree. 
\end{definition}
The class of dynamically convex contact structures were originally introduced in dimension three in \cite{HWZ_convex} and then defined in higher dimensions in \cite{Abreu_Mac, CieliebakOancea}. By the above discussion, we have the following proposition. 
\begin{proposition}\cite{BO2, EGH}\label{prop: zeroorbits}
If $(Y, \xi)$ is dynamically convex, then  all $\pi_1$-injective Liouville fillings of $(Y, \xi)$ have isomorphic $SH^{+}$. 
\end{proposition}
\begin{remark}\label{rem: H1}\
\begin{enumerate}[leftmargin=*]
\item Weinstein fillings $W^{2n}, n \ge 2,$ are $\pi_1$-injective since $\pi_1(W, \partial W) = 0$
if $2n - 1 > n$.  
\item Assuming transversality results, rigid J-holomorphic disks in $W$ have index zero and hence should be asymptotic to degree zero Reeb orbits. So the absence of such degree zero orbits should be enough to conclude that $SH^+(W)$ is independent of the filling. However in general transversality for $J$-holomorphic disks may fail and so we cannot use this argument. On the other hand, transversality does hold for Floer trajectories in $Y \times (0, \infty)$. If all Reeb orbits have \textit{positive} degree, then  transversality of these Floer trajectories can be used to exclude the formation of J-holomorphic discs in $W$ from Floer trajectories; see
Remark 12.5 of \cite{Ritter} and 
Proposition \ref{prop: neckstretch} below for details.
If all orbits have non-zero (rather than positive) degree, this argument does not seem to work. This is why Proposition \ref{prop: zeroorbits} requires all contractible orbits to have positive rather than just non-zero degree. 
\end{enumerate}
\end{remark}

For details about the proof of Proposition \ref{prop: zeroorbits}, see the proof of Proposition \ref{prop: nice_sh_independent} below, which follows similar lines. 
We also note that the first result of this type is due to
Eliashberg, Givental, and Hofer \cite{EGH}.  
They introduced the class of
\textit{nice} contact structures, which are similar to dynamically convex structures: these contact structures admit a contact form with no contractible Reeb orbits of degree $-1, 0, 1$.  Furthermore, they showed that  cylindrical contact homology, a $S^1$-equivariant J-holomorphic curve analog of $SH^+$,  can be defined for nice contact structures  (besides certain transversality issues)  and is a contact invariant, much like Proposition \ref{prop: zeroorbits} for $SH^+$. We also note that the contact homology algebra of dynamically convex contact forms has only the trivial augmentation \cite{EGH}. 

\begin{examples}\label{ex: surfaces_more} 
In Example \ref{ex: surfaces}, we saw that $SH^+(\Sigma_g)$ depends on whether $g  = 0$ or $g \ge 1$. The Reeb orbits of the contact boundary $S^1$ are not contractible in $S^1$ so this situation is slightly different from the setup we described above. However, all orbits are contractible in $\Sigma_0 = \mathbb{C}$. One orbit $\gamma$ has degree zero \cite{BO} and J-holomorphic disks asymptotic to this orbit affect the differential for $SH^+(\Sigma_0)$. For $g \ge 1$, the map $i_*: \pi_1(S^1) \rightarrow \pi_1(\Sigma_g)$ is injective and hence there are no $J$-holomorphic disks in $\Sigma_g$ asymptotic to Reeb orbits of $S^1$. In particular, for $g \ge 1$ the differentials for  $SH^+(\Sigma_g)$ agree, which is why $SH^+(\Sigma_g)$ are (ungraded) isomorphic for all $g \ge 1$, i.e. they are all infinite-dimensional with the same generators.
\end{examples}

\subsubsection{Definition of asymptotically dynamically convex contact structures}
One difficulty with dynamically convex contact structures is that they require precise control over the degrees of \textit{all} Reeb orbits. Except in certain special situations where all Reeb orbits can be described explicitly, we do not have much control over the Reeb orbits of a generic contact form or the degrees of these orbits. For example, contact surgery creates many wild orbits with arbitrarily large action. In particular, we do not have complete control over all orbits in the contact boundary of a flexible domain. 
However, since we are only interested in algebraic invariants, such complete control is not really required. For example, Bourgeois, Eliashberg, and Ekholm \cite{BEE12} showed that the orbits created by contact surgery with bounded action are much better behaved and by letting action go to infinity, they computed the relevant algebraic invariants in terms of these well-behaved orbits. 

We will take a similar approach here. We first define asymptotically dynamically convex contact structures which 
impose weaker constraints on their Reeb orbits than dynamically convex contact structures. Then we show that these weaker constraints are still enough to prove Proposition \ref{prop: zeroorbits} for ADC contact structures, i.e. $SH^{+}$ is independent of the filling for ADC contact structures. See Proposition \ref{prop: nice_sh_independent} below.
 
\begin{definition}\label{def: semigood}
A contact manifold $(Y,\xi)$ is \textit{asymptotically dynamically convex} if there exists a sequence of non-increasing contact forms $\alpha_1 \ge \alpha_2 \ge \alpha_3 \cdots$ for $\xi$ and increasing positive numbers $D_1 < D_2 < D_3 \cdots$ going to infinity such that all elements of $\mathcal{P}^{< D_k}(Y, \alpha_k)$ have positive degree. 
\end{definition}
\begin{remark}\label{rem: semigood}\

\begin{enumerate}[leftmargin=*]
\item Any contact structure that is contactomorphic to an ADC contact structure is also ADC: if $\phi: (Y, \xi) \rightarrow (Y', \xi')$ is a contactomorphism and $(\alpha', D')$ shows that $(Y', \xi')$ is ADC, then $(\phi^*\alpha', D')$ shows that $(Y,\xi)$ is ADC. Hence the property of being ADC is a contactomorphism invariant.
\item We also assume that all elements of $\mathcal{P}^{< {D_k}}(Y, \alpha_k)$ are non-degenerate for all $k$, which will suffice for applications; in general, $\alpha_k$ do not need to be non-degenerate. 
\item
Definition \ref{def: semigood} has three conditions involving the contact forms $\alpha_k$, constants $D_k$, and elements of $\mathcal{P}^{< D_k}(Y, \alpha_k)$. It is easy to show that if we drop any one of these conditions, then the other two are trivially satisfied for \textit{any} contact structure (for example, by positively scaling the contact forms) and so the definition of ADC becomes vacuous. However, requiring all three conditions to hold is a non-trivial constraint which does not always hold since there exist contact structures that are not ADC; see Example \ref{ex: not_nice} below.
\item 
For any constant $\epsilon < 1$, we can assume the sequence of contact forms satisfy the stronger inequality 
$\epsilon \alpha_k \ge \alpha_{k+1}$ for all $k$; in particular, we will often assume that the forms in Definition \ref{def: semigood} are decreasing instead of non-increasing.  Since $D_k$ tend to infinity, by taking a subsequence, we can first assume that $D_{k+1} \ge \frac{1}{\epsilon^2}D_k$. Now let $\alpha_k' := \epsilon^k \alpha_k, 
D_k' := \epsilon^k D_k$. Then 
$\epsilon \alpha_k' \ge \alpha_{k+1}'$ and $D_{k+1}' \ge \frac{1}{\epsilon} D_k'$ (so $D_k'$ tends to infinity). By Proposition \ref{prop: easy}, 
$\mathcal{P}^{< D_k'}(Y, \alpha_k')$ and  $\mathcal{P}^{< D_k}(Y, \alpha_k)$ are in grading-preserving bijection and so all elements of 
$\mathcal{P}^{< D_k'}(Y, \alpha_k')$ have positive degree. Hence $(\alpha_k', D_k')$ satisfies all conditions in Definition \ref{def: semigood} as well as the stronger inequality $\epsilon \alpha_k' \ge \alpha_{k+1}'$ as desired. 
\end{enumerate}
\end{remark}

We also note that asymptotically dynamically convex contact structures have some similarities to contact structures with \textit{convenient dynamics} defined in \cite{KvK}. These are also defined via a sequence of contact forms and are especially well-suited for situations involving Morse-Bott contact forms. 
It is quite likely that our main results about  asymptotically dynamically convex contact structures, 
Proposition \ref{prop: nice_sh_independent} and Theorem \ref{thm: semi-surgery}, can also be proven for convenient dynamics contact structures; indeed the related results Theorems \ref{thm: MLYau}, \ref{thm: MLYauindex2} were proven in \cite{KvK} for contact structures with convenient dynamics. 

\subsection{$SH^+$ independent of filling}

Although asymptotically dynamically convex contact structures are more general than dynamically convex contact structures, 
in this section we show that Proposition \ref{prop: zeroorbits} still holds for ADC structures. 

\begin{proposition}\label{prop: nice_sh_independent}
If $(Y^{2n-1}, \xi)$ is an asymptotically dynamically convex contact structure, then all 
$\pi_1$-injective Liouville filling of $(Y, \xi)$ have isomorphic $SH^{+}$. 
\end{proposition}
\begin{remark}\label{rem:semigood}\
\begin{enumerate}[leftmargin=*]
\item Consider the more restrictive class of contact structures such that every element of $\mathcal{P}^{< D_k}(Y, \alpha_k)$ has degree \textit{greater than $1$}. For such contact structures, we can define a version of $SH^+$ that does not require a filling of $(Y, \xi)$ and hence is manifestly a contact invariant. Compare this to the definition of cylindrical contact homology proposed in Section 1.6 of \cite{JP}: 
$CH^{cyl}(Y, \xi) =
\underset{\rightarrow }{\lim} \ CH^{cyl, <D}(Y, \alpha)$, where $(Y, \alpha)$ has no contractible orbits of degree $-1,0,1$ with action less than $D$. 
Although this class is slightly more restrictive than the class of asymptotically dynamically convex contact structures,  this is not really an issue for applications except for $n = 3$. For example, we will show that all orbits of $Y$ produced by flexible surgery have degree greater than $n-3$ and hence greater than $1$ if $n \ge 4$. In particular, it is expected that for $n \ge 4$, our main result Theorem \ref{thm: semi-surgery} can be stated in terms of this more restrictive class of contact structures and hence the relevant contact manifolds do not need to have Liouville fillings. 
\item 
The set of Liouville fillings of a contact manifold (up to symplectomorphism of their completions) is a contact invariant. To see this, suppose that $\phi: (Y, \xi) \rightarrow (Y', \xi')$ is a contactomorphism and $W$ is a Liouville filling of $(Y, \xi)$. Then $W' := W \cup_\phi Y'\times [1/2, 1]$ is a Liouville filling of $(Y', \xi')$ and $W, W'$ have symplectomorphic completions. A similar construction takes $W'$ to a Liouville filling of $(Y,\xi)$ whose completion is symplectomorphic to that of $W$. Hence there is a bijection between fillings. Since Liouville domains with symplectomorphic completions have isomorphic $SH^+$, the set of (isomorphism classes of) $SH^+$ of Liouville fillings is also a contact invariant. If the contact manifold is ADC, then all Liouville fillings have isomorphic $SH^+$ and so Proposition \ref{prop: nice_sh_independent} implies that $SH^+$ of any Liouville filling 
is a contact invariant for ADC contact structures. We will use this to show that certain contact structures are exotic as in Theorems \ref{thm: inf_contact_flex}, \ref{thm: boundedinfinite}.
\end{enumerate}
\end{remark}

Before proving Proposition
\ref{prop: nice_sh_independent}, we first describe how to stretch an almost complex structure along a contact hypersurface. This procedure is called stretching-the-neck.

Let $V \subset \widehat{W}$ be a Liouville subdomain with contact boundary $(Z, \alpha_Z)$. Consider a collar of $Z$ in $V$ symplectomorphic to $( Z\times [1-\delta, 1], d(t \alpha_Z))$ for small $\delta$. Let $J \in \mathcal{J}_{std}(W)$ be cylindrical in $Z\times [1-\delta, 1]$ and
 set $J' := J|_{Z\times [1-\delta, 1]}$.  For $0 < R <1 -\delta$, we extend $J'$ to a cylindrical  almost complex structure on $Z \times [R, 1]$, which we also call $J'$. 
Let $\phi^R$ be any diffeomorphism $[R, 1] \rightarrow [1-\delta, 1]$ whose derivative equals $1$ near the boundary. Following the notation in \cite{CieliebakOancea}, we define $J^R \in \mathcal{J}_{std}(W)$ to be $(Id \times \phi^R)_*J'$ on $Z\times [1-\delta, 1]$ and $J$ outside $Z\times [1-\delta, 1]$; the fact that $\phi^R$ has derivative 1 near the boundary implies that $J^R$ is smooth.  If $J_s \in \mathcal{J}_{std}(W)$ is a homotopy that is cylindrical and $s$-independent in $Z\times [1-\delta, 1]$, then let $J_s^R$ be the homotopy obtained by applying the previous construction to $J_s$  for all $s$.

The following proposition formalizes Section \ref{sssec: sh_independent}, where we explained that
positive Reeb orbit degrees imply that $SH^+$ is independent of the filling. This proposition essentially shows that the Floer trajectories defining the various operations on $SC^+$ from Section \ref{ssec: postive_sym_hom}, \ref{ssec: transfermaps} are contained in $\widehat{W} \backslash V$ after sufficiently stretching-the-neck 
$Z = \partial V$, assuming the Reeb orbits of $\partial V$ have positive degree. As a result, these operations are independent of $V$. To suit the particular applications we have in mind, this proposition is stated using orbits satisfying a fixed action bound. 

More precisely, let $(H_s,J_s)$ be a homotopy with 
$(H_s, J_s) = (H_-, J_-)$ for $s\ll 0$ and
$(H_s, J_s) = (H_+, J_+)$ for $s \gg 0$ such that one of the following cases holds:
\begin{itemize}
\item[i)]  $W = V$ and $H_s = H_W  \in \mathcal{H}_{std}(W)$,
$J_s = J_W  \in \mathcal{J}_{std}(W)$ are $s$-independent and $(H_W,J_W)$-trajectories define the differential 
$d: SC^+(H_W, J_W) \rightarrow SC^+(H_W, J_W)$
\item[ii)] $W= V$ and $H_s \in \mathcal{H}_{std}(W), J_s \in \mathcal{J}_{std}(W)$ for all $s$ with $H_+ = H_{W,+}, H_- = H_{W,-}$, and $(H_s, J_s)$-trajectories define the continuation map $\phi_{H_s}: SC^+(H_{W,+}, J_+) \rightarrow SC^+(H_{W,-}, J_-)$
\item[iii)] $H_+ = H_{W} \in \mathcal{H}_{std}(W), J_+ = J_W \in \mathcal{J}_{std}(W)$ and  $H_- = H_{W,V} \in \mathcal{H}_{step}(W,V),$ $J_- = J_{W,V} \in \mathcal{J}_{step}(W,V)$ and
$(H_s, J_s)$-trajectories defines the continuation map 
$SC^+(H_{W}, J_W) \rightarrow SC^+(H_{W,V}, J_{W,V})$, 
which induces (after taking quotients) the transfer map 
$SC^+(H_{W}, J_W) \rightarrow 
SC^+(i(H_{W,V}), i(J_{W,V}))$. 
\end{itemize}
Furthermore, we assume that  $H_s \equiv 0$ in $Z\times 
[1-\delta, 1] \subset V$. We also take $J_s$ to be  $s$-independent and cylindrical in  $Z\times 
[1-\delta, 1+\epsilon_V] \subset V$, with $\epsilon_V$ as in Section \ref{ssec: transfermaps}, so that we can construct $J_s^R$. Finally, let $x_+, x_-$ be Hamiltonian orbits of $H_+, H_-$ respectively in the source and target of the maps induced by $(H_s, J_s)$. 

\begin{proposition}\label{prop: neckstretch}
Suppose $i_*: \pi_1(Z) \rightarrow \pi_1(W)$ is injective and all elements of $\mathcal{P}^{< D}(Z, \alpha)$ have positive degree.
If $A_{H_+}(x_+)- A_{H_-}(x_-) < D$, then there exists $R_0 \in (0, 1-\delta)$ such that for any $R \le R_0$, 
all rigid $(H_s, J_s^R)$-Floer trajectories are contained in $\widehat{W}\backslash V$. 
\end{proposition}
\begin{remark}\label{rem: neck_stretching}\
\begin{enumerate}[leftmargin=*]
\item Since $H_+, H_-$ have finitely many orbits, we can assume that 
$R_0$ is small enough so that Proposition \ref{prop: neckstretch}
holds for \textit{all} orbits of $H_+, H_-$ in the source, target respectively that satisfy the appropriate action bounds. 
\item In general $R_0$ depends not just on $W\backslash V$
and $(H_s, J_s)|_{W\backslash V}$  but also on the filling $V$ of $(Z, \alpha_Z)$ and $(H_s, J_s)|_V$. If we take a different filling $V'$, then the conclusion of Proposition \ref{prop: neckstretch} still holds but with possibly smaller $R_0$. 
However for  any $R, R' \le R_0$, all $(H_s, J_s^R)$ and $(H_s, J_s^{R'})$-trajectories in $\widehat{W}$ coincide: by Proposition \ref{prop: neckstretch}, these trajectories occur in $\widehat{W}\backslash V$ and $J_s^R = J_s^{R'}$ in this region since the stretching occurs in $Z \times [1-\delta, 1] \subset V$.
\item 
We can also stretch along $Y=\partial W$ and $Z = \partial V$  simultaneously to show that cases i), ii), iii) can be satisfied with the same almost complex structure on $\widehat{W}$.  More precisely, suppose $H_1= H_W$ defines the differential in case i), 
$H_2 = H_s$ defines the continuation map in case ii), and $H_3 = H_s'$ from $H_{W,V}$ to $H_W$ defines the transfer map in case iii). Let $J^{A,B} \in \mathcal{J}_{std}(W)$ be stretched in $Y \times [1-\delta, 1], Z\times [1-\delta, 1]$ using $\phi^A, \phi^B$ respectively. Then there exist $R_0 \in (0, 1-\delta)$ such that if $A \le R_0$ (and any $B$), all rigid 
$(H_1, J^{A,B})$-trajectories  
and $(H_2, J^{A,B})$-trajectories 
stay in $\widehat{W}\backslash W$ 
and if $B \le R_0$ (and any $A$), all rigid $(H_3, J^{A,B})$-trajectories  stay in $\widehat{W}\backslash V$.  
Here we assume that the $x_+$ is a Hamiltonian orbit of 
$H_W$ and $x_-$ is a Hamiltonian orbit of either $H_W, H_{W,-},$ or $H_V$ depending on whether we consider case i), ii) or iii). 
\end{enumerate}
\end{remark}

\begin{proof}[Proof of Proposition \ref{prop: neckstretch}]
We will only prove case iii) since cases i), ii) are very similar. In fact, case i) was essentially proven in Section 5.2 of \cite{BO} and our proof closely follows the argument there.

For case iii), $H_s$ is a decreasing homotopy with 
$H_- = H_{W,V}$ and $H_+ = H_W$ and $H_s = h_s(t)$ is increasing in $t$ near $Z$. Also, $x_+$ is an orbit of $H_{W}$ located near $Y = \partial W$ and corresponds to a parametrized Reeb orbit of $(Y, \alpha) = \partial W$ since the source of the map we care about is $SC^+(H_W, J_W)$. Similarly $x_-$ is a type II orbit of $H_{W,V}$ located near $Z = \partial V$ and corresponds to a parametrized Reeb orbit of $(Z, \alpha_Z) = \partial V$ since the target is $SC^{II}(H_{W,V}, J_{W,V}) \cong SC^+(i(H_{W,V}), i(J_{W,V}))$. Furthermore, $|x_+|-|x_-| = 0$ because the transfer map is defined by counts of rigid parametrized Floer trajectories. 

Now suppose that there exists a sequence $R_k \rightarrow 0$ and $(H_s, J_s^{R_k})$-Floer cylinders $u_k \in  \mathcal{M}_0(x_-,x_+; H_s, J_s^{R_k})$ such that $u_k$ is not contained in $\widehat{W}\backslash V$. 
Then by the SFT compactness theorem \cite{BEHWZ_03}, $u_k$ converges to a broken Floer building. Our intermediate goal is to show that this building is just a single punctured $(H_s, J_s)$-Floer trajectory in $\widehat{W\backslash V}$ capped off by $J$-holomorphic buildings in $\widehat{V}$.  

A priori, this broken Floer building consists of Floer components in the completion $\widehat{W\backslash V} = \widehat{W}\backslash V \cup Z \times (0,1]$ of the Liouville cobordism $W\backslash V$, $J$-holomorphic components in (possibly several levels of) $Z \times (0, \infty)$ and $\widehat{V}$, and Floer components in $\widehat{V}$.  
To see that the usual SFT compactness theorem applies, note that even though $(H_s, J_s)$ depends on $s$, $H_s \equiv 0$ and $J_s$ is cylindrical and $s$-independent in $Z\times [1-\delta, 1]$, where the neck-stretching occurs. Furthermore, since $H_s$ is decreasing, these  trajectories do not escape to infinity by `no escape' Lemma \ref{lem: maximal_principle_param}. Finally, the Hofer energy of $u_k$ in $Y \times [1-\delta, 1]$ is bounded by $A_{H_+}(x_+) - A_{H_-}(x_-)  +  \int (\frac{\partial H_s}{\partial s} )(u_k) ds \wedge dt$ and hence uniformly bounded by $A_{H_+}(x_+) - A_{H_-}(x_-)$ since $H_s$ is decreasing; see the proof of Lemma 2.4 in \cite{CieliebakOancea}.

We now explain this broken building in more detail. As we stretch the almost complex structure on $\widehat{W}$, the almost complex structure on the domain $\mathbb{R}\times S^1$ and the Hamiltonian $H_s$ on $\widehat{W}$ are fixed. As a result, the only breaking can be bubbling off a $J$-holomorphic component at some finite point in the domain or Floer breaking at infinity in the domain. In particular, only one Floer component is a $(H_s, J_s)$-trajectory while the rest are $(H_+, J_+)$ or $(H_-, J_-)$ trajectories.
If a Floer component happens to map to $Z \times (0, \infty)$, then it is $J'$-holomorphic since $H_s$ is constant there. Also, all Floer components are punctured Floer trajectories that are asymptotic to Reeb orbits at the punctures and either to Hamiltonian orbits or Reeb orbits at the two cylindrical ends; in particular, a given component is asymptotic to at most two Hamiltonian orbits.  See Figure \ref{fig: breaking} for a possible broken Floer building and the corresponding broken domain.  In this figure, we consider the general case when $H_s$ is non-constant in $V$, which is why there is a $H_+$-trajectory in $\widehat{V}$; in the rest of the proof, $H_s$ is always zero in $V$.  
 \begin{figure}
     \centering
     \includegraphics[scale=0.4]{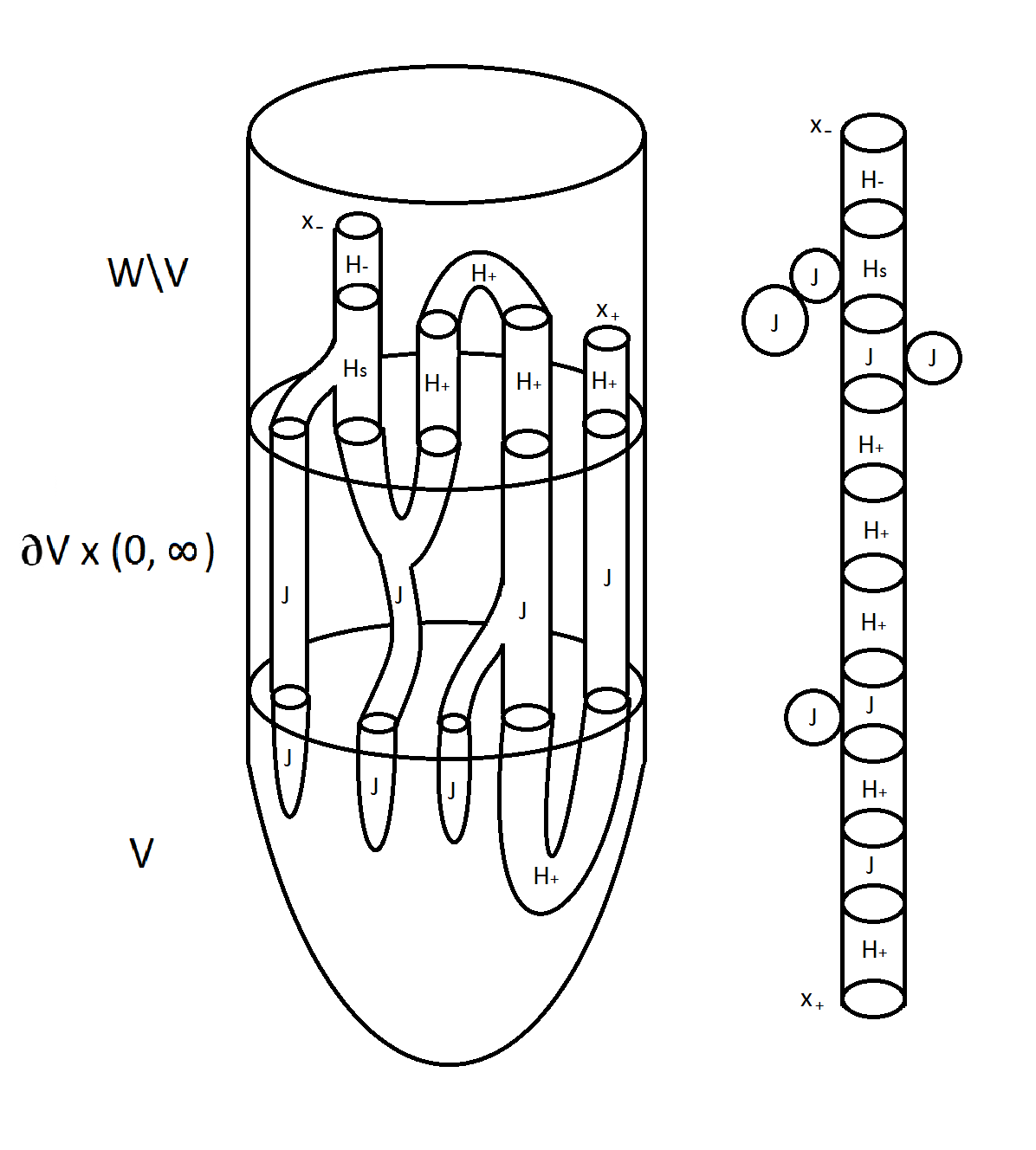}
     \caption{A broken Floer building and the corresponding broken domain
     }
     \label{fig: breaking}
 \end{figure}

For the proof, we shall be only concerned with the top part of the building in $\widehat{W\backslash V}$, which now describe explicitly.  First, note that we can extend $(H_s, J_s)$ from $\widehat{W}\backslash V$ to $\widehat{W\backslash V}$ since $(H_s, J_s)$ is constant near $\partial_- (W\backslash V) = Z$. To describe the broken Floer building, we use the notation from Section 5.2 of \cite{BO}, which describes the building when $H_s$ is $s$-independent. The broken Floer building in $\widehat{W\backslash V}$ corresponds to a tuple $(x_- = x_0, \tilde{u}_1, x_1,  \cdots, x_{m-1}, \tilde{u}_m, x_+= x_m), m \ge 1,$ with the following properties:
\begin{itemize}
\item $x_- = x_0, x_1, \cdots, x_j \in \widehat{W}\backslash V$ are Hamiltonian orbits of $H_- = H_{W,V}$ for some $0 \le j \le m-1$ and  $x_{j+1}, \cdots, x_+ =x_m\in \widehat{W}\backslash V$ are Hamiltonian orbits of $H_+ = H_W$. 
\item $\tilde{u}_i: \Sigma_i\backslash\{z_i^1, \cdots, z_i^{k_i}\} \rightarrow \widehat{W\backslash V}$, 
where $\Sigma_i =\mathbb{R} \times S^1$ or 
$\Sigma_i =\mathbb{R} \times S^1 \coprod
 \mathbb{R} \times S^1:= \overline{\Sigma}_i
  \coprod \underline{\Sigma}_i$, is a punctured $(H_-, J_-),(H_+, J_+)$-Floer cylinder  in $\widehat{W\backslash V}$ for  $i \le j, i \ge j+2$ respectively. 
  If $i = j+1$ and $\Sigma_{j+1}$ is connected, 
  $\tilde{u}_{j+1}$ is a punctured $(H_s, J_s)$-trajectory; if $\Sigma_{j+1}$ is disconnected, then either $\tilde{u}_{j+1}|_{\overline{\Sigma}_{j+1}}, \tilde{u}_{j+1}|_{\underline{\Sigma}_{j+1}}$ are  $(H_s, J_s), (H_-, J_-)$ - trajectories or 
   $\tilde{u}_{j+1}|_{\overline{\Sigma}_{j+1}}, \tilde{u}_{j+1}|_{\underline{\Sigma}_{j+1}}$ are $(H_+, J_+), (H_s, J_s)$-trajectories.
   \item If $\Sigma_i = \mathbb{R} \times S^1$, then 
$\underset{s\rightarrow - \infty}{\lim}
 \tilde{u}_i(s, \cdot) = x_{i-1}, 
\underset{s\rightarrow +\infty}{\lim}
 \tilde{u}_i(s, \cdot) = x_{i}$. 
\item If $\Sigma_i =\overline{\Sigma}_i
  \coprod \underline{\Sigma}_i$, then 
$\underset{s\rightarrow - \infty}{\lim} \tilde{u}_i|_{\underline{\Sigma}_i}(s, \cdot) = x_{i-1}, 
\underset{s\rightarrow + \infty}{\lim} \tilde{u}_i|_{\overline{\Sigma}_i}(s, \cdot) = x_{i}$. 
Furthermore, $\tilde{u_i}|_{\overline{\Sigma}_i}, 
\tilde{u_i}|_{\underline{\Sigma}_i}$ are asymptotic to Reeb orbits $\overline{\gamma}_i, \underline{\gamma}_i$ 
of $(Z, \alpha_Z)$ at their negative, positive ends respectively; in particular,  $\tilde{u_i}|_{\underline{\Sigma}_i}(s, \cdot) \subset Z\times (0,1] \subset \widehat{W\backslash V}$ for 
$s \gg 0$ and 
$\underset{s\rightarrow \infty}{\lim}
 t\circ \tilde{u_i}|_{\underline{\Sigma}_i}(s, \cdot)  = 0$,  
where $t$ is the second coordinate on $Z\times(0,1]$,  and similarly for $\tilde{u_i}|_{\overline{\Sigma}_i}$. 
\item Near the punctures $z_i^{k}$, $\tilde{u}_i$ is $J'$-holomorphic and asymptotic to Reeb orbits $\gamma_i^k$ of $(Z, \alpha_Z)$.
 \end{itemize}
For example in Figure \ref{fig: breaking}, $m = 4$ and $\Sigma_1, \Sigma_3$ are connected while $\Sigma_2, \Sigma_4$ are disconnected; also $j = 1$ so that $x_1$ is a Hamiltonian orbit of $H_-$ while $x_2$ is a Hamiltonian orbit of $H_+$.  It is also possible that the top part of the building has no $(H_s, J_s)$-component; this can happen when this component occurs in $\partial V \times (0, \infty)$ or in $\widehat{V}$. In this case, the proof is easier and so we focus on the situation described above when a $(H_s, J_s)$-component occurs in the top part of the building. 

A key part of the proof of this proposition involves restricting the possible broken Floer buildings, which a priori can be quite complicated as Figure \ref{fig: breaking} shows. As we noted before, we want to show that the building is actually 
just a punctured $(H_s, J_s)$-Floer trajectory in $\widehat{W\backslash V}$ capped off by $J$-holomorphic buildings in $\widehat{V}$.  Our first goal is to show that all $\Sigma_i$ are connected, which we do in two steps. 

First, we show that the Hamiltonian orbits $x_- = x_0, x_1, \cdots, x_j$ of $H_- = H_{W,V}$ are located near $Z = \partial V$  and 
correspond to parametrized Reeb orbits of $(Z, \alpha_Z)$; similarly, the orbits $x_{j+1}, \cdots, x_m = x_+$ of 
$H_+ =H_{W}$ are located near $Y= \partial W$ and correspond to parametrized Reeb orbits of 
$(Y, \alpha_Y)$. This already holds for $x_0$ and $x_m$ by assumption. Note that $A_{H_{W,V}}(x_{i-1}) < A_{H_{W,V}}(x_i)$ for $i \le j$ since, after gluing $J'$-holomorphic components in $Z \times (0, \infty)$ and Floer components in $\widehat{V}$ to the punctures, we can consider $\tilde{u}_i$ as a $(H_{W,V}, J_{W,V})$-Floer trajectory in $\widehat{W}$. Similarly, 
$A_{H_W}(x_{i-1}) < A_{H_W}(x_i)$ for $i \ge j+2$. Finally, 
$A_{H_{W,V}}(x_{j}) < A_{H_W}(x_{j+1})$  since $H_s$ is decreasing. From Section \ref{ssec: postive_sym_hom}, we know $A_{H_{W,V}}(x_-) > 0$ since $x_-$ is a II orbit. So by the previous inequalities, $A_{H}(x_i) > 0$ for all $i$ (where $H$ is $H_{W,V}$ or $H_W$ depending on whether $i \le j$ or 
$i \ge j+1$). In particular, for $i\le j$, $x_i$ cannot be a I, IV, or V orbit of $H_{W,V}$ which have non-positive action; we need $\epsilon_V \ge \frac{s_W}{s_V}$ for V orbits to have negative action but this already needed for the definition of the transfer map. Similarly, for $i \ge j+1$, $x_i$ cannot be a Morse critical point of $H_W$, which has non-positive action. Finally, we need to show that for $i \le j$, $x_i$ cannot be a III orbit of $H_{W,V}$.  Suppose a III orbit exists in the broken Floer building. Then there is a punctured $(H_{W,V}, J_{W,V})$-Floer trajectory from the last III orbit to a II orbit; to see this, note that the I, IV,V orbits have already been ruled out and we know that $x_-$ is a II orbit. But as we saw in Section \ref{ssec: transfermaps}, this is impossible by Lemma \ref{lem: rise_above} and Lemma \ref{lem: maximal_principle}.

Now we show that all Floer components are connected, i.e. $\Sigma_i = \mathbb{R} \times S^1$ for all $i$.
Suppose $\Sigma_i = \overline{\Sigma}_i \coprod 
\underline{\Sigma}_i$ is disconnected. We do only the case when $i = j+1$ and  $\tilde{u}_{j+1}|_{\overline{\Sigma}_{j+1}}, \tilde{u}_{j+1}|_{\underline{\Sigma}_{j+1}}$ are $(H_{W}, J_{W}), (H_s, J_s)$-trajectories since the other cases are similar. The punctured Floer trajectory 
$\tilde{u}_{j+1}|_{\underline{\Sigma}_{j+1}}$ is asymptotic to $x_j$ at its negative end and to the orbit $\underline{\gamma}_{j+1}$ at its positive end. By the discussion in the previous paragraph, the Hamiltonian orbit $x_{j}$ of $H_- = H_{W,V}$ is located near $Z = \partial V$. Note that $t|_{x_{j}}$ is constant because near $Z$, $H_{W,V}$ equals the function $h(t)$ and so $x_-$ lies in a level set of $t$. In particular, 
$\underset{s\rightarrow -\infty}{\lim} t\circ \tilde{u}_{j+1}|_{\underline{\Sigma}_{j+1}(s, \cdot)}  = t(x_{j})$. On the other hand, 
$\underset{s\rightarrow +\infty}{\lim} t\circ \tilde{u}_{j+1}|_{\underline{\Sigma}_{j+1}(s, \cdot)}  = 0$ since $\tilde{u}_{j+1}$ is asymptotic to  $\underline{\gamma}_{j+1}$ at its positive end. Since $t(x_{j}) > 0$, there exists $p \in \underline{\Sigma}_{j+1}$ such that $t(\tilde{u}_i|_{\underline{\Sigma}_{j+1}}(p)) < t(x_{j})$. At the same time, there also exists $q \in  \underline{\Sigma}_{j+1}$ such that $t(\tilde{u}_{j+1}(q)) > t(x_{j})$. To see this, consider $s \ll 0$ so that $H_s = H_-$. As we previously noted, $x_{j}$ occurs near $Z = \partial V$ and $H_- = h(t)$ is convex increasing here. So we can  apply  Lemma \ref{lem: rise_above} to $H_-$, which shows that  $\tilde{u}_{j+1}$ must rise above $x_{j}$
and hence there exists such $q$. The existence of $p$ and $q$ imply that $\tilde{u}_{j+1}$ violates the parametrized `no escape' Lemma \ref{lem: maximal_principle_param}; this lemma applies here because $H_s = h_s(t)$ is a decreasing homotopy that is increasing in $t$
near $Z\times \{1\}$, where $x_{j}$ is located, and $h_s(t)|_{Z\times \{1\}} = 0$  and $J|_{Z\times \{1\}}$ are $s$-independent. This is a contradiction and so $\Sigma_{j+1}$ must be connected. 
 
 Recall that the punctured Floer trajectories $\tilde{u}_i$ is asymptotic to the Reeb orbits $\gamma_i^1, \cdots, \gamma_i^{k_i}$ at its punctures $z_i^1, \cdots, z_i^{k_i}$. The original sequence $u_k \in  \mathcal{M}_0(x_-,x_+; H_s, J_s^{R_k})$  is not contained in $\widehat{W}\backslash V$ and so $k_i \ge 1$ for some $i$. Also, because
 the entire glued Floer building has genus 0, all
  $\gamma_i^k$ are contractible in $W$. Because $\pi_1(Z) \rightarrow \pi_1(W)$ is injective, $\gamma_i^k$ are actually contractible in $Z$.

We now show that all $\gamma_i^k$ have positive degree. If \textit{all} contractible Reeb orbits of $(Z, \alpha_Z)$ have positive degree, this follows trivially. However, 
the assumption in this proposition is that only elements of $\mathcal{P}^{< D}(Z, \alpha_Z)$ have positive degree. 
Therefore, to show that $\gamma_i^k$ have positive degree, it suffices to show that $A(\gamma_i^k) < D$. 
First, we note that $A_H(x_{i}) - A_H(x_{i-1}) - \sum_{k=1}^{k_i} A_H(\gamma_i^k)$ equals the Floer energy of $\tilde{u}_i$ for $i \ne j+1$ and hence is positive;  here $H = H_{W,V}$ or $H = H_W$ depending on whether $i \le j$ or $i \ge j+2$. We view $\gamma_i^k$ as a loop in $\widehat{W}$ lying in $Z = \partial V$ (see the proof of Lemma 2.4 in \cite{CieliebakOancea}). Since $H \equiv 0$ in $Z \times [1-\delta, 1]$, we have $A_H(\gamma_i^k) = A(\gamma_i^k) > 0$ and so $A_{H}(x_{i}) - A_{H}(x_{i-1}) > 0$ for $i \ne j+1$. 
Similarly,
 $A_{H_{W}}(x_{j+1}) - A_{H_{W,V}}(x_{j}) - \sum_{k=1}^{k_{j+1}} A(\gamma_i^k) + \int (\frac{\partial H_s}{\partial s})(u) ds \wedge dt$ equals the Floer energy of $\tilde{u}_{j+1}$ and hence is positive. Since $\frac{\partial H_s}{\partial s} \le 0$  and  $A(\gamma_i^k)> 0$,  then $A_{H_W}(x_{j+1}) - A_{H_{W,V}}(x_{j}) > 0$ as well.
 By assumption of the proposition, we have $A_{H_+}(x_+)- A_{H_-}(x_-)  < D$. 
Since $\sum_{i=1}^m (A_H(x_{i}) - A_H(x_{i-1})) = A_{H_+}(x_+)- A_{H_-}(x_-) < D$ and $A_H(x_{i}) - A_H(x_{i-1}) > 0$ for all $i$, we have  $A_H(x_{i}) - A_H(x_{i-1}) < D$ for all $i$. 
This implies $ \sum_{k=1}^{k_i} A(\gamma_i^k) < A_H(x_{i}) - A_H(x_{i-1}) < D$. Finally, since $A(\gamma_i^k) > 0$, we have $A(\gamma_i^k) < D$ as desired and so all $\gamma_i^k$ have positive degree.

The moduli space of punctured Floer cylinders asymptotic at their positive end, negative end, and punctures to  $x_i, x_{i-1},\gamma_i^{k}$ respectively has virtual dimension
\begin{equation}\label{eqn: dim_moduli}
|x_i| - |x_{i-1}| - \sum_{k = 1}^{k_i} |\gamma_i^{k}|.
\end{equation}
See \cite{BO}, Section 5.2. For $i \ne j+1$, we can quotient out by a free $\mathbb{R}$-action, in which case the resulting moduli space has virtual dimension one less than in Equation \ref{eqn: dim_moduli}. 
Just like in the non-punctured situation, transversality for these punctured Floer moduli spaces 
is guaranteed by using a time-dependent perturbation of the almost complex structure; see \cite{BO2} and \cite{Ritter}. Hence Equation \ref{eqn: dim_moduli} must be non-negative or positive, depending on whether $i = j+1$ or $i \ne j+1$. In particular, 
\begin{equation}\label{eqn: moduli_dim2}
|x_+| - |x_-| - \sum_{i=1}^m \sum_{k = 1}^{k_i} |\gamma_i^{k}| = \sum_{i=1}^m \left(|x_i| - |x_{i-1}| - \sum_{k = 1}^{k_i} |\gamma_i^{k}|\right) \ge m-1.
\end{equation}
By assumption, $|x_+| - |x_-| = 0$. Furthermore, 
 $|\gamma_i^k|$ are all positive and $m \ge 1$.
  So the left-hand-side of Equation \ref{eqn: moduli_dim2} is negative while the right-hand-side is non-negative.
 This is a contradiction and hence there exists $R_0$ such that all $(H_s,J_s^R)$-Floer trajectories stay in $\widehat{W}\backslash V$ for any $R \le R_0$. 
\end{proof}

Using Proposition \ref{prop: neckstretch}, we now prove Proposition \ref{prop: nice_sh_independent}. This proposition involves some basic homological algebra and is similar to the proof in \cite{BEE12} that linearized contact homology can be computed via the \textit{essential complex}.

\begin{proof}[Proof of Proposition 
\ref{prop: nice_sh_independent}]
Let $(Y, \xi)$ be an asymptotically dynamically convex contact structure. Let $\alpha_1 \ge \alpha_2 \ge \alpha_3 \cdots$ be contact forms for $(Y, \xi)$ and $D_1 < D_2 < D_3 \cdots $ a sequence of real numbers tending to infinity such that all elements of $\mathcal{P}^{< D_i}(Y, \alpha_i)$ have positive degree as in Definition \ref{def: semigood}. By Remark \ref{rem: semigood}, we can assume $\frac{1}{16}\alpha_i \ge \alpha_{i+1}$ for all $i$. 

Let $(W, \lambda)$ be any Liouville filling of $(Y, \xi)$. Let $\alpha = \lambda|_Y$ be the contact form on $Y$ induced by $(W, \lambda)$. Since $W$ is Liouville, we can symplectically embed the whole symplectization $(Y\times (0, \infty), r \alpha)$ of $(Y, \alpha)$ into the completion $\widehat{W}$. Recall from 
Section \ref{ssec: transfermaps} that any contact form $\beta$ for $(Y, \xi)$ gives rise to a graphical submanifold $Y_\beta$ of $Y\times (0, \infty)$ defined by $r \alpha|_{Y_{\beta}} = \beta$. 
Now let $Y_i  := Y_{\alpha_i} \subset Y\times (0, \infty) \subset \widehat{W}$ be the graphical submanifolds corresponding to the forms $\alpha_i$. 
The neighborhoods $(Y_i \times [1/2, 2], r\alpha_i) \subset (Y\times (0, \infty), r\alpha)$ of $Y_i$ are Liouville subcobordisms. In fact since $\frac{1}{16}\alpha_i \ge \alpha_{i+1}$, these subcobordisms are disjoint. For example, $(Y_{\alpha/16} \times [1/2, 2], r \alpha/16)$ is a Liouville subcobordism of $(Y\times (0, \infty), r\alpha)$ as $(Y\times [1/32, 1/8], r\alpha)$, which is disjoint from $(Y \times [1/2, 2], r\alpha)$ since $1/8 < 1/2$. 
Let  $W_i$ be the graphical subdomain equal to the bounded component of $\widehat{W}\backslash Y_i$. So $W_i$ is a Liouville filling of $(Y_i, \alpha_i)$ and all the completions $\widehat{W}_i$ coincide with $\widehat{W}$. Since  $\frac{1}{16}\alpha_i \ge \alpha_{i+1}$,  $W_i \supset W_{i+1}$ and the collars of $W_i$ are disjoint, i.e. 
$W_i \backslash (Y_i \times [1/2,1]) \supset W_{i+1} \cup (Y_{i+1} \times [1, 2])$. 

We now give an outline of the proof of Proposition \ref{prop: nice_sh_independent}. The idea is to construct certain Hamiltonians $H_i^{D_j} \in \mathcal{H}_{std}(W_i)$ and almost complex structures $J_i^{R_i} \in \mathcal{J}_{std}(W_i)$ (and the appropriate homotopies between them) that give rise to the following commutative diagram: 
\begin{equation}\label{eqn: CD_large}
\begin{CD}
SH^{+}(W_1; H_1^{D_1}, J_1^{R_1}) @>>> SH^{+}(W_1; H_1^{D_2}, J_1^{R_1})
@>>> \cdots @. \qquad @.  \underset{\longrightarrow j}{\lim}\; SH^{+}(W_1; H_1^{D_j}, J_1^{R_1}) \\ 
 @VVV  @VVV  @. \qquad @.  @VVV  \\
SH^{+}(W_2; H_2^{D_1}, J_2^{R_2}) @>>>  SH^{+}(W_2; H_2^{D_2}, J_2^{R_2})
 @>>> \cdots @. \qquad @.  \underset{\longrightarrow j}{\lim}\; SH^{+}(W_2; H_2^{D_j}, J_2^{R_2})\\
 @VVV  @VVV  @. \qquad @.  @VVV  \\
 \vdots  @. \vdots @. @. \qquad @.  \vdots \\
\underset{\longrightarrow i}{\lim}
   \; SH^{+}(W_i; H_i^{D_1}, J_i^{R_i}) @>>>  
 \underset{\longrightarrow i}{\lim} 
  \; SH^{+}(W_i; H_i^{D_2}, J_i^{R_i})
 @>>> \cdots
\end{CD}
\end{equation}
See \cite{BEE12} for a similar diagram involving linearized contact homology. The homotopies are chosen so that horizontal maps are continuation maps and the vertical maps are transfer maps. Furthermore, the $H_i^{D_j}$ are taken so that the horizontal limit
$\underset{\longrightarrow j}{\lim}\; SH^{+}(W_i; H_i^{D_j}, J_i^{R_i})$ is isomorphic to $SH^+(W_i)$ and the right-most vertical maps are all isomorphisms. 
Since direct limits commute, we can switch the $i$ and $j$ indices. As a result, we have
\begin{equation}\label{eqn: directlimit3}
SH^{+}(W) \cong 
\lim_{\longrightarrow i} SH^+(W_i)
\cong \lim_{\longrightarrow i} \lim_{\longrightarrow j} 
SH^{+}(W_i; H_i^{D_j}, J_i^{R_i}) 
\cong
\lim_{\longrightarrow j} \lim_{\longrightarrow i}
 SH^{+}(W_i; H_i^{D_j}, J_i^{R_i}).
\end{equation} 
Since direct limits do not depend on the first 
finitely many terms,
 $ \underset{\longrightarrow i}{\lim}
 SH^{+}(W_i; H_i^{D_j}, J_i^{R_i})$ is isomorphic to  
 $\underset{ {\substack{   \longrightarrow i\\
                   i \ge j }}}{\lim}
      SH^{+}(W_i; H_i^{D_j}, J_i^{R_i})$.  
Combining this with Equation \ref{eqn: directlimit3}, we get
\begin{equation}\label{eqn: directlimit4}
SH^+(W) \cong \underset{\longrightarrow j}{\lim}
\lim_{\substack{   \longrightarrow i\\
                   i \ge j }} 
      SH^{+}(W_i; H_i^{D_j}, J_i^{R_i}).
\end{equation}

The key point is that since $(Y, \xi)$ is ADC, we can construct $H_i^{D_j}, J_i^{R_i}$ so that the right-hand-side of Equation \ref{eqn: directlimit4} is independent of the filling $W$. More precisely, because all elements of $\mathcal{P}^{<D_j}(\alpha_i)$ have positive degree for $j \le i$, we can stretch-the-neck as in Proposition \ref{prop: neckstretch}  and take $H_{i}^{D_i}, J_i^{R_i}$ so that the Floer trajectories defining the differential on $SC^{+}(W_i; H_i^{D_j}, J_i^{R_i})$ and the continuation maps 
$SH^{+}(W_i; H_i^{D_j}, J_i^{R_i}) \rightarrow 
SH^{+}(W_i; H_i^{D_{j+1}}, J_i^{R_i})$ stay in $\widehat{W_i}\backslash W_i$ for $j \le i$. 
Similarly, we take the appropriate homotopies so that the trajectories defining the transfer map
$SH^{+}(W_i; H_i^{D_j}, J_i^{R_i}) 
\rightarrow 
SH^{+}(W_{i+1}; H_{i+1}^{D_j}, J_{i+1}^{R_{i+1}})$
stay in $\widehat{W}_{i}\backslash W_{i+1}$  for $j \le i$; more precisely, this map is chain-homotopic to a map defined by trajectories that stay in  $\widehat{W}_{i}\backslash W_{i+1}$ and so on homology, this map is defined by such trajectories. 

If $(Y, \xi)$ has another Liouville filling $V$, we repeat this procedure. We first take graphical subdomains $V_i$ of $\widehat{V}$ so that $\partial V_i = Y_i$. This implies that $\widehat{V}\backslash V_i
=\widehat{W}\backslash W_i$ and  $V_i \backslash V_{i+1} = W_i \backslash W_{i+1}$. Then we construct similar Hamiltonians $\overline{H}_i^{D_j} \in \mathcal{H}_{std}(V_i) $ 
and almost complex structures
$ \overline{J}_i^{R_i}\in \mathcal{J}_{std}(V_i)$
 that agree with $H_i^{D_j}, J_i^{R_i}$ in these regions. For $j \le i$, $(\overline{H}_i^{D_j},  \overline{J}_i^{R_i})$-Floer trajectories and  $(H_i^{D_j},  J_i^{R_i})$-Floer trajectories both occur in $\widehat{W}_i \backslash W_i$ and hence coincide.   
 This shows that $SH^+(V_i; \overline{H}_i^{D_j}, \overline{J}_i^{R_i})$ and 
$SH^+(W_i; H_i^{D_j}, J_i^{R_i})$ can be canonically identified for $j \le i$; the underlying chain complexes have the same generators for all $i,j$ and the same differential for $j \le i$.  Similarly,  the continuation maps and the transfers maps for $W,V$ can also be canonically identified for $j \le i$.
 As a result, the commutative diagrams 
 $\{SH^+(V_i; \overline{H}_i^{D_j}, \overline{J}_i^{R_i}),j \le i \},$
 $\{SH^+(W_i; H_i^{D_j}, J_i^{R_i}), j \le i \}$,
 are identical. Therefore we can conclude that 
$$
\lim_{\longrightarrow j} 
\lim_{\substack{   \longrightarrow i\\
                   i \ge j }} 
      SH^{+}(V_i; \overline{H}_i^{D_j}, \overline{J}_i^{R_i}) 
  =
  \lim_{\longrightarrow j} 
  \lim_{\substack{   \longrightarrow i\\
                     i \ge j }} 
        SH^{+}(W_i; H_i^{D_j}, J_i^{R_i})
$$  
and so by Equation \ref{eqn: directlimit3}, $SH^+(V) \cong SH^+(W)$ as desired.  
\\

To complete the proof, we need to  construct the special Hamiltonians and almost complex structures discussed above and show that they satisfy all the claimed properties. We do this just for $W$ since the procedure for $V$ is exactly the same.

Let $H_i^{D_j} \in \mathcal{H}_{std}(W_i)$ have slope $D_j$ in $\widehat{W_i}\backslash W_i$; in particular, $H_i^{D^j} \equiv 0$ in $Y_i \times [1/2,1] \subset W_i$. 
All non-constant Hamiltonian orbits of $H_i^{D_j}$ lie in $\widehat{W_i}\backslash W_i$
 and  have  action which is positive and bounded by $D_j$; indeed since $H_i^{D_j}\equiv 0$ in $W_i$, 
the Hamiltonian action of an Hamiltonian orbit is close to the action of the corresponding Reeb orbit, which is positive and bounded by $D_j$. In particular, any non-constant orbits $x_-, x_+$ of $H_i^{D_j}$ satisfy $A_H(x_+) - A_H(x_-) < D_j$. 
Also let $H_i^{D_j, D_{j+1}}$ be a decreasing homotopy of Hamiltonian functions in $\mathcal{H}_{std}(W_i)$ from  $H_- = H_i^{D_{j+1}}$ to $H_+ = H_i^{D_j}$; again we have
$H_i^{D_j, D_{j+1}}\equiv 0$ on $Y_i \times [1/2,1] \subset W_i$.
 
Let $J_i \in \mathcal{J}_{std}(W_i)$ be cylindrical  in $Y_i \times [1/2,1]$. For $R \in (0,1/2)$, let 
$J_i^R$ denote the result of stretching $J_i$ in $Y_i \times [1/2,1]$ using a diffeomorphism $\phi^R$ from  $[R,1]$ to $[1/2,1]$. 
Since $(Y, \xi)$ is ADC, all elements of $\mathcal{P}^{< D_i}(Y_i, \alpha_i)$  have positive degree;  in fact all elements of $\mathcal{P}^{< D_j}(Y_i, \alpha_i)$  have positive degree for all $j \le i$ because  the $D$'s are increasing. As discussed in the previous paragraph,  $A_H(x_+) - A_H(x_-) < D_j$ for non-constant orbits of $x_-, x_+$ of $H_i^{D_j}$. 
So for $j \le i$ we can apply case i) of Proposition \ref{prop: neckstretch} to $V = W = W_i$ and $H_s = H_i^{D_j}$ and conclude that there exists $R_i^j \in (0,1/2)$ such that all 
 $(H_i^{D_j}, J_i^{R_i^j})$-Floer trajectories in $\widehat{W_i}$ between non-constant orbits are contained in $\widehat{W_i}\backslash W_i$; in fact this holds
for $(H_i^{D_j}, J_i^{R})$-Floer trajectories for any $R \le R_i^j$. Note that we need $W$ to be a $\pi_1$-injective filling of $Y$ to use Proposition \ref{prop: neckstretch}. Similarly,  by case ii) of Proposition \ref{prop: neckstretch}, there exists $R_{i}^{j,j+1} \in (0, 1/2)$ such  that the analogous claim holds for all  $(H_i^{D_j, D_{j+1}}, J_i^{R_i^{j,j+1}})$-Floer trajectories for $j \le i$; here we consider $J_i^{R_i^{j,j+1}}$ as an $s$-independent homotopy.  
 Let $R_i : = \min\{R_i^j, R_i^{j,j+1}, j \le i\} >0$.
 Then for $j \le i$, all  $(H_i^{D_j}, J_i^{R_i})$
 and $(H_i^{D_j, D_{j+1}}, J_i^{R_i})$-Floer trajectories in $\widehat{W_i}$ between non-constant orbits are contained in $\widehat{W_i}\backslash W_i$; note that these are precisely the trajectories that define the differential for
 $SC^{+}(W_i; H_i^{D_j}, J_i^{R_i})$ and the horizontal continuation map 
 $SH^{+}(W_i; H_i^{D_j}, J_i^{R_i})
 \rightarrow 
 SH^{+}(W_i; H_i^{D_{j+1}}, J_i^{R_i})$. 
 Since the $D_j$'s tend to infinity, the sequence $SH^{+}(W_i; H_i^{D_j}, J_i^{R_i})$ is cofinal in $j$  and so for all $i$
$$
\lim_{\rightarrow j} SH^{+}(W_i; H_i^{D_j}, J_i^{R_i})
= SH^{+}(W_i),
$$
where the direct limit maps are induced by the homotopies $(H_i^{D_j, D_{j+1}}, J_i^{R_i})$.

Next we define the vertical transfer maps 
 $SH^{+}(W_i; H_i^{D_j}, J_i^{R_i}) \rightarrow SH^{+}(W_{i+1}; H_{i+1}^{D_j}, J_{i+1}^{R_{i+1}})$. 
Let $S_{i}^{D_j} \in \mathcal{H}_{step}(W_i, W_{i+1})$
such that $S_i^{D_j}$ has slope $D_j$ in $\widehat{W}_{i} \backslash W_{i} = Y_i \times [1, \infty)$, slope 
$D_{j+1}$ in $Y_{i+1} \times [1,2]$, and is constant otherwise; as a result, $S_i^{D_j}$ and $H_i^{D_j}$ differ by the constant $D_{i+1}$ in $(Y_i \times [1, \infty), r\alpha_i)$. 
In the notation of Section \ref{ssec: transfermaps}, $i(S_i^{D_j})= H_{i+1}^{D_j}$, $s_{W_i} = s_{W_{i+1}} = D_j$, and $\epsilon_{W_{i+1}} = 1$. 
Note that we can take $\epsilon_{W_{i+1}} = 1$ because the collar $Y_{i+1} \times [1, 2]$ is disjoint from the collar $Y_{i} \times [1/2, 1]$. 
For $A, B \in (0, 1/2)$, let 
$J_i^{A, B} \in \mathcal{J}_{step}(W_i, W_{i+1})$ be defined as in part iii) of Remark \ref{rem: neck_stretching}, i.e.  $J_i^{A,B}$ is stretched in $Y_{i}\times [1/2,1], Y_{i+1}\times [1/2,1]$ using $\phi^A, \phi^B$ respectively. Again this is possible because $Y_{i+1}\times [1,2]$ is disjoint from $Y_i \times [1/2, 1]$. 
Let $H_{i,i+1}^{D_j}$ be a decreasing homotopy from $S_i^{D_j}$ to $H_i^{D_j}$ such that $H_{i, i+1}^{D_j} \equiv 0$ in $W_{i+1}$ for all $s\in \mathbb{R}$ and $H_{i, i+1}^{D_j} = h_s(t)$ in $Y_{i+1} \times [1,2]$ is increasing in $t$, where $t$ is the coordinate on the $[1,2]$-component. 

 Because $1 = \epsilon_{W_{i+1}} \ge \frac{s_{W_i}}{s_{W_{i+1}}} = 1$, by the discussion in Section \ref{ssec: transfermaps} 
 $(H_{i,i+1}^{D_j}, J_i^{A, B})$-Floer trajectories
define the transfer map 
\begin{equation}\label{eqn: SHtransfer}
SH^{+}(W_i; H_i^{D_j}, 
J_i^{A, B}) 
\rightarrow 
SH^{+}
(W_{i+1}; i(S_i^{D_j}), i(J_i^{A, B}) ) = 
SH^{+}(W_{i+1}; H_{i+1}^{D_j}, J_{i+1}^B ).
\end{equation}
Here we consider $J_i^{A, B} \in J_{step}(W_i,W_{i+1})$ as an $s$-independent homotopy.
By part iii) of Remark \ref{rem: neck_stretching}), for $j \le i$
there exists $R_i' \in (0,1/2)$ 
such that if $A\le R_i'$ (and any $B$), all rigid 
$(H_i^{D_j}, J_i^{A,B})$-trajectories defining the differential for $SC^+(H_i^{D_j}, J_i^{A, B})$ stay in $\widehat{W}_{i} \backslash W_i$ and if $B \le R_i'$ (and any $A$), all  rigid 
$(H_{i,i+1}^{D_j}, J_i^{A, B})$-Floer trajectories defining the transfer map Equation \ref{eqn: SHtransfer} stay in $\widehat{W_i}\backslash W_{i+1}$.  We can also assume that 
$R_i' \le \min\{R_i, R_{i+1}\}$. 
Then the transfer map  $SH^{+}(W_i; H_i^{D_j}, J_i^{R_i}) \rightarrow SH^{+}(W_{i+1}; H_{i+1}^{D_j}, J_{i+1}^{R_{i+1}})$ will be the composition of the following maps: 
\begin{equation}\label{eqn: CD}
\begin{CD}
SH^{+}(W_i; H_i^{D_j}, J_i^{R_i}) @>>> 
SH^{+}(W_i; H_i^{D_j}, J_i^{R_i'}) @>>> 
SH^{+}(W_i; H_i^{D_j}, J_i^{R_i', R_i'}) \\ 
 @. @.  @VVV  \\
SH^{+}(W_{i+1}; H_{i+1}^{D_j}, J_{i+1}^{R_{i+1}}) @<<< 
SH^{+}(W_{i+1}; H_{i+1}^{D_j}, J_{i+1}^{R_{i}'}) @=
SH^{+}(W_{i+1}; i(S_{i}^{D_j}), i(J_{i}^{R_{i}', R_i'}))  \\
\end{CD}
\end{equation}
Here  the right vertical map is given by Equation \ref{eqn: SHtransfer} and the horizontal maps are induced by homotopies of almost complex structures.

Now we show that this transfer map 
is defined by trajectories that stay in $\widehat{W}_i \backslash W_{i+1}$ if $j \le i$. First we note that by choice of $R_i'$, the right vertical map in Equation \ref{eqn: CD}
is defined by trajectories in $\widehat{W}_i \backslash W_{i+1}$. 
Hence it will suffice to show that the horizontal maps in Equation \ref{eqn: CD} 
are identity maps between identical vector spaces. We will focus mainly on the upper left horizontal map $SH^{+}(W_i; H_i^{D_j}, J_i^{R_i}) \rightarrow SH^{+}(W_i; H_i^{D_j}, J_i^{R_i'})$ since the situation for the other horizontal maps is similar. 
First, we note that the domain and target of this map are identical vector spaces. They have the same Hamiltonian $H_i^{D_j}$ and so the underlying chain complexes have the same generators. Since $R_i' \le R_i$, these chain complexes also have the same differentials defined by trajectories in the symplectization $\widehat{W}_i \backslash W_i$. 

We now show that $SH^{+}(W_i; H_i^{D_j}, J_i^{R_i}) \rightarrow SH^{+}(W_i; H_i^{D_j}, J_i^{R_i'})$ is the identity map between identical vector spaces. This map is clearly an isomorphism since it is induced by a homotopy of almost complex structures with fixed Hamiltonians. To show that this map is the identity, for each $t \in [0,1]$ we consider the homotopy of  homotopies $J_\cdot^t$ from 
$J_i^{R_i}$ to $J_i^{(1-t)R_i + t R_i'}$ obtained by stretching $J_i^{R_i}$ to $J_i^{(1-t)R_i +tR_i'}$ in $Y_i \times [1/2,1]$. More precisely, let 
$f_t: (-\infty, +\infty) \rightarrow [R_i', R_i]$ be a family of smooth non-decreasing functions that are constant near $\pm \infty$ and $f_t(-\infty) = (1-t)R_i +t R_i', f_t(+\infty) = R_i$; let $J_s^t := J_i^{f_t(s)}$. 
Then $(H_i^{D_j}, J_s^1)$-trajectories induces the desired map
$SH^{+}(W_i; H_i^{D_j}, J_i^{R_i}) \rightarrow SH^{+}(W_i; H_i^{D_j}, J_i^{R_i'})$
since $J_s^1$ is a homotopy from $J_i^{R_i}$ to $J_{i}^{R_i'}$. Also, 
$J_s^0 = J_i^{R_i}$ is the constant homotopy and so index 0 $(H_i^{D_j}, J_s^0)$-Floer trajectories are constant in $s$; hence these trajectories induce the identity map 
$SH^{+}(W_i; H_i^{D_j}, J_i^{R_i}) \rightarrow SH^{+}(W_i; H_i^{D_j}, J_i^{R_i})$. Now consider $\cup_{t \in [0,1]} \mathcal{M}(x_-, x_+; H_i^{D_j}, J_s^t)$. The boundary of the 1-dimensional part of this moduli space includes the 0-dimensional parts of 
$\mathcal{M}(x_-, x_+; H_i^{D_j}, J_s^0)$ and  
$\mathcal{M}(x_-, x_+; H_i^{D_j}, J_s^1)$. It also includes broken Floer trajectories made up of index $(-1)$ elements of $\mathcal{M}(H_i^{D_j}, J_s^t)$ for some $t \in [0,1]$ and index 1 elements of $\mathcal{M}(H_i^{D_j}, J_{+\infty}^t) = \mathcal{M}(H_i^{D_j}, J_i^{R_i})$ or 
$\mathcal{M}(H_i^{D_j}, J_{-\infty}^t) = \mathcal{M}(H_i^{D_j}, J_i^{(1-t)R_i+ tR_i'})$, i.e. trajectories that are rigid up to $\mathbb{R}$-translation. Since $(1-t)R_i + tR_i' \le R_i$ (because $R_i' \le R_i$), 
rigid $(H_i^{D_j}, J_i^{(1-t)R_i+ tR_i'})$-Floer trajectories stay in the symplectization $\widehat{W}_i\backslash W_i$ for all $t \in [0,1]$. Since $J_i^{(1-t)R_i+ tR_i'}$ and $J_i^{R_i}$ coincide in the symplectization, these trajectories coincide with $(H_i^{D_j}, J_i^{R_i})$-trajectories, which induces the differential on $SC^+(H_i^{D_j}, J_i^{R_i})$. In particular, this implies that the count of $(H_i^{D_j}, J_s^t)$-Floer trajectories with index $-1$ induces a chain homotopy between the identity map and the desired map $SH^{+}(W_i; H_i^{D_j}, J_i^{R_i}) \rightarrow SH^{+}(W_i; H_i^{D_j}, J_i^{R_i'})$; see Remark \ref{rem: chainhomotopy} below. Thus, on the homology level, this map is just the identity map.

\begin{remark}\label{rem: chainhomotopy}\
Suppose $J_s, J_s'$ are two homotopies from $J_-$ to $J_+$; they induce the chain maps
$\phi_{J_s}, \phi_{J_s'}: SC^+(H_+, J_+) \rightarrow SC^+(H_-, J_-)$. Then the above proof that the map $SH^{+}(W_i; H_i^{D_j}, J_i^{R_i}) \rightarrow SH^{+}(W_i; H_i^{D_j}, J_i^{R_i'})$ is the identity is similar to the proof that $\phi_{J_s}, \phi_{J_s'}$ are chain-homotopic. This proof involves taking a homotopy of homotopies $J_{\cdot}^t$ that interpolates between $J_s, J_s'$ and has $J_{\pm \infty}^t = J_\pm$ for all $t$. In our setting, the homotopy of homotopies does not have this second property since $J_{-\infty}^t = J_i^{(1-t)R_i + tR_i'}$ depends on $t$. However, this is not an issue in our case because the index 1 components of broken trajectories that can occur in the compactification of the 1-dimensional moduli space are independent of $t$, i.e. rigid 
$(H_i^{D_j}, J_i^{(1-t)R_i+ tR_i'})$-Floer trajectories stay in $\widehat{W}_i \backslash W_i$  for all $t$ since we have sufficiently stretched the neck and hence agree with $(H_i^{D_j}, J_i^{R_i})$ -trajectories. 
\end{remark}

To show that the upper right map 
$SH^{+}(W_i; H_i^{D_j}, J_i^{R_i'}) \rightarrow SH^{+}(W_i; H_i^{D_j}, J_i^{R_i', R_i'})$ in Equation \ref{eqn: CD} is the identity map, we consider a similar homotopy of homotopies $J_\cdot^t$ from 
$J_i^{R_i'}$ to $J_{i}^{R_i', tR_i'}$ and use the fact that 
$(H_i^{D^j}, J_i^{R_i', B})$-trajectories stay in the symplectization $\widehat{W}_i\backslash W_i$ for any $B$. 
Similarly, the bottom left horizontal maps in Equation \ref{eqn: CD} is also the identity map. Therefore the composition of the maps in Equation \ref{eqn: CD} equals just the right vertical map, which is defined by trajectories that stay in $\widehat{W}_i \backslash W_{i+1}$. Furthermore, this composition coincides with the usual transfer map on homology. The transfer map 
$
SH^{+}(W_i; H_i^{D_j}, J_i^{R_i} )
\rightarrow 
SH^{+}(W_{i+1}; H_{i+1}^{D_j}, J_{i+1}^{R_{i+1}})
$
is defined using \textit{any} homotopy from $(H_i^{D_j}, J_i^{R_i})$ to 
$(S_{i}^{D_j}, J_i^{R_i', R_{i+1}})$ and hence is independent of the choice of homotopy. Here we have constructed a particular homotopy as a composition of several homotopies, namely the homotopies inducing the maps in Equation \ref{eqn: CD}. 

Now we show that the commutative diagram in Equation \ref{eqn: CD_large} has all the desired properties. As we noted before in the construction of $H_i^{D_j}$ and $H_{i}^{D_j, D_{j+1}}$, the horizontal limits 
$\underset{\rightarrow j}{\lim} SH^{+}(W_i; H_i^{D_j}, J_i^{R_i})
$
are isomorphic to $SH^{+}(W_i)$ for all $i$. Also, the induced right-most vertical maps $SH^+(W_i)\rightarrow SH^+(W_{i+1})$ in the commutative diagram are isomorphisms. 
To see this, note that they are transfer maps induced by the inclusion $W_{i+1} \subset W_i$. Since $W_{i}$ are all graphical subdomains of $\widehat{W}$, these maps are isomorphisms as discussed in Section \ref{sssec: graphical}. Finally, by choice of $R_i$,  for $j \le i$ the differential for $SC^{+}(W_i; H_i^{D_j}, J_i^{R_i})$ and the continuation maps $SH^{+}(W_i; H_i^{D_j}, J_i^{R_i}) \rightarrow SH^{+}(W_i; H_i^{D_{j+1}}, J_i^{R_i})$ are defined by trajectories that stay in $\widehat{W}_i\backslash W_i$. 
Similarly, for $j \le i$ the trajectories defining the vertical transfer maps stay in $\widehat{W}_{i} \backslash W_{i+1}$. So for such choices, the commutative diagram has all the desired properties.

If $V$ is another Liouville filling of $(Y, \xi)$, we repeat this procedure using Hamiltonians $\overline{H}_i^{D_j}$
and almost complex structures $\overline{J}_i^{\overline{R}_i}$ on $\widehat{V}$ that agree with those for $\widehat{W}$ at infinity as we explained in the outline. Note that the stretching-factors $R_i$ provided by Proposition \ref{prop: neckstretch} can depend on the filling; see 2) of Remark \ref{rem: neck_stretching}. For example, to ensure that 
 $(H_{i,i+1}^{D_j}, J_i^{A,B})$-trajectories and 
$(\overline{H}_{i,i+1}^{D_j}, \overline{J}_i^{\overline{A},\overline{B}})$-trajectories stay in $\widehat{W}_i\backslash W_{i+1} = \widehat{V}_i \backslash V_{i+1}$ 
\textit{and} 
$J_i^{A,B}, \overline{J}_i^{\overline{A},\overline{B}}$ coincide in 
$\widehat{W}_i\backslash W_{i+1} = \widehat{V}_i \backslash V_{i+1}$, we need to have 
$A = \overline{A} \le \min\{R_i',\overline{R}_i'\}$
and 
$B = \overline{B} \le \min\{R_i',\overline{R}_i'\}$.  Therefore, we should construct the Hamiltonians and almost complex structures for $W$ and $V$ simultaneously. Once this is done, the trajectories stay at infinity and hence coincide. Then $SH^+(W) \cong SH^+(V)$ as explained in the outline. 
\end{proof}

\begin{remark}\label{rem: positivedeg}\
\begin{enumerate}[leftmargin=*]
\item The direct limit $\underset{\rightarrow i}{\lim}
 SH^{+}(W_i; H_i^{D_i}, J_i^{R_i})$ of the \textit{diagonal} maps in the commutative diagram Equation \ref{eqn: CD_large}
 is isomorphic to 
 $\underset{\rightarrow i}{\lim} \underset{\rightarrow j}{\lim}
 SH^{+} (W_i; H_i^{D_j}, J_i^{R_i})$
and hence also equals $SH^{+}(W)$. In fact, this diagonal direct limit can be taken as the definition of $SH^{+}(W)$ even when $(Y, \xi)$ is not ADC; a similar statement is mentioned in \cite{JP} for  cylindrical contact homology. 
\item
The proof of Proposition \ref{prop: nice_sh_independent} requires us to embed every $(Y, \alpha_k)$ into $\widehat{W}$. In general, the sequence $\alpha_k$  might tend to zero in the sense that $\underset{k\rightarrow \infty}{\lim}
\alpha_k(x)/\alpha_1(x) = 0$ for some $x \in Y$. In this case, we need to embed the whole negative symplectization $(Y\times (0,1), r\alpha)$ into $\widehat{W}$, which only works because $W$ is a Liouville domain. This is another reason why we consider only Liouville domains in this paper (as opposed to the weaker notion of symplectic domains with contact type boundary).
\end{enumerate}
\end{remark}

\subsubsection{Homological obstruction}\label{sssec: homological_obstruction}
We now give a basic obstruction for a contact structure to be ADC and give an example of a contact structure that is not ADC. 

Recall that we initially started with a time-independent Hamiltonian $H$ whose Hamiltonian orbits came in $S^1$-families corresponding to parametrizations of Reeb orbits of $\alpha$ (below a fixed action). As explained in Section \ref{subsec: symhom}, after a small time-dependent perturbation of $H$, these families breaks up into finitely many non-degenerate orbits. For example, 
using a Morse function on $S^1$ with two critical points, the $S^1$-family corresponding to Reeb orbit $\gamma$ 
breaks up into two Hamiltonian orbits 
$\gamma_m, \gamma_M$ corresponding to 
the minimum and maximum of this Morse function. Then \begin{eqnarray}
\mu_{CZ}(\gamma_m) &=& \mu_{CZ}(\gamma)\\
\mu_{CZ}(\gamma_M) &=& \mu_{CZ}(\gamma)+1.
\end{eqnarray}
See \cite{BO}, p. 10. Here the Conley-Zehnder indices $\mu_{CZ}(\gamma_m), \mu_{CZ}(\gamma_M)$ are computed with respect to the linearized Hamiltonian flow on $TW$ while $\mu_{CZ}(\gamma)$ is computed using the linearized Reeb flow on $\xi$. 

If $(Y, \xi)$ is an ADC contact structure, then Reeb orbits $\gamma \in \mathcal{P}^{< D_k}(Y, \alpha_k)$ have positive degree $|\gamma| = \mu_{CZ}(\gamma)+n-3$. Hence $\mu_{CZ}(\gamma) \ge 4-n$ and so the corresponding Hamiltonian orbits $\gamma_m, \gamma_M$ satisfy 
$\mu_{CZ}(\gamma_m), \mu_{CZ}(\gamma_M) \ge 4-n$. As a result, the generators of the right-hand-side of Equation \ref{eqn: directlimit4} have degree at least $4-n$. 
So any Liouville filling $W$ of an ADC contact structure has $SH^{+}_k(W) = 0$ for $k \le 3-n$.   This observation can be used to show that certain contact manifolds are not ADC. 
\begin{examples}\label{ex: not_nice}
Suppose $W$ is a Liouville filling of $(Y, \xi)$ with $SH(W)=0$ and $H^k(W; \mathbb{Z})\ne 0$ for $k = 2n-2$ or $k = 2n-1$; note that $W^{2n}$ can be Weinstein only if $n =2$. Then by Equation \ref{eqn: linformula}, either 
$SH_{3-n}^{+}(W) \cong H^{2n-2}(W; \mathbb{Z})$ or $SH_{2-n}^{+}(W) \cong H^{2n-1}(W; \mathbb{Z})$ is non-zero and so $(Y, \xi)$ is not ADC. More concretely, if $W^4$ is Weinstein with $H^{2}(W; \mathbb{Z}) \ne 0$ and all critical handles attached along stabilized Legendrians, then $SH(W) = 0$ and so $\partial W^4$ is not ADC. 
\end{examples}

Therefore,  the non-vanishing of $SH^{+}_k(W)$ for $k \le 3-n$ is a basic homological obstruction for a contact structure to be asymptotically dynamically convex. Note that if $(Y, \xi)$ has a flexible Weinstein filling $W$, then $SH^{+}_k(W) = 0$ for $k \le 0$ by Remark \ref{rem: flexibledegs}. Since $n \ge 3$ for flexible domains, there is no homological obstruction for $(Y, \xi)$ to be ADC (at least from the filling $W$). 
The purpose of this paper is to show that the contact boundaries of flexible domains are in fact ADC; see Theorem \ref{thm: semi-surgery} below.

\subsection{Effect of contact surgery}\label{ssec: prop_nice_contact}
Here we state our main result that asymptotically dynamically convex contact structures are preserved under flexible surgery. We will later use this to prove that the contact boundary of a flexible domain is ADC; see Corollary \ref{cor: flexsemigood}. 

We first discuss subcritical surgery.
Although not all contact structures are ADC as we saw in Example \ref{ex: not_nice}, M.-L.Yau \cite{MLYau} proved that contact manifolds $(Y^{2n-1}, \xi)$ with subcritical fillings and $n \ge 2$  are ADC; this is precisely why these contact manifolds have linearized contact homology independent of the filling as we stated in the Introduction. More precisely, Yau showed that the new Reeb orbits (up to high action) that appear after subcritical surgery occur in the belt sphere of the subcritical handle and have positive degree; see Proposition \ref{prop: MLYau_action} below. This degree positivity essentially shows that ADC contact structures are preserved under subcritical contact surgery, although this result was not stated explicitly in \cite{MLYau}. We will prove this result in Section \ref{sec: surgerysemigood}.

\begin{theorem}\cite{MLYau}\label{thm: MLYau}
If $(Y_-^{2n-1},\xi_-), n \ge 2,$ is an asymptotically dynamically convex contact structure and $(Y_+, \xi_+)$ is the result of index $k\ne 2$ subcritical contact surgery on $(Y_-,\xi_-)$,  then $(Y_+, \xi_+)$ is also asymptotically dynamically convex. 
\end{theorem}
\begin{remark}\
\begin{enumerate}[leftmargin=*]
\item
Since $k \ne 2$, 
$c_1(Y_-, \xi_-) = 0$ implies 
$c_1(Y_+, \xi_+) = 0$ by Proposition \ref{prop: c1equivalence} and so it makes sense to say that $(Y_+, \xi_+)$ is ADC.
\item The non-degeneracy condition mentioned in Remark 
\ref{rem: semigood} is preserved under surgery. 
\item 
Lemma 5.18 of \cite{KvK} shows that contact structures with convenient dynamics are also preserved under subcritical contact surgery. 
\end{enumerate}
\end{remark}
The $k =2$ case is a bit more complicated. For example, if 
the isotropic attaching circle $\Lambda^1$ is contractible in $Y_-$, then $c_1(Y_-, \xi_-) = 0$ does not necessarily imply $c_1(Y_+, \xi_+) = 0$. In this case, it does not make sense to ask whether $(Y_+, \xi_+)$ is ADC.
However, there is a unique framing of the conformal symplectic normal bundle of $\Lambda^1$ that makes $c_1(Y_+, \xi_+) = 0$. 
On the other hand, if $k = 2$ and 
$\Lambda^1$ is non-contractible in $Y_-$
(or more precisely is non-torsion in $\pi_1(Y_-)$), then $c_1(Y_+, \xi_+)  = 0$. However Theorem \ref{thm: MLYau} still does not always hold. To see this, suppose there is a Reeb orbit $\gamma_- \subset (Y_-, \xi_-)$ in the same free homotopy class as $\Lambda^1$. Since $\gamma_-$ is a torsion-free element of $\pi_1(Y_-)$, there is a trivialization of the canonical bundle of $Y_-$ such that $\mu_{CZ}(\gamma_-)$ is an arbitrary integer; in particular we can choose a trivialization such that $\mu_{CZ}(\gamma_-) \le - n+3$ and so  $|\gamma_-| \le 0$.  Now choose a framing of the conformal symplectic normal bundle of $\Lambda^1$ compatible with the chosen trivialization of the canonical bundle of $(Y_-, \xi_-)$ and suppose that 
$(Y_+, \xi_+)$ is the result of surgery on $\Lambda^1$ with this framing. We can assume that $\gamma_- \cap \Lambda^1 = \varnothing$ and so $\gamma_-$ corresponds to an orbit $\gamma_+ \subset (Y_+, \xi_+)$. Then $\gamma_+$ is contractible in $Y_+$ via a parallel copy of the core of the Weinstein 2-handle in $Y_+$. Using this disk, we have $|\gamma_+| = |\gamma_-| \le 0 $ and so $(Y_+, \xi_+)$ is not ADC. Besides these topological issues, Theorem \ref{thm: MLYau} holds for the $k = 2$ case as well. The following result will be proven in Section \ref{sec: surgerysemigood}.

\begin{theorem}\cite{MLYau}\label{thm: MLYauindex2}
Let $(Y_-,\xi_-)$ be an asymptotically dynamically convex contact structure with $(\alpha_k, D_k)$ as in Definition \ref{def: semigood} and $(Y_+, \xi_+)$ be the result of index $2$ contact surgery on $\Lambda^1 \subset Y_-$. Suppose $c_1(Y_+, \xi_+) = 0$ and either $\Lambda^1$ or all orbits of $(Y_-, \alpha_k)$ with action less than $D_k$ are contractible for all $k$. Then $(Y_+, \xi_+)$ is also asymptotically dynamically convex. 
\end{theorem}

The main purpose of the paper is to generalize Theorems \ref{thm: MLYau}, \ref{thm: MLYauindex2} to the flexible case. 
We will prove the following theorem in Section \ref{sec: surgerysemigood} based on results in Section \ref{sec: reeb_chord_loose}.

\begin{theorem}\label{thm: semi-surgery}
If $(Y_-,\xi_-)$ is an asymptotically dynamically convex contact structure and $(Y_+, \xi_+)$ is the result of flexible contact surgery, then 
$(Y_+, \xi_+)$ is also asymptotically dynamically convex.
\end{theorem}
\begin{remark}\label{rem: semi-surgery}
By Proposition \ref{prop: c1equivalence}, 
 $c_1(Y_-,\xi_-) = 0$ implies $c_1(Y_+, \xi_+) = 0$ so it makes sense to say that $(Y_+, \xi_+)$ is ADC. Also, if the contact surgery is done along the loose Legendrian sphere
 $\Lambda^{n-1} \subset (Y^{2n-1}_-, \xi_-)$, then by Remark 
 \ref{rem: c1_lambda} below $c_1(Y_-, \xi_-) = 0$ and $n \ge 3$ imply that $c_1(Y_-, \Lambda) = 0$, which we need for the proof of Theorem \ref{thm: semi-surgery}.
\end{remark}
Theorems
\ref{thm: MLYau}, \ref{thm: MLYauindex2}, and \ref{thm: semi-surgery} \textit{do not} hold for the 
dynamically convex  contact structures from Definition 
\ref{def: dyn_convex }. As we noted before, contact surgery creates wild orbits with large action whose degrees we cannot control and so we cannot say anything about \textit{all} orbits. However by filtering out these wild orbits, we can prove Theorems
\ref{thm: MLYau}, \ref{thm: MLYauindex2}, and \ref{thm: semi-surgery} for ADC contact structures.

\section{Applications}\label{sec: applications}

Assuming Theorems \ref{thm: MLYau}, \ref{thm: MLYauindex2}, and \ref{thm: semi-surgery}, we now prove the applications stated in the Introduction.

Let $W$ be a flexible domain. As we mentioned in Section \ref{ssec: independencelin}, there is no homological obstruction for $(Y, \xi) = \partial W$ to be ADC (at least with respect to the filling $W$). 
Since $W$ satisfies an h-principle, one might hope that it is possible to remove all Reeb orbits of $(Y, \xi)$ that are homologically unnecessary (for fixed action) and thereby show that $(Y, \xi)$ is ADC. In the next corollary, which follows almost immediately from Theorem \ref{thm: semi-surgery}, we prove that this is indeed the case.

\begin{corollary}\label{cor: flexsemigood}
If $(Y^{2n-1},\xi), n \ge 3,$ has a flexible filling, then $(Y, \xi)$ is asymptotically dynamically convex. 
\end{corollary}
\begin{proof}
Let $W$ be a flexible filling of $(Y, \xi)$. So $(Y, \xi)$ is obtained  by successively doing subcritical or flexible contact surgery to $(S^{2n-1}, \xi_{std})$. Let $(Y_1, \xi_1), \cdots , (Y_k, \xi_k)$ denote the resulting contact manifolds with  $(Y_1, \xi_1) = (S^{2n-1}, \xi_{std})$ and $(Y_k, \xi_k) = (Y, \xi)$. 
 It can be shown explicitly that $(S^{2n-1}, \xi_{std})$ is ADC; see \cite{BEE12}.
By  repeatedly applying Theorem \ref{thm: MLYau} for subcritical surgery and Theorem \ref{thm: semi-surgery} for flexible surgery, we see that each $(Y_i, \xi_i)$ is also ADC. Note that $c_1(W) = 0$ implies that $c_1(Y_i) = 0$ for all $i$ so it makes sense to say that $(Y_i, \xi_i)$  is ADC. In particular, $(Y,\xi)$ is ADC. 

In the presence of 2-handles, we need to use Theorem \ref{thm: MLYauindex2}, which requires that all Reeb orbits are contractible.  After possibly a Weinstein homotopy of $(W, \lambda, \phi)$, we can assume that $\phi$ is self-indexing, i.e. critical points $p$ of $\phi$ with index $k$ have $\phi(p) = k$. 
Hence all 2-handles are attached to the contact manifold $(S^{2n-2} \times S^1 \sharp \cdots \sharp S^{2n-2} \times S^1, \xi_{std}) $ obtained by doing 1-surgery to $(S^{2n-1}, \xi_{std})$. By Proposition \ref{prop: MLYau_action}, all orbits of $(S^{2n-2} \times S^1 \sharp \cdots \sharp S^{2n-2} \times S^1, \xi_{std})$ are located in the $S^{2n-2}$ belt spheres of these 1-handles and hence are contractible since $\pi_1(S^{2n-2}) =0$ for $n \ge 3$. Therefore Theorem \ref{thm: MLYauindex2} applies in this situation. 
\end{proof}

The analogs of Corollary \ref{cor: flexsemigood} and Theorem \ref{thm: semi-surgery} do not hold for $n=2$ if the filling of $(Y^3, \xi)$ has critical index handles. As in Example \ref{ex: not_nice}, consider a Weinstein manifold $W^4$ such that  $H^2(W; \mathbb{Z}) \ne 0$ and all critical 2-handles are attached along stabilized 1-dimensional Legendrians. Then $\partial W$ is not ADC because the homological obstruction
from Section \ref{sssec: homological_obstruction} does not vanish. On the other hand, in higher dimensions, stabilized Legendrians are loose and so $W$ is the 4-dimensional analog of a flexible Weinstein domain. However  since stabilized 1-dimensional Legendrians do not satisfy an h-principle, $W$ does not satisfy an h-principle either and so  is not called flexible. However, if $(Y^3, \xi)$ has a \textit{subcritical} filling, i.e. $(Y^3, \xi) = (S^2 \times S^1 \sharp \cdots \sharp S^2 \times S^1, \xi_{std})$, then $(Y^3, \xi)$ is ADC by Theorem \ref{thm: MLYau}. 
For $n = 1$, these results  fail even if the filling is subcritical.  As we saw in Example \ref{ex: surfaces_more}, $S^1$ has a subcritical filling $\Sigma_0 = \mathbb{C}$ but also has Reeb orbits of degree zero and $SH^+$ depends on the filling.  

Corollary \ref{cor: flexsemigood} can be used to prove the following stronger version of Theorem \ref{thm: main}.
\begin{corollary}\label{thm: mainstronger}
If $(Y^{2n-1}, \xi)$ has a flexible filling $W^{2n}$, then all $\pi_1$-injective Liouville fillings $X$ of $(Y, \xi)$ with $SH(X) = 0$ satisfy $
H^*(X; \mathbb{Z}) \cong H^*(W; \mathbb{Z})$.
\end{corollary}

\begin{proof} 
By Corollary \ref{cor: flexsemigood} and Proposition \ref{prop: nice_sh_independent}, $SH_*^{+}(X) \cong SH_*^{+}(W)$.
Since $SH(W)$ vanishes by \cite{BEE12} and 
$SH(X)$ vanishes by assumption, we have
$SH_k^{+}(X) \cong 
H^{n-k+1}(X; \mathbb{Z})$ and
$SH_k^{+}(W) \cong 
H^{n-k+1}(W; \mathbb{Z})$
for all $k$ by Proposition \ref{prop: shcomputation}. So $H^{n-k+1}(X; \mathbb{Z}) \cong H^{n-k+1}(W; \mathbb{Z})$ for all $k$. 
\end{proof}

\begin{proof}[Proof of Theorem \ref{thm: main}]
Flexible Weinstein $W$ have vanishing 
$SH$ by \cite{BEE12} and are $\pi_1$-injective by Remark \ref{rem:semigood}. 
Hence Corollary \ref{thm: mainstronger} applies. 
\end{proof}

\begin{proof}[Proof of Corollary \ref{cor: example}]
The first statement follows from the fact that all flexible 2-connected Weinstein domains with boundary $Y$ are of the form $W_n$ with its unique flexible structure, for some $n$. The second statement uses Theorem \ref{thm: main}: since $(Y, \xi_n)$ has a flexible filling $W_n$ and $c_1(Y, \xi_n) = 0$, it remembers the cohomology of all its flexible fillings. As we explained before the statement of this corollary, the cohomology determines the symplectomorphism type for these examples.
\end{proof}

\begin{proof}[Proof of Corollary \ref{cor: displaceable_flexible_diffeomorphism}]

Suppose that $W_{flex}$ is displaceable in its completion $\widehat{W}_{flex} = W_{flex} \cup Y \times [1, \infty)$, i.e. there exists an embedding $\phi$ of $\widehat{W}_{flex}$
smoothly isotopic to the standard inclusion 
$W_{flex} 
\hookrightarrow \widehat{W}_{flex}$ such that 
$i(W_{flex}) \cap \phi(W_{flex}) = \emptyset$. Concretely, we will assume that $\phi(W_{	flex}) \subset W_{flex}^R := W_{flex} \cup
 Y \times [1,R]$ for some $R > 1$; hence 
$\phi(W_{	flex}) \subset Y\times [1,R] \subset W_{flex}^R$. 

Now suppose that $X$ is another flexible filling of $(Y, \xi)$  and consider the cut-and-pasted domain 
$X' := X \cup (W_{flex}^R \backslash i(W_{flex})) =  X \cup Y \times [1,R] \cong X$. 
Let $C: = X' \backslash \phi(W_{flex}) = (W_{flex}^R  \backslash ( i(W_{flex}) \cup \phi(W_{flex})))\cup X$ be the cobordism between 
$\partial_+ C = \partial W_{flex}^R$  and $\partial_- C = \partial \phi(W_{flex})$. Note that $C \cup \phi(W_{flex}) = X'\cong X$. We will show that $(C, \partial_- C)$ is an h-cobordism and hence an concordance since $\pi_1(Y) \cong \pi_1(W_{flex})= 0$ and $\dim Y \ge 5$. This will imply that $X \cong \phi(W_{flex}) \cup C$
is diffeomorphic to $\phi(W_{flex}) \cong W_{flex}$ as desired.   

By excision, $H_*(C, \partial_- C) \cong 
H_*(X', \phi(W_{flex}))$ and hence it suffices to show that the latter group vanishes. 
If $k \ge n+1$,  $H_k(X') \cong H_k(\phi(W_{flex})) = 0$ by Theorem \ref{thm: main} and so 
$H_k(X', \phi(W_{flex})) = 0$ for $k \ge n +2$.  Now we consider the case $k \le n-1$. 
We will use the commutative diagram in Figure \ref{diagram: large}. Here the main row is the long exact sequence of $(X', \phi(W_{flex}))$ and the main  column is the long exact sequence of $(X', Y\times [1,R])$. 
\begin{equation}\label{diagram: large}
\begin{CD}
@. @.  H_{k-1}(Y \times [1, R]) @. @.\\
@. @.  @AAA @. @.\\
@. @.  H_k(X', Y\times [1,R]) @. @.\\
@. @.  @AAA  @. @.\\
H_{k+1} (X', \phi(W_{flex}) ) @>>>
H_k(\phi(W_{flex} ) ) @>>> H_k(X') @>>>
H_k(X', \phi(W_{flex}))  \\
@. @VVV   @AAA  @. @.\\
@.  H_k(Y \times [1, R]) @= H_k(Y \times [1,R]) @. @.  \\
@. @VVV   @AAA @.\\
@. H_k(W_{flex}^R)  @. H_{k+1}(X', Y\times [1,R])
\end{CD}
\end{equation}
As seen in Figure \ref{diagram: large}, the 
map $H_k(\phi(W_{flex}))\rightarrow H_k(X')$ can be factored as 
$H_k(\phi(W_{flex}) ) \rightarrow H_k(Y \times [1,R]) \rightarrow H_k(X')$, both induced by inclusions. 
The second map $H_k(Y \times [1,R]) \rightarrow H_k(X')$ is an isomorphism for $k \le n-2$ by the long exact sequence of $(X', Y \times [1,R])$ and the fact that $H^{2n-k}(X') = H^{2n-k}(X) = 0$ for $k \le n - 1$. 
The first map $H_k(\phi(W_{flex}) ) \rightarrow H_k(Y \times [1,R])$
can be composed with the map $H_k(Y \times [1,R])  \rightarrow H_k(W_{flex}^R)$ induced by the inclusion $Y\times [1,R] \subset W_{flex}^R$. This composition 
 is an isomorphism since $\phi$ is smoothly isotopic to standard inclusion $i: W_{flex} \hookrightarrow W_{flex}^R$. The map $H_k(Y \times [1,R]) \rightarrow H_k(W_{flex}^R)$ is an isomorphism for $k \le n -2$  by the long exact sequence of $(W_{flex}^R, Y \times [1,R])$ and the fact that $H^{2n-k}(W_{flex}^R) = 0$ for $k \le n -1$. Hence $H_k(\phi(W_{flex}))\rightarrow H_k(Y \times [1,R])$ is also an isomorphism for $k \le n-2$.  Therefore  $H_k(\phi(W_{flex}))\rightarrow H_k(X')$ is an isomorphism for $k \le n-2$ and so $H_k(X', \phi(W_{flex})) = 0$ for $k \le n-2$ by the long exact sequence of $(X', \phi(W_{flex} ))$. For the $k = n-1$ case, we 
again use the long exact sequence of $(X', \phi(W_{flex} ))$ combined with the fact that $H_{n-1}(X') = 0$, which follows from the assumption that $H_{n-1}(W_{flex}) = 0$ and Theorem \ref{thm: main}.

Finally, we consider the cases $k = n, n +1$. 
Note that $H_{n+1}(X') \cong H_{n+1}(W_{flex}) = 0$ since $W_{flex}$ is Weinstein; also,   $H_{n-1}(\phi(W_{flex})) \cong H_{n-1} (W_{flex})$ vanishes by assumption. 
Hence to show that $H_{n}(X', \phi(W_{flex}))$ and  $H_{n+1}(X', \phi(W_{flex}))$ vanish, it suffices to show the map $H_n(\phi(W_{flex} ) ) \rightarrow H_n(X')$ is an isomorphism.
First we show that the map $H_n(\phi(W_{flex})) \rightarrow H_n(Y \times [1,R])$ is an isomorphism. The map $H_n(Y\times [1,R]) \rightarrow H_n(W_{flex}^R)$ is injective since $H_{n+1}(W_{flex}^R, Y\times [1,R]) \cong H^{n-1}(W_{flex}^R)$, which vanishes by assumption. It is surjective since the composition 
$H_n(\phi(W_{flex})) \rightarrow H_n(Y\times [1,R]) \rightarrow H_n(W_{flex}^R)$ is the identity map. Hence it is an isomorphism and so $H_n(\phi(W_{flex})) \rightarrow H_n(Y \times [1,R])$ is also an isomorphism. 

Now we show that the map $H_n(Y \times [1,R]) \rightarrow H_n(X')$ is surjective. We first note that $H_n(X', Y\times [1,R]) \cong H^n(X')$ and by Remark \ref{rem: fieldcoefficients,homology}, $H^n(X') \cong H^n(W_{flex})$. Next we show that $H^n(W_{flex}) \cong H_{n-1}(Y)$. 
The long exact sequence of $(W_{flex}, Y)$ is
$$
H_n(W_{flex})\rightarrow H_n(W_{flex}, Y) \rightarrow H_{n-1}(Y) \rightarrow H_{n-1}(W_{flex}).
$$
The rightmost term vanishes by assumption. The leftmost map is zero. To see this, note that it can be identified with the intersection form map $H_n(W_{flex}) \rightarrow Hom(H_n(W_{flex}), \mathbb{Z})$ under the isomorphisms 
$Hom(H_n(W_{flex}), \mathbb{Z}) \cong H^n(W_{flex};\mathbb{Z}) \cong H_n(W_{flex}, Y)$; the first isomorphism uses the universal coefficient theorem for cohomology along with the fact that  $\mbox{Ext}^1_{\mathbb{Z}}(H_{n-1}(W_{flex}), \mathbb{Z})$ vanishes since $H_{n-1}(W_{flex})$ vanishes by assumption. 
Now note that the intersection form map is zero since $W_{flex}$ is displaceable in its completion. Therefore the leftmost map is zero as well. Hence $H^n(W_{flex}) \cong H_n(W_{flex}, Y)  \cong H_{n-1}(Y)$ and so $H_n(X', Y \times [1,R]) \cong H^n(X') \cong H^n(W_{flex}) \cong H_{n-1}(Y)$. 
So the top vertical map 
$H_n(X', Y \times [1,R]) \rightarrow H_{n-1}(Y\times [1,R])$ in Figure \ref{diagram: large} is surjective since 
$H_{n-1}(X') \cong H_{n-1}(W')$ vanishes by assumption. So this map is a surjective map of isomorphic finitely-generated abelian groups. 
Finitely-generated abelian groups are \textit{Hopfian}, meaning that surjective endomorphisms are automatically isomorphisms, and so this map is in fact an isomorphism.  
So by the long exact sequence of $(X', Y\times [1,R])$, we have that $H_n(Y\times [1,R]) \rightarrow H_n(X')$ is surjective.  

Finally, since 
$H_n(\phi(W_{flex}))\rightarrow H_n(Y\times [1,R])$ is an isomorphism and 
$H_n(Y\times [1,R]) \rightarrow H_n(X')$ is surjective, the map $H_n(\phi(W_{flex})) \rightarrow H_n(X')$ is also surjective. 
By Theorem \ref{thm: main}, $H_n(\phi(W_{flex})) \cong H_n(X')$. Again using the Hopfian property, we conclude that this map is in fact an isomorphism, which finishes the $k= n+1, n$ cases. 

This shows that $X$ is diffeomorphic to $W_{flex}$. To complete the proof, we need to show that they are almost symplectomorphic. 
Since $X, W_{flex}$ are flexible Weinstein domains, this will show that they are Weinstein deformation equivalent  by the uniqueness h-principle \cite{CE12}; in particular, they will have symplectomorphic completions.

Using the diffeomorphism between $X$ and 
$C \cup \phi(W_{flex}) \cong \phi(W_{flex})$, the almost symplectic structure 
on $X$ pushes forward to the almost symplectic structure on $\phi(W_{flex})$ induced  from the inclusion  $\phi(W_{flex})\subset W_{flex}$, except in a small collar neighborhood of $\partial \phi(W_{flex})$ corresponding to the concordance $C$. 
However any almost symplectic structure on a concordance can be deformed relative to the negative boundary to a product structure. 
In particular, this push-foward structure can be deformed through almost symplectic structures to the induced-from-inclusion 
symplectic structure on $\phi(W_{flex})$. Since $\phi$ is isotopic to the identity map, $\phi(W_{flex})$ with this almost symplectic structure is almost symplectomorphism to $W_{flex}$ with the original symplectic structure. Hence $X$ and $W_{flex}$ are almost symplectomorphic, as desired. 
\end{proof}

\begin{proof}[Proof of Theorem \ref{thm: inf_contact_flex}]

Suppose $W^{2n}, n \ge 3,$ is an almost Weinstein filling of $(Y, J)$ and $M^{2n}$ is an almost Weinstein filling of $(S^{2n-1}, J_{std})$. 
By the Weinstein existence h-principle, there are flexible Weinstein domains that are almost symplectomorphic to $W, M$, respectively; we will also use $W,M$ to denote these flexible domains. Let $W_M := W\natural M$. Since $W_M$ is obtained by attaching a 1-handle to flexible $W$ and $M$, it is also flexible. Also note that $\partial W_M$ is in $(Y, J)$ since $\partial M$ is in $(S^{2n-1}, J_{std})$. This proves the first and third claim of Theorem \ref{thm: inf_contact_flex} with $(Y, \xi_M):= \partial W_M$. 
By Theorem \ref{thm: main}, all flexible fillings of $(Y, \xi_M) = \partial W_M$ have cohomology $H^*(W_M)$, which equals 
$H^*(W) \oplus H^*(M)$ except in degree zero. In particular, if $H^*(M)\not \cong H^*(N)$, then $H^*(W_M) \not \cong H^*(W_N)$ and so $\xi_M, \xi_N$ are not contactomorphic, which proves the second claim of Theorem \ref{thm: inf_contact_flex}. 

Finally, to construct infinitely many contact structures in $(Y, J)$ with flexible fillings, it suffices to produce infinitely many almost Weinstein fillings $M^{2n}$ of $(S^{2n-1}, J_{std}), n\ge 3$, with different $H(M)$. In particular, it is enough to find an infinite sequence of such $M_i$ with 
$\underset{i\rightarrow \infty}{\lim} \dim H^k(M_i) = \infty$ for some $k$. For $n$ odd, we can take $M_i$ to be (the almost Weinstein type of) certain Brieskorn manifolds. More precisely, 
let $M_i = \{z_0^i + z_1 ^2 +\cdots z_n^2 = \epsilon, \|z\| \le 1\} \subset B^{2n+2}$ for some small $\epsilon > 0$. Then if $i \equiv 1 \mod 2(n-1)!$, $\partial M_i$ is in $(S^{2n-1}, J_{std})$; we note that $\partial M_i = (S^{2n-1}, \xi_i)$ are precisely the contact structures studied by Ustilovsky in \cite{U}. These Brieskorn manifolds are plumbings of  an increasing number of $T^*S^n$'s and so $\underset{i\rightarrow \infty}{\lim} \dim H ^n(M_i) = \infty$.
For $n$ even, Brieskorn manifolds do not seem to suffice since their contact boundary do not seem to be in $(S^{2n-1}, J_{std})$. Instead, one can start with a certain Weinstein domain $M$ constructed  in \cite{Geiges} for $n \equiv 2 \mod 4$ and in \cite{DG} for $n \equiv 0 \mod 4$
for which $\partial M$ is in $(S^{2n-1}, J_{std})$ and $\dim H_{n}(M) \ge 1$. These domains are constructed using the Weinstein existence h-principle and so $M$ is already flexible. 
Let $M_i:=  \natural_{j=1}^i M$ be the $i$-fold boundary connected sum of $M$. Then $\partial M_i$ is in $(S^{2n-1}, J_{std})$ and $\dim H^n(M_i)  = i \dim H^n(M) \ge i$. In particular, $\underset{i\rightarrow \infty}{\lim} \dim H^n(M_i) = \infty$ as desired. 
\end{proof}
\begin{remark}\label{rem: ustilovksy}\

\begin{enumerate}[leftmargin=*]
\item 
Such an infinite collection of $M_i^{2n}$ 
only exist for $n \ge 3$: if $M^4$ is a Liouville filling of $(S^3, \xi)$, then $\widehat{M}^4$ is symplectomorphic to $(\mathbb{C}^2, \omega_{std})$ and $(S^3, \xi)$ must be $(S^3, \xi_{std})$.
\item We can generalize Theorem \ref{thm: inf_contact_flex} by considering a wider class of both $Y$ and $M$. Instead of $(Y, \xi)$ with a flexible filling, we can consider $(Y^{2n-1}, \xi), n \ge 3$, that is ADC and has a $\pi_1$-injective Liouville filling $W$ such that 
$SH_k(W), SH_{k+1}(W)$ are finite-dimensional for some 
$k$. Instead of almost Weinstein fillings $M$ of $(S^{2n-1}, J_{std})$, we can consider Liouville fillings $M$ of $(S^{2n-1}, J_{std})$ such that $SH(M) =0$ and $\partial M$ is ADC. We define $W_M = W\natural M$ and $(Y, \xi_M) = \partial W_M$ as before. Then $(Y, \xi_M)$ is in the same almost contact class as $(Y, \xi)$. Furthermore, $(Y, \xi_M)$ is also ADC since $(Y, \xi)$ and $\partial M$ are both ADC and 
by Theorem \ref{thm: MLYau}, subcritical surgery preserves ADC contact forms. If  $|\dim H^{n-k}(M) - \dim H^{n-k}(N)| > \dim SH_k(W)+ \dim SH_{k+1}(W)$, then $\xi_M, \xi_N$ are not contactomorphic, which generalizes Theorem \ref{thm: inf_contact_flex}.
 To see this, first note that $SH(W_M) = SH(W) \oplus SH(M) = SH(W)$ since subcritical handle attachment does not change $SH$ as shown by Cieliebak \cite{CieChord}. Then by the tautological long exact sequence, 
\begin{equation}\label{eqn: sh_inequality}
-\dim SH_{k+1}(W) 
\le \dim H^{n-k}(W_M) - \dim SH^{+}_{k+1}(W_M) 
\le 
\dim SH_k(W).
\end{equation}
Then $|\dim H^{n-k}(M) - \dim H^{n-k}(N)| > \dim SH_k(W)+ \dim SH_{k+1}(W)$ implies that
$
|\dim H^{n-k}(W_M) - \dim H^{n-k}(W_N)| > \dim SH_k(W)+ \dim SH_{k+1}(W).
$
Then by Equation \ref{eqn: sh_inequality} we can conclude that 
$SH_{k+1}^{+}(W_M) \not \cong SH_{k+1}^{+}(W_N)$. Since $\xi_M, \xi_N$ are ADC, Proposition \ref{prop: nice_sh_independent} implies that they are not contactomorphic. 
\end{enumerate}
\end{remark}

Now we prove the rest of the results stated in the Introduction.  

\begin{proof}[Proof of Remark \ref{rem: Cieliebak_Eliash_McLean}]

We first review the pairwise non-symplectomorphic  $W_k$ from \cite{CE12}. Suppose $(Y^{2n-1}, J), n \ge 3,$ has an almost Weinstein filling $W$. Then by the Weinstein existence h-principle, there is a flexible Weinstein domain almost symplectomorphic to $W$; we will also use $W$ to denote this flexible domain. Note that $\partial W$ is in $(Y, J)$. 
McLean \cite{MM} constructed a Weinstein structure $\mathbb{C}_M^n$ on $\mathbb{C}^n$ in the standard almost symplectic class with $SH(\mathbb{C}_M^n) \ne 0$; McLean worked over $\mathbb{Z}/2$ and we do the same in this section.
Since $SH$ is a unital ring with unit in degree $n$, we have $SH_n(\mathbb{C}_M^n) \ne 0$.
Then $W_k: = W \natural \mathbb{C}_M^n \natural \cdots \natural \mathbb{C}_M^n$, the boundary connected sum of $W$ with $k$ copies of $\mathbb{C}_M^n$, is almost symplectomorphic to $W$. 
Since $W$ is flexible, we have $SH(W) = 0$ and so $SH(W_k) = \oplus_{i=1}^k SH(\mathbb{C}_M^n)$. In particular, $SH_n(W_k) = \oplus_{i=1}^k SH_n(\mathbb{C}_M^n)$. 
The tautological long exact sequence for symplectic homology is
$$
H^0(W_k; \mathbb{Z}/2) 
\rightarrow SH_n(W_k)
\rightarrow SH_n^{+}(W_k)
\rightarrow H^1(W_k; \mathbb{Z}/2).
$$
Because  $\dim H^0(W_k; \mathbb{Z}/2) = 1$, this exact sequence implies 
\begin{equation}\label{eqn: inequalitysh}
\dim SH^{+}_n(W_k) 
 \ge \dim SH_n(W_k) -1.
\end{equation}
Since $SH_n(\mathbb{C}_M^n) \ne 0$, we have
$\dim SH_n(\mathbb{C}_M^n) \ge 1$ and so
$\dim SH_n(W_k) \ge k$. 
Hence by Equation (\ref{eqn: inequalitysh}) 
\begin{equation}\label{eqn: shbound2}
\dim SH_n^{+}(W_k) \ge k-1.
\end{equation}

Now suppose that the Cielieback-Eliashberg-McLean contact structure $(Y, \xi_k)= \partial W_k$ has a flexible filling $W$.
Then $(Y, \xi_k)$ is ADC by Corollary \ref{cor: flexsemigood}
and therefore 
$\dim SH_n^{+}(W_k) =\dim SH_n^{+}(W)$
by Proposition \ref{prop: nice_sh_independent}.
Since $SH(W) = 0$, by the tautological exact sequence $SH_n^{+}(W) \cong H^1(W; \mathbb{Z}/2)$. Since $W$ is Weinstein and $n \ge 3$, then $H^1(W; \mathbb{Z}/2) \cong H^1(Y; \mathbb{Z}/2)$ and so $\dim  SH_n^{+}(W) =
 \dim H^1(Y; \mathbb{Z}/2)$.
By Equation \ref{eqn: shbound2}, this implies that 
$\dim H^1(Y; \mathbb{Z}/2) \ge k -1$. 
So if $k \ge \dim H^1(Y; \mathbb{Z}/2) + 2$,   then $(Y, \xi_k) = \partial W_k$ cannot have \textit{any} flexible fillings. So at most $\dim H^1(Y; \mathbb{Z}/2) + 1$ of the Cieliebak-Eliashberg-McLean contact structures can have flexible fillings. 
\end{proof}
\begin{remark}
We do not know whether $(Y, \xi_k) = \partial W_k$ is ADC but in this proof, we only use that $\partial W$ is ADC. 
\end{remark}

\begin{proof}[Proof of Corollary \ref{cor: 57example}]
An almost contact structure $(Y, J)$ on 
a simply-connected 5-manifold or
 simply-connected 7-manifold with $\pi_2(Y)$ torsion-free has a Weinstein filling by \cite{Geiges}, \cite{BCS} respectively. We apply Theorem \ref{thm: inf_contact_flex} to classes with
 $c_1(J) = 0$
\end{proof}

\begin{proof}[Proof of of Corollary \ref{cor: s4k+3}]
Since $(S^{2n-1}, \xi_{std})$ has a Weinstein filling $(B^{2n}, \lambda_{std})$ and $c_1(\xi_{std}) = 0$,
Theorem \ref{thm: inf_contact_flex} produces infinitely many contact structures in $(S^{2n-1}, J_{std})$.  For $n$ odd, every almost contact class on $S^{2n-1}$ can be realized by a Weinstein-fillable contact structure \cite{U} and so the claim holds for every almost contact class. 

Next we show that the Ustilovsky contact structures $(S^{2n-1}, \xi_i)$, $n$ odd, do not have flexible fillings and hence are different from the contact structures we constructed in the previous paragraph. 
As we discussed in the proof of Theorem \ref{thm: inf_contact_flex}, $(S^{2n-1}, \xi_i)$ has the Brieskorn manifold $M_i$ as a Weinstein filling.  In Theorem 3.1 of \cite{Ue}, Uebele computed $SH_k^{+}(M_i)$ using Morse-Bott techniques and showed that for each $i$, $SH_k^+(M_i) \ne 0$ for infinitely many positive $k$; Uebele worked over $\mathbb{Z}/2$ and we do the same in this proof. 

 Now suppose that $(S^{2n-1}, \xi_i)$ has a flexible filling $W$. Then by Corollary \ref{cor: flexsemigood}, $(S^{2n-1}, \xi_i)$ is ADC. Since $W, M_i$ are both Weinstein and hence $\pi_1$-injective by Remark \ref{rem: H1}, then $SH^{+}(W) \cong SH^{+}(M_i)$ by Proposition 
\ref{prop: nice_sh_independent}. 
But $SH_k^{+}(W) = 0$ for $k \ge n+2$ by Proposition \ref{prop: shcomputation} while $SH_k^{+}(M_i) \ne 0$ for infinitely many positive $k$, which contradicts 
 $SH^{+}(W) \cong SH^{+}(M_i)$.
\end{proof}

\begin{proof}[Proof of Theorem \ref{thm: boundedinfinite}]
Let $(Y, J)$ have an almost Weinstein filling $W$. 
We first outline the construction of the contact structure $(Y, \xi_M)$, $M^n \in \Omega^n$. As described in the statement of Theorem \ref{thm: boundedinfinite}, $(Y, \xi_M)$ will be the boundary of a certain Weinstein domain $W_M$ that is almost symplectomorphic to $W\natural P$, where $P$ is a certain plumbing of $T^*S^n$'s. All of the $T^*S^n$'s in the plumbing will be flexible except for one, which will have an exotic non-flexible symplectic structure constructed by Eliashberg, Ganatra, and the author in \cite{EGL}.

We first review the exotic symplectic structures on $T^*S^n$ from \cite{EGL}. For 
$n \ge 3$, consider any closed manifold $M^n \in \Omega^n$, i.e. $M^n$ is closed, simply-connected, stably-parallelizable, and $\chi(M) = 2$ or $\chi_{1/2}(M) \equiv 1$ depending on whether $n$ is even or odd respectively. 
By  \cite{EGL}, there is a Weinstein domain $T^*S_M^n$ 
that is almost symplectomorphic to $(T^*S^n, \omega_{std})$ and contains $M$ as an exact Lagrangian. 
More explicitly, $T^*S_M^n$ is obtained by attaching subcritical and flexible handles to $T^*M$. 
Note that \cite{EGL} requires that the complexified tangent bundle of $M^n\backslash D^n$ is trivial, which follows from our condition that $M$ is stably parallelizable; see Theorem 9.1.5, \cite{Hu94}.

Since attaching subcritical or flexible handles does not change symplectic homology
by \cite{CieChord}, \cite{BEE12},  we have $SH(T^*S_M^n) \cong SH(T^*M)$. More precisely, subcritical and flexible handle attachment does not change \textit{full} $SH$, which generated by \textit{all} Hamiltonian orbits not just contractible ones as in our definition of $SH$. 
However since $\pi_1(T^*S_M^n) = \pi_1(S^n) =0$ and $\pi_1(T^*M) = \pi_1(M) = 0$ by assumption, this is not an issue here. Otherwise, non-contractible orbits of $\partial T^*M$ can become contractible in $\partial T^*S_M^n$, where they can have arbitrary degree depending on the framing of 2-handles that are attached to $T^*M$ to form $T^*S_M^n$; see Section \ref{ssec: independencelin} for a similar discussion. Furthermore $c_1(T^*S_M^n)$ and $c_1(T^*M)$ always vanish and so there are $\mathbb{Z}$-gradings on $SH(T^*S_M^n)$ and $SH(T^*M)$ and the isomorphism above is grading-preserving. 
Finally because $M$ is stably parallelizable, $M$ is spin and so $SH_*(T^*M) \cong H_*(\Lambda M; \mathbb{Z})$ by \cite{V_96}, \cite{K_07}. Therefore we have $SH_*(T^*S_M^n) \cong H_*(\Lambda M; \mathbb{Z})$.

Now we show that the contact boundary $\partial T^*S_M^n$ is ADC for $n \ge 4$. We first recall that for any closed manifold $M^n, n \ge 4$, the contact boundary $\partial T^*M$ of $T^*M$ is ADC \cite{EGH}. 
To see this, choose a metric on $M$ and consider the induced  contact form on $\partial T^*M$. The Reeb orbits of this contact form correspond to geodesics of $M$ with respect to the chosen metric. Using the trivialization of the canonical bundle of $T^*M$ induced by the Lagrangian cotangent fiber foliation,  the Conley-Zehnder index $\mu_{CZ}(\gamma)$ of a Reeb orbit $\gamma$ is the Morse index of the corresponding geodesic, which is always non-negative.
Since the Reeb orbits are graded by the \textit{reduced} Conley-Zehnder index
$|\gamma| = \mu_{CZ}(\gamma)+n-3$, 
the degree of \textit{all} orbits is positive for $n \ge 4$ with respect this trivialization and so $\partial T^*M$ is ADC; note that we do not need to take a sequence of contact forms as allowed in Definition \ref{def: semigood}. We will show that $\partial T^*M$ is ADC using a different approach in Section \ref{subsec: legwithflexfill}. 
Since $T^*S_M^n$ is obtained by attaching subcritical or flexible critical handles to $T^*M$ and $\partial T^*M$ is ADC, then $\partial T^*S_M^n$ is also ADC by Theorem \ref{thm: MLYau} and Theorem \ref{thm: semi-surgery}. For 2-handles, we use Theorem \ref{thm: MLYauindex2},
which also requires $\partial T^*M$ and hence $M$ to be simply-connected.

These exotic symplectic structures on $T^*S^n$ can be used to construct exotic structures on a plumbing of $T^*S^n$. As noted in 
\cite{Eexotic}, there exists a plumbing $P$ of copies of $T^*S^n$ such that $\partial P$ is in $(S^{2n-1}, J_P)$. Furthermore, for $n$ odd we can assume $J_P = J_{std}$; for example, we can take $P$ to be one of the Brieksorn manifolds $M_i$ discussed in the proof of Theorem \ref{thm: inf_contact_flex}.
Since $P$ is a plumbing of $T^*S^n$, we can view $P$ as the result of successively attaching some n-handles to $T^*S^n$, where $S^n$ is one of the zero sections of the plumbing; that is, $P = T^*S^n \cup C$ for some Weinstein cobordism $C$. For $M^n \in \Omega^n$, let 
$P_M  := T^*S_M^n \cup C_f$ be the exotic Weinstein structure on $P$ obtained by replacing $T^*S^n$ with $T^*S_M^n$ and $C$ with its flexibilization $C_f$. Since $T^*S_M^n, C_f$ are almost symplectomorphic to $(T^*S^n, \omega_{std}), C$ respectively, then $P_M$ is almost symplectomorphic to $P$ and $\partial P_M$ is in $(S^{2n-1}, J_P)$.
Finally, since  $P_M$ is obtained by attaching flexible handles to $T^*S^n_M$ and $\partial T^*S_M^n$ is ADC, then $\partial P_M$ is also ADC by Theorem \ref{thm: semi-surgery}.

Now suppose that $(Y, J)$ has an almost Weinstein filling $W$. By the Weinstein existence h-principle,
there is a flexible Weinstein domain almost symplectomorphic to $W$; we will also use $W$ to denote this domain. Note that $\partial W$ is in $(Y,J)$. Let $W_M := W \natural P_M$. Since $P_M$ is almost symplectomorphic to $P$, $W_M$ is almost symplectomorphic to $W \natural P$. 
If $n$ is odd, the contact boundary $\partial W_M$ is in $(Y,J)$ because $\partial P_M$ is in $(S^{2n-1}, J_{std})$. If $n$ is even, $\partial W_M$ is in $(Y,J')$, where $J' := J \sharp J_P$, because $\partial P$ is in $(S^{2n-1}, J_P)$; note that $c_1(J') = c_1(J) + c_1(J_P) = 0$. 
 This proves the first and third claims of Theorem \ref{thm: boundedinfinite} with $(Y, \xi_M) = \partial W_M$.   
Also, because $W_M$ is the boundary connected sum of  $P_M, W$ and $\partial P_M, \partial W$ are ADC (because $W$ is flexible), $\partial W_M$ is also ADC by Theorem \ref{thm: MLYau}. Finally, note that because $W$ and $C_f$ are flexible, $SH(W_M) \cong SH(P_M) \cong  SH(T^*S^n_M) \cong H(\Lambda M)$.

Consider two $M^n, N^n \in \Omega$.  Then $\partial W_M, \partial W_N$ are ADC contact structures and $SH^{+}$ is a contact invariant for ADC structures by Proposition \ref{prop: nice_sh_independent}. So if $SH^{+}(W_M) \not \cong SH^{+}(W_N)$, then $\partial W_M$ and $\partial W_N$ are non-contactomorphic.
The tautological long exact sequence for $SH(W_M)$ shows that 
\begin{equation}\label{eqn: SHinequality}
|\dim SH_k(W_M) - \dim SH_k^{+}(W_M)| \le \dim H^{n-k}(W_M) + \dim H^{n-k+1}(W_M) 
\end{equation}
for all $k$. 
Also, $H^k(W_M) \cong H^k(W\natural P) \cong H^k(W)$ for $k \le n-1$ since $P$ is $(n-1)$-connected. Hence 
$\dim H^{n-k}(W_M) + \dim H^{n-k+1}(W_M) = 
\dim H^{n-k}(W) + \dim H^{n-k+1}(W)$ for $k \ge 2$. 
So if
$$
|\dim H_k(\Lambda M) - \dim H_k(\Lambda N)| 
> 2\dim H^{n-k}(W) + 2\dim H^{n-k+1}(W)
$$ 
for some $k \ge 2$, then by Equation \ref{eqn: SHinequality} 
$SH_k^{+}(W_M) \not \cong SH_k^{+}(W_N)$ and so $\partial W_M, \partial N$ are not contactomorphic. Note that since $k \ge 2$, 
$ \dim H^{n-k}(W) = \dim H^{n-k}(Y)$ and $ \dim H^{n-k+1}(W) \le \dim H^{n-k+1}(Y)$. This proves the second claim of Theorem \ref{thm: boundedinfinite}.

To prove the last claim in Theorem \ref{thm: boundedinfinite}, it suffices to find infinitely  many
 $M^n \in \Omega^n$ such that $H_k(\Lambda M)$ are sufficiently different for some $k \ge 2$. 
 Note that since $M^n$ is simply-connected, 
 $\dim H_i(\Lambda M)$ is finite for all $i$; this is true for all simply-connected manifolds by applying
the theory of Serre classes to the relevant fibrations. 
 In particular, it is enough to find an infinite sequence $M^n_i \in \Omega^n$ such that $\underset{i\rightarrow \infty}{\lim} \dim H_k(\Lambda M_i) = \infty$ for some $k \ge 2$. 
  For any topological space $X$, $\dim H_k(\Lambda X) \ge \dim H_k(X)$
since the inclusion of constant loops $X \rightarrow \Lambda X$ and the evaluation map $\Lambda X \rightarrow X$ compose to the identity. So it suffices to find $M_i \in \Omega$ with $\underset{i\rightarrow \infty}{\lim} \dim H_k(M_i) = \infty$.

For $n \ge 6$ even, consider the CW-complex $X_i = \vee_{j= 1}^i (S^2\vee S^3)$. Note that $\chi(X_i) = 1$.  We can explicitly embed $X_i$ into $\mathbb{R}^{n+1}$ for $n \ge 3$. 
 A tubular neighborhood of $X_i\subset \mathbb{R}^{n+1}$ is a compact manifold $W^{n+1}_i \subset \mathbb{R}^{n+1}$. 
Since $W_i$ retracts to $X_i$, $\pi_1(W_i) = \pi_1(X_i) = 0$, $\dim H_2(W_i)  = \dim H_2(X_i) = i$, and $\chi(W_i) = \chi(X_i) = 1$. We now show that $M_i^n = \partial W_i^{n+1}$ is in $\Omega$ and $\underset{i\rightarrow \infty}{\lim} \dim H_2(M_i) = \infty$.
First of all, $M_i$ is stably parallelizable since $W$ is parallelizable. Furthermore, since $n \ge 5$, $\pi_1(M_i) = \pi_1(W_i) = 0$ since $W^{n+1}$ just has $0, 2, 3$-handles. For any manifold $X^{n+1}$ with $n$ even, $\chi(\partial X) = 2\chi(X)$ and so in our case $\chi(M_i) = 2 \chi(W_i) = 2$. So $M_i \in \Omega$. Finally since $n \ge 6$, 
$\dim H_2(M_i) = \dim  H_2(W_i) = i$ because $W_i$ has just $0, 2, 3$-handles 
and so 
$\underset{i\rightarrow \infty}{\lim} \dim H_2(M_i) = \infty$ as desired. 

For $n \ge 5$ odd, we take $X_i =  \vee_{j= 1}^i S^2$ and let $W_i, M_i$ be as before. To show that $M_i \in \Omega$, we just need
 $\chi_{1/2}(M_i) = \sum_{k=0}^{(n-1)/2} \dim H_k(M_i)  \equiv 1 \mod 2$. For $n \ge 5$, $H_k(M_i) \cong H_k(W_i)$ for $k \le (n-1)/2$ because $W^{n+1}$ just has $0, 2$-handles. Then 
$$
\chi_{1/2}(M_i) = \sum_{k=0}^{(n-1)/2}\dim H_k(M_i) = \sum_{k=0}^{(n-1)/2}\dim H_k(W_i)
= 1 + i
$$
So if $i$ is even, then $\chi_{1/2}(M_i) \equiv 1 \mod 2$. Also, since $n \ge 5$ and $W_i$ has only 0,2-handles, 
$\dim H_2(M_i) = \dim H_2(W_i) = i$ and so 
$\underset{i\rightarrow \infty}{\lim} \dim H_2(M_i) = \infty$ as desired. 
\end{proof}
\begin{remark}\
It is possible to construct infinitely many examples for $n = 4$ in a slightly different way. We can weaken the condition that $\pi_1(M) = 0$ to $H^1(M; \mathbb{Z})= 0$. In this case, \textit{full} $SH(T^*M)$ has a canonical grading and the isomorphisms $SH(T^*S_M^n) \cong SH(T^*M) \cong H(\Lambda M)$ are all grading-preserving; furthermore, $\partial T^*S_M^n$ is still ADC since \textit{all} orbits of $\partial T^*M$ have positive degree for $n \ge 4$. There are infinitely many $M$ with $H^1(M; \mathbb{Z}) = 0$ and sufficiently different 
$H_0(\Lambda M)$ which provide the necessary examples; see the proof of Theorem 4.7 in \cite{EGL}. In  Theorem \ref{thm: boundedinfinite}, we stated for simplicity that we need 
$H_k(\Lambda M)$ to be sufficiently different for $k \ge 2$ in order to distinguish the $\xi_M$ for simplicity but in fact $k \ge 0$ also works if $|\dim H_k(\Lambda M) - \dim H_k(\Lambda N)| \ge 2(\dim H(W) + \dim H(P))$, where $P$ is a fixed plumbing of $T^*S^n$. For $n =3$, our approach does not work since we need $n \ge 4$ for $\partial T^*M$ to be ADC. 
\end{remark}

\begin{proof}[Proof of Remark \ref{rem: boundedinfinite}]
Note that if $(Y, \xi)$ has a flexible filling $W$,  \textit{all} Liouville fillings $X$ of $(Y, \xi)$ have finite-dimensional $SH(X)$. To see this, note that $SH^{+}(W)\cong H(W)$ by the tautological exact sequence and so $SH^{+}(W)$ is finite-dimensional (working over $\mathbb{Q})$.  Also, $(Y, \xi)$ is ADC because it has a flexible filling and so $SH^{+}(W) \cong SH^{+}(X)$.
Therefore $SH^{+}(X)$ is finite-dimensional and so $SH(X)$ is also finite-dimensional again by the tautological exact sequence. 
On the other hand, for the fillings $W_M$ described in the proof of Theorem \ref{thm: boundedinfinite}, $SH(W_M) = H(\Lambda M; \mathbb{Q})$. Hence if $H(\Lambda M)$ is infinite-dimensional, $(Y, \xi_M) = \partial W_M$ has no flexible fillings. Finally, we note that because $M \in \Omega$ is simply-connected, $H(\Lambda M; \mathbb{Q})$ is always infinite-dimensional \cite{DSullivan}.
\end{proof}

\section{Reeb chords of loose Legendrians}\label{sec: reeb_chord_loose}

To prove Theorem \ref{thm: semi-surgery}, we need to understand the new Reeb orbits produced by doing flexible surgery.  As we explain below in Section \ref{ssec: chords_contact_surgery}, the new Reeb orbits formed by doing critical contact surgery correspond to words of Reeb chords of the Legendrian attaching sphere (up to cyclic permutation). For flexible surgery, the attaching sphere is a loose Legendrian and therefore, we need to understand Reeb chords of loose Legendrians spheres, which we study in Section \ref{sec: modifying}.

\subsection{Properties of Reeb chords}\label{ssec: prop_reeb_chords}

We first present some background about Reeb chords. Let  $\Lambda^{n-1}$ be a Legendrian in $(Y^{2n-1}, \xi)$. We will assume throughout that $\Lambda$ is connected; this is sufficient for most of our applications since contact surgery on  a Legendrian link is equivalent to a sequence of surgeries on connected Legendrians. 
Let $\alpha$ be a non-degenerate contact form for $(Y, \xi)$.  
We say that a Reeb chord $c$ of $\Lambda$ with endpoints $x_0, x_1\in \Lambda$ is \textit{non-degenerate} if the subspaces $T_{x_1}\Lambda$ and the image of $T_{x_0}\Lambda$ under the linearized Reeb flow are transverse; similarly, we say that $\Lambda$ is \textit{non-degenerate} if all Reeb chords of $\Lambda$ are non-degenerate. A generic Legendrian is non-degenerate and after $C^0$-small Legendrian isotopy, we can assume that any Legendrian is non-degenerate.  All Reeb chords that we discuss in this paper will be non-degenerate; in general Legendrians do not need to be non-degenerate since we only work with chords below some action bound.

As for Reeb orbits, we can define the contact action of a Reeb chord $c$ by
$$
A(c) := \int_c \alpha.
$$
We let $\mathcal{P}^{< D}(\Lambda, Y, \alpha)$ denote the set of Reeb chords of $\Lambda$ that have action less than $D$ and are zero in $\pi_1(Y, \Lambda)$. As in the contact case, the non-degeneracy of $\Lambda$ implies that $\mathcal{P}^{< D}(\Lambda, Y, \alpha)$ is a finite set for any fixed $D>0$.  Note that the definition of $\pi_1(Y, \Lambda)$ requires fixing a basepoint in $\Lambda$; we consider a chord $c$ as an element of $\pi_1(Y, \Lambda)$ by concatenating $c$ with a path in $\Lambda$ to this basepoint. The condition that $c$ is zero in $\pi_1(Y, \Lambda)$ is independent of the choice of this path and is equivalent to the existence of a disk $D^2\subset Y$ such that $\partial_- D^2 = c$ and $\partial_+ D^2 \subset \Lambda$. 
Also, let $\mathcal{W}^D(\Lambda, Y, \alpha)$ denote the set of cyclic equivalence classes of Reeb chord words that have total action less than $D$ and are zero in $\pi_1(Y, \Lambda)$; that is, 
$$
\mathcal{W}^D(\Lambda, Y, \alpha) = 
\{w:= c_1 \cdots c_k | 
A(w):=\sum_i A(c_i) < D \} / \sim
$$
with the equivalence given by 
$c_1 c_2 \cdots c_n \sim c_2 \cdots c_n c_1$. We consider $c_1 c_2 \cdots c_n$ as an element of $\pi_1(Y, \Lambda)$ by inserting paths in $\Lambda$ into this word that connect one endpoint of $c_1$ to the basepoint and connect the endpoint of $c_i$ to the endpoint of $c_{i+1}$ for all $i$. Note that the condition of being zero in $\pi_1(Y, \Lambda)$ is independent of the choice of these paths and is preserved under cyclic permutation. 

If both $c_1(Y) \in H^2(Y; \mathbb{Z})$ and $c_1(Y, \Lambda) \in H^2(Y, \Lambda; \mathbb{Z})$ of $\Lambda$ vanish, then we assign an integer Conley-Zehnder index $\mu_{CZ}(c)$ to each Reeb chord $c$ of $\Lambda \subset (Y, \xi)$; see \cite{EO_15}. Then the degree of $c$ is the Conley-Zehnder index minus one, i.e. 
$$
|c|= \mu_{CZ}(c)-1.
$$
This grading convention coincides with the grading used in Legendrian contact homology. 
The degree of a word of chords is the sum of the degrees of the chords in the word, i.e. $|c_1 \cdots c_k| = \sum_{i=1}^k |c_i|$. 
 As for Reeb orbits, in general the degree for chords depends on a trivialization of the canonical bundle but for chords that are zero in $\pi_1(Y, \Lambda)$, the degree is independent of the trivialization; this is also true for elements of $\mathcal{W}^{<D}(\Lambda, Y, \alpha)$.

\begin{remark}\label{rem: c1_lambda}
Note that $c_1(Y, \Lambda)$ maps to $c_1(Y)$ under the restriction $H^2(Y, \Lambda; \mathbb{Z}) \rightarrow H^2(Y; \mathbb{Z})$. If $H^1(\Lambda; \mathbb{Z}) = 0$ and $c_1(Y) =0$, then the cohomology long exact sequence shows that $c_1(Y, \Lambda)=0$ and so the Reeb chords of $\Lambda$ can be $\mathbb{Z}$-graded. 
\end{remark}

 Note that if $f: (Y, \alpha, \Lambda) \rightarrow (Y', \alpha', \Lambda')$ is a strict contactomorphism, 
 then $\mathcal{P}^{< D}(\Lambda, Y, \alpha)$ and
 $\mathcal{P}^{< D}(\Lambda', Y', \alpha')$ are in grading-preserving bijection and  similarly for $\mathcal{W}^D$. 
We also have analog of Proposition \ref{prop: easy} for Reeb chords. 

\begin{proposition}\label{prop: easychords}
For any $D, s>0$,   
$\mathcal{P}^{< D}(\Lambda, Y, s \alpha)$ and  $\mathcal{P}^{< D/s}(\Lambda, Y, \alpha)$ are in grading-preserving bijection. 
\end{proposition}

\subsection{Reeb chords and contact surgery}\label{ssec: chords_contact_surgery}

The following proposition, which was proven by Bourgeois, Ekholm, and Eliashberg \cite{BEE12}, allows us to translate our question about Reeb orbits to a question about Reeb chords of the attaching Legendrian. 

Before giving the precise statement, we fix some notation. Choose a Riemannian metric on $\Lambda$. Then let $U^{\epsilon}(\Lambda) \subset (J^1(\Lambda), \alpha_{std})$ be 
$\{ \|y\| < \epsilon, |z| < \epsilon\}$, where we use the metric on $\Lambda$ to define $\|y\|$. If $\Lambda \subset (Y,\alpha)$ is Legendrian, let $U^\epsilon(\Lambda, \alpha) \subset (Y, \alpha)$ be a neighborhood of $\Lambda$ that is strictly contactomorphic to $U^\epsilon(\Lambda)$. For the most part, the specific  contactomorphism between $U^\epsilon(\Lambda, \alpha)$ and $U^\epsilon(\Lambda)$ will not be important and just the existence of such a contactomorphism will be enough. 

 \begin{proposition}\cite{BEE12}\label{prop: BEE}
Let $\Lambda^{n-1} \subset (Y^{2n-1}_-, \alpha_-), n \ge 3$, be a Legendrian sphere. For any $D>0$, there exists $\epsilon = \epsilon(D)>0$ such that if 
$(Y_+, \alpha_+)$ is the result of contact surgery on $U^\epsilon(\Lambda, \alpha_-)$, then there is a (shifted) grading preserving bijection between 
$\mathcal{P}^{< D}(Y_+, \alpha_+)$ and $\mathcal{P}^{< D}(Y_-, \alpha_-) \cup \mathcal{W}^{<D}(\Lambda, Y_-, \alpha_-)$: if 
$\gamma_w$ is the orbit corresponding to a word of chords $w = c_1 \cdots c_k$, then 
$$
|\gamma_w| = |w|+n-3 = \sum_{i=1}^k |c_i| + n-3.
$$
\end{proposition}
\begin{remark}\label{rem: parametrization}\
\begin{enumerate}[leftmargin=*]
\item 
Strictly speaking, we should consider $\Lambda^{n-1} \subset Y^{2n-1}$ as a \textit{parametrized} Legendrian sphere, i.e. fix a diffeomorphism $S^{n-1} \rightarrow \Lambda^{n-1}$. 
This gives an identification of $U^\epsilon(\Lambda)$ with $U^\epsilon(S^{n-1})$, which we need in order to attach the Weinstein handle. Said differently, \textit{parametrized} Legendrian have a canonical framing, which makes contact surgery well-defined;  see \cite{Gbook}. However since the Reeb chords of $\Lambda^{n-1}$ do not depend on the parametrization of $\Lambda^{n-1}$, the conclusion of Proposition \ref{prop: BEE} holds for any parametrization. In this paper,  we are primarily concerned with Reeb chords and orbits  and so will rarely mention parametrizations. 
We also note that there are examples of different parameterizations of the same Legendrian that are smoothly isotopic in the ambient contact manifold but not Legendrian isotopic (through parametrized Legendrians); see \cite{ekholm_exotic_spheres}.  However, the notion of looseness is independent of parameterization and any two parametrizations of a loose Legendrian are Legendrian isotopic by Murphy's h-principle \cite{Murphy11}. 
\item 
Proposition \ref{prop: BEE} also requires that  all elements of
 $\mathcal{P}^{<D}(Y_-, \Lambda, \alpha_-)$ are non-degenerate 
and all elements of  $\mathcal{P}^{<D}(Y_-, \alpha_-)$ are disjoint from $\Lambda$.
This is true generically and can be achieved here by a $C^0$-small Legendrian isotopy of $\Lambda$ that does not affect the contactomorphism type of $(Y_+, \xi_+)$.
In this paper, we will always assume that these conditions hold  and explain how our modifications of Legendrians or contact forms preserve these conditions. 
Also, if all elements of  $\mathcal{P}^{<D}(Y_-, \alpha_-)$  are non-degenerate, then so are  all elements of  $\mathcal{P}^{<D}(Y_+, \alpha_+)$.
\item The contact form $\alpha_+$ in Proposition \ref{prop: BEE} is quite special. It equals $\alpha_-$  in $Y_+ \backslash U^\epsilon(\Lambda, \alpha_-)$ and has a certain standard form in handle region of $Y_+$; we will give a precise model in Proposition \ref{prop: techsurgery} below. 
\end{enumerate}
\end{remark}

Note that $c_1(Y_-) = 0$ implies 
$c_1(Y_+) = 0$ for $n \ge 3$ by Proposition \ref{prop: c1equivalence}. So the Reeb orbits of 
$(Y_+, \xi_+)$ can be $\mathbb{Z}$-graded. 
Similarly, $H^1(\Lambda; \mathbb{Z}) = 0$ and $c_1(Y_-) = 0$ imply that $c_1(Y_-, \Lambda) = 0$ by Remark \ref{rem: c1_lambda} and  so the Reeb chords of $\Lambda$ can also be $\mathbb{Z}$-graded. 
Also, $\pi_1(Y_-) \xrightarrow{\sim}  \pi_1(W) \xleftarrow{\sim}  \pi_1(Y_+)$, where $W$ is the elementary cobordism between $Y_-$ and $Y_+$. 
So any contractible orbit of $Y_+$ that corresponds to an old orbit of $Y_-$ must also be contractible in $Y_-$.  Otherwise, a non-contractible orbit of $Y_-$ might become contractible in $Y_+$ violating the claimed bijection in Proposition \ref{prop: BEE}.
Similarly, if $\gamma_w\in \pi_1(Y_+)$ is a new orbit of $(Y_+, \xi_+)$ corresponding to a word of chords $w$, then the corresponding element of $\pi_1(Y_-)$ under the isomorphism $\pi_1(Y_-) \cong \pi_1(Y_+)$ maps to $w \in \pi_1(Y_-, \Lambda)$ under the map $\pi_1(Y_-) \rightarrow \pi_1(Y_-, \Lambda)$. So if $\gamma_w$ is contractible in $Y_+$, then $w$ is zero in $\pi_1(Y_-, \Lambda)$. 

We note that Proposition \ref{prop: BEE} also holds for $n = 2$ under the additional assumptions that $c_1(Y_+, \xi_+) = 0$ and $\Lambda$ is contractible in $Y_-$. This case corresponds to critical surgery that is not flexible (since $n = 2$), which is generally not studied in this paper. One exception is Example \ref{ex: not_nice} but there $Y_-^3 = (S^3, \xi_{std})$  and so $\Lambda^1$ is contractible in $Y_-^3$; furthermore, by picking $\Lambda^1$ in the correct formal class, we can ensure that $c_1(W^4) = c_1(Y^3_+) = 0$.

 The analog of Proposition \ref{prop: BEE} for subcritical surgery was essentially proven by M.-L.Yau, who showed that all new orbits (up to fixed action) occur in the belt sphere of the subcritical handle and computed their degrees. 

\begin{proposition}\cite{MLYau}\label{prop: MLYau_action}
Let $\Lambda^{k-1} \subset (Y^{2n-1}_-, \alpha_-), n \ge 2,$ be an isotropic sphere with $k < n, k \ne 2$.
For any $D>0$ and integer $i > 0$, there exists $\epsilon = \epsilon(D, i)>0$ such that if $(Y_+, \alpha_+)$ is the result of contact surgery on $U^\epsilon(\Lambda, \alpha)$, then 
there is a grading preserving bijection between 
$\mathcal{P}^{< D}(Y_+, \alpha_+)$ and $\mathcal{P}^{< D}(Y_-, \alpha_-) \cup \{\gamma^1, \cdots, \gamma^i \}$, where 
$|\gamma^j| = 2n-k-4+2j$. 
If $k = 2$, we also need to assume that $c_1(Y_+, \xi_+) = 0$ and either $\Lambda$ is contractible or all orbits of $(Y_-, \alpha)$ with action less than $D$ are contractible. 
\end{proposition}

\begin{remark}
As in the critical case, we need to consider $\Lambda^{k-1}$ as a parametrized sphere. Furthermore, the conformal symplectic normal bundle of $\Lambda^{k-1} \subset Y^{2n-1}$ should be trivial, which is not always the case. Finally, we need to choose a trivialization for this bundle; see \cite{Gbook}. In our main applications Theorem \ref{thm: main} and \ref{thm: boundedinfinite}, this normal bundle will always be trivial since we already know all surgeries comes from an existing Weinstein domain and we just need to check that all Reeb orbits have positive degree. As a result, we will rarely mention this extra data since it does not affect Reeb chords and orbits. One exception is the $k = 2$ case, in which the trivialization of the normal bundle can affect the degrees of certain Reeb orbits, which we generally avoid; see the $k =2$ case in Proposition \ref{prop: MLYau_action} and the discussion before Theorem \ref{thm: MLYauindex2}. 
\end{remark}
\begin{proof}[Proof of Proposition \ref{prop: MLYau_action}]
 As explained in \cite{MLYau}, the surgery belt sphere 
 $S^{2n-k-1}$ contains a contact sphere 
$(S^{2n-2k-1}, \xi_{std})$. By taking the appropriate sequence of contact forms on $(Y_+, \xi_+)$, the Reeb orbits of  $(Y_+, \xi_+)$ correspond to the old Reeb orbits of $(Y_-, \xi_-)$ as well as new Reeb orbits of $(S^{2n-2k-1}, \xi_{std})$  inside this belt sphere (up to action $D$). The latter correspond to iterations $\gamma^1, \cdots, \gamma^i$ of a single Reeb orbit $\gamma$; see \cite{BEE12} or \cite{MLYau}. Furthermore, $\mu_{CZ}(\gamma^j) =  n-k - 1 + 2j$ and hence 
$|\gamma^j| = 2n - k -4 + 2j$. Also, by shrinking the handle, the action of $\gamma$ can be made arbitrarily small  and hence we can ensure that arbitrarily large iterations of $\gamma$ have action less than $D$; that is, $i$ can be taken to be arbitrarily large if $\epsilon$ is small enough. 
\end{proof}

If $W$ is subcritical, then $SH_{n-k+1}^+(W) \cong H^{k}(W; \mathbb{Z})$ by Proposition \ref{prop: shcomputation}. 
Indeed the proof of Proposition \ref{prop: simultaneoussurgerygood} shows that index $k$ contact surgery creates Reeb orbits $\gamma^j$ with Conley-Zehnder index $n-k-1+ 2j, j \ge 1$. As we saw in Section \ref{sssec: homological_obstruction}, each Reeb orbit $\gamma^j$ gives rise to two Hamiltonian orbits $\gamma_m^j, \gamma_M^j$ with degrees $n-k-1+ 2j, n-k + 2j$ respectively. 
Then $\gamma_m^1$ has degree $n-k+1$ and is the generator of $SH_{n-k+1}^+(W)$; the rest of the orbits $\gamma_m^j, j \ge 2,$ and $\gamma_M^j, j \ge 1,$ cancel out algebraically in $SH^+(W)$. 

Also, note that the degree of $\gamma^j$ is always positive for $n \ge 2$. Furthermore, $\gamma^j$ is contractible in $Y_+$ since $\pi_1(S^{2n-1 -k}) = 0$. Hence Proposition \ref{prop: MLYau_action} is almost enough to prove Theorems \ref{thm: MLYau}, \ref{thm: MLYauindex2} but the situation is complicated by the fact that ADC contact structures are defined by a sequence of \textit{non-increasing} contact forms, i.e. it is not clear that if we start with a sequence of non-increasing contact forms, then the  contact forms after surgery 
are also non-increasing. We resolve this issue in Section \ref{sec: surgerysemigood}. 

As discussed in Section \ref{ssec: independencelin}, we must deal with the cases $k = 2$ and $k \ne 2$ in Proposition \ref{prop: MLYau_action} separately. For $k = 2$, we need to make the extra assumptions stated in the proposition  in order to ensure that $c_1(Y_+, \xi_+) = 0$ and that all contractible orbits of $Y_+$ that corresponds to an old orbit of $Y_-$ are also be contractible in $Y_-$.

\subsection{Removing Reeb chords of loose Legendrians}\label{sec: modifying}

Although loose Legendrians satisfy an h-principle, they still have some non-trivial J-holomorphic curve invariants (analogous to the fact that $SH^+(W)$ does not vanish for flexible $W$). These invariants place some non-trivial restrictions on the Reeb chords of such Legendrians. 
For example, given a Liouville domain $X$ and 
a Legendrian $\Lambda \subset \partial X$, 
the Legendrian contact homology $LHA(\Lambda, X)$ is an invariant up to Legendrian isotopy; see \cite{BEE12} for details. For our purposes, it is enough to note that the chain complex of $LHA$ is generated by words of Reeb chords of $\Lambda$ (including the empty word). 
If $\Lambda$ is a loose Legendrian, then it is shown in \cite{EES, Murphy11} that 
$LHA(\Lambda; X)$ vanishes, which implies that $\Lambda$ must have a Reeb chord of degree $1$. Another J-holomorphic curve  invariant of $\Lambda$ is $LH^{Ho, +}(\Lambda, X)$, which is generated by two copies of words of Reeb chords; see \cite{BEE12}. If $SH(X) = 0$, then $LH^{Ho,+}(\Lambda, X)$ is non-zero only in degree $2$; this follows from the fact that $LH^{Ho,+}(\Lambda, X)$ can be realized as the mapping cone of a certain transfer map; see Theorem 5.4
of \cite{BEE12}. In particular, both of these invariants vanish in non-positive degrees. Since loose Legendrians satisfy an h-principle, one might hope that it is possible to remove all homologically unnecessary Reeb chords  and in particular, show that all chords of $\Lambda$ have positive degree (again for fixed action). The follow lemma  shows that it is indeed possible to do this via a $C^0$-small formal Legendrian isotopy 
and hence by a genuine Legendrian isotopy in the case of loose Legendrians. 

\begin{lemma}\label{lem: maingeo}
For any Legendrian $\Lambda^{n-1} \subset (Y^{2n-1}, \alpha), n \ge 3,$ and any positive $\epsilon$ and $D$, there exists a Legendrian $\Lambda' \subset U^\epsilon(\Lambda, \alpha)$ such that $\Lambda'$ is both loose and formally isotopic to $\Lambda$ in $U^\epsilon(\Lambda, \alpha)$ and all elements of $\mathcal{P}^{< D}(\Lambda', Y, \alpha)$ have positive degree.
\end{lemma}
\begin{remark}\

\begin{enumerate}[leftmargin=*]
\item 
Lemma \ref{lem: maingeo} also holds for $\mathcal{W}^{<D}(\Lambda', Y, \alpha)$ instead of $\mathcal{P}^{< D}(\Lambda', Y, \alpha)$; in this case, we work with an arbitrary trivialization of the canonical bundle of $(Y, \alpha)$ but the degree of elements of $\mathcal{W}^{<D}(\Lambda', Y, \alpha)$ is independent of this trivialization since they vanish in $\pi_1(Y,\Lambda')$.
\item 
If we drop the condition that the formal Legendrian isotopy stays in a small neighborhood of $\Lambda \subset Y$, then Lemma \ref{lem: maingeo} becomes vacuous in some cases. For example, if $Y = P \times \mathbb{R}$, where $P$ an exact symplectic manifold, there is a contact isotopy expanding the contact form. 
So for any fixed $D$, there is Legendrian isotopy to some $\Lambda'$ such that $\mathcal{P}^{< D}(\Lambda', Y, \alpha)$ is empty; however, this isotopy is not $C^0$-small. This is analogous to the situation in Remark \ref{rem: semigood}, where  the definition of ADC becomes vacuous if we are allowed to scale the contact forms.
\item 
If $\Lambda$ is loose in $(Y, \alpha)$, by the uniqueness portion of Murphy's h-principle for loose Legendrians, $\Lambda$ and $\Lambda'$ are Legendrian isotopic in $Y$. However, this isotopy is not $C^0$-small; in particular, it must leave $U^\epsilon(\Lambda)$ since $\Lambda'$ is loose in $U^\epsilon(\Lambda)$ but $\Lambda$ is not. 
\item Our proof of Lemma \ref{lem: maingeo}, which uses Legendrian stabilization, does not work for  $n = 2$ because stabilization does not preserve the formal isotopy class. In fact, this is precisely why Murphy's h-principle and the theory of flexible Weinstein domains break down for $n =2$; see Remark \ref{rem: n=2}. Nonetheless,  for $n \ge 2$  we can construct a Legendrian \textit{regular homotopy} to $\Lambda'$; see Lemma \ref{lem: reghtpy} below.
\item We can also assume that all elements of $\mathcal{P}^{< D}(Y, \Lambda', \alpha)$ are non-degenerate and all elements of $\mathcal{P}^{< D}(Y, \alpha)$ are disjoint from $\Lambda'$. 
\end{enumerate}
\end{remark}

We prove Lemma \ref{lem: maingeo}
by modifying the Legendrian to increase the degrees of its Reeb chords. This is done in two steps. First, we prove Lemma \ref{lem: reghtpy}, a weaker version 
of 
Lemma \ref{lem: maingeo}, that uses a Legendrian regular homotopy instead of a formal Legendrian isotopy.  
Then we calculate the self-intersection invariant of this regular homotopy and show that if $n \ge 3$, it can be chosen equal zero, which implies the existence of a formal Legendrian isotopy between the two Legendrians.
\begin{lemma}
\label{lem: reghtpy}
For any Legendrian $\Lambda^{n-1} \subset (Y^{2n-1}, \alpha), n \ge 2,$ and any positive $\epsilon$ and $D$, there exists a Legendrian $\Lambda' \subset U^\epsilon(\Lambda, \alpha)$ such that $\Lambda'$ is both loose and regular homotopic to $\Lambda$ in $U^\epsilon(\Lambda, \alpha)$ and  all elements of $\mathcal{P}^{< D}(\Lambda', Y, \alpha)$ have positive degree. 
\end{lemma}
\begin{proof}[Proof of Lemma \ref{lem: reghtpy}]
Possibly after a $C^0$-small Legendrian isotopy, we can assume that $\Lambda$ is non-degenerate and is disjoint from $\mathcal{P}^{< D}(Y, \alpha)$. Now suppose that $\mathcal{P}^{< D}(\Lambda, Y, \alpha) = \{c_1, \cdots, c_k\}$; note that this is a finite set because all Reeb chords are non-degenerate. To prove this lemma, we will do a local modification near the positive endpoints of $c_1, \cdots, c_k$ that makes their degree positive and does not create any new Reeb chords of non-positive degree. 

Let $B = [0,1] \subset \mathbb{R}$.
We begin by describing a Legendrian in $(J^1(B), \alpha_{std}) \subset (\mathbb{R}^3, \alpha_{std})$ via its front projection. The front projection of $\Lambda^{n-1} \subset (\mathbb{R}^{2n-1}, \alpha_{std}), \alpha_{std} = dz - \sum_{i=1}^{n-1} y_i dx_i,$ is the image of $\Lambda$ in $\mathbb{R}^{n}$ under the projection to the $(x, z)$-plane. The front projection is an immersion except at a finite collection of points where it is a cuspidal singularity; away from these cuspidal points, it is the graph of a multi-valued function. Note that the Legendrian can be recovered from its front projection by setting $y_i = \frac{\partial z}{\partial x_i}$. 
Now let $\gamma_m' \subset B \times [-1,1] \subset \mathbb{R}^2$ be the concatenation of a path from the origin to $(\frac{1}{4}, -\frac{1}{8})$ with slope between $-\frac{2}{3}$ and $0$ and a path from  $(\frac{1}{4}, -\frac{1}{8})$ to  $(\frac{1}{2},0)$ 
with slope with slope between $0$ and $\frac{2}{3}$, 
that contains $m$ non-overlapping zig-zags. In particular, assume that the $k$th zig-zag occurs occurs in the rectangular region 
$(\frac{1}{4} + \frac{2k-2}{8m}, \frac{1}{4} + \frac{2k-1}{8m}) \times 
(-\frac{1}{4} +  \frac{2k-2}{8m}, -\frac{1}{4} + \frac{2k-1}{8m}) \subset \mathbb{R}^2$. Finally, we assume that $\gamma_m'$ has zero slope near $(0,0)$ and $(\frac{1}{2},0)$. Let $\gamma_m$ be the path obtained by concatenating $\gamma_m'$ with its reflection about the vertical line $x= \frac{1}{2}$. See Figure \ref{fig: front}. 
Note that because of the constraint on the slope of $\gamma_m$, the Legendrian lift of $\gamma_m$ (which we also denote by $\gamma_m$) is contained in  $
B \times [-1,1]^2 \subset J^1(B)$.
 \begin{figure}
     \centering
     \includegraphics[scale=0.25]{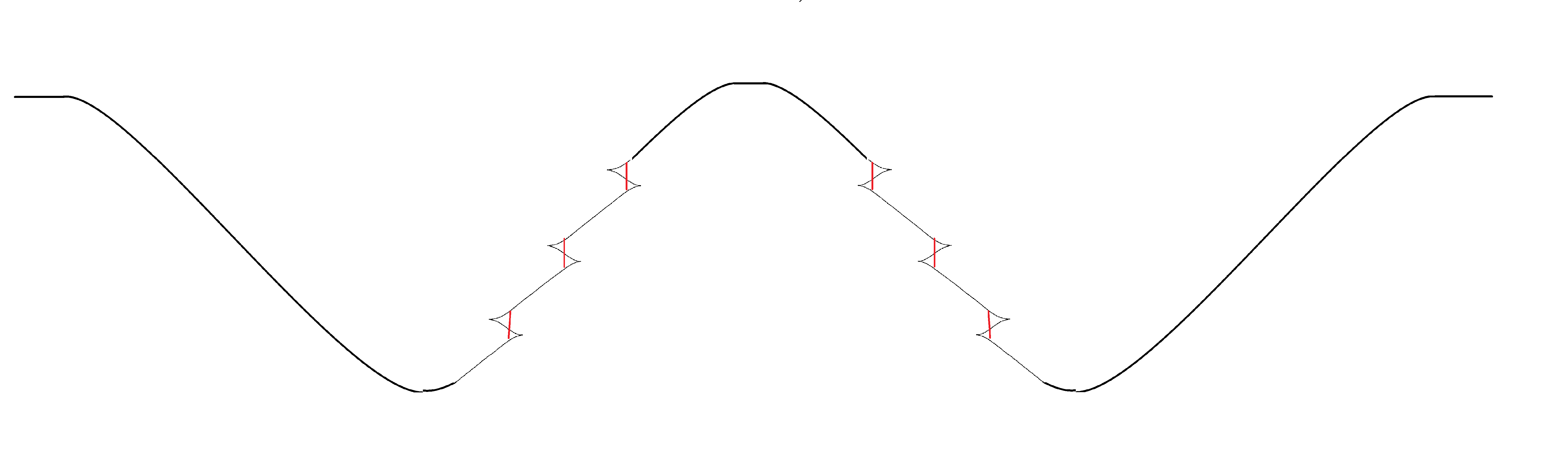}
     \caption{The front projection of $\gamma_m$ when $m = 3$; Reeb chords are drawn as red lines}
     \label{fig: front}
 \end{figure}
 
Now we study the Reeb chords of $\gamma_m$. Suppose that two branches of $\gamma_m$ are given in the front projection by the graphs of $h_1, h_2$. Since $R_{\alpha_{std}} = \partial_z$, the Reeb chords of $\gamma_m$ with endpoints on these two branches correspond to critical points of the height difference function $h_1 - h_2$.  In particular, for each zig-zag in the front projection, there is one Reeb chord in the Legendrian lift of $\gamma_m$ into $\mathbb{R}^3$; see the vertical red lines in Figure \ref{fig: front}. Therefore, the lift has $2m$ Reeb chords. Note that the zig-zags can be made arbitrarily small and therefore the action of the Reeb chords can be made arbitrarily small. 

We now show that each chord of $\gamma_m$ has degree $1$. As we noted, each chord $c$ corresponds to a critical point $p_c$ of the height difference $h_1-h_2$; we will assume that near $p_c$, the graph of $h_1$ defines the branch with greater $z$-coordinate. This critical point is non-degenerate if the Reeb flow is non-degenerate and so $\mbox{Ind}_{h_1 - h_2}(p_c)$ is defined. Lemma 3.4 of \cite{EES} gives the following formula for the degree of $c$ in terms of the front projection of $\gamma_m$: 
\begin{equation}
    \label{eqn: degformula}
  |c| = D-U + \mbox{Ind}_{h_1 - h_2}(p_c) -1,
\end{equation}
  where $D, U$ is the number of times a generic path in $\gamma'_m$ from $a$ to $b$ traverses a cusp downward, upward respectively. 
  Note that order matters and it is important that we consider $h_1-h_2$ and not $h_2-h_1$. 
 From Figure \ref{fig: front}, we see that each Reeb chord $c$ of $\gamma_m$ has $D = 2,  U = 0$ and 
 $\mbox{Ind} = 0$ since the height difference from top endpoint to bottom endpoint has a local minimum at the Reeb chord. Therefore $|c| = 1$. 

Let $Q^{n-2}$ be a oriented, closed $n-2$ dimensional manifold that has an embedding into $\mathbb{R}^{n-1}$ and let $A: =B \times Q$. Note that $A$ also has an embedding into $\mathbb{R}^{n-1}$ since the normal bundle of $Q^{n-2} \subset \mathbb{R}^{n-1}$ is trivial. 
Then $(J^1(A), \alpha_{std}) = (J^1(B) \times T^*Q, \alpha_{std} - \lambda_{std})$. 
Let $j_A: A \hookrightarrow J^1(A), j_Q: Q \hookrightarrow T^*Q$ denote the inclusions of the respective zero sections. Let $p := (\frac{1}{2}, p_0) \in A$ for some $p_0 \in Q$.

Now consider the Legendrian embedding 
$\Gamma_m: A \rightarrow J^1(A)$ given by 
$$
\Gamma_m(x, z) = (\gamma_m(x), j_Q(z) ) \subset  B \times [-1,1]^2\times T^*Q \subset J^1(A),
$$
which agrees with $j_A$ near $\partial A$. Since $\gamma_m$ has zero slope near $(\frac{1}{2}, 0)$, $\Gamma_m$ also agrees with $j_A$ near $p$. The Reeb chords of $\Gamma_m$ are degenerate; in particular, $\Gamma_m$ has a $Q$-family of Reeb chords for each Reeb chord of $\gamma_m$. We can use Morse function on $Q$ to make a $C^2$-small perturbation of $\Gamma_m$ in the 
$Q$-direction in a neighborhood of the \textit{positive} chord endpoints. Then all Reeb chords of $\Gamma_m$ are non-degenerate. Furthermore, each of these chords has positive degree. To see this, suppose $h$ is a Morse function on $Q$ with  critical points $p_1, \cdots, p_q$.  Each critical point $p_j$ corresponds to $2m$ Reeb chords $d_{p_j}^1, \cdots, d_{p_j}^{2m}$ of $\Gamma_m$ (one for each zig-zag). Since all of the chords of $\gamma_m$  have degree $1$, by Equation \ref{eqn: degformula} 
\begin{equation}
\label{eqn: degindex}
|d_{p_j}^i| = 1 + \mbox{Ind}_h(p_j).
\end{equation}
To see this, note that the only term in Equation \ref{eqn: degformula} that changes for chords of $\Gamma_m$ from chords of $\gamma_m$ is the Morse index, which increases by  $\mbox{Ind}_h(p_j)$, the index of the height difference of $\Gamma_m$ in the $Q$-direction. Since $\mbox{Ind}_h(p_j)$ is always non-negative, we have $|d_{p_j}^i| \ge 1$. Therefore, $\Gamma_m$ has $2m q$ Reeb chords, all with positive degree.

\begin{remark}
Note that $\Gamma_m$  is loose (relative to its boundary) for $n \ge 3$ and hence from Section \ref{sec: reeb_chord_loose}, we know that the Legendrian contact homology of $\Gamma_m$ vanishes.
On the other hand, any Morse function $h$ on $Q$ has a minimum $p$ for which $\mbox{Ind}_h(p) = 0$.
Therefore if $c$ is a Reeb chord of $\Gamma_m$ corresponding to $p$
(any one of the $2m$ such chords), then by Equation (\ref{eqn: degindex}) $|c| = 1$. In fact, $d^{leg} c = 1$, where $d^{leg}$ is the differential for Legendrian contact homology of $\Gamma_m$. So $c$ is the chord responsible for killing the Legendrian contact homology of $\Gamma_m$. 
\end{remark}

 \begin{figure}
         \centering
         \includegraphics[scale=0.22]{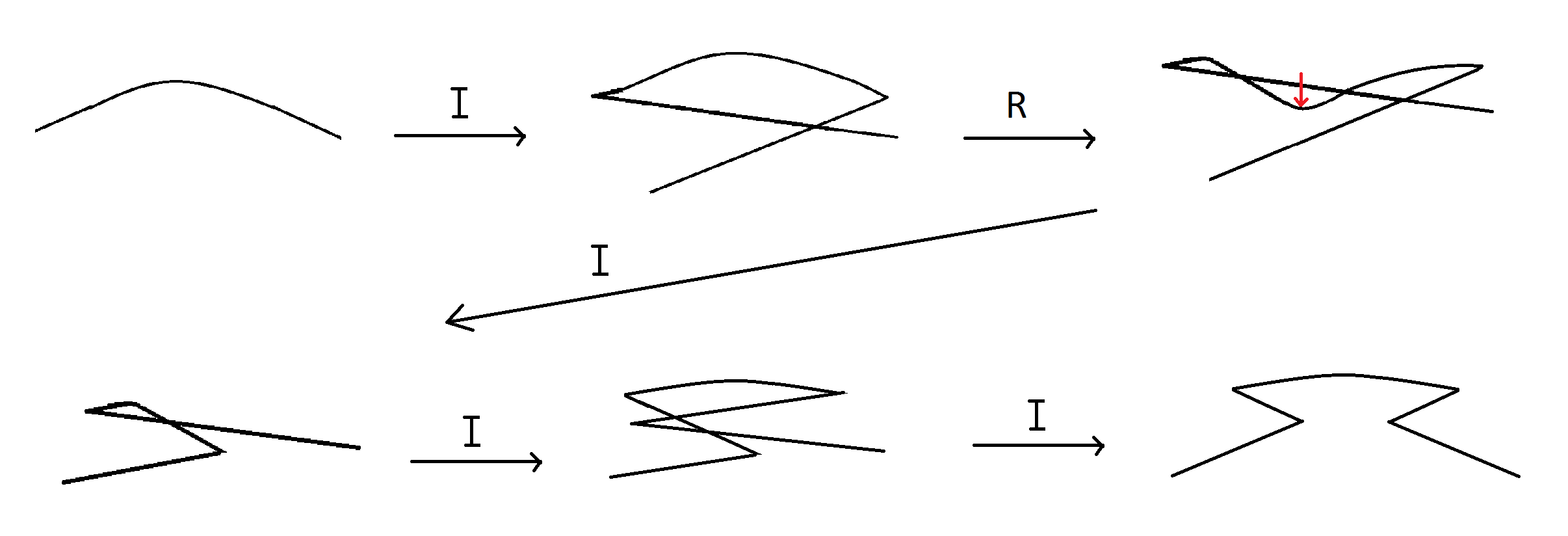}
         \caption{Legendrian regular homotopy $f_{B,t}$ from $j_B$ to  $\Gamma_1$ in the front projection;
        I, R denote Legendrian isotopy, Legendrian regular homotopy with single self-intersection respectively}
         \label{fig: isotopy}
     \end{figure}

Now we show that 
$j_A$ and $\Gamma_m$ are Legendrian regular homotopic inside $J^1(A)$ relative $\partial A$. 
First note that $\gamma_m$ is Legendrian regular homotopic to the inclusion of the zero section 
$j_B: B \rightarrow J^1(B)$ relative $\partial B$. 
Figure \ref{fig: isotopy} depicts one possible Legendrian regular homotopy near a single cusp-pair. We have broken down the regular homotopy into several steps. All steps are Legendrian Reidemeister moves, which are Legendrian isotopies, except for the second step, which is Legendrian stabilization near a cusp and hence a Legendrian regular homotopy.  Let $f_{B,t}$ denote the entire Legendrian regular homotopy from $j_B$ to $\gamma_m$. 
Note that $f_{B,t}$ has $m$ transverse self-intersection points, one for each zig-zag; see the arrow in third diagram in Figure \ref{fig: isotopy}. Finally, we define the Legendrian regular homotopy $f_{A,t}$ from $j_A$ to $\Gamma_m$ to be $f_{B,t}$ on the $B$ component of $A$ and the constant on the $Q$ component. Furthermore, by 
using a Morse function on $Q$ when performing the stabilization step of the regular homotopy, we can assume that all intersection points of $f_{A,t}$ are transverse; see Section 7.4 of \cite{CE12}. Again, if we use a Morse function on $Q$ with $q$ critical points, then $f_{A,t}$ will have $mq$ transverse self-intersection points. 

Having described the local modification, we want apply it to the original Legendrian $\Lambda \subset (Y, \alpha)$. To do so, we need a \textit{strict} 
 version of the Weinstein neighborhood theorem.
\begin{proposition}\cite{W}
\label{thm: Geiges}
For any Legendrian $\Lambda \subset (Y, \alpha)$, there exists a neighborhood $U$ of $\Lambda$ and a neighborhood $V$ of the zero section of $J^1(\Lambda)$ and a strict contactomorphism  $\phi: (U, \alpha) \rightarrow (V, \alpha_{std})$ mapping $\Lambda \subset U$ to the zero section of $J^1(\Lambda)$. 
\end{proposition}

Recall that $\mathcal{P}^{< D}(\Lambda, Y, \alpha) = \{c_1, \cdots, c_k\}$. 
Let $c_i^+ \in \Lambda$ be the positive endpoint of $c_i$; positive/negative endpoints are defined by the condition that  the Reeb flow takes the negative endpoint to the positive endpoint.
Identity a small neighborhood of $c_i^+$ in $\Lambda^{n-1}$ with $\mathbb{R}^{n-1}$. Then because $A$ embeds into $\mathbb{R}^{n-1}$, for each $i$, there exists an embedding of $A$ into a small neighborhood of $c_i^+$ in $\Lambda$ taking $p \in A$ to $c_i^+$. In particular, by Proposition \ref{thm: Geiges}, there exist disjoint $U_1, \cdots, U_k \subset \Lambda$ such that $c_i^+ \in U_i$ and $(U_i, U_i \cap \Lambda, c_i^+)\subset (Y, \alpha)$ is strictly contactomorphic to 
$(U^\delta(A), A, p) \subset J^1(A)$ for some sufficiently small $\delta$,
smaller than the $\epsilon$ in the statement of this lemma. Since there are finitely many $c_i$, we also can assume that  $c_i \cap U_j = \varnothing$ for $i \ne j$; there may be higher action chords that intersect $U_j$ but we do not care about these chords.

Let $N: = -\underset{1\le i\le k}{\min} |c_i|$ +1. We can assume that $\underset{1\le i\le k}{\min} |c_i| \le 0$ since otherwise all chords already have positive degree; so $N \ge 1$. We will use $\Gamma_N$ to increase the degree of $c_i$. By scaling $\Gamma_N$ and $f_{A,t}$ by $(x_i, y_i, z) \rightarrow (x_i, a y_i, az)$ for some sufficiently small $a$, we can assume that their images are contained in 
$U^\delta(A) = \{\|y\|, |z| < \delta  \} \subset J^1(A)$. Using the strict contactomorphism 
between
$(U^\delta(A), A, p)$
and 
 $(U_i, U_i \cap \Lambda, c_i^+)\subset (Y, \alpha)$, we can transfer the local Legendrian regular homotopy $f_{A,t}$ to a Legendrian regular homotopy $f_{Q,N,t}$ from $\Lambda \subset Y$ to a Legendrian embedding $\Lambda_{Q,N} \subset Y$. Note that $f_{Q,N,t}$ is supported in the union of the $U_i$'s and hence in $U^\epsilon(\Lambda, \alpha) \subset (Y, \alpha)$. 

We now show that $\Lambda_{Q,N}$ has the desired properties. First note that $\Lambda_{Q,N}$ is loose in $U^\epsilon(\Lambda, \alpha)$ by the discussion in Section \ref{subsection: loose}.  
Now we study the Reeb chords of $\Lambda_{Q,N}$. Since $\Lambda, \Lambda_{Q,N}$ agree near the chord endpoints $c_i^+$ and $c_1, \cdots, c_k$ are Reeb chords of $\Lambda$, they are also Reeb chords of $\Lambda_{Q,N}$. Furthermore, they have the same action and hence also belong to $\mathcal{P}^{< D}(\Lambda_{Q,N}, Y, \alpha)$. By making zig-zags in $\gamma$ arbitrarily small, we can get the action of the $2Nq$ Reeb chords $d_{p_j}^i$ arbitrarily small so that $d_{p_j}^i \in \mathcal{P}^{< D}(\Lambda_{Q,N}, Y, \alpha)$ as well; note that $d_{p_j}^i$ is zero in $\pi_1(Y, \Lambda_{Q,N})$ since $d_{p_j}^i$ is contained in a small chart where topologically, $(Y, \Lambda_{Q,N})$ looks like $(B^{2n-1}, B^{n-1})$. In fact, by scaling $\Gamma_N$ down enough, we can assume that
\begin{equation}\label{eqn: chordsenumerate}
\mathcal{P}^{< D}(\Lambda_{Q,N}, Y, \alpha) = \{c_1, \cdots, c_k, d_{p_j}^i \}.
\end{equation}
To see this, suppose that there is a Reeb chord in $\mathcal{P}^{< D}(\Lambda_{Q,N}, Y, \alpha)$ which does not appear on the right-hand-side of Equation \ref{eqn: chordsenumerate} no matter how much we scale $\Gamma_N$. As a result, we get a sequence of chords which either converge to a chord 
$c_i \in \mathcal{P}^{< D}(\Lambda, Y, \alpha)$ or shrink to the empty set. In the first case, we note that because $\Gamma_N$ agrees with $A$ near $p$ (which equals to $c_i^+$), scaling does nothing here. So eventually the chords in the sequence coincide with a chord $c_i \in \mathcal{P}^{< D}(\Lambda, Y, \alpha)$. In the second case, chords in the sequence have arbitrarily small action and therefore are eventually contained in $U_i$ and hence must coincide with one of $d_{p_j}^i$.
Since $\Lambda$ is disjoint from all elements of $\mathcal{P}^{< D}(Y, \alpha)$,  we can also assume that this is true for $\Lambda_{Q,N}$ by scaling $\Gamma_N$ down enough. Also, note that all elements of $\mathcal{P}^{< D}(\Lambda_{Q,N}, Y, \alpha)$ are non-degenerate.

Finally, we show that all elements of $\mathcal{P}^{< D}(\Lambda_{Q,N}, Y, \alpha)$ have positive degree.
 Note that  since $\Lambda_{Q,N}$ is Legendrian regular homotopic to $\Lambda$ and $\Lambda$ has vanishing Maslov class $c_1(Y, \Lambda)$, so does $\Lambda_{Q,N}$; hence it makes sense to talk about the degree of Reeb chords of $\Lambda_{Q,N}$. 
 As we noted in Equation \ref{eqn: degindex}, the degrees of the $d_{p_j}^i$ are always positive. Although $c_i$ are Reeb chords of both $\Lambda$ and $\Lambda_{Q,N}$, they have different degrees; in particular, let $|c_i|, |c_i|'$ denote the degree of $c_i$ as a Reeb chord of $\Lambda, \Lambda_{Q,N}$ respectively. Now we show that  
\begin{equation}\label{eqn: degchange}
|c_i|' - |c_i| = 2N
\end{equation}
for all $i$. Since $\Lambda$ and $\Lambda_{Q,N}$ differ just inside the $U_i$'s but agree inside a smaller neighborhood of the $c_i^+$, then the change in degrees can be calculated locally in $U_i$; see the general formula in \cite{BEE12}. In particular, it is important that the rest of the $U_j$ are disjoint from the entire chord $c_i$ since otherwise, the $U_j$ modification can also affect the degree of $c_i$. 
Since $U_i \subset (\mathbb{R}^{2n-1},\alpha_{std})$, we can use Equation \ref{eqn: degformula} to show that $|c_i|' - |c_i| = D - U$, where $D,U$ are the number of down,up cusps respectively that a generic path from $c_i^+$ to $\partial A$ traverses in the front projection. Since $c_i^+$ corresponds to $p = (\frac{1}{2}, p_0) \in B \times Q$, we can take the path to be $\gamma_N$. Because $\gamma_N$ has $N$ zig-zags, each of which have two down-cusps and no up-cusps, we see that $D-U = 2N$, which proves Equation \ref{eqn: degchange}. Since  $N = -\mbox{min}|c_i| +1$, we have that $|c_i|' \ge (-N+1) + 2N = N+1 \ge 2$; in particular $|c_i|'$ is positive. Note that by taking $N$ arbitrarily large and introducing more $d_{p_j}^i$, we can make $|c_i|'$ arbitrarily large.
\end{proof}

\begin{remark}
The proof of Lemma \ref{lem: reghtpy} involved modifying the Legendrian by adding zig-zags near the endpoints of Reeb chords. 
A similar modification is performed in Lemma 4.2 of \cite{EM} in order to change the action of Lagrangian self-intersection points, which can then be cancelled. Our modification was used to change chord degree rather than action. Furthermore, the modification in \cite{EM} is not $C^0$-small, which we shall need; see the proof of Lemma \ref{lem: techkey} in the Appendix. 
Nonetheless, this cancellation can be used to remove Reeb chords of Legendrians in certain special contact manifolds. For example, Reeb chords of $\Lambda \subset P \times \mathbb{R}$ correspond to self-intersection points of the Lagrangian projection of $\Lambda$ in $P$. The procedure from \cite{EM} can be used to remove most (but not all) Reeb chords of $\Lambda$, as done for $P = \mathbb{C}^n$ in \cite{RGEstimating}.
\end{remark}

Although $\Lambda$ and 
$\Lambda'=\Lambda_{Q,N}$ are Legendrian regular homotopic, they are not necessarily formally Legendrian isotopic. For instance, if $\Lambda, \Lambda_{Q,N}$ are null-homologous in $Y$, then their Thurston-Bennequin invariants are an obstruction to the existence of a formal isotopy between them; see Appendix B of \cite{CE12}. However, we will show in Lemma \ref{lem: maingeo} that for appropriate $Q$, $\Lambda$ and $\Lambda_{Q,N}$ are actually formally Legendrian isotopic.

We begin by reviewing some facts about the self-intersection index, which is closely related to the relative Thurston-Bennquin invariant; see Appendix B of \cite{CE12}. If $f_t$ is a (smooth) regular homotopy of an $(n-1)$-dimensional manifold $\Lambda^{n-1}$ in an oriented $(2n-1)$-dimensional manifold and $f_t$ has transverse self-intersection points, the self-intersection index $I(f_0, f_1; f_t)$ of $f_t$ is defined to be the signed count of the number of self-intersection points of $f_t$. Note that in our situation, the contact form $\alpha$ provides $Y$ with a canonical orientation given by $\alpha\wedge(d\alpha)^{n-1}$. This index takes values in $\mathbb{Z}$ if $n$ is even and $\Lambda$ is orientable and in $\mathbb{Z}_2$ otherwise. If $Y^{2n-1}, n \ge 3$, is simply-connected and oriented, Whitney proved that $f_t$ can be deformed to an isotopy relative to its endpoints.
However, this deformation is not necessarily through regular homotopies. In fact Whitney also proved that if $Y^{2n-1}, n \ge 3,$ is simply-connected and oriented, $f_t$ can be deformed through regular homotopies with fixed endpoints to an isotopy if and only if $I(f_0, f_1; f_t)$ vanishes; see Chapter 7 of \cite{CE12}. 
This vanishing will be relevant to our situation because in certain cases the resulting isotopy can be lifted to a formal Legendrian isotopy. 

We now compute the self-intersection index for the regular homotopy $f_{Q, N, t}$ constructed in the previous lemma. Since $f_{Q,N,t}$ is a collection of distinct Legendrian regular homotopies occurring in the disjoint subsets $U_i$, we compute just $I(\Lambda, \Lambda_{Q, N}; f_{Q,N,t}|_{U_i})$ in the following lemma; here 
$f_{Q,N,t}|_{U_i}$ means 
$f_{Q,N,t}|_{U_i}: U_i \cap \Lambda \rightarrow U_i$.  
Furthermore, since all the homotopies in the $U_i$'s are modelled on the same homotopy, $I(\Lambda, \Lambda_{Q,N}; f_{Q,N,t}|_{U_i})$ is independent of $i$.

\begin{lemma}
\label{lem: indexcalculation}
The self-intersection index of $f_{Q,N,t}|_{U_i}$ is 
$$
I(\Lambda, \Lambda_{Q,N}; f_{Q, N, t}|_{U_i}) 
= (-1)^{(n-1)(n-2)/2}  N \cdot \chi(Q).
$$ 
\end{lemma}
\begin{proof}
Since $\gamma_{N}$ has $N$ zig-zag pairs,  $f_{Q,N,t}|_{U_i}$ can be broken up into $N$ repetitions of the Legendrian regular homotopy  depicted in Figure \ref{fig: isotopy} and so $I(\Lambda, \Lambda_{Q,N}, f_{Q, N,t}|{U_i})$ is just $N$ times the self-intersection index of that homotopy. 
Note that all steps in Figure \ref{fig: isotopy} are Legendrian isotopies except for the stabilization step. This step is essentially the same as the stabilization procedure defined in \cite{CE12}. There stabilization over $M^{n-1} \subset B^{n-1}$ is defined by pushing a Legendrian $B^{n-1}$ upward past a parallel Legendrian $B^{n-1}$ along $M$. Lemma 7.14 of \cite{CE12} proves that the self-intersection index of this stabilization is $(-1)^{(n-1)(n-2)/2}\chi(M)$ by computing the local index of \textit{each} self-intersection point.
 Our stabilization procedure is slightly different; namely, we push a Legendrian $B\times Q$ \textit{down} past a parallel Legendrian $B\times Q$ along a smaller subset $B\times Q$. The self-intersection index of each self-intersect point is computed locally and hence these local indices coincide; the fact that the stabilization in \cite{CE12} moves up and ours moves down contributes a sign change of $(-1)^n$, which has no effect since the index is
$\mathbb{Z}/2$-valued for $n$ odd. Therefore, the \textit{total} self-intersection indices of the two stabilizations also coincide and our lemma follows.
\end{proof}

Now we complete the proof of Lemma \ref{lem: maingeo} using Lemma \ref{lem: indexcalculation}.

\begin{proof}[Proof of Lemma \ref{lem: maingeo}]

We first show that if $I(\Lambda, \Lambda_{Q,N}; f_{Q,N,t}|_{U_i}) = 0$ for all $i$, then $\Lambda$ and $\Lambda_{Q,N}$ are formally Legendrian isotopic in $U^\epsilon(\Lambda, \alpha)$. 
Let $V_i \subset U^\epsilon(\Lambda) \subset Y$ be a subset containing $U_i$ 
such that $(V_i,U_i)$ is strictly contactomorphic to 
$(U^\delta(B^{n-1}), U^\delta(A))  \subset (J^1(B^{n-1}), J^1(A))$; here we think of $A \subset B^{n-1}$ as subsets of $\Lambda$.  
Since $V_i$ is an arbitrarily small neighborhood of  $c_i^+$, we can also assume that $V_i \cap U_j = \varnothing$ for $i\ne j$. Note that $V_i \simeq  B^{2n-1}$
is simply-connected, $I(\Lambda, \Lambda_{Q,N}; f_{Q,N,t}|_{V_i}) = I(\Lambda, \Lambda_{Q,N}; f_{Q,N,t}|_{U_i}) = 0$, and $f_{Q,N,t}|_{\partial V_i \cap \Lambda}$ is $t$-independent. So we can apply a relative version of Whitney's Theorem to 
the regular homotopy $f_{Q, N,t}|_{V_i}: (V_i\cap \Lambda, \partial V_i \cap \Lambda) \rightarrow (V_i, \partial V_i)$; note that this requires $n \ge 3$.  Hence there is a family of (smooth) regular homotopies $f_{s,t}:  V_i\cap \Lambda \rightarrow V_i$ 
from $f_{Q, N,t}|_{V_i}$ to a smooth isotopy $g_t$ such that $f_{s,t}|_{\partial  V_i \cap \Lambda}$ is $(s,t)$-independent and therefore agrees with the inclusion $\partial V_i \cap \Lambda \rightharpoondown \partial V_i$; furthermore $f_{s, t}$ has fixed endpoints, i.e. 
$f_{s, t}$ agrees with $\Lambda, \Lambda_{Q,N}$ for $t = 0, 1$ respectively,  for all $s$.
Also, note that since we are working locally in $V_i$, we do not need $Y$ to be simply-connected as required for the global Whitney Theorem. Since $V_i$ is strictly contactomorphic to the ball $(U^\delta(B^{n-1}), \alpha_{std}) \subset (\mathbb{R}^{2n-1}, \alpha_{std})$, we can trivialize the Legendrian planes in $V_i$. 
Using this trivialization and the family of regular homotopies $f_{s,t}$ (whose image is contained in $V_i$), we get a family of 
tangent planes over $g_t$ that start at $Tg_t$ and end at Lagrangian planes since $f_{Q,N,t}  = f_{1, t}$ is a Legendrian regular homotopy. 
In particular, $g_t$ is a formal Legendrian isotopy in each $U_i$ and so the global isotopy obtained by composing all the $g_t$'s for all $i$ is a formal Legendrian isotopy between $\Lambda$ and $\Lambda_{Q,N}$.

In view of the above discussion and Lemma \ref{lem: indexcalculation}, to complete the proof of this lemma, we just need to find an orientable manifold $Q^{n-2}$ that embeds into $\mathbb{R}^{n-1}$ and has $\chi(Q) = 0$. If $n \ge 4$, we can take $S^1\times S^{n-3}$ and if $n =3$, we can take $S^1$. 
\end{proof}

\begin{remark}\label{rem: n=2}\
\begin{enumerate}[leftmargin=*]
\item Note that if $n =2$, all manifolds $Q^0$ have $\chi(Q^0)> 0 $. This is precisely why Lemma \ref{lem: maingeo} only holds for $n \ge 3$. 
\item 
The last part of the proof of Lemma \ref{lem: maingeo} is similar to the Appendix of \cite{Murphy11}, which describes all formal Legendrians in $(\mathbb{R}^{2n-1}, \alpha_{std})$ up to formal Legendrian isotopy. Both rely on the fact that the Legendrian planes of $(\mathbb{R}^{2n-1}, \alpha_{std})$ can be trivialized. The Appendix of \cite{Murphy11} uses this to address the question of when a given smooth isotopy lifts to a formal Legendrian isotopy. Our situation is simpler because we already know that our smooth isotopy can be deformed through regular homotopies to a Legendrian regular homotopy. 
\item 
The proof of Lemma \ref{lem: maingeo} also shows that after Legendrian isotopy, \textit{all} chords of a loose Legendrian $\Lambda\subset (P\times \mathbb{R}, \alpha_{std})$ have positive degree 
(not just those with bounded action). This is because a generic Legendrian $\Lambda \subset (P\times \mathbb{R}, \alpha_{std})$ has finitely many Reeb chords and so we do not need to work  below a fixed action bound. 
But due to the existence of wild chords with large action, it is not clear whether this is possible for a Legendrian in a general contact manifold.
\end{enumerate}
 
\end{remark}

\section{Contact surgery and asymptotically dynamically convex structures}\label{sec: surgerysemigood} 

In this section, we prove Theorems
\ref{thm: MLYau}, \ref{thm: MLYauindex2}, and \ref{thm: semi-surgery} that subcritical and flexible surgery preserve asymptotically dynamically convex contact structures.
We mostly focus on the flexible case since the main argument in the subcritical case was proven by M.-L.Yau.
The underlying geometric content in the flexible case is in Lemma \ref{lem: maingeo}, which shows that any loose Legendrian can be isotoped to another Legendrian whose Reeb chords with fixed action bound have positive degree. Using Proposition \ref{prop: BEE}, this translates into the condition that all Reeb orbits with fixed action bound of the resulting surgered contact manifold have positive degree. This is almost enough to prove that the resulting contact structure is ADC, except that we have little control over the resulting contact forms $\alpha_k$ and so the non-increasing condition 
$ \alpha_1 \ge \alpha_2 \ge \alpha_3 \cdots$ in Definition \ref{def: semigood} is not necessarily satisfied. As was mentioned in Remark \ref{rem: semigood}, all conditions in the Definition \ref{def: semigood}
are crucial since otherwise this definition becomes vacuous. However, the following more refined version of Lemma \ref{lem: maingeo}, which we will prove in the Appendix,  does include control over the contact forms. 

\begin{lemma}\label{lem: techkey}
Consider Legendrians $\Lambda_0 \subset (Y, \alpha)$ and $\Lambda \subset U^{\epsilon/4}(\Lambda_0, \alpha)$ such that $\Lambda$ is loose and formally Legendrian isotopic to $\Lambda_0$ inside $U^{\epsilon/4}(\Lambda_0, \alpha)$. 
Then for any $D>0$ and $\delta < 1$, there is a contactomorphism $h$ 
of  $(Y, \alpha)$ such that 
\begin{itemize}[leftmargin=10.5mm]
\item $h$ is supported in $U^\epsilon(\Lambda_0, \alpha)$ and $h^* \alpha < 4\alpha$
\item all elements of 
$\mathcal{P}^{< D}(h(\Lambda), Y, \alpha)$ have positive degree
\item  $h(\Lambda) \subset U^{\delta}(\Lambda_0, \alpha)$ is both loose and formally Legendrian isotopic to $\Lambda_0$ in $U^{\delta}(\Lambda_0, \alpha)$. 
\end{itemize}
\end{lemma}

\begin{remark}\
\begin{enumerate}[leftmargin=*]
\item Again, this lemma holds for $\mathcal{W}^{<D}$ instead of $\mathcal{P}^{< D}$. 
\item We can assume $h(\Lambda)$ is generic in the sense that all elements of $\mathcal{P}^{< D}(h(\Lambda), Y, \alpha)$ are non-degenerate and all elements of  $\mathcal{P}^{< D}(Y, \alpha)$ are disjoint from $h(\Lambda)$.  
\item 
If a chord $c$ has $|c| > 3-n$, then the corresponding orbit $\gamma_c$ has $|\gamma_c| > 0 $; see Proposition \ref{prop: BEE}. However the condition $|c| > 3-n $ is not enough to conclude that all orbits have positive degree. 
For example, consider $c$ with $|c| = -1$, which is greater than $3-n$ for $n \ge 5$. Then the word $w$ consisting of $n-3$ copies of $c$ has 
$|w| = -(n-3)$ and so $|\gamma_w| = |w| + n-3 = 0$, which violates the definition of ADC contact forms. This is why we require the chords in Lemma \ref{lem: techkey} to have positive degree (instead of just degree bigger than $n-3$).
\end{enumerate}
\end{remark}

The definition of asymptotically dynamically convex contact structures involves a sequence of decreasing contact forms. When we do contact surgery along an isotropic submanifold, we need to assume that the resulting sequence of Weinstein handles are nested inside one another so that the new forms are still decreasing; furthermore, the handles need to have model contact forms so that the Reeb flow inside the handles is standard.  As the following proposition shows, we can achieve this if the original sequence of contact forms is well-behaved in a neighborhood of the isotropic attaching sphere. 

\begin{proposition}\label{prop: techsurgery}
Let $\Lambda \subset (Y^{2n-1}_-, \xi_-), n\ge 3,$ be a Legendrian sphere and  $(Y_+, \xi_+)$ the result of contact surgery on $\Lambda$. 
Suppose $(Y_-, \xi_-)$ is an asymptotically dynamically convex contact structure with $(\alpha_k, D_k)$ as in Definition \ref{def: semigood}. If  $\alpha_k|_{U^{\epsilon}(\Lambda, \alpha_1)}
 = c_k \alpha_1|_{U^{\epsilon}(\Lambda, \alpha_1)}$
for some constants $\epsilon, c_k$  
and all elements of $\mathcal{W}^{<D_k}(Y_-, \Lambda, \alpha_k)$ have positive degree, then $(Y_+, \xi_+)$ is also asymptotically dynamically convex. 
\end{proposition}

\begin{remark}
As mentioned in Remark \ref{rem: parametrization}, we need to consider $\Lambda$ as a parametrized sphere for contact surgery to be well-defined; however, since this parametrization does not affect our proof, we do not include it here. Furthermore, we need to assume that $\Lambda$ is generic in the sense that all elements of 
$\mathcal{P}^{<D_k}(Y_-, \Lambda, \alpha_k)$ are non-degenerate 
and all elements of $\mathcal{P}^{<D_k}(Y_-, \alpha_k)$ are disjoint from $\Lambda$. 
\end{remark}
\begin{proof}[Proof of Proposition \ref{prop: techsurgery}]
We first review Weinstein n-handles, following \cite{W}. Consider $(\mathbb{R}^{2n}, \omega_{std})$ with Liouville form 
$\lambda = \sum_{i=1}^n(2q_i dp_i + p_i dq_i)$ and Liouville vector field 
$v = \sum_{i=1}^n (2q_i \frac{\partial }{\partial q_i} - p_i\frac{\partial }{\partial p_i})$. Note that $v$ has only one zero, which occurs at the origin. Also $v$ is transverse to the hypersurfaces $X_- := \{\sum_{i=1}^n (q_i^2 -\frac{1}{2}p_i^2) =-1 \} \cong S^{n-1} \times \mathbb{R}^{n}$ and 
$X_+ := \{\sum_{i=1}^n q_i^2 = \delta\} \cong \mathbb{R}^n \times S^{n-1}$, the radius-$\delta$ cotangent bundle $S^\delta T^*\mathbb{R}^n$  of the Lagrangian $\mathbb{R}^n = \{q_1 = \cdots = q_n = 0\}$.  Therefore, $\lambda|_{X_-}, \lambda|_{X_+}$ are contact forms.

The Weinstein handle $H \cong D^n \times D^n$ is the compact subset of $\mathbb{R}^{2n}$ bounded by $X_-$ and $X_+$.
The boundary of $H$ has two components $\partial_\pm H \subset X_{\pm}$ that meet at the corner $\partial_+H \cap \partial_- H = X_- \cap X_+ \cong S^{n-1} \times S^{n-1}$. Since $\partial_\pm H \subset X_{\pm}$,  $\lambda|_{\partial_\pm H}$ are contact forms. Note that $X_-$ contains the Legendrian sphere $S^{n-1} =  X_- \cap \{q_1 = \cdots = q_n = 0 \}$. Hence there is a neighborhood 
$U^{\epsilon}(S^{n-1}, \lambda|_{X_-})$ of $S^{n-1}$ in $X_-$ that is strictly contactomorphic to 
$U^{\epsilon}(S^{n-1}) \subset 
(J^1(S^{n-1}),  \alpha_{std})$.
The lower boundary $\partial_-H = S^{n-1} \times D^n$ is a small neighborhood of $S^{n-1} \subset X_-$ and 
by shrinking $\delta$, we can assume that $\partial_-H \subset U^{\epsilon}(S^{n-1}, \lambda|_{X_-})$. To emphasize this dependence on $\epsilon$, we denote this Weinstein handle $H_1^\epsilon$. We also take $\delta$'s so that $H_1^{\epsilon'} \subset H_1^\epsilon$ if $\epsilon' < \epsilon$. See Figure \ref{fig: handle}.

\begin{figure}
     \centering
     \includegraphics[scale=0.5]{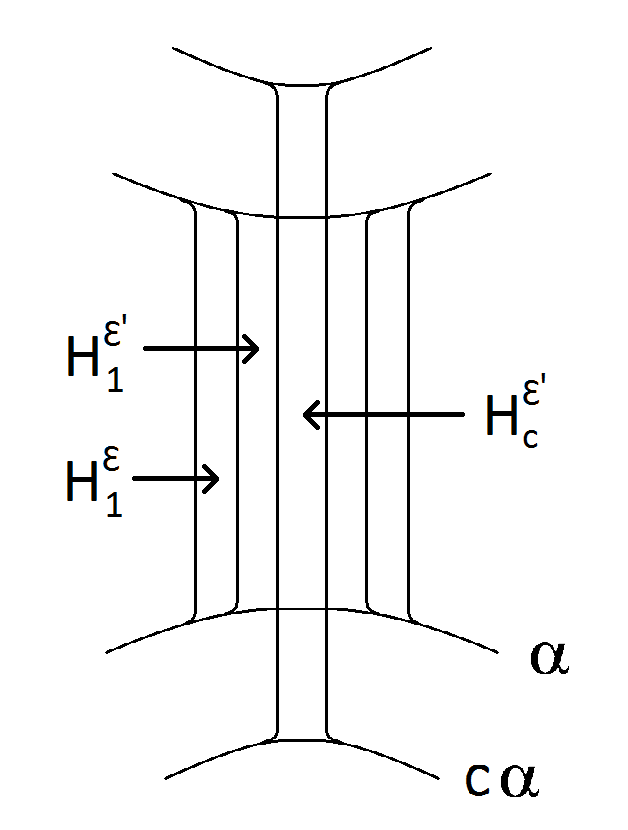}
     \caption{Weinstein handles $H_1^{\epsilon},
H_1^{\epsilon'},$ and $H_c^{\epsilon}$}
     \label{fig: handle}
 \end{figure}

By modifying $X_+$, we can  make $X_+$ agree with $X_-$ agree along $X_- \backslash U^{\epsilon}(S^{n-1}, \lambda|_{X_-})$ and hence smooth the corners between $\partial_+ H_1^\epsilon$ and $X_-$.
We modify $X_+$ by setting $X_+ = \{F(\sum_{i=1}^n q_i^2, \sum_{i=1}^n p_i^2) = 0\}$, where $F: \mathbb{R}^2 \rightarrow \mathbb{R}$ is any smooth function such that $F(x,y) = x- \delta$ away from $X_-$ and $F(x,y) = x - \frac{1}{2}y + 1$ near 
$X_- \backslash U^{\epsilon}(S^{n-1}, \lambda|_{X_-})$. Also, to ensure that $\lambda|_{X_+}$ is still a contact form, we require that the partial derivatives of $F$ do not have the same sign, $\frac{\partial F}{\partial x}$ is not zero when $y=0$,  $\frac{\partial F}{\partial y}$ is not zero when $x=0$, and $F(0,0)\neq 0$; see \cite{W}.

Furthermore, we can take $F$ so that  $\partial_-H_1^\epsilon$ is a \textit{convex} subset of $U^\epsilon(S^{n-1}) \subset (J^1(\Lambda), \alpha_{std})$ under the strict contactomorphism from
$U^{\epsilon}(S^{n-1}, \lambda|_{X_-})$ to
$U^\epsilon(S^{n-1})$. Here convexity in $U^\epsilon(S^{n-1}) \subset J^1(S^{n-1})$ is with respect to the linear $y, z$ coordinates on $J^1(S^{n-1})$. To see this, we construct the strict contactomorphism by first taking a diffeomorphism $\phi$ from a neighborhood $U$ of $S^{n-1} \subset X_-$ to a neighborhood $U'$ of $S^{n-1} \subset J^1(S^{n-1})$ such that $\phi|_{S^{n-1}}= Id$ and $\phi(\partial_-H_1^\epsilon)$ is convex in $J^1(S^{n-1})$ for all sufficiently small $\epsilon$. Then we take a strict contactomorphism $\phi'$ from $(U', \phi_*\alpha) \subset J^1(S^{n-1})$ to $(U'', \alpha_{std}) \subset J^1(S^{n-1})$ such that $\phi'|_{S^{n-1}} = Id$. Since $\phi'$ is a diffeomorphism such that $\phi'|_{S^{n-1}} = Id$, $\phi'$ takes sufficiently small convex neighborhoods of $S^{n-1}$ to convex neighborhoods of $S^{n-1}$. So under the strict contactomorphism $\phi'\circ \phi$, we have that $\phi' \circ \phi(\partial_-H_1^\epsilon)$ is a convex subset of $U^\epsilon(S^{n-1}) \subset J^1(S^{n-1})$ for all  sufficiently small $\epsilon$ as desired.

Now we construct nested handles that are contained in $H_1^\epsilon$. We first note that the image of $U^{\epsilon}(S^{n-1}, \lambda|_{X_-})$  in $\mathbb{R}^{2n}$ under the flow of $-v$ can be identified with the negative symplectization $(U^{\epsilon}(S^{n-1}) \times (0, 1], t\alpha_{std})$. 
Therefore $H_1^{\epsilon}\cup (U^{\epsilon}(S^{n-1}) \times (0, 1], t\alpha_{std})$ can be considered a subset of $\mathbb{R}^{2n}$. 
For  $c \in (0,1]$, let $H_c^{\epsilon} \subset \mathbb{R}^{2n}$ be the image of $H_1^{\epsilon}$ under the time 
$\ln (c^{-1})$ flow of $-v$. 
Since $v$ expands $\lambda$ exponentially, $\partial_-H_c^\epsilon \subset U^{\epsilon}(S^{n-1}) \times \{c\}$. So we can attach $H_c^\epsilon$ to $U^{\epsilon}(S^{n-1}) \times (0, c]$; note that $H_c^{\epsilon} \cup U^{\epsilon}(S^{n-1}) \times (0,c]$ is precisely the image of $H_1^{\epsilon} \cup U^{\epsilon}(S^{n-1}) \times (0,1]$ under the time $\ln (c^{-1})$ flow of $-v$.
Also, if $0 < c' < c \le 1$ and $\epsilon' < \epsilon$, then 
$ H_{c'}^{\epsilon'} \cup U^{\epsilon'}(S^{n-1}) \times (0,c']\subset 
H_c^\epsilon \cup U^{\epsilon}(S^{n-1}) \times (0, c]$. This is because 
$H_1^{\epsilon'} \cup U^{\epsilon}(S^{n-1}) \times (0, 1]
 \subset H_1^\epsilon \cup U^{\epsilon}(S^{n-1}) \times (0, 1]$ if $\epsilon' < \epsilon$  and 
$-v$ is inward transverse (or tangent) to the boundary of 
$H_1^{\epsilon} \cup U^{\epsilon}(S^{n-1})\times (0, 1]$. As a result, we also see that $v$ is transverse to $\partial_{\pm} H_c^\epsilon$ and hence $\lambda|_{\partial_{\pm} H_c^\epsilon}$ is a contact form.

There is a strict contactomorphism between $\partial_-H_1^{\epsilon} \subset U^{\epsilon}(S^{n-1})$ and a neighborhood of the Legendrian attaching sphere $\Lambda$ contained in $U^{\epsilon}(\Lambda, \alpha_1) \subset (Y_-, \alpha_1)$. We can use this contactomorphism to attach $H_1^\epsilon$ to $(Y_-, \alpha_1)$. Then $H_1^\epsilon \cup (U^{\epsilon}(S^{n-1}) \times (0, 1], t\alpha_{std})$ is a subset of 
$H_1^\epsilon \cup (Y_- \times (0, 1], t\alpha_1)$. Since $\alpha_k < \alpha_1$, we can consider $\alpha_k$ as a graphical submanifold $Y_k$ of $(Y_- \times (0, 1], t\alpha_1)$. By assumption,  $\alpha_k|_{U^{\epsilon}(\Lambda, \alpha_1)}
 = c_k\alpha_1|_{U^{\epsilon}(\Lambda, \alpha_1)}$, where $c_k \le 1$ are decreasing.  Therefore, $Y_k \cap (U^{\epsilon}(S^{n-1}) \times (0, 1]) = U^{\epsilon}(S^{n-1}) \times \{c_k\}$. 
Let $\epsilon_k \in (0, \epsilon)$ be a decreasing sequence. 
Since $\partial_- H_{c_k}^{\epsilon_k} \subset 
 U^{\epsilon_k}(S^{n-1}) \times \{c_k\}
 \subset 
  U^{\epsilon}(S^{n-1}) \times \{c_k\}$,  we can attach $H_{c_k}^{\epsilon_k}$ to $Y_k$.     
  Because $\epsilon_{k}, c_k, \alpha_k$ are all decreasing, we have $H_{c_{k+1}}^{\epsilon_{k+1}}\cup (Y_{k+1} \times (0,1], t\alpha_{k+1})
  \subset 
 H_{c_k}^{\epsilon_k}\cup (Y_k \times (0,1], t\alpha_k) $ for all $k$.

 There is a Liouville form on $H_{1}^{\epsilon_1}\cup Y_- \times (0,1]$ that equals $\lambda$ on $H_{1}^{\epsilon_1}$ and $t\alpha_1$ on
 $Y_- \times (0,1]$; correspondingly, $H_{1}^{\epsilon_1}\cup Y_- \times (0,1]$ has a Liouville vector field that is
$v$ in  $H_{1}^{\epsilon_1}$ and   $\frac{\partial }{\partial t}$ in $(Y_- \times (0,1], t\alpha_1)$. In particular, this vector field has only one zero, namely at the origin in $H_{1}^{\epsilon_1}$. Furthermore, the Liouville structure on $H_{1}^{\epsilon_1}\cup Y_- \times (0,1]$ restricts to a Liouville structure on  $H_{c_k}^{\epsilon_k}\cup (Y_k \times (0,1], t\alpha_k)$ for all $k$. 

Now let $(Y_+, \alpha_k')$ be
 $\partial (H_{c_k}^{\epsilon_k} \cup (Y_-\times (0,1], t\alpha_k))$. 
  Note that $U^{\epsilon}(S^{n-1}) \times \{c_k\} \subset (J^1(S^{n-1}) \times (0, \infty), t\alpha_{std})$ is strictly contactomorphic to $U^{c_k\epsilon}(S^{n-1}) \subset 
    (J^1(S^{n-1}), \alpha_{std})$ and so $\partial_-H_{c_k}^{\epsilon_k} \subset U^{\epsilon_k}(S^{n-1}) \times \{c_k\} = U^{c_k\epsilon_k}(S^{n-1}) = U^{c_k\epsilon_k}(\Lambda, \alpha_k)$. 
Therefore $(Y_+, \alpha_k)$ equals $(Y_-, \alpha_k)$ outside $(\partial_- H_{c_k}^{\epsilon_k}, \lambda)
\subset U^{c_k\epsilon_k}(\Lambda, \alpha_k)$ and equals $(\partial_+H_{c_k}^{\epsilon_k},  \lambda)$  otherwise. We choose $\epsilon_k$ to be small enough so that Proposition \ref{prop: BEE} applies to $\Lambda \subset (Y_-, \alpha_k)$ with a neighborhood of size $c_k\epsilon_k$ and action bound $D_k$. Then there is a (shifted) grading-preserving bijection between $\mathcal{P}^{< D_k}(Y_+, \alpha_k')$ and $\mathcal{P}^{< D_k}(Y_-, \alpha_k) \cup \mathcal{W}^{<D_k}(\Lambda, Y_-, \alpha_k)$. Proposition \ref{prop: BEE} assumes that the handle is of the form 
 $(S^\delta T^*D^n =\{ \sum_{i=1}^n q_i^2 = \delta\}, \lambda_{std}= \sum_{i=1}^n q_i dp_i)$, which has Reeb vector field
 $R_{std} = \frac{1}{\delta}\sum_{i=1}^n q_i \frac{\partial }{\partial p_i}$; here $q_i, p_i$ are switched from the usual conventions. In our situation, the handle is 
 $(S^\delta T^* D^n =  \{\sum_{i=1}^n q_i^2 = \delta\}, \lambda  = \sum_{i=1}^n (2q_i dp_i + p_i dq_i))$  away from $X_-$. 
  However, $R_\lambda$ equals 
 $\frac{1}{2}R_{std}$ and so Proposition \ref{prop: BEE}, which uses only the Reeb dynamics in the handle,  still applies. We also need $\partial_- H_{c_k}^{\epsilon_k}$ to be convex in $U^{c_k\epsilon_k}(\Lambda, \alpha_k)$ so that the only Reeb trajectories of $(Y_-, \alpha_k)$ that leave and come back to  $\partial_-H_{c_k}^{\epsilon_k}$ correspond to Reeb chords of $\Lambda \subset (Y_-, \alpha_k)$. This follows because $\partial H_1^\epsilon$ is convex in $U^\epsilon(\Lambda, \alpha_1)$.

Finally, we prove that $(Y_+,\xi_+)$ is ADC by showing that $(\alpha_k', D_k)$ satisfy the conditions in Definition \ref{def: semigood}. Since $(\alpha_k, D_k)$ satisfy Definition \ref{def: semigood}, $D_k$ is 
increasing and tends to positive infinity, $\alpha_k$ are decreasing, and  all elements of $\mathcal{P}^{< D_k}(Y_-, \alpha_k)$ have positive degree. Furthermore, we assumed that all elements of $\mathcal{W}^{<D_k}(\Lambda, Y_-, \alpha_k)$ have positive degree. By choice of $\epsilon_k$, we have a bijection between $\mathcal{P}^{< D_k}(Y_+, \alpha_k')$ and  
$\mathcal{P}^{< D_k}(Y_-, \alpha_k) \cup 
\mathcal{W}^{<D_k}(\Lambda, Y_-, \alpha_k)$, which shifts the grading by $n-3$. Since $n -3 \ge 0$, all elements of $\mathcal{P}^{< D_k}(Y_+, \alpha_k')$ also have positive degree. 

Finally we show that $\alpha_k'$ is decreasing. 
Recall that $(Y_+, \alpha_k') = \partial (H_{c_k}^{\epsilon_k} \cup (Y_-\times (0,1], t\alpha_k))$. Also, 
$H_{c_{k+1}}^{\epsilon_{k+1}}\cup (Y_{k+1} \times (0,1], t\alpha_{k+1})
  \subset 
 H_{c_k}^{\epsilon_k}\cup (Y_k \times (0,1], t\alpha_k) $ and the cobordism 
 $H_{c_k}^{\epsilon_k}\cup Y_k \times (0,1] \backslash H_{c_{k+1}}^{\epsilon_{k+1}}\cup Y_{k+1} \times (0,1]$ admits a Liouville vector field that is transverse to its boundaries and has no zeroes. Therefore, the Liouville flow in this region takes $(Y_+, \alpha_{k+1}')$ to 
 $(Y_+, \alpha_k')$. Since the Liouville vector field expands the Liouville form and the contact forms are induced by the Liouville forms,   we have that $\alpha_{k+1}'  < \alpha_k'$ on $Y_+$ as desired. 
\end{proof}
\begin{remark}\label{rem: subcritical_supernice}
By using Proposition \ref{prop: MLYau_action}, one can show that the analog of Proposition \ref{prop: techsurgery} also holds for subcritical surgery. In this case, there are no  requirements on  $\mathcal{W}^{<D_k}(Y_-, \Lambda, \alpha_k)$ since all new Reeb orbits occur in the belt sphere. To apply Proposition \ref{prop: MLYau_action} for index $2$ surgery,  we also need to assume that $c_1(Y_+, \xi_+) = 0$ and either $\Lambda^1$ or all orbits of $(Y_-, \alpha_k)$ with action less than $D_k$ are contractible.
\end{remark}

Using Lemma \ref{lem: techkey} and Proposition \ref{prop: techsurgery}, we now prove our main result Theorem \ref{thm: semi-surgery} that flexible surgery preserves ADC contact forms.

\begin{proof}[Proof of Theorem \ref{thm: semi-surgery}]
Let $(Y_-, \xi_-)$ be an ADC contact structure and let $(Y_+, \xi_+)$ be the result of contact surgery on a loose Legendrian $\Lambda \subset (Y_-, \xi_-)$. 
Since $(Y_-, \xi_-)$ is ADC, there are decreasing contact forms $\alpha_k$ for $(Y_-, \xi_-)$ and increasing $D_k$ tending to infinity so that all elements of $\mathcal{P}^{< D_k}(Y_-, \alpha_k)$ have non-zero degree. By taking a subsequence of $(\alpha_k, D_k)$, we can assume that $D_k > k 4^k$. 
We will use $\alpha_k$ to construct 
a sequence of forms $\alpha_k''$ for 
$(Y_-, \xi_-)$ and numbers $D_k''$ such that 
\begin{enumerate}
 \item $\alpha_k''$ are decreasing and $D_k''$ are increasing and tend to infinity
 \item all elements of $\mathcal{P}^{< D_k''}(Y_-, \alpha_k'')$ have positive degree
 \item all elements of $\mathcal{W}^{<D_k''}(\Lambda, Y_-, \alpha_k'')$ have positive degree
 \item $\alpha_k''|_{U^\epsilon(\Lambda, \alpha_1'')} =c''_k \alpha_1''|_{U^\epsilon(\Lambda, \alpha_1'')}$ for some constants $c_k'' \le 1$  and $\epsilon > 0$. 
\end{enumerate}
Note that  conditions (1) and (2) are precisely the conditions from Definition \ref{def: semigood}.
We also need the usual genericity assumption: 
all elements of $\mathcal{P}^{< D_k''}(\Lambda, Y_-, \alpha_k'')$
 are non-degenerate and all elements of $\mathcal{P}^{< D_k''}(Y_-, \alpha_k'')$ are disjoint from $\Lambda$.  Then by Proposition \ref{prop: techsurgery}, $(Y_+, \xi_+)$ is also ADC. 
\begin{remark}\label{rem: goodlegdef}
Conditions (1) and (3) implies that loose Legendrians are asymptotically dynamically convex; see Definition \ref{def: leggood} and Proposition \ref{prop: loose_nice_leg} below. Conditions (1), (2), (3) imply that a loose Legendrian in an ADC contact manifold always form an ADC pair; see Definition  \ref{def: pair}.
\end{remark} 
 
Choose $\epsilon_k$ small enough so that the neighborhoods $U^{\epsilon_k/4}(\Lambda, \alpha_k) \subset Y_-$ of $\Lambda$  
satisfy
 $ U^{\epsilon_{k+1}/4}(\Lambda, \alpha_{k+1})
  \subset
  U^{\epsilon_k/4}(\Lambda, \alpha_k)$ for all $k$. 
  As an intermediate step, we will inductively construct contactomorphisms $\phi_k$ of $(Y_-, \xi_-)$ and Legendrians $\Lambda_k \subset (Y_-, \xi_-), \Lambda_k = \phi_k(\Lambda_{k-1}),$ satisfying
  \begin{enumerate}
  \item[(1)\textquotesingle] $\phi_k^* \alpha_k < 4\alpha_k$ for $k > 1$
  \item[(2)\textquotesingle]  elements of
   $\mathcal{W}^{<k4^k}(\Lambda_k, Y_-, \alpha_k)$ have positive degree
   \item[(3)\textquotesingle] $\Lambda_k \subset U^{\epsilon_{k+1}/4}(\Lambda, \alpha_{k+1})$ is both loose and formally isotopic to $\Lambda$ in 
    $U^{\epsilon_{k+1}/4}(\Lambda, \alpha_{k+1})$
  \end{enumerate} 
  plus the usual genericity assumptions: 
  all elements of 
  $\mathcal{P}^{<k4^{k}}(\Lambda_k, Y_-, \alpha_{k})$ are non-degenerate and all elements of 
  $\mathcal{P}^{<k4^{k}} (Y_-, \alpha_{k})$ are disjoint from $\Lambda_k$. 
  
Let $\delta_k$ be a sequence so that 
$U^{\delta_k}(\Lambda, \alpha_k) \subset 
U^{\epsilon_{k+1}/4}(\Lambda, \alpha_{k+1})$.
Applying Lemma \ref{lem: maingeo} to $\Lambda \subset (Y_-, \alpha_1), \epsilon = \delta_1,$ and $D = 4$, there is a Legendrian $\Lambda_1 \subset U^{\delta_1}(\Lambda, \alpha_1)$
such that all elements of $\mathcal{W}^{<4}(\Lambda_1, Y_-, \alpha_1)$ have positive degree and $\Lambda_1$ is both loose and formally Legendrian isotopic to $\Lambda$ in 
$U^{\delta_1}(\Lambda, \alpha_1)$  (along with the genericity conditions). Since  $U^{\delta_1}(\Lambda, \alpha_1)\subset 
U^{\epsilon_2/4}(\Lambda, \alpha_2)$, the 
last claim about looseness and formal isotopy also holds for $U^{\epsilon_2/4}(\Lambda, \alpha_2)$ by Remark \ref{rem: looseleg}.
Since $\Lambda$ is loose in $(Y_-, \xi_-)$, by Murphy's h-principle there exists a genuine Legendrian isotopy $f_t$ in $Y_-$ from $\Lambda$ to $\Lambda_1$.  This Legendrian isotopy can be extended to an ambient contact isotopy $\phi_t$ of $(Y_-, \xi_-)$; see Theorem 2.6.2 of \cite{Gbook}. So $\Lambda_1 = \phi_1(\Lambda)$ satisfies (2)\textquotesingle \ and (3)\textquotesingle. Since $\Lambda_1 = f_1(\Lambda)$ is loose in $J^1(\Lambda)$ but $\Lambda$ is not, $f_t$ is a large isotopy and $f_t(\Lambda)$ necessarily leaves $U^{\epsilon_2/4}(\Lambda, \alpha_2)$. Similarly, $\phi_t$ is also large and does not necessarily satisfy $\phi_1^* \alpha_1 < 4\alpha_1$; however we do not require this for $k =1$. 

Suppose we have constructed $\Lambda_{k-1}$ and $\phi_{k-1}$ and want to construct $\Lambda_{k}$ and $\phi_{k}$. By assumption, $\Lambda_{k-1} \subset U^{\epsilon_{k}/4}(\Lambda, \alpha_{k})$ is both loose and formally isotopic to $\Lambda$ in 
$U^{\epsilon_{k}/4}(\Lambda, \alpha_{k})$. 
Applying Lemma \ref{lem: techkey} to $\Lambda_{k-1} \subset U^{\epsilon_{k}/4}(\Lambda, \alpha_{k}), \epsilon = \delta_k$, and $D = k4^k$,  there is a contactomorphism $\phi_{k}$ of  $(Y_-, \xi_-)$ such that 
$\phi_{k}^* \alpha_{k} < 4 \alpha_{k}$, all elements of $\mathcal{W}^{<k4^{k}} (\phi_{k}(\Lambda_{k-1}), Y_-, \alpha_{k})$ have positive degree, and 
$\phi_{k}(\Lambda_{k-1}) \subset U^{\delta_k}(\Lambda, \alpha_{k})$ is both loose and formally Legendrian isotopic to $\Lambda$ in $U^{\delta_k}(\Lambda, \alpha_{k})$; furthermore the required genericity holds. Since
$U^{\delta_k}(\Lambda, \alpha_{k})
\subset U^{\epsilon_{k+1}/16}(\Lambda, \alpha_{k+1})$, the last claim also holds for 
$U^{\epsilon_{k+1}/16}(\Lambda, \alpha_{k+1})$ by Remark \ref{rem: looseleg}.
So $\phi_k$ and $\Lambda_k: = \phi_{k}(\Lambda_{k-1})$ satisfy (1)\textquotesingle, (2)\textquotesingle, (3)\textquotesingle. Note that 
$\Lambda_k = \phi_k \circ \cdots \circ \phi_1(\Lambda)$.

Now let $\alpha_{k}': = 
\frac{1}{4^k} (\phi_k \circ \cdots \circ \phi_1)^*\alpha_k 
=
\frac{1}{4^k} \phi_1^* \cdots \phi_k^* \alpha_{k}$.
Note that $\alpha_{k}'$ is a contact form for $\xi_-$ because $\phi_i$ are all contactomorphisms.  We will now show that $\alpha_k'$ satisfies (1), (2), (3).
We first show that (1) holds. 
Because $\alpha_{k} > \frac{1}{4} \phi_k^* \alpha_{k}$ for $k >1$ by (1)\textquotesingle, we have
$\phi_1^* \cdots \phi_{k-1}^*\alpha_{k} > \frac{1}{4}\phi_1^* \cdots \phi_{k-1}^*\phi_k^* \alpha_k$. Dividing by $4^{k-1}$ and using the fact that $\alpha_{k-1} > \alpha_k$, 
$$
\alpha_{k-1}' = \frac{1}{4^{k-1}} \phi_1^* \cdots \phi_{k-1}^*\alpha_{k-1} >
\frac{1}{4^{k-1}} \phi_1^* \cdots \phi_{k-1}^*\alpha_{k}
> 
 \frac{1}{4^k}\phi_1^* \cdots \phi_{k-1}^*\phi_k^* \alpha_k
=\alpha_{k}'
$$
for $k>1$ as desired. It does not matter that $\phi_1^*\alpha_1$ might be much larger than $\alpha_1$. 
Now we show that (2) holds with $D_k' = k$. Note that  
$\mathcal{W}^{<k}(\Lambda, Y_-, \alpha_{k}')$ equals
\begin{equation}\label{eqn: chords1}
\mathcal{W}^{<k}(\Lambda, Y_-, \frac{1}{4^k} \phi_1^* \cdots \phi_k^* \alpha_{k})
= \mathcal{W}^{<k4^k}(\Lambda, Y_-, \phi_1^* \cdots \phi_k^* \alpha_{k})
=\mathcal{W}^{<k4^k}(\Lambda_k, Y_-, \alpha_{k})
\end{equation}
by Proposition
\ref{prop: easychords}.
However, all elements of $\mathcal{W}^{<k4^k}(\Lambda_k, Y_-, \alpha_{k})$ have positive degree by (2)\textquotesingle \ and therefore so do all elements of 
$\mathcal{W}^{<k}(\Lambda, Y_-, \alpha_{k}')$. 

We can similarly show that (3) holds with $D_k' = k$. 
By Proposition \ref{prop: easy}
\begin{equation}\label{eqn: orbits1}
\mathcal{P}^{< k}(Y_-, \alpha_k') = 
\mathcal{P}^{< k4^k}(Y_-, \phi_1^* \cdots \phi_k^* \alpha_{k})
= \mathcal{P}^{< k4^k}(Y_-, \alpha_k).
\end{equation}
Since $D_k > k4^k$, we have 
$\mathcal{P}^{<k4^k}(Y_-, \alpha_k) \subseteq  \mathcal{P}^{< D_k}(Y_-, \alpha_k)$. Because $(Y_-,\xi_-)$ is ADC,  elements of $\mathcal{P}^{< D_k}(Y_-, \alpha_k)$ have positive degree and therefore so do elements of $\mathcal{P}^{<k4^k}(Y_-, \alpha_k)$.
Finally, note that all elements of $\mathcal{P}^{< D_k'}(\Lambda, Y_-, \alpha_k')$ are non-degenerate and all elements of $\mathcal{P}^{< D_k'}(Y_-, \alpha_k')$ are disjoint from $\Lambda$ since 
$\mathcal{P}^{< D_k'}(\Lambda_k, Y_-, \alpha_k)$ is non-degenerate and 
 $\mathcal{P}^{< D_k'}(Y_-, \alpha_k)$ is disjoint from $\Lambda_k$. 

In general $\alpha_k'$ do not satisfy $(4)$. However, we can use the following proposition, which will be proven in the Appendix, to further modify the $(\alpha_k', D_k')$ and get contact forms that satisfy all the conditions $(1),(2), (3), (4)$. 

\begin{proposition}\label{prop: nice_to_supernice}
Let $\alpha_1 > \alpha_2$ be contact forms for $(Y,\xi)$ and let 
$\Lambda \subset (Y, \xi)$ be an isotropic submanifold with trivial symplectic conormal bundle. Then for any sufficiently small $\delta_1, \delta_2$, then  there exists a contactomorphism $h$ of $(Y, \xi)$ such that 
\begin{itemize}
\item $h$ is supported in $U^\epsilon(\Lambda, \alpha_1)$,  $h|_{\Lambda} = Id$, and $h^*\alpha_2 < 4 \alpha_1$
\item $h^*\alpha_2|_{U^{\delta_1}(\Lambda, \alpha_1)} = c \alpha_1|_{U^{\delta_1}(\Lambda, \alpha_1)}$ for some constant $c$ (depending on $\delta_1, \delta_2$)
\item $h(U^{\delta_1}(\Lambda, \alpha_1)) \subset U^{\delta_2}(\Lambda, \alpha_2).$
\end{itemize}
\end{proposition}

Now we proceed as in the first half of this proof. By taking a subsequence of $\alpha_k'$, we can assume that $D_k' > k4^k$.
Let $\delta_1^k, \delta_2^k$ be a sequence of sufficiently small numbers
so that $\delta_2^{k-1} < \delta_1^{k}$ and Proposition \ref{prop: nice_to_supernice} applies to the contact forms $\alpha_{k-1}'> \alpha_{k}'$ constructed above. 
Let $h_1 = Id$ and for $k > 1$, let $h_k$ be the contactomorphism of $(Y_-, \xi_-)$ provided by Proposition \ref{prop: nice_to_supernice} such that
$h_k^*\alpha'_k < 4\alpha_{k-1}$,  $h_k^*{\alpha'_{k}}|_{U^{\delta_1^{k-1}}(\Lambda, {\alpha'_k})}
= c_k {\alpha'_{k-1}}|_{U^{\delta_1^{k-1}}(\Lambda, 
{\alpha'_{k-1}})}$ 
for some constant $c_k$, and  $h_k(U^{\delta_1^{k-1}}(\Lambda, \alpha'_{k-1}) ) \subset U^{\delta_2^{k-1}}(\Lambda, \alpha'_{k})
\subset 
U^{\delta_1^{k}}(\Lambda, \alpha'_{k}).$
Let 
$\alpha_k'' := \frac{1}{4^{k}}(h_{k} \circ \cdots \circ h_1)^* \alpha_{k}'$.
By repeating the approach in the first half of this proof, we can show that $(\alpha_k'', D_k'' = k)$ satisfy $(1), (2), (3)$. Finally, since 
$h_k(U^{\delta_1^{k-1}}(\Lambda, \alpha_{k-1}') ) \subset 
U^{\delta_1^{k}}(\Lambda, \alpha_{k}')$, we can use induction to show that 
$\alpha_k''|_{U^{\delta_1}(\Lambda, \alpha_1'')} = \left(\frac{\prod_{i=1}^{k} c_i }{4^{k}}\right)
\alpha_1''|_{U^{\delta_1}(\Lambda, \alpha_1'')}$, where $c_1 = 1$. Therefore condition $(4)$ holds with 
$c_k'' := \left(\frac{\prod_{i=1}^{k} c_i }{4^{k}}\right)$
and $\epsilon := \delta_1$. 
We also note that the required genericity still holds for $(\alpha_k'', D_k'')$ and $\Lambda$. 
\end{proof}

Finally, we give a proof of Theorems \ref{thm: MLYau}, \ref{thm: MLYauindex2} that subcritical surgery preserves asymptotically dynamically convex contact forms. Besides M.-L.Yau's calculation of the degrees of the Reeb orbits in subcritical handles, the proof is along the same lines as the proof in the flexible case. 

\begin{proof}[Proof of Theorem \ref{thm: MLYau}]
Suppose $(Y_-, \xi_-)$ is ADC with $(\alpha_k, D_i)$ as in Definition \ref{def: semigood}. Using Proposition \ref{prop: nice_to_supernice} like in the proof of Theorem \ref{thm: semi-surgery}, we can furthermore assume that $\alpha_k|_{U^\epsilon(\Lambda, \alpha_1)} = c_k \alpha_1|_{U^\epsilon(\Lambda, \alpha_1)}$ holds. 
Then by the subcritical version of Proposition \ref{prop: techsurgery} discussed in Remark 
\ref{rem: subcritical_supernice}, 
$(Y_+, \xi_+)$ is also ADC.
\end{proof}

\begin{proof}[Proof of Theorem \ref{thm: MLYauindex2}]
We proceed exactly as in the proof of Theorem \ref{thm: MLYau} except now we need to assume that $c_1(Y_+, \xi_+) = 0$ and either $\Lambda^1$ or all orbits of $(Y_-, \alpha_k)$ with action less than $D_k$ are contractible for all $k$ in order to apply Proposition  \ref{prop: techsurgery}.
\end{proof}

\section{Results for Legendrians}\label{sec: legendrian_results}

We now discuss the Legendrian analogs of our results. Because the main ideas are quite similar, we give less details here than in the contact case. We first define asymptotically dynamically convex \textit{Legendrians} and use these to define asymptotically dynamically convex \textit{Weinstein domains} and \textit{Lagrangians}.

In this section, we assume that all Legendrians $\Lambda^{n-1} \subset Y^{2n-1}$ and Lagrangians $L^n \subset W^{2n}, n \ge 3,$ are oriented and spin and that $c_1(Y, \Lambda), c_1(W, L)$ vanish; also $c_1(Y), c_1(W)$ vanish as before. 
After choosing a trivialization of the canonical bundle of $(Y, \xi)$, we can $\mathbb{Z}$-grade the Reeb chords of Legendrians. 
For simplicity, we will also assume that $\pi_1(Y, \Lambda) = 0$, in which case all chords have a well-defined grading that is independent of the choice of trivialization. This assumption also allows us to avoid discussing $\pi_1$-injective Lagrangian fillings; see Section \ref{ssec: independencelin} for a discussion in the contact case. Below we will explain when these assumptions are preserved under surgery. Finally, we assume that all  Legendrians and contact forms satisfy the necessary genericity conditions, which we do not discuss here.  

\subsection{Asymptotically dynamically convex Legendrians}\label{ssec: ACD_Legendrians}
\subsubsection{Definitions}
First, we define the Legendrian analog of an asymptotically dynamically convex contact structure. Let $\Lambda \subset (Y, \xi)$ be a Legendrian and let $\alpha$ be a contact form for $\xi$. Recall that 
$\mathcal{P}^{< D}(\Lambda, Y, \alpha)$ is the set of Reeb chords of $\Lambda$ with action less than $D$; since $\pi_1(Y, \Lambda) =0$, we do not need to make the additional requirement that the chords are zero in $\pi_1(Y, \Lambda)$.

\begin{definition}\label{def: leggood}
A Legendrian $\Lambda \subset (Y, \xi)$ is \textit{asymptotically dynamically convex} if there exists a sequence of non-increasing contact forms $\alpha_1 \ge \alpha_2 \ge \alpha_3 \cdots$ for $\xi$ and a sequence of increasing 
$D_1 < D_2 < D_3 \cdots $ going to infinity such that all elements of $\mathcal{P}^{< D_k}(\Lambda, Y, \alpha_k)$ have positive degree.
\end{definition}
\begin{remark}
Asymptotically dynamically convex Legendrians generalize the dynamically convex Legendrians defined on p.80, \cite{CieliebakOancea}; these Legendrians have positive degree Reeb chords for \textit{some fixed} contact form on $(Y, \xi)$. 
\end{remark}

Note that by the Legendrian isotopy extension theorem \cite{Gbook}, any Legendrian that is Legendrian isotopic to an ADC Legendrian is also ADC. Also note that Definition \ref{def: leggood} does not require the Legendrian to be parametrized.

The contact manifold $(Y, \xi)$ in Definition \ref{def: leggood} does not have to be ADC. However, we can also consider ADC Legendrians in ADC contact manifolds.  It is not clear that the same sequence of contact forms can be used to show that both the contact manifold and the Legendrian are asymptotically dynamically convex. We will use the following definition when this is possible. 

\begin{definition}\label{def: pair}
$(Y,\Lambda, \xi)$ is an \textit{asymptotically dynamically convex pair} if there exist a sequence of contact forms $\alpha_1 \ge \alpha_2 \ge \alpha_3 \cdots$ for $\xi$ and a sequence of increasing  $D_1 < D_2 < D_3 \cdots $ going to infinity such that all elements of 
$\mathcal{P}^{< D_k}(Y, \alpha_k)$ and $\mathcal{P}^{< D_k}(\Lambda, Y, \alpha_k)$ have positive degree.
\end{definition}

\begin{remark}
We do not know whether $(Y, \Lambda, \xi)$ is always an ADC pair if $(Y, \xi)$ is an ADC contact structure and $\Lambda \subset (Y, \xi)$ is an ADC Legendrian.
\end{remark}

The following proposition reformulates our main result Lemma \ref{lem: maingeo} about Reeb chords of loose Legendrians; see Remark \ref{rem: goodlegdef}.

\begin{proposition}\label{prop: loose_nice_leg}
If $\Lambda \subset (Y, \xi)$ is a loose Legendrian, then $\Lambda$ is asymptotically dynamically convex in $(Y, \xi)$. Furthermore, if $(Y, \xi)$ is an asymptotically dynamically convex contact structure, then 
$(Y, \Lambda, \xi)$ is an asymptotically dynamically convex pair.  
\end{proposition}

\subsubsection{Wrapped Floer homology}

Let $W$ be a Liouville filling of $(Y, \xi)$. 
For an exact Lagrangian $L \subset W$ with Legendrian boundary $\Lambda \subset Y$, wrapped Floer homology $WH(L, L; W)$ and its positive version $WH^+(L, L; W)$ are the Legendrian analogs of  $SH(W)$ and $SH^+(W)$ respectively. The generators of $WH(L, L; W)$ are Hamiltonian chords of $L$, which for certain admissible Hamiltonians $H$ correspond to Morse critical points of $H|_L$ and Reeb chords of $\Lambda \subset Y$; see \cite{Abouzaid_Seidel} and \cite{Ritter} for details. Unlike for $SH$, where each Reeb orbit gives rise to two (non-constant) Hamiltonian orbits, for $WH$ there is one-to-one correspondence between Reeb chords and (non-constant) Hamiltonian chords. The generators of $WH^+(L, L; W)$ are just Reeb chords of $\Lambda$. 

We now explain our grading for $WH$. Hamiltonian chords corresponding to Reeb chords are graded with the Reeb chord grading described in Section \ref{ssec: prop_reeb_chords}; Hamiltonian chords corresponding to Morse critical points
of $H|_L$ are graded by $n- Ind(p) -2$, where $Ind(p)$ is the Morse index of $H|_L$. The resulting grading coincides with the grading of $WH$ in \cite{BEE12}. 

With this grading convention, the tautological long exact sequence for $WH$ takes the form
$$
\cdots \rightarrow H^{n-k-2}(L; \mathbb{Z}) \rightarrow WH_k(L, L; W) \rightarrow WH^+_k(L, L; W) \rightarrow 
H^{n-k-1}(L; \mathbb{Z}) \rightarrow \cdots
$$
In particular, if $WH(L, L; W) = 0$ (which will usually be the case in this paper), then 
$WH_k^+(L, L; W) \cong H^{n-k-1}(L; \mathbb{Z}) $. 
Under this grading convention, 
$WH_k(T^*_x M, T^*_x M; T^*M) \cong H_{k-n+2}(\Omega M; \mathbb{Z})$, where $\Omega M$ is the based loop space of $M$.
Another important property is that $WH(L, L; W)$ is a module over $SH(W)$, which can be endowed with a certain unital ring structure that we do not discuss here;  see Theorem 6.17 of \cite{Ritter}. In particular, if $SH(W) = 0$, then $WH(L, L; W) = 0$ for any exact Lagrangian $L \subset W$. 

If $\Lambda$ is asymptotically dynamically convex, then $WH_k^+(L, L; W) = 0$ for $k \le 0$ for all exact Lagrangian fillings $L \subset W$ of $\Lambda$; see 
Section \ref{sssec: homological_obstruction} for a similar statement in the contact case. This can be used to show that not all Legendrians are ADC. 

\begin{examples}\label{ex: not_ADC_legendrian} 
Let $M^n$ be a manifold with non-empty, connected boundary and $H^{n-1}(M; \mathbb{Z}) \ne 0$. Then $T^*M$ is a subcritical Weinstein domain and so $SH(T^*M) = 0$, which implies that $WH(M, M; T^*M)= 0$. 
Hence $WH_0^+(M, M; T^*M)\cong H^{n-1}(M; \mathbb{Z}) \ne 0$. 
This does not directly show that $\partial M \subset \partial T^*M$ is not ADC since the degree $0$ chord $c$ is non-zero in $\pi_1(\partial T^*M, \partial M) \cong \pi_1(M, \partial M)$; see Proposition \ref{prop: simultaneoussurgerygood} below.
By slightly modifying this situation, we can produce an example that does not have this issue. First note that $c$ \textit{does} vanish in $\pi_1(T^*M, M) = 0$, i.e. the filling is not $\pi_1$-injective; in fact, $c$ has degree 0 with respect to the trivialization induced by the disk in $T^*M$ that contracts this chord. Now attach a Weinstein 2-handle to $T^*M$ along an isotropic $\Lambda^1 \subset  \partial T^*M$ that maps to $c$ under $\pi_1(\partial T^*M) \rightarrow \pi_1(\partial T^*M, \partial M)$; frame $\Lambda^1$ using the framing induced from $c$. The result is a subcritical Weinstein domain $W$ that has $M$ as an exact Lagrangian. Furthermore, $c$ vanishes in $\pi_1(\partial W, \partial M)$ and represents a non-zero element of $WH_0^+(M, M; W) \cong H^{n-1}(M; \mathbb{Z})$; therefore 
$\partial M \subset \partial W$ is not ADC. See the discussion before 
Theorem \ref{thm: MLYauindex2} for a similar situation. 
\end{examples}

There are many other examples of non-ADC Legendrians. The Legendrian contact homology algebras of the exotic Legendrians in \cite{Che, EES} have non-trivial augmentations with different linearized contact homologies. This requires degree zero chords and hence all these examples are non-ADC.

Note that  $WH(L, L; W)$ and $WH^+(L, L; W)$ 
are invariants of $L$ up to Hamiltonian isotopy in $W$. However, as in the contact setting, $WH^+(L,L; W)$ is not a Legendrian invariant. 
\begin{examples}\label{ex: sabloff} (Theorem 1.5, \cite{Sabloff})
For any integer $N$, there is a Legendrian sphere 
$\Lambda^{n-1}\subset (S^{2n-1}, \xi_{std}), n \ge 3,$ with at least $N$ exact Lagrangian fillings $L \subset (B^{2n}, \omega_{std})$ with different cohomology $H^*(L; \mathbb{Z})$. Since $SH(B^{2n}) = 0$, we have $WH(L, L; B^{2n}) = 0$. Then the tautological long exact sequence for $WH$ shows that 
$WH^+_k(L, L; B^{2n}) \cong H^{n-k-1}(L; \mathbb{Z})$. So different fillings $L$ have different $WH^+(L, L; B^{2n})$ and therefore $WH^+(L, L; B^{2n})$ is not an invariant of $\Lambda$. 
\end{examples} 
Like in the contact case, $WH^+$ fails to be a Legendrian invariant due to the presence of certain J-holomorphic curves. These are J-holomorphic disks in $W$ that have boundary on $L$ and are asymptotic to degree $0$ Reeb chords of $\Lambda$. The following proposition
states that $WH^+(L, L; W)$ is an invariant for asymptotically dynamically convex Legendrians, which do not have such chords (for fixed action bounds going to infinity). It is the Legendrian analog of Proposition \ref{prop: nice_sh_independent} and is proven in the same way. 

\begin{proposition}\label{prop: Legindependence}
If $\Lambda \subset (Y,\xi)$ is an asymptotically dynamically convex Legendrian and $W$ is a Liouville filling of $Y$, then all exact Lagrangian fillings of $\Lambda$ in $W$ have isomorphic $WH^{+}$. 
\end{proposition}
\begin{remark}\label{rem: niceLegendrianindependent}\
\begin{enumerate}[leftmargin=*]
\item Proposition \ref{prop: Legindependence} holds vacuously for loose Legendrians, which have no exact Lagrangian fillings \cite{Murphy11}.
\item Note that $(Y, \xi)$ does not have to be ADC.
\item 
Example \ref{ex: sabloff} shows that the conclusion of Proposition \ref{prop: Legindependence} does not hold in general. Therefore, the ADC condition is crucial.  In particular, the  Legendrians from Example \ref{ex: sabloff} are not ADC. 
\end{enumerate}
\end{remark}

For ADC \textit{pairs}, a stronger version of Proposition \ref{prop: Legindependence} holds for Lagrangians in \textit{different} Liouville filling. 

\begin{proposition}
If $(Y, \Lambda, \xi)$ is an asymptotically dynamically convex pair, then all exact 
Lagrangian fillings of $\Lambda$ in all
Liouville filling of $(Y, \xi)$ have isomorphic $WH^{+}$;
that is, if $W_1, W_2$ are two Liouville filling of $(Y, \xi)$ and $L_k \subset W_k, k=1,2,$ are two exact Lagrangian fillings of $\Lambda$, then $WH^+(L_1, L_1; W_1) \cong WH^+(L_2, L_2; W_2)$. 
\end{proposition}

\subsubsection{Surgery operations}
Theorems \ref{thm: MLYau}, \ref{thm: MLYauindex2}, and \ref{thm: semi-surgery} show that asymptotically dynamically convex contact structures are preserved under certain contact surgeries. We now present several Legendrian analogs of these results. 

We begin with \textit{ambient Legendrian surgery}, which was introduced by Rizell in \cite{Riz}. Given a Legendrian $\Lambda_-^{n-1} \subset (Y^{2n-1}, \xi)$ and a framed isotropic disk $D^k \subset Y, k < n,$ with $\partial D^k \subset \Lambda_-$ and Int $D^k \subset Y\backslash \Lambda_-$, Rizell constructs another Legendrian $\Lambda_+ \subset (Y, \xi)$ and an elementary Lagrangian cobordism $L \subset Y\times (0, \infty)$ between $\Lambda_-$ and $\Lambda_+$ with one critical point of index $k$.  Note that the construction does not change the contact manifold.
Following Rizell, we call the surgery subcritical if $k < n-1$; this only makes sense for $n \ge 3$. On the other hand, \textit{contact} surgery is called subcritical if $k < n$.

Rizell also computed the degrees of the new Reeb chords produced by ambient Legendrian surgery, essentially proving the 
following result.

\begin{proposition}\label{prop: ambientsurgery}\cite{Riz}
If $\Lambda_- \subset (Y, \xi)$ is an asymptotically dynamically convex Legendrian and $\Lambda_+ \subset (Y,\xi)$ is the result of subcritical ambient Legendrian surgery, then $\Lambda_+$ is also an asymptotically dynamically convex Legendrian. 
\end{proposition}
\begin{remark}\
\begin{enumerate} [leftmargin=*]
\item If $k \ne 1$ and $c_1(Y, \Lambda_-) = 0$,  the cohomology long exact sequence of the triple $(Y \times (0, \infty), L, \Lambda_-)$ shows that $c_1(Y, \Lambda_+) = 0$;  for $k = 1$, there is a unique choice of framing (up to homotopy) such that $c_1(Y, \Lambda_+) = 0$. See Section 4.4 of \cite{Riz}. If $k \ne n-1$, then $\Lambda_+$ is connected and by the homotopy long exact sequence $\pi_1(Y, \Lambda_-) = 0$ implies $\pi_1(Y, \Lambda_+) = 0$. 
In the critical case $k = n-1$, these do not necessarily hold; more importantly, critical surgery creates Reeb chords of degree 0 and hence we must exclude this case completely.
\item 
Consider the special case of Proposition \ref{prop: ambientsurgery}
when $\Lambda_-$ is loose (and hence ADC by Proposition \ref{prop: loose_nice_leg}). 
For $k < n-1$, we can make $D^k$ disjoint from the loose chart of $\Lambda_-$. Hence this remains a loose chart for $\Lambda_+$ and so $\Lambda_+ \subset (Y, \xi)$ is also loose (and therefore ADC). 
\end{enumerate}
\end{remark}

\begin{proof}
The Reeb chords of $\Lambda_+$ correspond to the old chords of $\Lambda_-$ and a new chord $c$ near the belt sphere of $\Lambda_+$, which has $|c| = n-k-1$; see Remark 4.6 of \cite{Riz}. More precisely, we need to consider a sequence of decreasing contact forms for $(Y, \xi)$ obtained by shrinking the surgery neighborhood of $D^k$. In the subcritical case $k < n-1$,  the chord $c$ has positive degree and hence if $\Lambda_-$ is an ADC Legendrian, so is $\Lambda_+$.
\end{proof}

As we noted before, the Legendrians in Example \ref{ex: sabloff} are not ADC. However they are created via ambient Legendrian surgery on the standard Legendrian unknot, which is ADC. But the construction of these examples uses critical surgery, which is why the resulting Legendrians are not ADC. This shows that Proposition \ref{prop: ambientsurgery} fails in the critical case.  Indeed for a Lagrangian $L \subset (B^{2n}, \omega_{std})$, we have $WH^+_0(L, L; B^{2n}) \cong H^{n-1}(L; \mathbb{Z})$, which may be non-zero in the presence of critical surgery. We also note that the examples from \cite{EES} rely on index 1 ambient Legendrian surgery between \textit{disconnected} Legendrians, which is why they are not ADC. 

Now we consider surgery that \textit{does} change the contact manifold. First, we assume that the surgery changes the contact manifold but does not change the Legendrian.

\begin{proposition}\label{prop: nonsimultaneousgood}
If $\Lambda_- \subset (Y_-, \xi_-)$ is an asymptotically dynamically convex Legendrian 
and $(Y_+, \xi_+)$ is the result of contact surgery on $\Lambda^{k-1} \subset Y_- \backslash \Lambda_-$, which is subcritical or loose in the complement of $\Lambda_-$, then $\Lambda_+ = \Lambda_- \subset (Y_+, \xi_+)$ is also an asymptotically dynamically convex Legendrian. If $(Y_-, \Lambda_-, \xi_-)$ is an asymptotically dynamically convex pair, then $(Y_+, \Lambda_+, \xi_+)$ is also an asymptotically dynamically convex pair. 
\end{proposition}
\begin{remark}\
\begin{enumerate}[leftmargin=*]
\item If $k \ne 2$ and $c_1(Y_-, \Lambda_-) = 0$,  the cohomology long exact sequence of the triple $(W, Y_-, \Lambda_-)$, where $W$ is the Weinstein cobordism between $Y_-$ and $Y_+$, implies that $c_1(Y_+, \Lambda_+) = 0$; also $c_1(Y_+) = 0$ by Proposition \ref{prop: c1equivalence}. For $k = 2$, these results hold with the correct framing of $\Lambda^{k-1}$; see Section \ref{ssec: independencelin}. 
Also, if $k \ne 1$ and $\pi_1(Y_-, \Lambda_-) = 0$, the homotopy long exact sequence implies that $\pi_1(Y_+, \Lambda_+) = 0$. 
\item If $\Lambda_-$ is loose (hence ADC by Proposition \ref{prop: loose_nice_leg}) and $\Lambda^{k-1}$ is subcritical or loose in the complement of $\Lambda_-$, we can make $\Lambda^{k-1}$ disjoint from the loose chart of $\Lambda_-$ and so $\Lambda_+ = \Lambda_-$ has a loose chart in $(Y_+, \xi_+)$. Therefore $\Lambda_+ \subset (Y_+, \xi_+)$ is also loose and hence ADC. 
\end{enumerate}
\end{remark}
\begin{proof}
When $\Lambda^{k-1}$ is subcritical, the Reeb chords of $\Lambda_+ \subset (Y_+, \xi_+)$ correspond to Reeb chords of $\Lambda_-$ since generically there are no chords with endpoints on both $\Lambda_-$ and $\Lambda^{k-1}$. When $\Lambda^{k-1}$ is a Legendrian, there is a grading-preserving correspondence between chords of $\Lambda_+ \subset (Y_+, \xi_+) $ and words of chords of $\Lambda_- \cup \Lambda^{k-1}$ that begin and end on $\Lambda_-$; see Theorem 5.10 of \cite{BEE12}. 
More precisely, these are either chords of $\Lambda_-$ or words of the form $a c_1 \cdots c_k b$, where $a$ is a chord from $\Lambda_-$ to $\Lambda$, $b$ is a chord from $\Lambda$ to $\Lambda_-$, and $c_i$ are chords of $\Lambda$.
Because $\Lambda^{k-1}$ is loose in the complement of $\Lambda_-$, we can increase the degree of any chord with at least one endpoint on $\Lambda^{k-1}$ as in Lemma \ref{lem: maingeo}; hence the words $ac_1 \cdots c_k b$ have positive degree. In addition, the chords of $\Lambda_- \subset (Y_-, \xi_-)$ have positive degree since $\Lambda_-$ is ADC. Therefore $\Lambda_+$ is also ADC. 

If $(Y_-, \Lambda_-, \xi_-)$ is an ADC pair, the new Reeb orbits of $(Y_+, \xi_+)$ have positive degree by Proposition \ref{prop: MLYau_action} in the subcritical case and by Lemma \ref{lem: maingeo} and Proposition \ref{prop: BEE} in the flexible case (we automatically have $n-3 \ge 0$ since $\Lambda^{n-1}$ is loose). So $(Y_+, \xi_+)$ is also ADC. The sequence of contact forms that show $(Y_+, \xi_+)$ is an ADC contact structure and $\Lambda_+ \subset (Y_+, \xi_+)$ is an ADC Legendrian are obtained by shrinking the handle around $\Lambda^{k-1}$ and applying Legendrian isotopies in the flexible case. 
Since the same sequence of forms is used, $\Lambda_+ \subset (Y_+, \xi_+)$ is an ADC pair. 
\end{proof}

\begin{remark}\label{rem: nonsimultaneousgood} \
\begin{enumerate}[leftmargin=*]
\item
Note that in general the grading of chords with endpoints on different Legendrians is not well-defined. However we only consider words of the form $a c_1 \cdots c_k b$. These words are \textit{cyclically composable} in the sense of \cite{BEE12} and have well-defined grading; see \cite{knotcontact} or \cite{EO_15}.
\item
Proposition \ref{prop: nonsimultaneousgood} does not necessarily hold if $\Lambda^{n-1}$ is an ADC Legendrian that is not loose in the complement of $\Lambda_-$. If $\Lambda^{n-1}$ is ADC, then all chords of $\Lambda^{n-1}$ have positive degree but this might not be true for chords between $\Lambda^{n-1}$ and $\Lambda_-$, which is what we need for the proof of Proposition \ref{prop: nonsimultaneousgood}. 
However, if $\Lambda_- \cup \Lambda^{n-1}$ is an ADC \textit{link}, i.e. all chords of $\Lambda_-$ and $\Lambda^{n-1}$ \textit{and} chord words $ab$ between them have positive degree, then Proposition
\ref{prop: nonsimultaneousgood} holds once more. Indeed, the proof of Proposition \ref{prop: nonsimultaneousgood} shows that if $\Lambda_-$ is ADC and $\Lambda$ is loose in the complement of $\Lambda_-$, then $\Lambda_- \cup \Lambda^{n-1}$ is an ADC link. 
\end{enumerate}
\end{remark}

Now we consider surgery that changes the Legendrian and the contact manifold simultaneously. Consider a Legendrian $\Lambda_-^{n-1} \subset (Y_-^{2n-1}, \xi_-)$ and a framed isotropic sphere $\Lambda^{k-1} \subset \Lambda_-^{n-1}$ with $k < n$. Let $(Y_+, \xi_+)$ be result of index $k$ contact surgery on $\Lambda^{k-1}$ (which is subcritical since $k < n$) 
and let $W$ be the elementary Weinstein cobordism between $(Y_-, \xi_-)$ and $(Y_+, \xi_+)$. Then $(Y_+, \xi_+)$ contains a Legendrian $\Lambda_+$ which is smoothly the result of index $k$ surgery on $\Lambda_-$ \cite{EGL}, \cite{CieChord} and $W$ contains an elementary Lagrangian cobordism $L$ between $\Lambda_-$ and $\Lambda_+$. We will call this operation \textit{Legendrian surgery} and the $k < n-1$ case subcritical.

\begin{proposition}\label{prop: simultaneoussurgerygood}
If $\Lambda_- \subset (Y_-, \xi_-)$ is an asymptotically dynamically convex Legendrian 
and $\Lambda_+\subset (Y_+, \xi_+)$ is the result of subcritical Legendrian surgery on $\Lambda^{k-1} \subset \Lambda_-$,  then $\Lambda_+ \subset (Y_+, \xi_+)$ is also an asymptotically dynamically convex Legendrian. If $(Y_-, \Lambda_-, \xi_-)$ is an asymptotically dynamically convex pair, then $(Y_+, \Lambda_+, \xi_+)$ is also an asymptotically dynamically convex pair. 
\end{proposition}
\begin{remark}
If $k \ne 2$, then $c_1(Y_-) = 0$ implies $c_1(Y_+) = 0$ by Proposition \ref{prop: c1equivalence}; furthermore, $\pi_1(Y_-) \cong \pi_1(Y_+)$ in this case. If $k \ne 1,2$, then $c_1(Y_-, \Lambda_-) = 0$ implies $c_1(Y_+, \Lambda_+) = 0$ by the five-lemma applied to the cohomology long exact sequences of $(W, L)$ and $(Y_-, \Lambda_-)$. If $k = 2$, these results hold for the correct framing of $\Lambda^{k-1}$. 
Also, if $k \ne n-1$, then $\Lambda_+$ is connected and $\pi_1(Y_-, \Lambda_-) = 0$ implies $\pi_1(Y_+, \Lambda_+) = 0$.
In the critical case $k = n-1$, these topological assumptions do not necessarily hold; more importantly, critical surgery creates Reeb chords of degree 0 and hence we must exclude this case completely.
\end{remark}

\begin{proof}
If $k < n -1$, then $\Lambda_+$ is connected and has a portion contained in the contact handle.  As we noted in the proof of Proposition \ref{prop: MLYau_action}, the belt sphere $S^{2n-k-1}$ contains a contact sphere $(S^{2n-2k-1}, \xi_{std})$. Furthermore, $\Lambda_+ \cap S^{2n-k-1}$ is the Legendrian unknot 
$S^{n-k-1} \subset (S^{2n-2k-1}, \xi_{std})$. The new Reeb chords of $\Lambda_+$ correspond to Reeb chords of this unknot, which has a single Reeb chord $c$ with $|c| = n-k-1$; see \cite{BEE12} or \cite{EES2}.   In the subcritical case $k < n - 1$, this chord has positive degree; in fact, $S^{n-k-1} \subset (S^{2n-2k-1}, \xi_{std})$ is an ADC pair for $k < n-1$.  Therefore if $\Lambda_-$ is an ADC Legendrian, so is  $\Lambda_+$. 
Finally, since the corresponding contact surgery is subcritical, Legendrian surgery on an ADC pair results in another ADC pair.
\end{proof}

Note that if $M$ has non-empty boundary, then $SH(T^*M) = 0$ and therefore by the tautological long exact sequence
$WH_{n-k-1}^+(M, M; T^*M) \cong H^{k}(M; \mathbb{Z})$. Indeed as we saw in the proof of Proposition \ref{prop: simultaneoussurgerygood}, index $k$ Legendrian surgery creates a Reeb chord of degree $n-k-1$. 

The critical case $n = 2, k = 1$ was analyzed by Ekholm and Ng \cite{EkholmNg};  in their situation $\Lambda_+^1 \subset Y_+^3$ is connected but $\Lambda_-^1 \subset Y^3_-$ consists of two disconnected circles.  
They showed that $\Lambda_+$ has infinitely many new Reeb chords that correspond to (copies of) the Reeb orbits of $Y_+$ in the belt sphere $S^2$ of the Weinstein 1-handle. This is because $\Lambda^+ \cap S^2 \subset S^1$ equals $S^0 \subset S^1$ and so all the Reeb orbits of $S^1$ necessarily hit $S^0$; in the subcritical situation,
$S^{n-k-1}= \Lambda^+ \cap S^{2n-k-1} \subset S^{2n-2k-1}$ avoids the Reeb orbits of $S^{2n-2k-1}$.

When $k = n$ and $\Lambda^{n-1}_- = S^{n-1}$, the surgery caps off $\Lambda_-$, which does not survive to the new contact manifold $Y_+$. However, this operation gives rise to the belt sphere $\Lambda_+ = S^{n-1}$ of the Weinstein n-handle, which is a Legendrian in $Y_+$. 
 
\begin{proposition}\label{prop: surgerybeltsphere}
If $\Lambda_-^{n-1} \subset (Y_-^{2n-1}, \xi_-)$ is an asymptotically dynamically convex Legendrian sphere and $(Y_+, \xi_+)$ is the result of contact surgery on $Y_-$ along $\Lambda_+$, the belt sphere $\Lambda_+ \subset (Y_+, \xi_+)$ is also asymptotically dynamically convex. If $(Y_-^{2n-1}, \Lambda_-, \xi_-)$ is an asymptotically dynamically convex pair and $n \ge 3$, then $(Y_+,\Lambda_+, \xi_+)$ is also an asymptotically dynamically convex pair.
\end{proposition}
\begin{remark}\
\begin{enumerate}[leftmargin=*]
\item If $n \ne 2$ and $c_1(Y_-) = 0$, then $c_1(Y_+) = 0$ by Proposition \ref{prop: c1equivalence} and 
$c_1(Y_+, \Lambda_+)  = 0$ by the cohomology long exact sequence for $(Y_+, \Lambda_+)$.
Also, if $\pi_1(Y_-, \Lambda_-) = 0$, then 
$\pi_1(Y_+, \Lambda_+) = 0$. 
\item Even if $\Lambda_-$ is loose, then $\Lambda_+$ is usually not loose. For example, if $Y_-$ has a Liouville filling $W_-$, then $\Lambda_+$ bounds the Lagrangian co-core $D^n$ in $W_+$ and so is not loose; here $W_+$ is the result of attaching a Weinstein handle along $\Lambda_-$ to $W_-$. Note that $D^n$ is an example of a semi-flexible Lagrangian; see Corollary \ref{cor: flexsemigood} below. 
However, if $\Lambda_-$ is loose and $Y_-$ is overtwisted, then $Y_+$ is also overtwisted by Proposition 2.8 of  \cite{CMP}. Furthermore, the belt sphere $\Lambda_+ \subset Y_+$ is loose since it avoids the overtwisted disk in $Y_+$; note that in this case, $Y_-$ has no Liouville fillings. 
\end{enumerate}
\end{remark}
\begin{proof}
The Reeb chords of $\Lambda_+$ correspond to words of chords of $\Lambda_-$ with a grading shift of $n-2$. More precisely, if $c_w$ is a Reeb chord of $\Lambda_+$ corresponding to a word $w$ of Reeb chords of $\Lambda_-$, then $|c_w| = |w| + n-2$; see Theorem 5.8 of \cite{BEE12}.
Since $\Lambda_-$ is ADC and $n \ge 2$ always,  $\Lambda_+$ is ADC as well.  
By Proposition \ref{prop: BEE},  the new Reeb orbits of $(Y_+, \xi_+)$ correspond to cyclic equivalence classes of words of chords of $\Lambda_-$ with grading shifted by $n-3$, i.e. 
$|\gamma_w| = |w| + n-3$. Since $\Lambda_-$ is ADC and $n \ge 3$, these orbits have positive degree and so $(Y_+, \xi_+)$ is ADC. Since the same sequence of contact forms is used for both chords and orbits, $(Y_+, \Lambda_+, \xi)$ is an ADC pair. 
\end{proof}

\subsection{Legendrians with flexible Lagrangian fillings}\label{subsec: legwithflexfill}

 Eliashberg, Ganatra, and the author  \cite{EGL} defined \textit{flexible} Lagrangians and showed that they satisfy an h-principle. We say that an exact Lagrangian $L\subset  (W, \omega)$  with Legendrian boundary is {\em flexible}  if  $(W \backslash T^*L, \omega|_{W \backslash T^*L})$ is a flexible Weinstein cobordism.
This implies that $W$ itself is flexible. For example, $L \subset T^*L$ has trivial Weinstein complement and hence is tautologically flexible.  
 An exact Lagrangian $L \subset  (W, \omega)$  is {\em semi-flexible}  if  the cobordism $(W \backslash T^*L, \omega|_{W \backslash T^*L})$  can be presented as a composition of a Weinstein domain $W'$ and a  flexible cobordism $W''$ with $\partial_-W'' = ST^*L \sqcup\partial W'$, where 
 $\partial T^*L = ST^*L \cup (T^*(\partial L)\times \mathbb{R})$. Note that in this case $W$ itself does not have to be flexible. Unless we specify otherwise, all flexible and semi-flexible Lagrangians will have non-empty boundary (although the definitions above make sense in the closed case).

In an attempt to distinguish our uses of `subcritical', we say $L^n$ 
is \textit{Morse} subcritical if it admits a proper Morse function without any critical points of index $n-1$. Note that this makes sense both if $L^n$ is closed or has non-empty boundary. If $L^n$ is closed and Morse subcritical, then $L^n \backslash D^n$ is also Morse subcritical and admits a Morse function with critical points of index 
less than $n-1$.  

\begin{lemma}\label{lem: cotangentLaggood}
If $L$ is Morse subcritical with non-empty boundary, then $(\partial T^*L, \partial L)$ is an asymptotically dynamically convex pair. In particular, $\partial L \subset \partial T^*L$ is asymptotically dynamically convex. 
\end{lemma}
\begin{remark}\label{rem: cotangentLaggood}
Since $T^*L$ retracts to $L$, we have $c_1(T^*L, L) = 0$ and hence $c_1(\partial T^*L, \partial L) = 0$. 
Also, since $L$ is Morse subcritical, $\pi_1(L, \partial L) = 0$. 
By the homotopy long exact sequence of $(T^*L, \partial T^*L, \partial L)$, this implies $\pi_1(\partial T^*L, \partial L) = 0$ as desired. 
\end{remark}

\begin{proof}
The pair $(\partial T^*L, \partial L) = \partial (T^*L, L)$ can be constructed by doing a sequence of Legendrian surgeries to $(S^{2n-1}, S^{n-1})= \partial (B^{2n}, B^n)$; here $S^{n-1}$ is a Legendrian unknot in $S^{2n-1}$ and hence $(S^{2n-1}, S^{n-1})$ an ADC pair. Because $L$ is Morse subcritical, these surgeries have index less than $n-1$ and hence are subcritical. 
 Since $(S^{2n-1}, S^{n-1})$ is an ADC pair, $(\partial T^*L, \partial L)$ is also an ADC pair by repeated use of Proposition \ref{prop: simultaneoussurgerygood}.
 Note that index $2$ case works
 since $c_1(T^*L, L) = 0$
 and all Reeb orbits and chords obtained after doing index $1$ surgery occur in the belt sphere $S^{2n-2}$ and hence are contractible.
\end{proof}

Lemma \ref{lem: cotangentLaggood} provides another prospective on cotangent bundles of closed manifolds. Consider a closed, Morse subcritical manifold $L^n$. Then $L \backslash D^n$ is also Morse subcritical
and so 
$(\partial T^*(L\backslash D^n), \partial (L \backslash D^n))$ is an ADC pair by Lemma \ref{lem: cotangentLaggood}. 
Then  for $n \ge 3$,  $(\partial T^*L, \partial T_x^*L)$ is also an ADC pair by Proposition \ref{prop: surgerybeltsphere}. On the other hand, recall that the proof of Theorem \ref{thm: boundedinfinite} showed that $\partial T^*L$ is an ADC contact structure for \textit{any} closed $L$ with $n\ge 4$ by interpreting the Reeb flow as geodesic flow. Similarly, it is possible to use geodesic flow to show that $\partial T_x^*L \subset \partial T^*L$ an ADC pair for any $L^n, n \ge 4$. 

We now present the Legendrian analog of Corollary \ref{cor: flexsemigood}.
\begin{corollary}\label{cor: semiflexlaggood}
If $L \subset W$ is a Morse subcritical, semi-flexible Lagrangian, then $\partial L \subset \partial W$ is asymptotically dynamically convex. 
If $L \subset W$ is a Morse subcritical, flexible Lagrangian, then $(\partial W, \partial L)$ is an asymptotically dynamically convex pair. 
\end{corollary}
\begin{remark}\label{rem: semiflexlaggood}
Although $\pi_1(\partial T^*L, \partial L) = 0$ by Remark 
\ref{rem: cotangentLaggood}, in general it is not true that $\pi_1(\partial W, \partial L)  = 0$ and hence we will need to make this extra assumption here. 
\end{remark}
\begin{proof}
Since $L$ is Morse subcritical, $\partial L \subset \partial T^*L$ is an ADC Legendrian by Lemma \ref{lem: cotangentLaggood}.
Since $W$ is connected, there is a one-handle in $W''$ between $T^*L$ and $W'$; let $U = T^*L \natural W'$ be the result of attaching this one-handle. One-handles are subcritical and so all Reeb chords of $\partial L \subset \partial U$ are contained in $\partial T^*L$. Therefore $\partial L \subset \partial U$ is also ADC. 

Because $W''$ is a flexible cobordism, $W$ is obtained by successively attaching subcritical or flexible handles to $U$. For example, suppose we need to attach a flexible handle $H^n$ to $\partial U$ in the construction of $W''$. 
The attaching sphere $\partial H^n$ and its loose chart are contained in $ST^*L \sharp \partial W' \subset \partial U$ because $\partial_-W'' = ST^*L \cup \partial W'$.  On the other hand, $\partial L \subset T^*(\partial L) \times \mathbb{R} \subset \partial U$ which is disjoint from $ST^*L \sharp \partial W'$. Therefore the attaching sphere $\partial H^n$ is loose in the complement of $\partial L \subset \partial U$. 
So by Proposition \ref{prop: nonsimultaneousgood}, $\partial L \subset \partial (U \cup H^n)$ is also ADC.
The same is true after successive subcritical and flexible handles are attached to successive contact manifolds. In particular,  $\partial L \subset \partial W$ is ADC.

If $L$ is flexible, then $W$ is a flexible domain. So we already know by Corollary \ref{cor: flexsemigood} that $\partial W$ is ADC.
To show that $(\partial W, \partial L)$ is an ADC pair,  note that $(\partial T^*L, \partial L)$ is an ADC pair by Lemma \ref{lem: cotangentLaggood} and attaching flexible handles whose attaching spheres are loose in the complement of $\partial L$ results in another ADC pair. 
\end{proof}

As noted in Section 5 of \cite{EGL}, the Lagrangian co-core of a flexible Weinstein handle is semi-flexible.  Since co-cores are discs, they are Morse subcritical. Therefore, by Corollary 
\ref{cor: semiflexlaggood}, belt spheres of flexible handles are ADC Legendrians, which gives another proof of Proposition \ref{prop: surgerybeltsphere} in the special case when $\Lambda_-$ is loose. 

Now we prove the results for Legendrians stated in the Introduction. 
\begin{proof}[Proof of Theorem \ref{thm: flexlag}]
Let $W$ be a flexible Weinstein filling of $(Y,\xi)$ and let $L \subset W$ be a flexible Lagrangian filling of $\Lambda \subset (Y, \xi)$. If $\pi_1(L, \Lambda) = 0$ and $n \ge 5$, by Smale's handle-trading trick $L$ is Morse subcritical. So by Corollary \ref{cor: semiflexlaggood}, $(Y, \Lambda, \xi)$ form an ADC pair. 
If $W'$ is another flexible filling of $(Y, \xi)$ and $L' \subset W'$ is another exact Lagrangian filling of $\Lambda$, then 
$WH^+(L', L'; W') \cong WH^+(L, L; W)$ by Proposition \ref{prop: Legindependence}; we use the fact that $\pi_1(Y, \Lambda) = 0$ to avoid discussing $\pi_1$-injective Lagrangian fillings. 

Since $W, W'$ are flexible, $SH(W) = SH(W')= 0$ and therefore we have 
$WH(L, L; W)= WH(L', L'; W') = 0$ as well. By the tautological exact sequence for $WH$, these vanishing results imply  $WH^+_k(L, L; W) \cong H^{n-k-1}(L)$ and $WH^+_k(L', L'; W') \cong H^{n-k-1}(L')$.
Since $WH^+_*(L', L'; W') \cong WH^+_*(L, L; W)$, we have $H^*(L') \cong H^*(L)$. Since $L, L'$ are spin by assumption, this holds over $\mathbb{Z}$. 
\end{proof}

\begin{proof}[Proof of Theorem \ref{thm: leginfinite}]
Let $M^n \in \Omega^n, n \ge 3,$ and $M' := M \backslash D^n$.
Then by the h-principle for flexible Lagrangians \cite{EGL}, $M'$ has a flexible Lagrangian embedding into $(B^{2n}, \omega_{std})$ such that the Legendrian boundary $\partial M' \subset (S^{2n-1}, \xi_{std})$ is formally isotopic to the Legendrian unknot. For $n \ge 3$,  $\pi_1(M) = 0$ implies $\pi_1(M') = 0$ and $\pi_1(M', \partial M') = 0$, which for $n\ge 5$ shows that $M'$ is Morse subcritical by Smale's handle-trading trick.

Suppose $\Lambda \subset (Y,\xi)$ has flexible filling $L \subset W$. If $\pi_1(L, \Lambda) = 0$ and $n \ge 5$, $L$ is Morse subcritical again by Smale's handle-trading trick.  Consider $L_M := L \natural M' \subset W \natural B^{2n} = W$. Because $\partial M'$ is formally isotopic to the Legendrian unknot, $\partial L_M = \partial L \natural \partial M'$ is formally isotopic to $\Lambda = \partial L$. Since $L_M$ is obtained by attaching a 1-handle to flexible $M'$ and $L$, $L_M$ is also flexible, which proves the first and third claims of Theorem \ref{thm: leginfinite} with $\Lambda_M := \partial L_M$.

Because $M'$ and $L$ are Morse subcritical, so is $L_M$ and hence $\partial L_M \subset \partial W$ is an ADC Legendrian by Corollary \ref{cor: semiflexlaggood}. So by Proposition \ref{prop: Legindependence} $WH^+_k(L_M, L_M; W)$, which equals $H^{n-k-1}(L_M)$ since $SH(W) = 0$, 
is a Legendrian invariant. Note that 
$H^*(L_M)\cong H^*(L) \oplus H^*(M')$ in non-zero degrees. Also, by the Mayer-Vietoris exact sequence $H^*(M) \cong H^*(M')$ except in top degree. 
Therefore if $H^*(M)\not \cong H^*(N)$, then $H^*(L_M) \not \cong H^*(L_N)$ and so $\Lambda_M, \Lambda_N$ are not Legendrian isotopic, which proves the second claim of Theorem \ref{thm: leginfinite}.
Finally, note that $M \in \Omega^n$ is stably parallelizable and hence is spin; also, $L$ is spin by our assumption that all Lagrangians are spin. Therefore $L_M$ is also spin and so the claim hold over $\mathbb{Z}$.

To construct infinitely many non-isotopic Legendrians that are formally isotopic and have flexible fillings, it suffices to produce infinitely many $M \in \Omega$ with different 
$H^*(M; \mathbb{Z})$. This was done for $n \ge 5$ in the proof of Theorem \ref{thm: boundedinfinite}.
\end{proof}

\begin{proof}[Proof of Corollary \ref{cor: nearby}]
Let $L \subset T^*M$ be an exact Lagrangian that has zero Maslov class and intersects $T^*M_q$ transversely in a single point for some $q\in M$. 
First, we can Hamiltonian isotope $L^n$ so that $L^n \cap T^*D^n = D^n$ for some small disk $D^n \subset M^n$ containing $q$. 
More precisely, for small enough $D^n$, $L \cap T^*D^n$ is the graph $\Gamma(df)$ of some $f: D^n\rightarrow \mathbb{R}$. Consider the Hamiltonian isotopy $L_t = \Gamma(d(\delta_t \cdot f))$, where $\delta_t(x) = 1- tB(x)$ for some bump function $B: D^n \rightarrow \mathbb{R}$ that equals 1 on a smaller disk $D^n_0 \subset D^n$ and vanishes near $\partial D^n$. This is an isotopy since there is only one sheet of $L$ over $D^n$; furthermore, 
$L_0 = L$ and $L_1 \cap T^*D_0^n = D_0^n$.

Since $L \cap T^*D^n = D^n$, then $L \backslash D^n \subset T^*(M\backslash D^n)$ is an exact Lagrangian with 
Legendrian boundary $\partial(L\backslash D^n) \subset \partial T^*(M\backslash D^n)$ equal to $\partial (M\backslash D^n)$. Therefore $L\backslash D^n$ is an exact Lagrangian filling of $\partial (M\backslash D^n)$. Since $L$ has zero Maslov class, $c_1(T^*M, L) = 0$ and so $c_1(T^*(M\backslash D^n), L\backslash D^n) = 0$ as well.
Since $\pi_1(M) = 0$ and $n \ge 5$, by Smale's handle-trading trick $M$ has no critical points of index $n-1$ and hence $M \backslash D^n$ is Morse subcritical. 
So $\partial (M\backslash D^n) \subset \partial T^*(M\backslash D^n)$ is an ADC Legendrian by Lemma \ref{lem: cotangentLaggood}. Then $H^*(L\backslash D^n; \mathbb{Z}) \cong H^*(M \backslash D^n; \mathbb{Z})$ by Theorem \ref{thm: flexlag}.
Using Mayer-Vietoris again, we have $H^*(L; \mathbb{Z}) \cong H^*(L\backslash D^n; \mathbb{Z})$ and $H^*(M; \mathbb{Z}) \cong H^*(M\backslash D^n; \mathbb{Z})$ except in top degree. Hence  $H^*(L; \mathbb{Z}) \cong H^*(M; \mathbb{Z})$ in all degrees; both $L, M$ are closed oriented manifolds and hence the top degree cohomology equals $\mathbb{Z}$ for both. By Poincar\'e
duality, we also have $H_*(L; \mathbb{Z}) \cong 
H_*(M; \mathbb{Z})$. 

We now prove that $\pi_*: H_*(L; \mathbb{Z}) \rightarrow H_*(M; \mathbb{Z})$ is an isomorphism. Since $L, M$ are oriented manifolds of the same dimension and $L$ intersects $T^*_q M$ transversely in a single point, $\deg(\pi) = \pm 1$ and so $\pi_*: H_*(L; \mathbb{Z}) \rightarrow H_*(M; \mathbb{Z})$ is surjective. Furthermore, 
$H_*(M; \mathbb{Z}), H_*(N; \mathbb{Z})$ are finitely-generated abelian groups since
$L, M$ are manifolds. This implies that $\pi_*$ is an isomorphism; any surjective epimorphism of a Hopfian group is automatically an isomorphism and finitely-generated abelian groups are known to be Hopfian. 

Any homology equivalence between simply-connected CW complexes is a homotopy equivalence. So if $L$ is simply-connected, then $\pi$ is a homotopy equivalence; note that we have already assumed throughout that $M$ is simply-connected. 

The above proof works for $n \ge 5$. For $n = 2,3,4$, 
we use an idea from \cite{kragh_ss}. We repeat the proof with $M' := M \times S^3$ and the closed exact Lagrangian $L' := L \times S^3 \subset T^*M'$ that has zero Maslov class and intersects $T^*_{q\times p} M'$ transversely in a single point for any $p \in S^3$. We can then conclude that
$\pi_*: H_*(L'; \mathbb{Z}) \rightarrow H_*(M'; \mathbb{Z})$ is an isomorphism. Therefore 
$\pi_*: H_*(L; \mathbb{Z}) \rightarrow H_*(M; \mathbb{Z})$ is an isomorphism as well. 
\end{proof}
Note that unlike Theorem \ref{thm: flexlag}, Corollary \ref{cor: nearby}  does  not use any h-principles or involve loose Legendrians. Also, if $L$ intersects $T_q^*M$ in $k$ points, we can assume $L \cap T^*D$ consists of $k$ parallel copies of $D^n$. So $L\backslash D^n$ is an exact Lagrangian filling of a Legendrian link given by copies of the ADC Legendrian $\partial (M\backslash D^n)$.  By studying fillings of such Legendrian links, it might be possible to show that $\deg(\pi) = 1$ and  $\pi_*: H(L; \mathbb{Z}) \rightarrow H(M; \mathbb{Z})$ is an isomorphism without any restrictions.

\subsection{Asymptotically dynamically convex Weinstein domains}\label{sec: ADS_Weinstein}

In this section, we present the class of asymptotically dynamically convex \textit{Weinstein domains}. This class contains  flexible Weinstein domains and all of our previous exotic examples. These domains can potentially be used to create even more exotic symplectic and contact manifolds; see Theorem \ref{prop: ADC_weinstein_boundary_ADC} below.  We also state some properties of these domains, in the process generalizing several of our results.
As usual, all Weinstein domains $W$ have $c_1(W) = 0$.

\begin{definition}\label{def: ADC_weinstein}
A Weinstein domain $(W^{2n}, \lambda, \phi)$ is 
\textit{asymptotically dynamically convex} if there exist regular values $c_1, \cdots, c_k$ of $\phi$ such that $c_1 < \min \phi < c_2 < \cdots < c_{k-1} < \max \phi < c_{k}$ and for all $i = 1, \cdots, k-1$, the Weinstein cobordism $\{c_i \le \phi \le c_{i+1} \}$ has a single critical point $p$ and either the attaching sphere $\Lambda_p$ is either subcritical or $(Y^{c_i}, \Lambda_p, \lambda|_{Y^{c_i}})$ form an asymptotically dynamically convex pair.   
\end{definition}
In other words, $W^{2n}$ is obtained by successively attaching Weinstein handles to $(B^{2n}, \omega_{std})$ such that each handle is subcritical or is attached along a Legendrian sphere 
in a contact manifold forming an ADC pair. Definition \ref{def: ADC_weinstein} is motivated by the definition of flexible Weinstein domains, whose index n handles are attached along loose Legendrians; see Definition \ref{def: weinstein_flexible}. 
 Similarly, we can define ADC Weinstein \textit{cobordisms}.
 Note that if $(Y^{2n-1}, \xi)$ is a level set of $\phi$, then $c_1(Y) = 0$; also, if $\Lambda^{n-1} \subset (Y^{2n-1}, \xi)$ is an attaching Legendrian and $n \ge 3$, then $c_1(Y,\Lambda) =0$ as well. So in this case, it always makes sense to say that $\Lambda^{n-1} \subset (Y^{2n-1}, \xi)$ is ADC. 

We now give some examples. For the following, let $M, N$ be closed, Morse subcritical manifolds.

\begin{proposition}\label{prop: examples_ADC_Weinstein}
 The following Weinstein domains are asymptotically dynamically convex: 
\begin{enumerate}
\item Flexible Weinstein domains 
\item The cotangent bundle $T^*M$ or the plumbing 
$T^*M \sharp_p T^*N$
\item The exotic cotangent bundle $T^*S^n_M$ from Theorem \ref{thm: boundedinfinite}.
\end{enumerate} 
\end{proposition}
\begin{proof}
A flexible Weinstein domain $W$ is constructed by successively attaching subcritical or flexible handles to
$(B^{2n}, \omega_{std})$. As in Corollary \ref{cor: flexsemigood}, we can inductively show that each successive contact boundary is ADC. Furthermore, a loose Legendrian in any ADC contact manifold is an ADC pair, which proves (1). 

For  $T^*M$, there is a single index n handle. This handle is attached along $\partial (M\backslash D^n) \subset \partial T^*(M\backslash D^n)$, which by Lemma \ref{lem: cotangentLaggood} is an ADC pair since $M^n$ is Morse subcritical. This proves the first part of (2). 
For the plumbing case, we note that 
$T^*M \sharp_p T^*N$ is obtained by attaching handles to $\partial T^*M$.
More precisely, Legendrian surgeries are done on $\partial T^*_x M \subset \partial T^*M$, which by Lemma \ref{lem: cotangentLaggood} is an ADC pair since $M$ is Morse subcritical. Since $N$ is also Morse subcritical, all of these surgeries have index less than $n-1$ except for one surgery of index $n$. So by 
Proposition \ref{prop: simultaneoussurgerygood}, the subcritical surgeries result in new Legendrians and contact manifolds that  form ADC pairs. Hence the index n handle is also attached to a Legendrian in an ADC pair, proving the second part of (2).

For (3), recall that $T^*S^n_M$ is constructed by attaching subcritical or flexible handles to $T^*M$, which is an ADC Weinstein domain by (2).
\end{proof}

The following result generalizes Corollary \ref{cor: flexsemigood} for flexible Weinstein domains and the proof of Theorem \ref{thm: boundedinfinite}.

\begin{proposition}\label{prop: ADC_weinstein_boundary_ADC}
If $W^{2n}$ is an asymptotically dynamically convex Weinstein domain and $n \ge 3$, then $\partial W^{2n}$ is an asymptotically dynamically convex contact structure. 
\end{proposition}
\begin{proof}
By Theorems \ref{thm: MLYau} and \ref{thm: MLYauindex2}, the result of subcritical contact surgery on an ADC contact manifold is another ADC contact manifold. By Proposition \ref{prop: surgerybeltsphere}, which requires $n \ge 3$, the result of contact surgery along a Legendrian in a contact manifold forming an ADC pair is another ADC contact manifold. 
\end{proof}

Now we generalize Corollary \ref{cor: nearby} to some other ADC Weinstein domains. 
For the following, let $M, N$ be closed, Morse subcritical manifolds. 
\begin{proposition}\label{prop: ADC_nearby}
Let $L$ be a closed exact Lagrangian with zero Maslov class such that either
\begin{enumerate}
\item $L\subset T^*S^n_M$ and $L$ intersects $T^*_x S^n_M \subset T^*S^n_M$ transversely in a single point for some $x \in M$, or
\item $L\subset T^*M\sharp_p T^*N$ and $L$ intersects $T^*_x M \subset T^*M\sharp_p T^*N$ transversely in a single point for some $x \in M$ and is disjoint from $T^*_y N \subset T^*M\sharp_p T^*N$ for some $y \in N$. 
\end{enumerate}
Then $H^*(L; \mathbb{Z}) \cong H^*(M; \mathbb{Z})$. 
\end{proposition}
\begin{proof}
In the first case, we remove the cotangent fiber $T^*_x M$ from $T^*S^n_M$, which is the co-core of the index n handle we attached to $(B^{2n}, \omega_{std})$ to create $T^*S^n_M$;  the resulting Weinstein domain is therefore $B^{2n}$. 
Since $L$ intersects $T^*_x M$ once, we have $L \backslash D^n \subset B^{2n}$.
Also, $M\backslash D^n \subset B^{2n}$ is a flexible, Morse subcritical Lagrangian with $\partial(M\backslash D^n) = \partial (L\backslash D^n)$. Since $M \backslash D^n$ is flexible and Morse subcritical, $\partial (M\backslash D^n) \subset S^{2n-1}$ is ADC by Corollary \ref{cor: flexsemigood} and so 
$WH^+(L\backslash D^n,L\backslash D^n; B^{2n} ) \cong WH^+(M \backslash D^n, M\backslash D^n; B^{2n})$. 
Also, $SH(B^{2n}) = 0$ implies that $WH(L\backslash D^n,L\backslash D^n; B^{2n} ) = WH(M \backslash D^n, M\backslash D^n; B^{2n}) = 0$. 
By the tautological long exact sequence for $WH$, this shows that $H^*(M \backslash D^n; \mathbb{Z}) \cong H^*(L \backslash D^n; \mathbb{Z})$. Finally $H^*(M; \mathbb{Z}) \cong H^*(L; \mathbb{Z})$ by the Mayer-Vietoris exact sequence. 
 
In the second case, we remove both $T^*_x M$ and $T^*_y N$ from 
$T^*M\sharp_p T^*N$. The result is a subcritical Weinstein domain $W$ obtained by attaching some subcritical handles to $T^*(M \backslash D^n)$. Since $L$ intersects $T^*_x M$ once and  is disjoint from $T^*_y N$, we have $L\backslash D^n \subset W$. Also, $M \backslash D^n \subset W$ is a flexible, Morse subcritical Lagrangian with $\partial (M\backslash D^n) = \partial(L\backslash D^n)$. Since $W$ is subcritical, we can conclude that $SH(W) = 0$ and the rest of the proof is exactly the same as before.
\end{proof}

\subsection{Asymptotically dynamically convex Lagrangians}

We now introduce the class of asymptotically dynamically convex Lagrangians which generalize flexible Lagrangians. It should be possible to use these Lagrangians to construct even more exotic Legendrians; see Proposition \ref{prop: ADC_Legendrian} below.  
\begin{definition}
An exact Lagrangian $L \subset (W, \omega)$ is \textit{asymptotically dynamically convex} if the complement $(W\backslash T^*L, \omega|_{W\backslash T^*L})$ is an asymptotically dynamically convex Weinstein cobordism.
\end{definition}

This definition is motivated by the definition of flexible Lagrangians given in \cite{EGL} and makes sense for both Lagrangians with and without boundary. However if $L$ has boundary,
the situation is a bit more complicated. In this case, the cobordism $W\backslash T^*L$ has corners. In particular, the negative boundary $\partial_- (W\backslash T^*L)$ itself has boundary $ST^*(\partial L) \times \{\pm 1\} \cong \partial (T^*(\partial L) \times \mathbb{R})$ that make up the corners.
The definition of an ADC cobordism should be relative to these corners. 
That is, we not only require that all Reeb chords of an attaching sphere $\Lambda^{n-1} \subset ST^*L$ have positive degree but also that chord words $ab$, where $a$ goes from $\Lambda^{n-1}$ to $\partial L$ and $b$  goes in the opposite direction, have positive degree. Note that these words have well-defined degree even though $\Lambda^{n-1} \cup \partial L$ is disconnected; see Remark \ref{rem: nonsimultaneousgood}.
This should also hold for all successive handles that we attach to successive contact manifolds. Finally, note that $(W, \omega)$ is always a Weinstein domain. In fact, if $L$ has non-empty boundary or $L$ is Morse subcritical, then $(W, \omega)$ is an ADC Weinstein domain. 

Now we give some examples. In the following, $M, N$ are closed, Morse subcritical manifolds.

\begin{proposition}
The following Lagrangians are asymptotically dynamically convex:
\begin{enumerate}
\item Flexible Lagrangians
\item The cotangent fiber $T^*_x M$ of $T^*M$,  $T^*M\sharp_p T^*N$, or $T^*S^n_M$ 
\item The zero-section $M$ of $T^*M$, $T^*M \sharp_p T^*N$, or 
$T^*S^n_M$.
\end{enumerate}
\end{proposition}
\begin{proof}
Flexible Weinstein cobordisms are ADC cobordisms, which proves (1) if $L$ is closed. If $L$ has boundary, all words $a b$ (described above) also need to have positive degree; this is true because we can make the degree arbitrarily high 
of any Reeb chord with at least one endpoint on a Legendrian loose in the complement of $\partial L$; see Proposition \ref{prop: nonsimultaneousgood}.

For (2), note that the co-core of an index n handle attached   to an ADC Weinstein domain $W$ along a Legendrian $\Lambda \subset \partial W$ forming an ADC pair is an ADC Lagrangian.
The proof of this is the same as the proof of Lemma 5.1, \cite{EGL}: the co-core of a flexible handle attached to a flexible Weinstein domain is a flexible Lagrangian. 
The cotangent fiber $T^*_x M \subset T^*M$ is the co-core of a handle attached to the ACD domain $T^*(M\backslash D^n)$ along $\partial (M\backslash D^n) \subset \partial T^*(M\backslash D^n)$, which is an ACD pair. Hence  $T^*_x M \subset T^*M$ is an ADC Lagrangian, which proves the first part of (1). 
For the plumbing case, we note that $T^*M \sharp_p T^*N$ is constructed from $T^*N$ by doing a sequence of Legendrian surgeries on $\partial T^*_y N \subset \partial T^*N$, which is an ADC pair since $N$ is Morse subcritical. Then  $T^*_x M \subset T^*M \sharp_p T^*N$ is the co-core of the index n handle that is attached last. Finally, note that $T^*_x M \subset T^* S^n_M$ is the co-core of the index n handle attached to $B^{2n}$ along $\partial (M\backslash D^n) \subset S^{2n-1}$, which is an ADC pair. 

The zero-section of a cotangent bundle has trivial Weinstein complement and so the first part of (3) holds tautologically. The second part of (3) follows from the proof that $T^*M \sharp_p T^*N$ is an ADC Weinstein domain in Proposition \ref{prop: examples_ADC_Weinstein}. For the third part of (3), recall that $T^*S^n_M$ is constructed by attaching subcritical or flexible handles to $T^*M$.
\end{proof}
The following result generalizes Proposition \ref{cor: semiflexlaggood} for flexible Lagrangians and gives another proof of Proposition \ref{prop: surgerybeltsphere}. It can also be thought of as the relative version of Proposition \ref{prop: ADC_weinstein_boundary_ADC}.

\begin{proposition}\label{prop: ADC_Legendrian}
If $L \subset W$ is an asymptotically dynamically convex Lagrangian that has non-empty boundary and is Morse subcritical, then $\partial L \subset \partial W$ is an asymptotically dynamically convex pair. 
\end{proposition}
\begin{proof}
Since $L$ is Morse subcritical, $\partial L \subset  \partial T^*L$ is an ADC pair by Lemma \ref{lem: cotangentLaggood}. 
Since $L \subset W$ is an ACD Lagrangian, $W$ is formed by successively attaching Weinstein handles to $ST^*L \subset \partial T^*L$ along Legendrians that form ACD pairs with each successive contact manifolds. Then by Proposition \ref{prop: nonsimultaneousgood}, $\partial L$ also forms an ACD pair with each successive contact manifold; in particular, $\partial L \subset \partial W$ is also an ADC pair. 
Note that for Proposition \ref{prop: nonsimultaneousgood} to hold, we need the condition that all words $ab$ from the attaching spheres to $\partial L$ have positive degree;  this is precisely the  condition mentioned in Remark \ref{rem: nonsimultaneousgood} that $\Lambda_- \cup \Lambda^{n-1}$ form an ACD link. Here all words $ab$ have positive degree since the  cobordism $W\backslash T^*L$ is ADC relative to its corner. 
\end{proof}

\section{Open problems}\label{sec: open_problems}
  
We now present some open problems; we focus on the contact case but many similar problems apply to Legendrians. 
\subsection{Topology of fillings}
Suppose that $(Y^{2n-1}, \xi)$ has a flexible filling $W^{2n}$. Corollary \ref{thm: mainstronger} showed that $(Y, \xi)$ remembers the cohomology of fillings with \textit{vanishing} $SH$. However, it is not known whether all fillings have this property for $n \ge 2$; for $n=1$ this is false because $S^1$ has Weinstein fillings with vanishing and non-vanishing $SH$, as shown in Example \ref{ex: surfaces}. \\

\noindent\textbf{Problem 1:} If $(Y, \xi)$ has a flexible filling $W$, do all Liouville fillings $X$ of $(Y, \xi)$ have the same cohomology as $W$? More generally, do all Liouville fillings of an arbitrary asymptotically dynamically convex $(Y, \xi)$ have the same cohomology? 
\\

Work of Barth, Geiges, and Zehmisch \cite{Geiges_subcritical} showed that the first part of Problem 1 has an affirmative answer if $W$ is subcritical and simply-connected. Using this result, we showed in Proposition \ref{thm: reduce_subcritical} that there is also an affirmative answer when the flexible domain admits an almost symplectic embedding into a subcritical domain.

In a slightly different direction, we can also ask whether the contact boundary of a flexible domain remembers more than just the cohomology. 
\\

\noindent\textbf{Problem 2:}
Do all flexible Weinstein fillings of $(Y, \xi)$ have the same almost Weinstein class?
\\

In this case, the Weinstein fillings would be Weinstein deformation equivalent by the uniqueness h-principle  \cite{CE12}
and hence have symplectomorphic completions.
In Corollary \ref{cor: displaceable_flexible_diffeomorphism}, we showed that this is the case if $(Y, \xi)$ has a flexible filling that is smoothly displaceable in its completion.

Theorem \ref{thm: boundedinfinite} shows that for any $M^n \in \Omega^n$, there is a contact structure $\xi_M$ in $(Y^{2n-1}, J)$ or $(Y^{2n-1}, J')$, depending on the parity of $n$. Furthermore, if $H_k(\Lambda M; \mathbb{Z})$, $H_k(\Lambda N; \mathbb{Z})$ are sufficiently different, then $\xi_M, \xi_N$ are non-contactomorphic. This shows that $\xi_M$ remembers part of the homology of $\Lambda M$. 
\\

\noindent\textbf{Problem 3:}
Does $\xi_M$ remember the cohomology of $M$?  
\\

Problem 3 is the contact analog of Problem 4.12 in \cite{EGL}, which asked a similar question for the Weinstein domains $T^*S^n_M$ from Theorem \ref{thm: boundedinfinite}.

\subsection{Exotic contact structures}
By connect summing with exotic contact spheres, we showed in Theorem \ref{thm: inf_contact_flex} that any almost contact manifold with an almost Weinstein filling has infinitely many contact structures. We can consider the following more general scenario. 
Suppose $(Y^{2n-1}, \xi), n \ge 3,$ has a Liouville filling $X$. Let $M_i$ be the infinite sequence of Weinstein domains with $\partial M_i$ in $(S^{2n-1}, J_{std})$ from the proof of Theorem \ref{thm: inf_contact_flex}; as we noted in Remark \ref{rem: ustilovksy}, these examples exist only for $n \ge 3$. Let $X_i := X \natural M_i$ and $(Y, \xi_i) := \partial X_i$. Since $\partial M_i$ in $(S^{2n-1}, J_{std})$, $(Y, \xi_i)$ is almost contactomorphic to $(Y, \xi)$ for all $i$. 
However it is not clear whether different $(Y, \xi_i)$ are contactomorphic or not. 
Either infinitely many of the $(Y, \xi_i)$  are different or only finitely many of them are different, in which case infinitely many of the $(Y, \xi_i)$ coincide; in the latter case, there exists a contact structure with infinitely many fillings $X_i$,  all of which have different homotopy types. 
To summarize: for any $(Y^{2n-1}, \xi), n \ge 3,$ with a Liouville filling $X$, either
\begin{enumerate}
\item $Y$ has infinitely many contact structures  almost contactomorphic to $(Y, \xi)$, or 
\item there exists a contact structure on $Y$ that is almost contactomorphic to $(Y, \xi)$ and has infinitely many Liouville fillings with different homotopy types. 
\end{enumerate}

In Theorem \ref{thm: inf_contact_flex}, we showed that if $(Y^{2n-1}, \xi), n \ge 3,$ has a flexible Weinstein filling, then (1) holds. We can ask whether (1) holds more generally. 
\\

\noindent \textbf{Problem 4:} 
If $(Y^{2n-1}, \xi), n \ge 3,$ has a Liouville filling, does $Y$ have infinitely many different contact structures in the same almost contact class?
\\

As we discussed after the statement of Theorem \ref{thm: inf_contact_flex} in the Introduction, claim (1) does not always hold for $n = 2$; hence the condition $n \ge 3$ is crucial. 
 We also note that (2) holds for certain $(Y^{2n-1}, \xi)$; Ozbagci and Stipsicz  \cite{OS} constructed examples for $n =2$ and Oba \cite{Oba_infinite_fillings} constructed examples for $n \ge 4$ and even. 

By taking the boundary connected sum with the exotic plumbings from Theorem \ref{thm: boundedinfinite}, we can obtain slightly stronger versions of (1), (2). More precisely, suppose $(Y^{2n-1}, \xi)$ has a Liouville filling $X$ such that  $SH(X)$ is finite-dimensional (or at least is finite-dimensional in certain degrees). If $n \ge 4$ is odd, then either 
\begin{enumerate}
\item[(1)\textquotesingle] $Y$ has infinitely many contact structures that are almost contactomorphic to $(Y, \xi)$ and admit almost symplectomorphic Liouville fillings, or
\item[(2)\textquotesingle] there exists a contact structure on $Y$ that is almost contactomorphic to $(Y, \xi)$ and has 
infinitely many almost symplectomorphic fillings with different symplectomorphism types.  
\end{enumerate}
Theorem \ref{thm: boundedinfinite} shows that (1)\textquotesingle \ holds if $(Y^{2n-1},\xi)$ has a flexible Weinstein filling and $n \ge 4$ is odd. There are no known examples  in any dimension where (2)\textquotesingle \ holds, i.e. in all similar examples \cite{Oba_infinite_fillings, OS}, fillings of the same contact manifold are distinguished by algebraic topology.

It is also not known whether all almost contact manifolds have a fillable contact structure. Hence it is unclear when the situations described above apply. On the other hand, \textit{all} almost contact manifolds have at least one contact structure \cite{BEMtwisted}, which is overtwisted and hence non-fillable. Unfortunately it is not clear how to use this contact structure to create infinitely many other contact structures. If we connect sum an overtwisted structure with our exotic contact spheres, the result is still an overtwisted structure in the same formal class and hence by the h-principle for overtwisted structures is contactomorphic to the original one.

\section{Appendix: scaling map lemma}

In this section, we prove Lemma \ref{lem: techkey}, a more refined version of Lemma \ref{lem: maingeo}, 
and Proposition \ref{prop: nice_sh_independent}. These results are needed to prove our main result Theorem \ref{thm: semi-surgery}, i.e. flexible surgery preserves asymptotically dynamically convex contact structures.  
The proofs of Lemma \ref{lem: techkey} and Proposition \ref{prop: nice_sh_independent} involve another somewhat technical but geometric lemma, which can be stated without any reference to the ambient contact manifold or looseness. 
Fix a metric on $\Lambda$. As before, we let  $U^\epsilon = U^{\epsilon}(\Lambda) \subset (J^1(\Lambda), \alpha_{std})$ denote $\{ \|y\| < \epsilon, |z| < \epsilon\}$; here $\alpha_{std} = dz - \sum_{i=1}^n y_i dx_i$, where $x_i$ are any local coordinates on $\Lambda$ and $y_i$ are the dual coordinates on $T^*\Lambda$. For $c>0$, let $s_c$ be the scaling contactomorphism of $J^1(\Lambda)$ 
\begin{equation}
s_c(x,y,z) = (x,cy, cz).
\end{equation} 
Note that we can multiply the $y$ coordinate by $c$ since it is the coordinate on the vector space fibers of $T^*\Lambda$. Also, note that $s_c^*\alpha = c\alpha$. In the following lemma, we construct a contactomorphism supported in $U^1$ that scales $U^{1/2}$ down an arbitrary amount using $s_c$ but does not change the contact form by more than a fixed bounded amount. To do so, we first construct such a contactomorphism  without the correct support and then modify this contactomorphism using contact Hamiltonians. 

\begin{lemma}\label{lem: scalingmap}
For any positive $\epsilon$ and $\delta$ such that 
$\delta \le \epsilon$, there exists a contactomorphism 
$f_{\delta}$ of $(J^1(\Lambda), \alpha)$ supported in $U^\epsilon$ such that  $f_\delta|_{U_{\epsilon/2}}= 
s_\delta$ 
and $f_\delta^* \alpha < 4 \alpha$. 
\end{lemma}
\begin{remark}\
\begin{enumerate}[leftmargin=*]
\item
By cutting off the contact Hamiltonian for $s_c$ with a bump function, it is easy to construct such $f$ with  $f_\delta^* \alpha < C_\delta \alpha$ for some constant $C_\delta$ depending on $\delta$. However it is unclear what effect this cutting off has on $C_\delta$. The point is that there is a bound on $f_\delta^* \alpha$ independent of $\delta$, which we will show by explicit computation.  
\item 
Note that $(f_\delta|_{U_{\epsilon/2}})^*\alpha = \delta \alpha$. However, this condition by itself is not enough for the proof of Lemma \ref{lem: techkey}; we will actually need  $f_\delta|_{U_{\epsilon/2}}= s_\delta$.
\end{enumerate}
\end{remark}
\begin{proof}
It is enough to prove the case $\epsilon = 1$. 
If we have constructed $f_\delta$ for $\epsilon=1$, in general we can take as our contactomorphism $s_\epsilon\circ f_\delta \circ s_\epsilon^{-1}$. This has the desired support, 
$s_\epsilon\circ f_\delta \circ s_\epsilon^{-1}|_{U^{\epsilon/2}} = s_\delta$, and                                                                                                                                                                                                                                                                                                                                                                                                                                                                                                                                                                                                                                                                                                                                                                                                                                                                                                                                                                                                                                                                                                                                                                                                                                                                                                                                                                                                                                                                                                                                                                                                                                                                                                                                                        
$(s_\epsilon \circ f_\delta \circ s_\epsilon^{-1})^*\alpha
= \epsilon (f_\delta \circ s_\epsilon^{-1})^* \alpha
\le 4 \epsilon (s_\epsilon^{-1})^* \alpha
= 4 \alpha$. 

We first explain how to construct a contactomorphism of $J^1(\Lambda)$ with all the desired properties except support in $U^1$. For $t\in [0,1)$, we consider the diffeomorphism $h_t: [0,1] \rightarrow [0, 1]$ that is the identity map near $\{1\}$ and is obtained by smoothing 
the piecewise-linear map whose graph in $\mathbb{R}^2$ connects the points 
\begin{itemize}
\item $(0,0)$ and $(1/2, (1-t)/2)$ with a line of slope $1-t$
\item 
 $(1/2, (1-t)/2)$ and $(1,1)$
with a line of slope $1+t $.
\end{itemize}
 \begin{figure}
     \centering
     \includegraphics[scale=0.25]
     {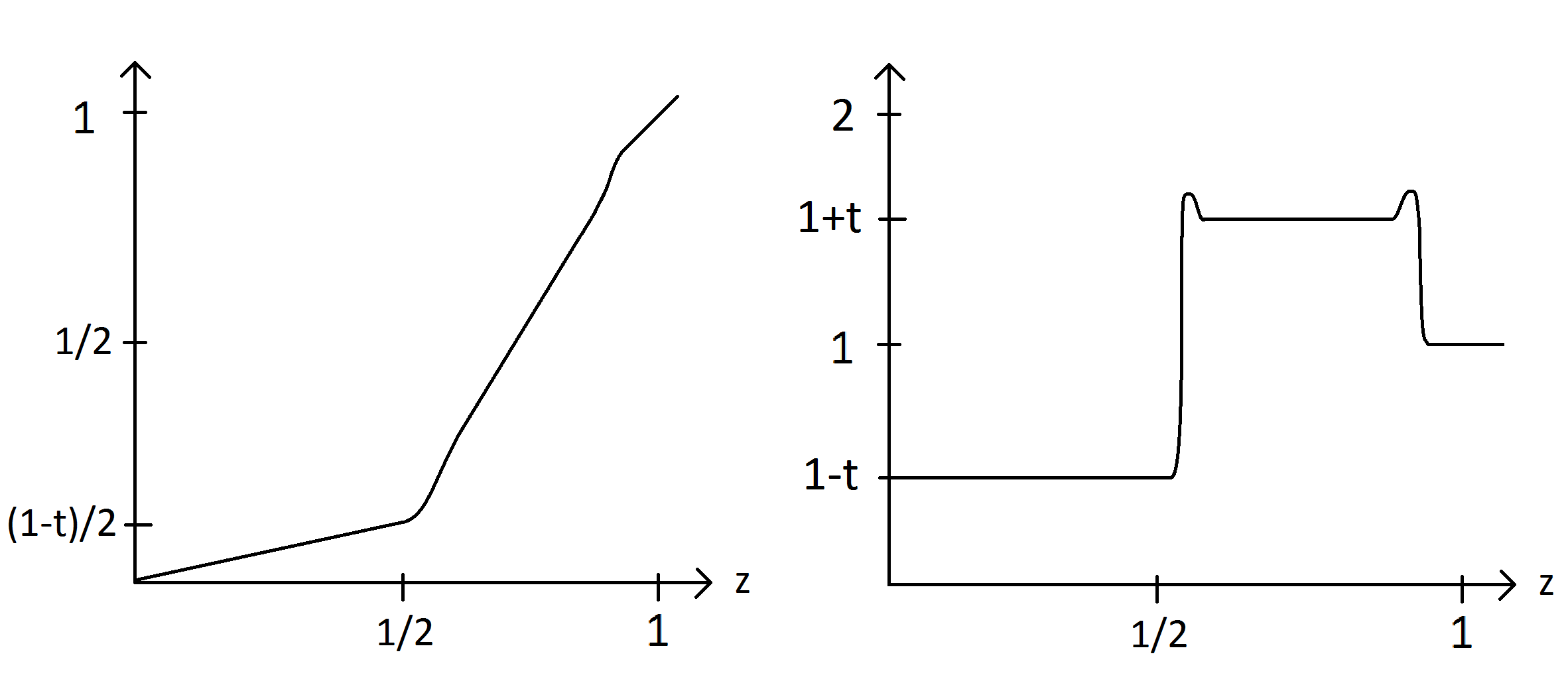}
     \caption{Graph of $h_t$ and $\frac{dh_t}{dz}$}
     \label{fig: graphft}
 \end{figure}
See Figure \ref{fig: graphft}.
Note that $h_0 = Id$ and as $t$ approaches $1$, the  graph of $h_t$ approaches a line of slope $0$ and a line of slope $2$. 
Note that $1-t, 1+t < 2$ for $t < 1$. In particular, we can assume that $h_t(z) = (1-t) z$  for $z\in [0,1/2]$ and $\frac{d h_t}{d z} \le 2$ everywhere;
also $\frac{d h_t}{d z} > 0$ since $h_t$ is a diffeomorphism. 
Extend $h_t$ to a diffeomorphism $h_t: \mathbb{R} \rightarrow \mathbb{R}$ satisfying
$h_t(-z) = -h_t(-z)$ and $h_t(z) = z$ for $|z|> 1$.
Let $\phi_t: J^1(\Lambda)\rightarrow J^1(\Lambda)$ be the contactomorphism defined by
$\phi_t(x, y, z) = (x, \frac{d h_t}{dz}y, h_t(z))$;
because $\frac{d h_t}{dz} > 0$, this is a diffeomorphism. In particular, we have $\phi_t|_{U_{1/2}}=  
 s_{1-t}$. Also, note that $\phi_t^* \alpha = \frac{d h_t}{d z} \alpha$ and so $\phi_t^* \alpha \le 2\alpha$ for all $t \in [0,1)$.  Then 
 $\phi_{1-\delta}$ satisfies all conditions except it is not supported in $U^{1}$; the issue is that    if $ \frac{d h_t}{d z}(z) \ne 1$, then $(x,  \frac{d h_t}{d z}y, h_t(z)) \ne (x,y,z)$ for \textit{all} $y \in \mathbb{R}^n \backslash \{0\}$.

We now explain to how to modify $\phi_{t}$ to get the correct support while preserving the other properties.  We first describe $h_t$ more carefully. In particular, we will use the following lemma about smooth functions, which we prove at the end of this section.

\begin{lemma}\label{lem: techlem2}
There exists a smooth family of diffeomorphisms $h_t= h(t, \cdot): \mathbb{R} \rightarrow \mathbb{R}$ 
defined for $t\in [0,1)$ such that 
$h_0(z) = z$, 
$h_t(z) = (1-t)z$ for $z\in [-\frac{1}{2}, \frac{1}{2}]$,
$h_t(z) = z$ for $|z| > 1$
and 
$$
\max_{z\in \mathbb{R}}
\left(\frac{\partial h_t}{\partial z} \right)^{-1}
\frac{\partial^2 h_t} {\partial t \partial z}
\le \frac{5}{4}
$$
for all $t\in [0,1)$.
\end{lemma}
\begin{remark}
In fact, we can replace $\frac{5}{4}$ by any number bigger than $1$. 
\end{remark}
Our old $h_t$ satisfies all conditions except for possibly the last one and the $h_t$ constructed in this lemma look very much like the $h_t$ we constructed earlier. 

Again we consider the contact isotopy $\phi_t$ of $J^1(\Lambda)$ with
$\phi_t(x,y,z) = $ $(x, \frac{\partial h_t }{\partial z}y, h_t(z))$. The vector field $X_t = (\frac{d\phi_t}{dt}) \circ \phi_t^{-1}$ is a time-dependent contact vector field whose flow is defined for all $t\in [0,1)$; for example, see Section 2 of \cite{Topological_contact_dynamics}. 
Therefore, there exists a time-dependent contact Hamiltonian $H_t: J^1(\Lambda) \rightarrow \mathbb{R}$ such that the corresponding contact vector field is precisely $X_t$; we do not need 
an explicit formula for $H_t$ but note that $H_t = \alpha(X_t)$. 
Because $\phi_t$ is the identity map for $|z|>1$, $X_t$ and hence $H_t$ vanish for $|z|>1$. 
Let $b: \mathbb{R}^+ \rightarrow [0,1]$ be a smooth non-increasing function supported on $[0, 1)$ such that $b=1$ on $[0, 3/4]$. 
Let $G_t: J^1(\Lambda) \rightarrow \mathbb{R}$ be defined by $G_t(x,y,z) = b(\|y\|^2)H_t(x,y,z)$, where we use the metric on $\Lambda$ to define 
$\|y\|^2$. Note that $G_t$ is supported in $U^1$. Let $\psi_t$ be the contact isotopy obtained by integrating the contact vector field $Y_t$ of $G_t$. Since $G_t$ is supported in $U^1$, so is $\psi_t$ and therefore $\psi_t$ is defined for all $t\in [0,1)$. 
We will show that $\psi_{1-\delta}$ is the desired contactomorphism. Since $G_t = H_t$ in $U^{1/2}$, we have $Y_t = X_t$ in $U^{1/2}$. Furthermore, $\phi_t(U^{1/2}) \subseteq U^{1/2}$ for all $t$ and so $\psi_t|_{U^{1/2}} =
\phi_t|_{U^{1/2}} = s_{1-t}$. Therefore $\psi_{1-\delta}|_{U^{1/2}} = s_\delta$. It remains to show that $\psi_{1-\delta}^*\alpha < 4\alpha$; we will show that $\psi_t^*\alpha < 4 \alpha$ for all $t\in [0,1)$. 

We now recall some properties of contact Hamiltonians. 
Suppose $\phi_t$ is a contact isotopy induced by an arbitrary contact Hamiltonian $H_t$. Then $\phi_t^*\alpha = \lambda_t \alpha$ for some $\lambda_t$. We have the following explicit formula for $\lambda_t$
\begin{equation}\label{eqn: lambda_formula}
\lambda_t = e^{\int_0^t \mu_s ds },
\end{equation}
where
$
\mu_t= dH_t(R_\alpha)\circ \phi_t
$. For completeness, we review the proof of this formula, which is given in \cite{Gbook}, p. 63. Note that 
$$
\frac{d\lambda_t}{dt}\alpha = \frac{d\phi_t^*\alpha}{dt} = \phi_t^*L_{X_t} \alpha = \phi_t^* (dH_t(R_\alpha) \alpha) = 
\mu_t \lambda_t \alpha
$$
where the second-to-last equality follows from the definition of the contact Hamiltonian. So we get 
$\frac{d\lambda_t}{dt} = \mu_t \lambda_t$, which proves Equation \ref{eqn: lambda_formula}. In particular, to bound $\lambda_t$ it is sufficient to bound $dH_t(R_\alpha)$. 
Solving for
$dH_t(R_\alpha)\circ \phi_t$ in terms of $\lambda_t$, we get
$$
 dH_t(R_\alpha)\circ \phi_t =  \mu_t  = \lambda_t^{-1} \frac{d \lambda_t}{dt}.
$$
In our situation, we have $\lambda_t = \frac{d h_t(z)}{d z}$ and so 
$$
 dH_t(R_\alpha)\circ \phi_t =  \left(\frac{\partial h_t}{\partial z} \right)^{-1}
\frac{\partial^2 h_t} {\partial t \partial z}.
$$
By the last condition of Lemma 
\ref{lem: techlem2}, 
$\underset{J^1(\Lambda)}{\max} \ dH_t(R_\alpha) 
= \underset{J^1(\Lambda)}{\max} \ dH_t(R_\alpha) \circ \phi_t 
\le 5/4$ for all $t\in [0,1)$. 

We are interested in bounding the function $\gamma_t$ defined by $\psi_t^*\alpha = \gamma_t\alpha$. Recall that $\psi_t$ is generated by $G_t = b(\|y\|^2)H_t$. 
Since $b$ is independent of $z$ and 
$R_\alpha = \partial_z$, we have that $dG_t(R_\alpha) = b(\|y\|^2) dH_t(R_\alpha)$; the fact that we can just factor out $b$ and not consider derivatives of $b$ is key. Let $\nu_t = dG_t(R_\alpha)\circ \psi_t$. 
We have $\underset{J^1(\Lambda)}{\max} \ dH_t(R_\alpha) \ge 0$ because $H_t = 0$ for $|z| > 1$ for all $t$. Also,  $0 \le b(\|y\|^2) \le 1$. Therefore
$$
\max_{J^1(\Lambda)}\nu_t = \max_{J^1(\Lambda)}{dG_t(R_\alpha)} 
\le 
\max_{J^1(\Lambda)}{dH_t(R_\alpha)} 
=\max_{J^1(\Lambda)}\mu_t
\le \frac{5}{4}.
$$
Therefore by Equation \ref{eqn: lambda_formula}
$$
\gamma_t = e^{\int_0^t \nu_s ds } \le 
e^{\int_0^t \frac{5}{4} ds} = e^{\frac{5}{4}t} < e^\frac{5}{4} < 4
$$
for all $t\in [0,1)$ as desired. 
\end{proof}
\begin{remark}
We knew a priori that 
$\int_0^t \mu_s ds = \ln \lambda_t \le \ln 2$ for all points in $J^1(\Lambda)$ but this is not enough to conclude that $\int_0^1 \max_{J^1(\Lambda)}\mu_t dt < \infty$, as required for the last part of the proof of Lemma \ref{lem: scalingmap}. This is why we needed to us Lemma \ref{lem: techlem2} to bound  $\max_{J^1(\Lambda)}\mu_t$.
\end{remark}

\begin{proof}[Proof of Lemma \ref{lem: techlem2}]
 \begin{figure}
     \centering
     \includegraphics[scale=0.27]{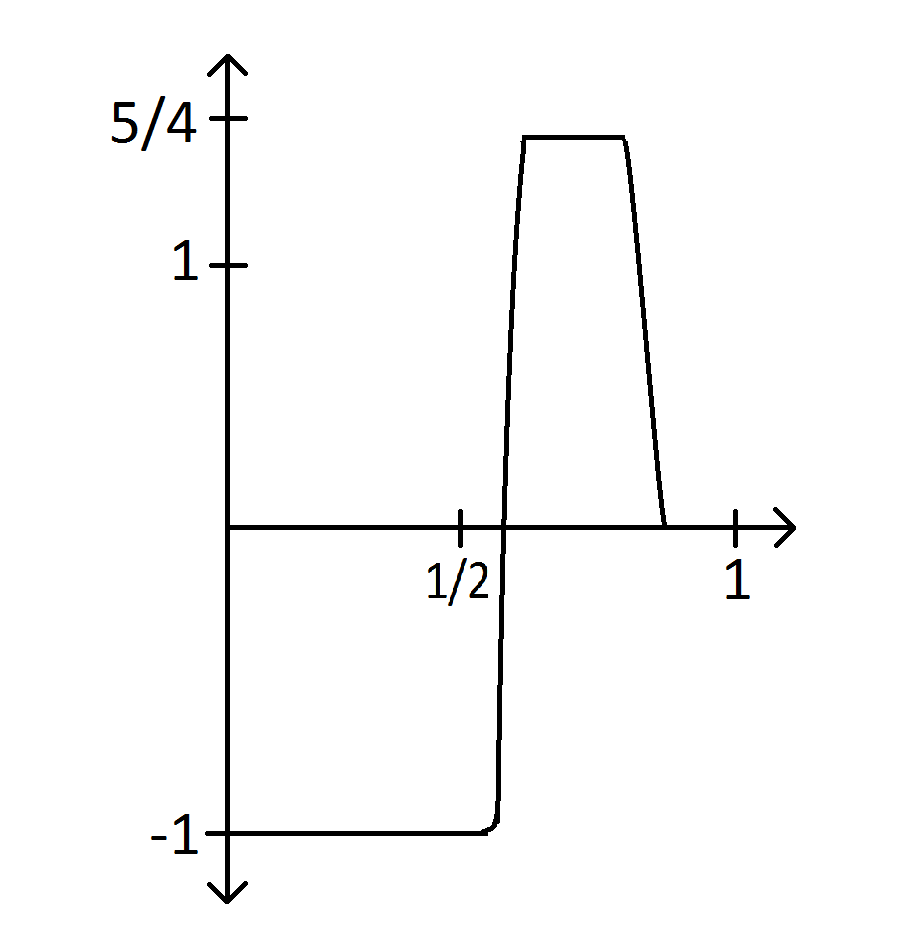}
     \caption{Graph of $g$ }
     \label{fig: graphgt}
 \end{figure}
Now we explain how to construct the desired $h_t$.
Let $g:[0,1] \rightarrow [-1, \frac{5}{4}]$ be a smooth function such that $g$ equals $-1$ on $[0,1/2]$, vanishes near $1$, and $\int_0^1 g(x) dx = 0$; see Figure \ref{fig: graphgt}.
Define $h_t: [0,1] \rightarrow [0,1]$ by $h_t(z) = \int_0^z (tg(s)+1)ds$.
  Since $\frac{\partial h_t}{\partial z} = tg(z)+1 > 0$ for $t < 1$, $h_t$ is a smooth family of increasing (in $z$) functions. Also, $h_0(z) =z, h_t(z) = (1-t)z$ for $z\in [0,1/2],$ and $h_t(z) = z$ near $1$. We need to bound 
$$
\left(\frac{\partial h_t}{\partial z}\right)^{-1}
\frac{\partial^2 h_t} {\partial t \partial z} = 
(tg(z)+1)^{-1} \frac{\partial (tg(z)+1)}{\partial t}
 = \frac{g(z)}{tg(x)+1}
$$
in $z$ for fixed $t<1$. If $g(z)\le 0$, then because $tg(z)+1 >0$ for $t<1$, we have that $g(z)/(tg(z)+1) \le 0$. If $g(z)>0$, then $tg(z)+1 \ge 1$ and so $g(z)/(tg(z)+1)) \le g(z) \le \frac{5}{4}$. So $g(z)/(tg(z)+1) \le \frac{5}{4}$ for all $z\in [0,1]$ and $ t < 1$, as desired. Finally, we extend $h_t$ to a diffeomorphism $h_t: \mathbb{R} \rightarrow \mathbb{R}$ such that $h_t(z)= -h_t(-z)$ and $h_t(z) = z$ for $|z|>1$. Since $\frac{\partial h_t}{\partial z}(z)= \frac{\partial h_t}{ \partial z}(-z)$, this bound holds for all $z\in \mathbb{R}$. 
\end{proof}

\begin{remark}\label{rem: appendix}
Lemma \ref{lem: scalingmap} also holds for 
$(J^1(\Lambda) \times E, \alpha)$. Here $E$ is a trivial conformal symplectic bundle over $\Lambda$ and $\alpha = dz - y dx - vdu$, where $(x,y,z, u, v)$ are local coordinates on $J^1(\Lambda) \times E$. In this case, we take 
$s_c(x,y,z,u,v) := (x, cy, cz, \sqrt{c}u, \sqrt{c}v)$ 
and $U^\epsilon = \{\|y\|, |z|, |u|, |v| < \epsilon\}$  so that $s_c^*\alpha = c\alpha$ and $s_c(U^\epsilon) \subset U^{\epsilon \sqrt{c}}$. 
We modify the proof a bit by using  $\phi_t(x,y,z,u,v) = (x,\frac{\partial h_t}{\partial z}y,h_t(z), \sqrt{\frac{\partial h_t}{\partial z}}u, \sqrt{\frac{\partial h_t}{\partial z}} v)$, which is smooth since $\frac{\partial h_t}{\partial z}$ is never zero, and $G_t = b(\|y\|^2) b(|u|^2)b(|v|^2)H_t$. 
Note that any isotropic submanifold with trivial symplectic conormal bundle has a neighborhood that is strictly contactomorphism to a neighborhood of $\Lambda$ in 
$(J^1(\Lambda) \times E, \alpha)$. 
When $\Lambda$ is a sphere, this is precisely the neighborhood of $\Lambda$ used in subcritical surgery. 
\end{remark}

Using Lemma \ref{lem: scalingmap} together with Lemma \ref{lem: maingeo}, we finally prove Lemma \ref{lem: techkey}.

\begin{proof}[Proof of Lemma \ref{lem: techkey}]

By Lemma \ref{lem: maingeo}, for any $D> 0$,
$\Lambda_0$ is formally Legendrian isotopic in $U^{\epsilon/4}(\Lambda_0, Y, \alpha)$ to a loose Legendrian $\Lambda'$ such that all elements of $\mathcal{P}^{< D}(\Lambda', Y, \alpha)$ have positive degree. 
Since $\Lambda$ is formally Legendrian isotopic to $\Lambda_0$ in $U^{\epsilon/4}(\Lambda_0, Y,\alpha)$ by assumption, $\Lambda$ is also formally Legendrian isotopic to $\Lambda'$ in $U^{\epsilon/4}(\Lambda_0, Y, \alpha)$. Since $\Lambda$ and $\Lambda'$ are loose in 
$U^{\epsilon/4}(\Lambda_0, Y, \alpha)$ and are formally Legendrian isotopic, by the relative version of Murphy's h-principle there exists a genuine Legendrian isotopy between $\Lambda$ and $\Lambda'$ contained in $U^{\epsilon/4}(\Lambda_0, Y, \alpha)$. We can extend this to an ambient contact isotopy $\phi_t$ supported in a small neighborhood of the Legendrian isotopy; see the proof of Theorem 2.6.2 of \cite{Gbook}. Hence we can assume $\phi_t$ is supported in 
$U^{\epsilon/2}(\Lambda_0, Y, \alpha)$. 

Note that we do not have any control over $\phi = \phi_1$ and indeed $\phi^*\alpha$ might be much larger than $\alpha$. We now explain how to modify $\phi$ to get a new contactomorphism $h$ which does not have this problem. 
Suppose that $\phi^* \alpha|_{U^{\epsilon/2}(\Lambda_0)} < 2^k \alpha$ for some possibly very large $k$. By taking larger $k$ if necessary, we can  assume that $\epsilon/2^k < \delta$.
Note that $(\phi^* \alpha)|_{U^{\epsilon}(\Lambda_0) 
\backslash {U^{\epsilon/2}(\Lambda_0)}}  = \alpha$ because $\phi$ is supported in $U^{\epsilon/2}(\Lambda_0, \alpha)$.
 By Lemma \ref{lem: scalingmap}, there exists a contactomorphism $f_{k}$ supported in $U^{\epsilon}(\Lambda_0, \alpha)$ 
such that $f_{k}^*\alpha < 4\alpha$ and $f_{k}|_{U^{\epsilon/2}(\Lambda_0)} = s_{1/2^k}$. 
We will show that the contactomorphism $h = f_{k} \circ \phi$ has the desired properties. 

First note that $h$ is supported in $U^\epsilon(\Lambda_0)$ because this is true for $f_{k}$ and $\phi$. 
Now we give bounds on $h^* \alpha$ in  $U^{\epsilon/2}(\Lambda_0, \alpha)$ and $Y\backslash U^{\epsilon/2}(\Lambda_0,  \alpha)$.
On $U^{\epsilon/2}(\Lambda_0, \alpha)$, 
we have $f_k^* \alpha|_{U^{\epsilon/2}(\Lambda_0)} = \frac{1}{2^k}\alpha$ and $\phi^*\alpha|_{U^{\epsilon/2}(\Lambda_0)} < 2^k \alpha$. Since $\phi(U^{\epsilon/2}(\Lambda_0)) = U^{\epsilon/2}(\Lambda_0)$, 
$$
h^*\alpha|_{U^{\epsilon/2}(\Lambda_0)}
= (f_k \circ \phi)^*\alpha|_{U^{\epsilon/2}(\Lambda_0)}
=
\phi^*(f_k^* \alpha|_{U^{\epsilon/2}(\Lambda_0)})|_{U^{\epsilon/2}(\Lambda_0)}
= \frac{1}{2^k}\phi^* \alpha|_{U^{\epsilon/2}(\Lambda_0)} 
< \alpha|_{U^{\epsilon/2}(\Lambda_0)}.
$$
On $Y\backslash U^{\epsilon/2}(\Lambda_0)$, we
have $f_k^* \alpha < 4 \alpha$ and  $\phi^*\alpha|_{Y\backslash U^{\epsilon/2}} = \alpha$ and so
$$
h^*\alpha|_{Y \backslash {U^{\epsilon/2}(\Lambda_0)}} 
 = (f_k\circ \phi)^*\alpha|_{Y\backslash {U^{\epsilon/2}(\Lambda_0)}} 
 < 4\alpha|_{Y\backslash {U^{\epsilon/2}(\Lambda_0)}}.
$$
Therefore, we have $h^*\alpha < 4 \alpha$ on all of $Y$ as desired. 

We now show that all elements of $\mathcal{P}^{< D}(h(\Lambda), Y, \alpha)$ have positive degree. Note that there is a degree-preserving bijection between 
$\mathcal{P}^{< D}(h(\Lambda), \alpha)$ and $\mathcal{P}^{< D}(\Lambda', \alpha)$. 
To see this, recall from Lemma \ref{lem: maingeo} that $\Lambda'$ is just $\Lambda_0$ with many zig-zags in $U^{\epsilon/2}(\Lambda_0)$. These zig-zags are so small that there is a bijection between  $\mathcal{P}^{< D}(\Lambda', Y, \alpha)$ and the union of global Reeb chords $\mathcal{P}^{< D}(\Lambda_0, Y, \alpha)$ of $\Lambda_0$ and local Reeb chords contained in $U^{\epsilon/2}(\Lambda_0)$. 
On the other hand, because $\Lambda' \subset U^{\epsilon/2}(\Lambda_0, \alpha)$, and $f_k|_{U^{\epsilon/2}(\Lambda_0)} = s_{1/2^k}$, we have that 
$h(\Lambda) = f_k(\Lambda') = s_{1/2^k}(\Lambda')$. 
By scaling $\Lambda'$ to $s_{1/2^k}(\Lambda')$, neither the set of global nor local chords change and so we have a grading-preserving bijection between $\mathcal{P}^{< D}(\Lambda', Y, \alpha)$ and 
$\mathcal{P}^{< D}(h(\Lambda), Y, \alpha)$; see the proof of Lemma \ref{lem: maingeo}. 
Since all elements of 
$\mathcal{P}^{< D}(\Lambda', Y, \alpha)$ have positive degree by construction, so do all elements of $\mathcal{P}^{< D}(h(\Lambda), Y, \alpha)$.

Since $\Lambda$ is loose in 
$U^{\epsilon/4}(\Lambda_0, \alpha)$, then by Remark \ref{rem: looseleg}, $h(\Lambda)$ is loose in $h(U^{\epsilon/4}(\Lambda_0, \alpha))$. 
Since $\epsilon/2^k < \delta$, we have $h(U^{\epsilon/4}(\Lambda_0, \alpha)) = s_{1/2^k}(U^{\epsilon/2}(\Lambda_0, \alpha)) = U^{\epsilon/2^{k+1}}(\Lambda_0, \alpha) \subset U^\delta(\Lambda_0, \alpha)$ and so $h(\Lambda)$ is loose in 
$U^\delta(\Lambda_0, \alpha)$. 
Also, since $\Lambda'$ is formally isotopic to $\Lambda_0$ in $U^{\epsilon/4}(\Lambda_0, \alpha)$, $h(\Lambda) = s_{1/2^k}(\Lambda')$ is formally isotopic to $s_{1/2^k}(\Lambda_0) = \Lambda_0$ in $s_{1/2^k}(U^{\epsilon/4}(\Lambda_0, \alpha)) \subset U^\delta(\Lambda, \alpha)$.
\end{proof}

We also need to prove Proposition \ref{prop: nice_to_supernice}, which produces contact forms for nice contact structures that are standard in a neighborhood of an isotropic submanifold. 

\begin{proof}[Proof of Proposition \ref{prop: nice_to_supernice}]

Let $\alpha_1 > \alpha_2$ be contact forms for $(Y, \xi)$ and let $\Lambda \subset (Y, \xi)$ be an isotropic submanifold with neighborhood $U^\epsilon(\Lambda, \alpha_1)$. 
Also, let $\epsilon_1 = \epsilon$
and let $\epsilon_2$ be sufficiently small so that 
$U^{\epsilon_2}(\Lambda, \alpha_2) \subset U^{\epsilon_1}(\Lambda, \alpha_1)$ and 
$U^{\epsilon_2/2}(\Lambda, \alpha_2) \subset U^{\epsilon_1/2}(\Lambda, \alpha_1)$.

We first show that there is a  contactomorphism $\phi$ of $(Y, \xi)$ supported in 
$U^{\epsilon_2/2}(\Lambda, \alpha_2)$ such that $\phi|_\Lambda = Id$ and  $\phi^*\alpha_2|_{U^{\delta_1}(\Lambda, \alpha_1)}=
 \alpha_1|_{U^{\delta_1}(\Lambda, \alpha_1)}$ for some sufficiently small $\delta_1$ such that 
$U^{\delta_1}(\Lambda, \alpha_1)\subset 
U^{\epsilon_2/2}(\Lambda, \alpha_2) \subset U^{\epsilon_1/2}(\Lambda, \alpha_1)$.
This follows essentially from the proof of strict Weinstein neighorhood theorem for isotropic submanifolds. The proof involves constructing a contact vector field in a neighborhood of $\Lambda$ whose time-1 flow $\phi$ is defined in a small neighorhood of $\Lambda$  and satisfies $\phi^*\alpha_2 = \alpha_1$ and $\phi|_\Lambda = Id$.  
We can ensure that this contactomorphism is supported in $U^{\epsilon/2}(\Lambda, \alpha_2)$ by proceeding as in the proof of Lemma \ref{lem: scalingmap} and using a bump function to cut off the contact Hamiltonian corresponding to the contact vector field used to construct $\phi$. Since the bump function equals $1$ near $\Lambda$, we can ensure that  $\phi^*\alpha_2= 
 \alpha_1$ near $\Lambda$ and   $\phi|_{\Lambda} = Id$ still hold. 

Note that $\phi^*\alpha_2$ might be much larger than $\alpha_1$ in 
$U^{\epsilon_2/2}(\Lambda, \alpha_2)\backslash 
U^{\delta_1}(\Lambda, \alpha_1)$; suppose that $\phi^*\alpha_2 < 2^k \alpha_1$ for some possibly large $k$. Now we proceed as in the proof of Proposition \ref{prop: techsurgery}
and show that $h:= f_{1/2^k} \circ \phi$ has the desired properties. Here 
$f_{1/2^k}$ is the contactomorphism from Lemma \ref{lem: scalingmap} defined for $U^{\epsilon_2}(\Lambda, \alpha_2) \subset 
U^{\epsilon_1}(\Lambda, \alpha_1)$;  when $\Lambda$ is a subcritical isotropic submanifold, we use Remark \ref{rem: appendix}.

Since $\phi$ and $ f_{1/2^k}$ are both supported in $U^{\epsilon_1}(\Lambda, \alpha_1)$, so is $h$. Since $\phi|_\Lambda =  f_{1/2^k}|_\Lambda = Id$, then $h|_\Lambda = Id$ as well. Also, 
$f_{1/2^k}|_{U^{\epsilon_2/2}(\Lambda, \alpha_2)} = s_{1/2^k}$ by construction, where $s_{1/2^k}^*\alpha_2 = \frac{1}{2^k}\alpha_2$.
Since $\phi$ is supported in 
$U^{\epsilon_2/2}(\Lambda, \alpha_2)$ and 
$U^{\delta_1}(\Lambda, \alpha_1)\subset U^{\epsilon_2/2}(\Lambda, \alpha_2)$, we have $\phi(U^{\delta_1}(\Lambda, \alpha_1))\subset U^{\epsilon_2/2}(\Lambda, \alpha_2)$.
Hence
 $h^*\alpha_2|_{U^{\delta_1}(\Lambda, \alpha_1)} = \phi^*(s_{1/2^k}^*\alpha_2|_{ U^{\epsilon_2/2}(\Lambda, \alpha_2)})|_{U^{\delta_1}(\Lambda, \alpha_1)} 
 = \frac{1}{2^k}\phi^*\alpha_2|_{U^{\delta_1}(\Lambda, \alpha_1)}  = \frac{1}{2^k}\alpha_1|_{U^{\delta_1}(\Lambda, \alpha_1)}$; so we can take $c = \frac{1}{2^k}$ in the statement of the proposition. Next, we need to show $h^*\alpha_2 < 4\alpha_1$. In $U^{\epsilon_2/2}(\Lambda, \alpha_2)$, we have 
 $h^*\alpha_2|_{U^{\epsilon/2}(\Lambda, \alpha_2)} = \phi^*(s_{1/2^k}^*\alpha_2|_{U^{ \epsilon_2/2}(\Lambda, \alpha_2)})|_{U^{\epsilon_2/2}(\Lambda, \alpha_2)} 
  = \frac{1}{2^k}\phi^*\alpha_2|_{U^{\epsilon_2/2}(\Lambda, \alpha_2)}  <  \alpha_1|_{U^{\epsilon_2/2}(\Lambda, \alpha_2)}$ since $\phi^* \alpha_2 < 2^k \alpha_1$. 
  In $U^{\epsilon_1}(\Lambda, \alpha_1)\backslash 
  U^{\epsilon_2/2}(\Lambda, \alpha_2)$,  
  we have $\phi = Id$ and therefore
  $h^*\alpha_2|_{U^{\epsilon_1}(\Lambda, \alpha_1)\backslash 
    U^{\epsilon_2/2}(\Lambda, \alpha_2)} $ $
    = f_{1/2^k}^*\alpha_2|_{U^{\epsilon_1}(\Lambda, \alpha_1)\backslash 
        U^{\epsilon_2/2}(\Lambda, \alpha_2)} < 4 \alpha_2$. This implies that 
        $h^*\alpha_2|_{U^{\epsilon_1}(\Lambda, \alpha_1)\backslash 
            U^{\epsilon_2/2}(\Lambda, \alpha_2)} < 4\alpha_1|_{U^{\epsilon_1}(\Lambda, \alpha_1)\backslash 
            U^{\epsilon_2/2}(\Lambda, \alpha_2)}$ 
            since $\alpha_1 > \alpha_2$. 
 Finally, we need to show that for any sufficiently small $\delta_2$ (and fixed $\delta_1$), 
 $h(U^{\delta_1}(\Lambda, \alpha_1) ) \subset U^{\delta_2}(\Lambda, \alpha_2)$. As noted before,  $\phi(U^{\delta_1}(\Lambda, \alpha_1))\subset U^{\epsilon_2/2}(\Lambda, \alpha_2)$ and 
 $f_{1/2^k}|_{U^{\epsilon_2/2}(\Lambda, \alpha_2)} = s_{1/2^k}$. So 
 $h(U^{\delta_1}(\Lambda, \alpha_1)) \subset U^{\epsilon_2/2^{k+1}}(\Lambda, \alpha_2)$. By taking $k$ large enough so that $\delta_2 < \epsilon_2/2^{k+1}$, we have $h(U^{\delta_1}(\Lambda, \alpha_1)) \subset U^{\delta_2}(\Lambda, \alpha_2)$ as desired.  
\end{proof}

\bibliographystyle{abbrv}
\bibliography{sources}

\end{document}